\documentclass[12pt,CJK]{article}
\usepackage{amssymb}
\usepackage{stmaryrd} 
\usepackage{amsfonts,amsmath,amssymb,amsthm,mathrsfs}
\usepackage{latexsym}
\usepackage{graphicx}
\usepackage{indentfirst}
\usepackage{color}
\usepackage{fancyhdr}
\usepackage[colorlinks,linktocpage,linkcolor=blue]{hyperref}

\usepackage{tikz}               
\usetikzlibrary{positioning, fit, arrows.meta,calc}
\usetikzlibrary{decorations.markings} 

\usepackage{tabularx} 
\usepackage{booktabs} 
\usepackage{longtable} 

\usepackage{authblk} 

\theoremstyle{plain}                       
\newtheorem{lemma}{Lemma}[section]
\newtheorem{theorem}[lemma]{Theorem}
\newtheorem{corollary}[lemma]{Corollary}
\newtheorem{remark}[lemma]{Remark}
\newtheorem{definition}[lemma]{Definition}
\newtheorem{proposition}[lemma]{Proposition}

\numberwithin{equation}{section}

\theoremstyle{remark}

\def\Xint#1{\mathchoice
  {\XXint\displaystyle\textstyle{#1}}%
  {\XXint\textstyle\scriptstyle{#1}}%
  {\XXint\scriptstyle\scriptscriptstyle{#1}}%
  {\XXint\scriptscriptstyle\scriptscriptstyle{#1}}%
  \!\int}
\def\XXint#1#2#3{{\setbox0=\hbox{$#1{#2#3}{\int}$}
  \vcenter{\hbox{$#2#3$}}\kern-.5\wd0}}

\def\dashint{\Xint-}

\DeclareMathOperator*{\esssup}{ess\,sup}

\author[a]{
Jun Geng \thanks{Email: gengjun@lzu.edu.cn.}
}

\author[a]{
Qiang Xu \thanks{Email: xuq@lzu.edu.cn.}
}

\affil[a]{School of Mathematics and Statistics, Lanzhou University, Lanzhou, 730000, China.}

\title{\textbf{Sharp error estimates in stochastic homogenization
of parabolic systems with time-dependent coefficients}}

\begin{document}

\allowdisplaybreaks

\maketitle

%

%
%
%

%

\begin{abstract}
This article mainly proves the existence of stationary correctors under space-time spectral gap conditions, which exhibit different properties from those of elliptic operator correctors. Additionally, new flux correctors and their fluctuation estimates
are introduced. Based on this, we obtain the optimal homogenization  error in the sense of strong and weak norms on $C^1$ cylinders by using the duality and distance-weighted arguments, in which the (weighted) annealed Calder\'on-Zygmund estimates coupled with a novel form of the minimal radius are developed.  Throughout the paper, no small-scale smoothness of the coefficients is used.
\\
\textbf{Key words:} Stochastic homogenization; parabolic systems;
homogenization errors;
annealed Calder\'on-Zygmund estimates.
\end{abstract}

\small
\tableofcontents

\normalsize

\section{Introduction}
\noindent
In recent years, Armstrong, Bordas, and Mourrat \cite{Armstrong-Bordas-Mourrat-18} considered uniformly
parabolic operators with random coefficients depending on space and time variables (with a parabolic scaling):
\begin{equation}\label{operator}
\partial_t + \mathcal{L}_\varepsilon :=
\partial_t -\nabla \cdot a^\varepsilon\nabla
=\partial_t -\nabla\cdot a(\cdot/\varepsilon,\cdot/\varepsilon^2)\nabla,
\qquad \varepsilon>0,
\end{equation}
and they developed a quantitative theory of stochastic homogenization
(e.g. homogenization errors and large-scales $C^{k,1}$-type estimates)
by assuming that the law of the random field $a$ is invariant under $\mathbb{Z}^d\times\mathbb{Z}$-translations and has a finite-range
dependence, building on the ideas introduced by
Armstrong et al. \cite{Armstrong-Kuusi-Mourrat-19,Armstrong-Smart-16}
\footnote{We note
that Armstrong and Kuusi have already upgraded their quantitative theory under very general mixing
conditions in \cite{Armstrong-Kuusi-22} compared to the previous monograph \cite{Armstrong-Bordas-Mourrat-18}, which includes the spectral-gap type inequalities used in this article. This line of arguments is now recognized as the coarse-grained theory \cite{Armstrong-Kuusi-22,Armstrong-Kuusi-25}. }.
Additionally, Bella, Chiarini, and Fehrman studied $\eqref{operator}$
and established qualitative properties of correctors
along with large-scale parabolic excess decay estimates,
under the assumption that the law of the coefficient field $a$ is stationary with respect to space-time translations and ergodic, drawing inspiration from Gloria, Neukamm, and Otto \cite{Gloria-Neukamm-Otto-20}. In an almost periodic setting,
the first author and Shi \cite{Geng-Shi-25} also developed
a quantitative theory for $\eqref{operator}$ in compliance with
its elliptic counterpart by Armstrong and Shen \cite{Armstrong-Shen-16,Shen-15}. Regarding high contrast settings, Lau \cite{Lau-26} introduced a parabolic coarse-graining framework extending the recent work of Armstrong and Kussi \cite{Armstrong-Kuusi-25} on elliptic operators. The operators $\eqref{operator}$ also arise naturally in the intersection of homogenization and regularity structure theory, as evidenced by the recent works  \cite{Clozeau-Singh-25,Hairer-Singh-25}, while research on such multiscale parabolic operators can be traced back to the classical literature \cite{Bensoussan-Lions-Papanicolaou-11,Zhikov-Kozlov-Oleinik-82}.

Compared with the quantitative theory of periodic homogenization
investigated by the first author and Shen \cite{Geng-Shen-15,Geng-Shen-17}, the main challenge in the aperiodic case lies in the study of the following corrector equation:
\begin{equation}\label{corrector}
\partial_t\phi_j -\nabla\cdot a(\nabla\phi_j+e_j) = 0
\quad \text{in}\quad\mathbb{R}^{d+1},
\quad \text{with}~j=1,\cdots,d,
\end{equation}
subject to an underlying ensemble of coefficient fields.
However, to the best of the authors' knowledge,
the existence of stationary solutions to $\eqref{corrector}$ has not yet been addressed in the aforementioned literature \textemdash a gap that is particularly relevant for optimal homogenization error estimates and further applications.
The main motivation and objective of this article
are to close this gap by establishing the existence of the stationary solution to $\eqref{corrector}$
under a spectral gap inequality (e.g., \eqref{c:3}), adopting an approach that has been well developed to the study of
quantitative stochastic homogenization by
Gloria et al. \cite{Gloria-Otto-11,Gloria-Neukamm-Otto-15,Gloria-Neukamm-Otto-20,
Gloria-Neukamm-Otto-21}, originally inspired by Naddaf and Spencer \cite{Naddaf-Spencer-98,Naddaf-Spencer-97}. The present contribution is twofold: (i) we provide a rigorous proof of the existence of stationary correctors; (ii) we demonstrate how this result can be utilized to improve the quantitative estimates in the stochastic homogenization framework for parabolic systems.

\subsection{Setup and main results}
\noindent
Precisely, let $a$ satisfy the $\mu$-uniform ellipticity condition: for some $\mu\in(0,1)$, it holds that
\begin{equation}\label{c:1}
 \mu  |\xi|^2 \leq a_{ij}(z)\xi_i \xi_j \leq \mu^{-1} |\xi|^2
 \qquad
\text{for~all~}\xi\in\mathbb{R}^{d}\text{~and~}z:=(x,t)\in\mathbb{R}^{d+1}.
\end{equation}
(Einstein's summation convention for repeated indices is
used throughout.)
The configuration space can be introduced as the set
of coefficient fields satisfying $\eqref{c:1}$, equipped with a
probability measure \textemdash  referred to as the ensemble
\textemdash and the expectation is denoted by $\langle\cdot\rangle$.
This ensemble is assumed to be \emph{stationary}\footnote{The stationarity in the literature postulates that
the coefficients are statistically homogeneous in time and space.}, i.e., for all shift vectors $z=(x,t)\in\mathbb{R}^{d+1}$, $a(\cdot+x,\cdot+t)$ and $a(\cdot,\cdot)$ have the same law under $\langle\cdot\rangle$.  Additionally, the ensemble satisfies the \emph{spectral gap condition}
\begin{equation}\label{c:3}
 \big\langle (F-\langle F\rangle)^2 \big\rangle \leq \lambda_1
 \Big\langle\int_{\mathbb{R}^{d+1}} dz\Big(\dashint_{Q_1(z)}\big|\frac{\partial F}{\partial a}\big|\Big)^2 \Big\rangle,
\end{equation}
where the notation $Q_r(z):=B_r(x)\times(t-r^2,t+r^2)$
with $z=(x,t)$ is known as the parabolic cube, and
the random tensor field $\frac{\partial F}{\partial a}$ (depending on $(a,z)$) is the functional derivative of $F$ with respect to $\delta a$
(see $\eqref{functional}$ for the concrete definition).

The first result of this paper is now presented.
\begin{theorem}[Stationary correctors]\label{thm:1.0}
Let $d\geq 2$ and $1\leq p<\infty$. Suppose that the ensemble $\langle\cdot\rangle$ satisfies the stationarity condition
associated with $\eqref{c:1}$ and the spectral gap condition $\eqref{c:3}$. Let $1\leq i,j,k\leq d$ and
$1\leq k'\leq d+1$ be integers.
Then, there exist random fields $\phi=\{\phi_j\}$ and   $\sigma=\{\sigma_{k'ij}\}$  with the following properties:
\begin{enumerate}
  \item The random fields $\phi_j$  and $\sigma_{kij}$ are
the unique stationary solutions\footnote{If we say a random field $X(a,x)$ is stationary if it is shift-covariant in the sense of $X(a,x+z) = X(a(z),x)$ for any $x,z\in\mathbb{R}^{d+1}$ and
$\langle\cdot\rangle$-a.e. $a$. This implies that $X(a,\cdot)$ and $X(a,\cdot+z)$ have the same law under
$\langle\cdot\rangle$.}
to the following equations:
\begin{equation}\label{pde:0}
\begin{aligned}
\partial_t\phi_j -\nabla\cdot a(\nabla\phi_j+e_j) &= 0;\\
\Delta_{d+1}\sigma_{k'ij}
 &=\partial_{k'}q_{ij}-\partial_i q_{k'j};\\
 \partial_{k'}\sigma_{k'ij}
 &= q_{ij};\\
 \partial_i\sigma_{(d+1)ij}
 &= \langle \phi_j\rangle-\phi_j,
\end{aligned}
\end{equation}
in the weak sense on $\mathbb{R}^{d+1}$, where
it is noted that $\sigma_{(d+1)ij}$ is not stationary, and
\begin{equation}\label{eq:1.1}
q_{ij}:= \bar{a}_{ij}
- a_{ij}
- a_{ik}\partial_k\phi_j,
\quad
q_{(d+1)j} :=\phi_j-\langle\phi_j\rangle,
\quad
 \bar{a}_{ij}:=\big\langle e_i\cdot a(\nabla\phi_j+e_j)\big\rangle.
\end{equation}
  \item The field $\sigma$ (referred to as the flux corrector) is skew-symmetric in its first two indices, i.e.,
      \begin{equation}\label{skew-symmetric}
      \sigma_{k'ij} = -\sigma_{ik'j}.
      \end{equation}
  \item  The following quantitative estimates hold:
  \begin{subequations}
  \begin{align}
  & \bigg\langle\Big(\dashint_{Q_1(z)}\big|(\nabla \phi_j,\nabla \sigma_{k'ij})\big|^{2}\Big)^p\bigg\rangle^{1/p}\lesssim_{\mu,\lambda_1,d,p}1;
  \label{pri:5.2}\\
  &\bigg\langle \Big|\dashint_{Q_1(z)}\big(\phi_j,\sigma_{kij}\big) - \dashint_{Q_1(0)}\big(\phi_j,\sigma_{kij}\big)\Big|^{2p}\bigg\rangle^{1/p}
\lesssim_{\mu,\lambda_1,d,p}1;
  \label{pri:5.4}\\
  &\bigg\langle \Big|\dashint_{Q_1(z)}\sigma_{(d+1)ij} - \dashint_{Q_1(0)}\sigma_{(d+1)ij}\Big|^{2p}\bigg\rangle^{1/p}
\lesssim_{\mu,\lambda_1,p} \bar{\mu}_d^2(z)
\qquad \forall z\in\mathbb{R}^{d+1},
\label{pri:5.4b}
  \end{align}
  \end{subequations}
\end{enumerate}
where the notation ``$\lesssim_{\mu,\lambda_1,p,d}$'' indicates that the multiplicative constant depends on $\mu,\lambda_1,p,d$. Let 
$\text{d}(z,0):=|x|+\sqrt{|t|}$ be the parabolic distance, and the weight function $\mu_d$ is defined as
\begin{equation}\label{pri:5.5}
\bar{\mu}_d(z):=\left\{\begin{aligned}
&\ln^{\frac{3}{4}}\big(2+\text{d}(z,0)\big) &\quad&\text{if}~~d=2;\\
& 1  &\quad&\text{if}~~d>2.
\end{aligned}\right.
\end{equation}
\end{theorem}

Let $\Omega\subset\mathbb{R}^d$ with $d\geq 2$ be a Lipschitz domain. For $0<T<\infty$, the parabolic cylinder is defined as $\Omega_T= \Omega\times(0,T]$, and the parabolic boundary of $\Omega_T$ is denoted by
$\partial_p\Omega_T := \overline{\Omega}_T\setminus\Omega_T$.
For given data $F$, we consider the following parabolic system with a zero initial-Dirichlet condition:
\begin{equation}\label{pde:1.1}
(\text{DP}_\varepsilon)\left\{\begin{aligned}
\big(\partial_t + \mathcal{L}_\varepsilon\big)(u_\varepsilon)
& = F &~&\text{in}~ ~\Omega_T;\\
u_\varepsilon & = 0 &~&\text{on}~ \partial_p\Omega_T.
\end{aligned}\right.
\end{equation}
If the ensemble $\langle\cdot\rangle$ is merely \emph{ergodic}
rather than satisfying $\eqref{c:3}$, which means
that every translationally invariant function of the coefficient field is constant\footnote{If
$X(a)=X(a(z))$ holds for any $z\in\mathbb{R}^{d+1}$, then
one asserts $X=c$ for $\langle\cdot\rangle$-a.s.,
where $c$ is a constant.}, then it follows from Tartar's test function method
(see e.g. \cite{Papanicolaou-Varadhan-79})
that, for $\langle\cdot\rangle$-a.e. $a$, the following strong convergence holds:
\begin{equation}\label{pri:1.0}
  \lim_{\varepsilon\to 0}\Big\langle\int_{\Omega_T}|u_\varepsilon-u_0|^2
  \Big\rangle = 0.
\end{equation}
This convergence is indeed based on the sublinearity of correctors studied by Bella et al. \cite[Proposition 1]{Bella-Chiarini-Fehrman-19} (and in
the literature \cite{Zhikov-Kozlov-Oleinik-82}), where
the effective solution $u_0$ satisfies the following homogenized equations:
\begin{equation}\label{pde:1.2}
(\text{DP}_0)\left\{\begin{aligned}
\big(\partial_t + \mathcal{L}_0\big)(u_0)
& = F &~&\text{in}~ ~\Omega_T;\\
u_0 & = 0 &~&\text{on}~ \partial_p\Omega_T,
\end{aligned}\right.
\end{equation}
where $\mathcal{L}_0 :=-\nabla\cdot \bar{a}\nabla$, and $\bar{a}$ is the effective coefficient as defined in $\eqref{eq:1.1}$.

On account of the results presented in Theorem $\ref{thm:1.0}$ and
the quantified ergodicity provided by $\eqref{c:3}$,
the next result is devoted
to quantifying the strong convergence described in $\eqref{pri:1.0}$,
which will benefit the real implementation of the numerical algorithm (see e.g. \cite{Clozeau-Josien-Otto-Xu,Gloria-Neukamm-Otto-15}).

\begin{theorem}[Homogenization errors I]\label{thm:1.1}
Let $\Omega\subset\mathbb{R}^d$ be a bounded $C^1$ domain with
$d\geq 2$, $\varepsilon\in(0,1]$, and $T>0$.
Suppose that the ensemble $\langle\cdot\rangle$ is stationary
with respect to $\eqref{c:1}$ and satisfies
the spectral gap condition $\eqref{c:3}$.
Given $F\in L^2(\Omega_T)$ such that $\|F\|_{L^2(\Omega_T)}=1$, let
$u_\varepsilon$ and $u_0$ be the weak solutions of $\eqref{pde:1.1}$
and $\eqref{pde:1.2}$, respectively. Then,
for any $p\in[1,\infty)$, it holds that
\begin{equation}\label{pri:1.1}
\bigg\langle\Big(\int_{\Omega_T}|u_\varepsilon - u_0|^2\Big)^{\frac{p}{2}}\bigg\rangle^{\frac{1}{p}}
\lesssim_{\mu,\lambda_1,d,p,\Omega,T} \varepsilon\mu_d(R_0/\varepsilon)\ln(R_0/\varepsilon).
\end{equation}
In particular,
for the whole space-time region $\mathbb{R}^{d+1}$ and $F\in L^q(\mathbb{R}^{d+1})\cap L^2(\mathbb{R}^{d+1})$ with $q> 2$ and $\|F\|_{L^2(\mathbb{R}^{d+1})}
+\|\bar{\mu}_d^{1+2/d}F\|_{L^q(\mathbb{R}^{d+1})}=1$,
and for any $p\in[1,\infty)$, the optimal error estimate holds:
\begin{equation}\label{pri:1.2}
\begin{aligned}
\bigg\langle\Big(\int_{\mathbb{R}^{d+1}}|u_\varepsilon
&- u_0|^{\frac{2(d+2)}{d}}\Big)^{\frac{pd}{2(d+2)}}\bigg\rangle^{\frac{1}{p}}\\
&+\Bigg\langle\bigg(\int_{\mathbb{R}^{d+1}}dz\Big(
\dashint_{Q_\varepsilon(z)}|u_\varepsilon - u_0|^2\Big)^{\frac{q(d+2)}{2d}}\bigg)^{\frac{pd}{q(d+2)}}
\Bigg\rangle^{\frac{1}{p}}
\lesssim_{\mu,\lambda_1,q,p,d} \varepsilon\mu_d(1/\varepsilon),
\end{aligned}
\end{equation}
where $\mu_d$ is merely a rephrasing of $\bar{\mu}_d$, and given by 
\begin{equation}\label{pri:5.5**}
\mu_d(r):=\left\{\begin{aligned}
&\ln^{\frac{3}{4}}\big(2+r\big) &\quad&\text{if}~~d=2;\\
& 1  &\quad&\text{if}~~d>2.
\end{aligned}\right.
\end{equation}
\end{theorem}


Once the fluctuation estimates
$\eqref{pri:5.2}$, $\eqref{pri:5.4}$, and $\eqref{pri:5.4b}$
in Theorem $\ref{thm:1.0}$ have been established, the
optimal error estimate $\eqref{pri:1.1}$ follows from
a dual method\footnote{It was also known as the Aubin-Nitsche's approach in the numerical fields.}, originally inspired by references
\cite{Kenig-Lin-Shen-12,Suslina-13,Xu-16}. However,
the duality argument fails to establish fluctuation estimates for convergence rates, 
owing to the presence of boundary layer effects.
Therefore, recourse is made to the weighted annealed Calder\'on-Zygmund estimates to accelerate the convergence rate and overcome the negative impact caused by the boundary layer phenomenon. This idea was recently developed
by the second author and Wang
for the optimal convergence rate of
linear elasticity systems on the perforated
domains \cite{Wang-Xu-Zhao21}, and this method is specifically designed to address situations in which the dual argument fails.

In the process of establishing the weighted annealed Calder\'on-Zygmund estimates, the main contribution is to introduce a new definition of the minimum radius, originally inspired by Gloria et al. \cite{Gloria-Neukamm-Otto-20,Gloria-Neukamm-Otto-21},
\begin{equation}\label{min_rad*}
\chi_{*}(0;\theta):=
 \inf\Bigg\{l>0:
 \frac{1}{R}
  \Big(\dashint_{Q_{R}}
 \big|(\bar{\phi},\bar{\sigma},\nabla\tilde{\sigma}_{(d+1)})\big|^{2}\Big)^{\frac{1}{2}}
+ \frac{1}{R^2}\Big(\dashint_{Q_{R}}
 |\tilde{\sigma}_{(d+1)}|^{2}\Big)^{\frac{1}{2}}
 \leq \theta;~\forall R\geq l\Bigg\}\vee 1,
\end{equation}
in which the extended correctors $(\bar{\phi},\bar{\sigma},\tilde{\sigma}_{(d+1)})$ are given by
\begin{equation*}
\begin{aligned}
&\bar{\phi}_j:=\phi_j-\dashint_{Q_R}\phi_j;
\qquad \bar{\sigma}_{ikj}(x,t)
:=\sigma_{kij}(x,t)-\dashint_{B_R}\sigma_{kij}(\cdot,t);\\
&\tilde{\sigma}_{(d+1)}(x,t)
:= \sigma_{(d+1)}(x,t) - \dashint_{Q_R}\sigma_{(d+1)}
-x\cdot\dashint_{Q_R}\nabla\sigma_{(d+1)}, \qquad i,j,k\in\{1,\cdots,d\},
\end{aligned}
\end{equation*}
where $\theta>0$ is an arbitrarily small quantity that can be determined according to the concrete content,
and we usually omit the center point and $\theta$ without causing any confusion. It is not hard to observe that the definition of $\chi_*$ depends on the correctors at the gradient level, and thus it is well-defined even under a stationary and ergodic ensemble (see Remark $\ref{remark:1.2}$).

For the reader's convenience, we present some geometric notation
on integral regions. For any $z\in\Omega_T$, we set
$U_{*,\varepsilon}(z):=\Omega_T\cap Q_{r}(z)$
with $r=\varepsilon\chi_{*}(z/\varepsilon)
=\varepsilon\chi_{*}(x/\varepsilon,t/\varepsilon^2)$,
and $U_\varepsilon(z):=Q_\varepsilon(z)\cap\Omega_T$.
Let $\Omega_0\supseteq\Omega$
be a  $C^1$ domain such that $\partial\Omega_0$ is
the hypersurface at the distance of $4\varepsilon$ from $\partial\Omega$
and parallel to it. Now, we can present another important result in the paper.


\begin{theorem}[Calder\'on-Zygmund estimates]\label{thm:C-Z}
Let $\Omega\subset\mathbb{R}^d$ be a $C^1$ domain with
$d\geq 2$, $\varepsilon\in(0,1]$, and $T>0$.
Suppose that the ensemble $\langle\cdot\rangle$ is stationary and
ergodic.
Let $u_\varepsilon$ and $f$ be associated with the following equations:
\begin{equation}\label{pde:A}
\left\{\begin{aligned}
\partial_t u_\varepsilon+\mathcal{L}_\varepsilon(u_\varepsilon)
&= \nabla\cdot f
&\quad&\text{in}~~\Omega_T;\\
u_\varepsilon &= 0
&\quad&\text{on}~~\partial_p\Omega_T.
\end{aligned}\right.
\end{equation}
Then, there exists a stationary random field $\chi_*$
(referred to as the minimal radius) such that,
for any $1<p<\infty$,
the quenched Calder\'on-Zygmund estimate holds:
\begin{equation}\label{pri:A}
\bigg(\int_{\Omega_T}dz\Big(
\dashint_{U_{*,\varepsilon}(z)}|\nabla u_\varepsilon|^2
\Big)^{\frac{p}{2}}\bigg)^{\frac{1}{p}}
\lesssim_{\mu,d,\partial\Omega,p}
\bigg(
\int_{\Omega_T}dz\Big(\dashint_{U_{*,\varepsilon}(z)}|f|^2\Big)^{\frac{p}{2}}
\bigg)^{\frac{1}{p}}.
\end{equation}
Moreover, if the ensemble $\langle\cdot\rangle$ additionally satisfies
$\eqref{c:3}$,
then, for any $1<p,q<\infty$ and $\bar{p}>p$,
the weighted annealed Calder\'on-Zygmund estimate holds:
\begin{equation}\label{pri:B}
\bigg(\int_{\Omega_T}dz\Big\langle
\Big(\dashint_{
U_\varepsilon(z)}
|\nabla u_\varepsilon|^2 \Big)^{\frac{p}{2}}
\Big\rangle^{\frac{q}{p}}\omega_\sigma(z)\bigg)^{\frac{1}{q}}
\lesssim_{\mu,\lambda_1,d,\partial\Omega,p,\bar{p},q} \bigg(\int_{\Omega_T}dz
\Big\langle\Big(\dashint_{
U_{\varepsilon}(z)}
|f|^2 \Big)^{\frac{\bar{p}}{2}}\Big\rangle^{\frac{q}{\bar{p}}}
\omega_\sigma(z)\bigg)^{\frac{1}{q}},
\end{equation}
where $\omega_\sigma(z):=[\text{dist}(x,\partial\Omega_0)]^{\underline{q}-1}$
with $0<\underline{q}<q$ and
$z=(x,t)\in\Omega_T$ is the
wight function.  Finally, for
$-(d+2)<\alpha<(d+2)(q-1)$,
the estimate $\eqref{pri:B}$
with $\omega_\sigma(z):=(|x|+\sqrt{|t|}+\varepsilon)^{\alpha}$, alone with
$\eqref{pri:A}$,
also holds in the case where $\Omega_T=\mathbb{R}^{d+1}$.
\end{theorem}

Theorem $\ref{thm:C-Z}$ concerns regularity estimates in quantitative homogenization theory, which can be traced back to the contributions of Avellaneda and Lin \cite{Avellaneda-Lin-87}. The first work on quenched Calder\'on-Zygmund estimate
such as $\eqref{pri:A}$ was developed by Armstrong and Daniel \cite{Armstrong-Daniel-16}. Taking into account the influence of boundary smoothness, the present estimate $\eqref{pri:A}$
was established for a stationary and ergodic ensemble, which is novel
even among results on the elliptic boundary value problems. The weighted annealed Calder\'on-Zygmund estimate $\eqref{pri:B}$ has proven to be very useful in quantitative stochastic homogenization theory. For example, it can be used to study the fluctuation properties of high-order correctors, homogenization errors, and boundary correctors
(see e.g., \cite{Clozeau-Josien-Otto-Xu,Wang-Xu-25,Wang-Xu-26}).

Although the complete proof of the above theorem is far from easy,
the basic idea for obtaining estimates $\eqref{pri:A}$ and $\eqref{pri:B}$ is conceptually clear: via the two-scale expansion, one has the following decomposition:
\begin{equation*}
\nabla u_\varepsilon
 = \underbrace{\nabla u_\varepsilon
 -\big(e_i+\nabla\phi_i(\cdot/\varepsilon,\cdot/\varepsilon^2)\big)\partial_i
 \bar{u}_{\text{approx}}}_{W_\varepsilon}
 + \underbrace{\big(e_i+\nabla\phi_i(\cdot/\varepsilon,\cdot/\varepsilon^2)\big)
 \partial_i\bar{u}_{\text{approx}}}_{V}.
\end{equation*}
Roughly speaking, $V$ represents the ``regular'' part in which the approximating function $\bar{u}_{\text{approx}}$
has the same type estimates as the stated estimates $\eqref{pri:A}$ and $\eqref{pri:B}$. If we want to pass the regularity of $V$ to $\nabla u_\varepsilon$, then it is required that $W_\varepsilon$ should be sufficiently small under the appropriate norm and scales. Thereupon,
the so-called minimum radius $\chi_{*}$ becomes relevant, and it exactly characterizes the scale at which the approximating term $W_\varepsilon$
can be as desired small (see the concrete statement in Lemma \ref{lemma:approxi}). We will implement this basic idea by using Shen's lemma (i.e., Lemma $\ref{shen's lemma2}$), which can be traced back to Duerinckx and Otto's work \cite{Duerinckx-Otto-20} for random elliptic operators. To obtain the estimate $\eqref{pri:B}$ from $\eqref{pri:A}$,
the key idea is the repetitive application of Shen's lemma.

In order to precisely describe this local approximation,
additional notation must be introduced.
Let $R_0$ represent the diameter of $\Omega$, and $R_*:= R_0\vee T$. Then,  $\Omega_*:=\Omega\times I_{R_*}$ (with $I_{R_*}:=(-R_*^2,R_*^2)$) is defined as an extension cylinder of
$\Omega_T$, and $\partial_{\shortparallel}\Omega_*:=\partial\Omega\times I_{R_*}$ represents the lateral boundary of $\Omega_*$.
For any $z\in\partial_{\shortparallel}\Omega_T$,
let $D_r(z):=Q_r(z)\cap\Omega_{*}$,  $\Delta_r(z):=Q_r(z)\cap\partial_{\shortparallel}\Omega_{*}$,
and we will omit the center point of
$D_r(z),\Delta_r(z)$ when $z=0$. Let $\partial_p D_r$ denote the parabolic boundary of $D_r$.

\begin{lemma}[Qualitative theory]\label{lemma:approxi}
Let $\Omega$ be a $C^1$ domain with $\partial\Omega\ni\{0\}$, $T>0$, and  $\varepsilon\in(0,1]$.
Assume that the ensemble $\langle\cdot\rangle$
is stationary and ergodic.
Then, for any $\nu>0$,  there exists
a stationary random field
$\chi_*$ as given in $\eqref{min_rad*}$,  such that,
for any solution $u_\varepsilon$ to $\big(\partial_t+\mathcal{L}_\varepsilon\big) (u_\varepsilon) = 0$
in $D_{8r}$
with $u_\varepsilon = 0$
on $\Delta_{8r}$, there exists
a weak solution $\bar{u}_r$ satisfying
$\big(\partial_t+\mathcal{L}_0\big)(\bar{u}_r) = 0$ in $D_{2r}$ with $\bar{u}_r = u_\varepsilon$ on $\partial_p  D_{2r}$, such that the estimate
\begin{equation}\label{pri:1.5}
\dashint_{D_r}|\nabla u_\varepsilon -
\big(e_j+\nabla\phi_j(\cdot/\varepsilon,\cdot/\varepsilon^2)\big)\partial_j\bar{u}_r|^2 \leq 
\nu^2
 \dashint_{D_{4r}}|\nabla u_\varepsilon|^2
\end{equation}
holds for any $r\geq \varepsilon \chi_*$. Moreover,
if the ensemble $\langle\cdot\rangle$ satisfies the spectral gap condition $\eqref{c:3}$, then the stationary random field $\chi_*$ satisfies
the following:
\begin{equation}\label{m-1}
\langle
\chi_*^\beta\rangle\lesssim_{\mu,\lambda_1,d,\beta} 1
\qquad \forall\beta\in[1,\infty).
\end{equation}
\end{lemma}

In stochastic homogenization theory, by appealing to
the correctors, the oscillations of $\nabla u_\varepsilon$ can be expressed via the two-scale asymptotic expansion as in the case of periodic homogenization:
\begin{equation*}
  \nabla u_\varepsilon \thickapprox
  \big(e_j+\nabla\phi_j(\cdot/\varepsilon,\cdot/\varepsilon^2)
  \big)\partial_ju_0.
\end{equation*}
This is exactly the content of the estimate $\eqref{pri:1.5}$ at large scales (through the notion of the minimum radius). However, as opposed to the deterministic case, the stochastic setting additionally involves the study of the fluctuations of $\nabla u_\varepsilon$ alongside its oscillations.
Concerned with random elliptic operators, the fluctuations exhibit a central limit theorem (CLT) scaling $O(\varepsilon^{\frac{d}{2}})$, in which the above two-scale expansion is not accurate in general. This was first observed by Gu and Mourrat \cite{Gu-Mourrat-16}. Later, by introducing the standard homogenization commutator, defined as
\begin{equation}\label{eq:1}
  \Xi_i^{\text{ell}}:= (a^{\text{ell}}-\bar{a}^{\text{ell}})\big(e_i+\nabla\phi_i^{\text{ell}}\big),
\end{equation}
where we use the superscript ``ell'' to denote the quantities associated with the elliptic operator,
Duerinckx, Gloria, and Otto \cite{Duerinckx-Gloria-Otto-20} discovered that the two-scale expansion of the homogenization commutator of the solution $u_\varepsilon^{\text{ell}}$
\begin{equation*}
  (a^{\text{ell}}(\cdot/\varepsilon) - \bar{a}^{\text{ell}})\nabla u_\varepsilon^{\text{ell}} -
  \big\langle(a^{\text{ell}}(\cdot/\varepsilon) - \bar{a}^{\text{ell}})\nabla u_\varepsilon^{\text{ell}}\big\rangle
  \thickapprox \Xi_{i}^{\text{ell}}(\cdot/\varepsilon)\partial_iu^{\text{ell}}_0
\end{equation*}
is accurate in the fluctuation scaling. Below, we establish results for parabolic operators that are parallel to those for elliptic ones.
Here, it is emphasized that the two-scale expansion used for parabolic operator differs from that given for the elliptic operator. Correction
for the time variable must be taken into account simultaneously, i.e.,
\begin{equation*}
   u_\varepsilon \thickapprox \big(1 + \varepsilon\phi_i(\cdot/\varepsilon,\cdot/\varepsilon^2)\partial_i
   + \varepsilon^2\tilde{\sigma}_{i(d+1)j}(\cdot/\varepsilon,\cdot/\varepsilon^2)
   \partial_{ij}\big)u_0,
\end{equation*}
which will be conducive to establishing the equations satisfied by the error of the two-scale expansions (see Lemma $\ref{lemma:3.4}$).
In other words, it has been confirmed that the (first-order) homogenization commutator of the parabolic operators has the same structure as its elliptic counterpart. However, this by no means implies a similar conclusion for high-order homogenization commutators, and the related issues will be discussed in a separated work. In addition, the following results are also an important application of the weighted annealed Calder\'on-Zygmund estimates.

\begin{theorem}[Homogenization errors II]\label{thm:1.2}
Let $2\leq q<\infty$.
Let $\Omega\subset\mathbb{R}^d$ be a bounded $C^1$ domain with
$d\geq 2$, $\varepsilon\in(0,1]$, and $T>0$.
Suppose that the ensemble $\langle\cdot\rangle$ is stationary
with respect to $\eqref{c:1}$ and satisfies $\eqref{c:3}$. Given $f\in C_0^1(\Omega_T;\mathbb{R}^d)$,  assume that $u_\varepsilon$ and $u_0$ are the weak solutions of the zero initial-Dirichlet problems
in $\eqref{pde:1.1}$ and $\eqref{pde:1.2}$, respectively, with $F=\nabla\cdot f$.
Let $\varphi_j\in W_{2}^{2,1}(\Omega_T)$ be chosen such that it vanishes near the parabolic boundary $\partial_{\shortparallel}\Omega_T$.
Then, for any $p\in[1,\infty)$, one can obtain that
\begin{equation}\label{pri:1.3}
\begin{aligned}
\Bigg\langle
\bigg(\int_{\Omega_T}dz \Big(\dashint_{U_\varepsilon(z)}
|\nabla u_\varepsilon - \nabla u_0 -\nabla\phi_j\varphi_j|^2\Big)^{\frac{q}{2}}\omega_\sigma(z)
\bigg)^{\frac{p}{q}}\Bigg\rangle^{\frac{1}{p}}
\lesssim
\varepsilon^{1-}\mu_d(R_0/\varepsilon)
\Big(\int_{\Omega_T}
 |R_0\nabla f|^{q}\Big)^{\frac{1}{q}},
\end{aligned}
\end{equation}
where the weight function $\omega_\sigma$ is given as in Theorem $\ref{thm:C-Z}$.
Moreover, for any $h\in C_0^{2}(\Omega_T;\mathbb{R}^{d+1})$ adhering to
a parabolic scaling,
the random variable $H^\varepsilon$ is defined as follows:
\begin{equation}\label{eq:3-5-1}
H^\varepsilon : = \int_{\Omega}
h\cdot (a^\varepsilon-\bar{a})\big(\nabla u_\varepsilon-\nabla u_0
-\nabla\phi_i(\cdot/\varepsilon,\cdot/\varepsilon^2)\varphi_i\big).
\end{equation}
Then, for any $p\in[1,\infty)$, we have
\begin{equation}\label{pri:1.4}
\begin{aligned}
\varepsilon^{-\frac{d+2}{2}} \big\langle (H^\varepsilon
&-\langle H^\varepsilon\rangle)^{2p} \big\rangle^{\frac{1}{2p}}\\
&\lesssim_{\mu,\lambda_1,d,\Omega,T,p,s}
\varepsilon^{1-}\mu_{d}(R_0/\varepsilon)
\Big(\int_{\Omega_T}|R_0(\partial_t,\partial^2) h|^{2s}\Big)^{\frac{1}{2s}}
\Big(\int_{\Omega_T}
 |R_0\nabla f|^{2s'}\Big)^{\frac{1}{s'}},
\end{aligned}
\end{equation}
where $s,s'>1$ are associated with $1/s'+1/s=1$ and $0<s'-1\ll 1$.
\end{theorem}

\subsection{Novelties and remarks on the main results}

\begin{remark}
\emph{While the notation $Q_1(z)$ in $\eqref{c:3}$ refers to a parabolic cube centered at $z=(x,t)\in\mathbb{R}^{d+1}$, it is important to clarify that the variable $t$ in $\eqref{c:3}$ does not coincide with the time variable in the parabolic system $\eqref{pde:1.1}$. The time variable in an evolutionary system possesses a directional structure, whereas the spectral gap condition here does not encode such information. 
}

\emph{The central idea of this article is to extend the sensitive estimates developed by Gloria et al. \cite{Gloria-Otto-11,Gloria-Neukamm-Otto-15,Gloria-Neukamm-Otto-20,Gloria-Neukamm-Otto-21}
for elliptic operators to the parabolic setting. Specifically, we employ functional inequalities as the primary tool to quantify ergodicity, thereby enabling a precise quantification of the sublinear growth of the correctors. A distinguishing feature of our results is that the existence of stationary correctors associated with the parabolic operator $\eqref{operator}$, as established in Theorem $\ref{thm:1.0}$, is independent of the spatial dimension, which markedly contrasts with the dimension-dependent results for elliptic operators.}
\emph{Furthermore, compared with the work of Bella et al. \cite{Bella-Chiarini-Fehrman-19}, a novel structural refinement is introduced by expanding $q_{ij}$ into $(q_{ij},q_{(d+1)j})$ as shown in $\eqref{eq:1.1}$. This endows the new vector field $q_{\cdot j}$ with two crucial properties: it is \emph{divergence-free} and exhibits \emph{mean zero} (see Proposition $\ref{P:5.3}$). This structural innovation allows us to exploit the intrinsic geometry of the equation $\eqref{corrector}$, leading to the derivation of the new system $\eqref{pde:0}$. Although this observation was initially recognized in the periodic case by the first author and Shen \cite{Geng-Shen-17}, the extension to the stochastic setting and the establishment of the new estimates $\eqref{pri:5.2}$, $\eqref{pri:5.4}$, and $\eqref{pri:5.4b}$ require substantially more sophisticated analytical techniques (see Section $\ref{section:2}$).}
\end{remark}

\begin{remark}\label{remark:1.2}
\emph{As noted in \cite{Gloria-Neukamm-Otto-20}, the minimum radius $\chi_*$ introduced in this paper exhibits both qualitative and quantitative characteristics within the framework of stochastic homogenization theory. While its rudimentary form was already present in \cite[Proposition 1]{Bella-Chiarini-Fehrman-19},
this quantity was not formally defined in the work of Bella et al. \cite{Bella-Chiarini-Fehrman-19}. A key innovation of this paper is the introduction of the extended flux correctors $(\bar{\sigma}_k,\tilde{\sigma}_{(d+1)})$ used to define $\chi_*$\footnote{It is further reflected in the fact that the minimum radius $\chi_*$ defined here is more amenable to be quantified in terms of stochastic integrability compared with the related construction in \cite{Bella-Chiarini-Fehrman-19}.}.
Below, we briefly take $\bar{\sigma}_k$ as an example to
illustrate that this minimum radius is well-defined
even under stationary and ergodic ensembles.
By virtue of $\langle|\nabla \sigma_{kij}|^2\rangle<\infty$ and
$\langle \nabla \sigma_{kij}\rangle = 0$, it can be established
that for each $i,j,k\in\{1,\cdots,d\}$,
\begin{equation}\label{pri:1.6}
 \limsup_{R\to\infty}\frac{1}{R^2}\dashint_{Q_R}\big|\sigma_{kij}
 -\dashint_{B_R}dy\sigma_{kij}(y,\cdot)\big|^2 = 0
 \qquad \text{for}~\langle\cdot\rangle\text{-a.e.~} a
\end{equation}
(see \cite[Lemma 5]{Bella-Chiarini-Fehrman-19}),
and the stationarity of $\nabla\sigma$ follows immediately from
the second equation of $\eqref{pde:0}$. The argument for $\tilde{\sigma}_{(d+1)}$ is more involved
and is deferred to Remark $\ref{remark:3.2}$.
It is desirable to obtain the optimal stochastic integrability for $\chi_*$
(e.g., under stronger functional inequalities such as the Logarithmic Sobolev Inequality (LSI)), though this direction is not pursued here. It is also noted that Armstrong et al. \cite{Armstrong-Bordas-Mourrat-18} addressed similar issues in a relatively straightforward manner and
achieved the sharp stochastic integrability under the framework therein.}

\emph{Equipped with the minimum radius $\chi_*$, one can, roughly speaking,  further perform large-scale regularity estimates analogous to those
developed by Armstrong et al. \cite{Armstrong-Bordas-Mourrat-18,Armstrong-Kuusi-Mourrat-19,Armstrong-Smart-16} or Bella et al. \cite{Bella-Chiarini-Fehrman-19,Gloria-Neukamm-Otto-20}.
The core principle is to convert the (boundary) regularity estimates  into a local approximation lemma (e.g.,
Lemma $\ref{lemma:approxi}$) based on a qualitative (or quantitative) homogenization theory\footnote{If one studies boundary Lipschitz (or Schauder) estimates, a quantitative statement
analogous to \cite[Lemma 2.8]{Wang-Xu-26} seems to be inevitable, which
could be observed even in a periodic setting (see \cite{Shen18}).}.
Since the definition of $\chi_*$ relies entirely on correctors and flux correctors, its advantage lies in the ability to ``separate'' large-scale estimates from quantifying the sublinearity of correctors. Although boundary Lipschitz estimates do not appear in  \cite{Armstrong-Bordas-Mourrat-18} or \cite{Bella-Chiarini-Fehrman-19}, they become tractable under the assumption of quantitative ergodicity (see e.g., \cite{Armstrong-Kuusi-Mourrat-19,Wang-Xu-26}).}
\emph{Moreover, we pose the following open questions to guide future research:
\emph{Can boundary Lipschitz estimates be derived under stationary and ergodic ensembles? Similarly, for a Lipschitz cylinder, can the estimate $\eqref{pri:A}$ be obtained as the same in Theorem $\ref{thm:C-Z}$?}}
\end{remark}

\begin{remark}
\emph{Another  effort of this paper is the establishment of homogenization errors  under low regularity of boundary geometries in Theorems $\ref{thm:1.1}$ and $\ref{thm:1.2}$.
A key methodological contribution here is the development of a unified framework that reduces error estimates to boundary-layer-type estimates (encompassing both spatial and temoral boundary layers, see Fig.$\ref{pic:1.2}$) and co-layer-type estimates for the homogenized solution (see Lemma \ref{lemma:3.2-ap}), even for domains with minimal regularity.
This approach not only connects to the study of low-regularity boundary value problems \cite{Brown-89,Fabes-Riviere-79} but also demonstrates that providing a detailed argument remains essential to fully capture  this meaningful connection,
particularly given that low boundary regularity inherently leads to convergence rate losses (see e.g., the estimates $\eqref{pri:1.1}$ and $\eqref{pri:1.4}$). While extending results from $C^1$ domains to Reifenberg-flat-type regions (or from Dirichlet to other type boundary conditions) is entirely feasible, generalizing to Lipschitz domains remains constrained by integrability index conditions (see e.g., \cite[Theorem 1.1]{Wang-Xu-25}). Notably, the dependence of the aforementioned conclusions $\eqref{pri:1.1}$, $\eqref{pri:1.3}$ and $\eqref{pri:1.4}$ on the domain diameter (which limits direct extension to the whole space) is addressed here: with slight modifications to the proof, this diameter dependence is stransformed into a weight function in the norms as seen in $\eqref{pri:1.2}$. Overall, the whole-space estimate is more straightforward than the bounded domain case, primarily because it avoids boundary layer complications.}
\end{remark}

\begin{remark}
\emph{The investigation of higher-order correctors not only carries profound theoretical significance (see e.g., the connection to the Bourgain-Spencer conjecture \cite{Duerinckx-Gloria-Lemm-19}) but also has substantial practical utility in numerical applications, particularly in analyzing systematic errors within the Representative Volume Element (RVE) method \cite{Clozeau-Josien-Otto-Xu} and in addressing the optimal artificial boundary condition problem \cite{Lu-Otto-Wang-24}. Through asymptotic analysis, the governing equation for the second-order corrector can be derived as follows:
\begin{equation}\label{corrector-2}
\big(\partial_t -\nabla\cdot a\nabla\big)(\phi_{ij}^{(2)})
=\nabla\cdot\big(a\phi_{i}-\sigma_{i}
+a\partial\sigma_{(d+1)i}\big)e_j
\qquad\text{in}\quad\mathbb{R}^{d+1}.
\end{equation}
By Theorem $\ref{thm:1.0}$, the divergence term on the right-hand side of $\eqref{corrector-2}$ corresponds to a stationary field. Consequently, analogous to Theorem $\ref{thm:1.0}$, there is strong potential for identifying stationary solutions for the equation $\eqref{corrector-2}$ for $d>2$. However, extending the methodology for obtaining second-order correctors to higher orders proves exceptionally challenging for the parabolic operators in $\eqref{operator}$. A core practical obstacle lies in the scaling equivalence between the time derivative and the second-order spatial derivative, which currently lacks an accurate and structured representation in formal expansions. Thus, the exploration of higher-order correctors and associated homogenization commutators for parabolic operators remains an interesting topic. Additionally, the correlation structure of (first-order) correctors and the variance structure of homogenization commutators will be elaborated on in a subsequent paper.}
\end{remark}

\subsection{Outline of the paper}
\noindent
Although this article encompasses a significant amount of material, its primary line of reasoning follows a structured progression. As shown in Tab.$\ref{table1}$, Theorems $\ref{thm:1.0}$ and $\ref{thm:C-Z}$ form the most fundamental theories of this article, while Theorems $\ref{thm:1.1}$ and $\ref{thm:1.2}$ serve as crucial bridges connecting to applied fields such as computational mathematics.

Section $\ref{section:2}$ is dedicated to the study of the corrector and flux correctors, and the proof structure of Theorem $\ref{thm:1.0}$ is shown in Fig.$\ref{pic:2.1}$. Section $\ref{section:4}$ investigates
two-scale asymptotic expansions, weighted estimates of smoothing operators, and the proof of Lemma $\ref{lemma:approxi}$, which leads to the definition of the minimum radius $\chi_*$, as depicted in $\eqref{min_rad*}$. Section $\ref{section:5}$ focuses on Calder\'on-Zygmund theory, and the proof flowchart for Theorem $\ref{thm:C-Z}$ is presented in
Fig.$\ref{pic:4.1}$. Section $\ref{section:3}$ examines the homogenization errors in the sense of oscillation and fluctuation, and the structural proof of Theorems $\ref{thm:1.1}$ and $\ref{thm:1.2}$ are provided in
Figs.$\ref{pic:5.1}$ and $\ref{pic:5.2}$, respectively.


\begin{longtable}{|>{\arraybackslash}m{5.1cm}|>{
\centering\arraybackslash}m{5.1cm}|>{\centering\arraybackslash}m{5.2cm}|}
\caption{The basic classification of the main results}\label{table1}\\
\hline
%
%
%
Theoretical foundation work &\quad~ Theorem $\ref{thm:1.0}$ \newline (Correctors' theory) &\quad~ Theorem $\ref{thm:C-Z}$ \newline (Calder\'on-Zygmund theory) \\
\hline
Applied mathematics work & \quad~
Theorem $\ref{thm:1.1}$ \newline (Oscillation of the errors) &\quad~ Theorem $\ref{thm:1.2}$ \newline (Fluctuation of the errors)\\
\hline
\end{longtable}

\subsection{Notation}
\noindent
An apology is offered at the outset for introducing a substantial amount of notation here. This is particularly necessary because the analysis of parabolic equations requires the definition of more intricate symbols than their elliptic counterparts, and therefore it is recommended that this section should be consulted as a reference.
\begin{enumerate}
  \item Notation for estimates.
\begin{enumerate}
  \item $\lesssim$ and $\gtrsim$ stand for $\leq$ and $\geq$
  up to a multiplicative constant,
  which may depend on some given parameters in the paper
  but never on $\varepsilon$. The subscript form $\lesssim_{\lambda,\cdots,\lambda_n}$ indicates the constant depends only on parameters $\lambda,\cdots,\lambda_n$.
  In addition, superscripts such as $\lesssim^{(2.1)}$ are used to indicate the referenced formula or estimate.
  The notation $\sim$ is used when both $\lesssim$ and $\gtrsim$ hold.
  \item We use $\gg$ instead of $\gtrsim$ to indicate that the multiplicative constant is much larger than 1 (but still finite),
      and it is similarly for $\ll$.
  \item We write $\varepsilon^{s-}$ to represent the quantity of the form $\varepsilon^{s-\kappa}$, provided that $0<\kappa\ll 1$.
\end{enumerate}
  \item Notation for derivatives.
  \begin{enumerate}
  \item Temporal derivative: $\partial_t := \partial u/\partial t$ is the derivative with respect to the time variable.
    \item Spatial derivatives:
  $\nabla v = (\partial_1 v, \cdots, \partial_d v)$ is the gradient of $v$, where
  $\partial_i v = \partial v /\partial x_i$ denotes the
  $i^{\text{th}}$ derivative of $v$.
  $\nabla^2 v$  denotes the Hessian matrix of $v$;
  $\nabla\cdot v=\sum_{i=1}^d \partial_i v_i$
  denotes the divergence of $v$, where
  $v = (v_1,\cdots,v_d)$ is a vector-valued function.
$\Delta_{d+1}:=\sum_{i=1}^{d+1}\partial_i^2$ represents the $(d+1)$-dimensional Laplace operator, while $\overline{\nabla}
=(\nabla,\partial_{d+1})$ denotes the $(d+1)$-dimensional gradient operator.
   \item Functional (or vertical) derivative:  the random tensor field $\frac{\partial F}{\partial a}$ (depending on $(a,z)$) is the functional derivative of $F$ with respect to $a$, defined by
       \begin{equation}\label{functional}
       \lim_{\varepsilon\to 0} \frac{F(a+\varepsilon\delta a)
       -F(a)}{\varepsilon}
       =\int_{\mathbb{R}^{d+1}}dz \frac{\partial F(a)}{\partial a_{ij}(z)}
       (\delta a)_{ij}(z).
       \end{equation}
  \end{enumerate}

  \item Geometric notation.
  \begin{enumerate}
  \item Let $d\geq 2$ be the dimension, and let $R_0$ denote
  the diameter of $\Omega$. For $0<T<\infty$, we define the parabolic cylinder as $\Omega_T= \Omega\times(0,T]$, and the parabolic boundary of $\Omega_T$ is denoted by
$\partial_p\Omega_T := \overline{\Omega}_T\setminus\Omega_T$, which consists of the ``lateral'' and ``bottom'' boundaries of $\Omega_T$, respectively, as follows:
\begin{equation*}
 \partial_p\Omega_T=\partial_{\shortparallel}\Omega_T
 \cup \partial_b\Omega
 :=\big\{\partial\Omega\times (0,T]\big\}
 \cup \big\{\Omega\times \{t=0\}\big\} .
\end{equation*}
  Let $R_*:= R_0\vee T$, where it is always assumed
   that $R_0\geq 1$ and $T\geq 1$ throughout. Then,  $\Omega_*:=\Omega\times I_{R_*}$ with
  $I_{R_*}:=(-R_*^2,R_*^2)$ is defined as an extension cylinder of
$\Omega_T$, and $\partial_{\shortparallel}\Omega_*:=\partial\Omega\times I_{R_*}$ denotes the lateral boundary of $\Omega_*$.
  \item Let $Q_r(z):=B_r(x)\times(t-r^2,t+r^2)$
be the parabolic cube with $z=(x,t)\in\mathbb{R}^{d+1}$.  For a parabolic cube $Q$, set $Q=Q_{r_Q}(z_Q)$ and, by abuse of notation,  write
  $\alpha Q:=Q_{\alpha r_Q}(z_B)$.
\item For any $z\in\partial_{\shortparallel}\Omega_T$,
let $D_r(z):=Q_r(z)\cap\Omega_{*}$,  $\Delta_r(z):=Q_r(z)\cap\partial_{\shortparallel}\Omega_{*}$, and
\begin{equation*}
\partial_p D_r(z):= (\partial_{\shortparallel}\cup\partial_b)D_r(z)
= \big\{\partial (B_r(x)\cap\Omega) \times
 I_{r}(t)\big\}\cup \big\{(B_r(x)\cap\Omega)\times\{t=-r^2\}\big\},
\end{equation*}
where $I_r(t):=t+I_r=(t-r^2,t+r^2)$, and $\partial_p D_r$ denotes the parabolic boundary of
$D_r$ which consists of the lateral and
bottom parts, denoted by $\partial_{\shortparallel} D_r$ and
$\partial_b D_r$, respectively.
The center point of
$Q_r(z),D_r(z)$, and $\Delta_r(z)$ is omitted only when $z=0$.
For any $z=(x,t)\in\Omega_T$, we introduce $U_r(z):=Q_r(z)\cap\Omega_T$,
and $U_{*,\varepsilon}(z):=
\Omega_T\cap Q_{*,\varepsilon}(z)$ with
$Q_{*,\varepsilon}(z):=Q_{\varepsilon\chi_{*}(z/\varepsilon)}(z)
=Q_{\varepsilon\chi_{*}(x/\varepsilon,t/\varepsilon^2)}(z)$,
where $\chi_*$ is the minimum radius given
in Lemma $\ref{lemma:approxi}$.

\item
$\Sigma_{r}^T = \Sigma_{r}\times [r^2,T-r^2]$ is interpreted as the time-space co-layer of $\Omega_T$,
      where $\Sigma_{r}
      = \big\{x\in\Omega:\text{dist}(x,\partial\Omega) \geq r\big\}$.
      $\boxbox_{r}:=\Omega_T\setminus \Sigma_{r}^T$ represents the time-space layer of
      $\Omega_T$. Let $(O_r)_T$ be the lateral layer of $\Omega_T$,
      where $O_r:=\{x\in\Omega:\text{dist}(x,\partial\Omega)<r\}$,
      and $(\Omega\setminus O_r)_T:=
(\Omega\setminus O_r)\times (0,T] = \Sigma_{r}\times(0,T]$ is
the lateral co-layer of $\Omega_T$.
\end{enumerate}
\includegraphics[width=5.5in]{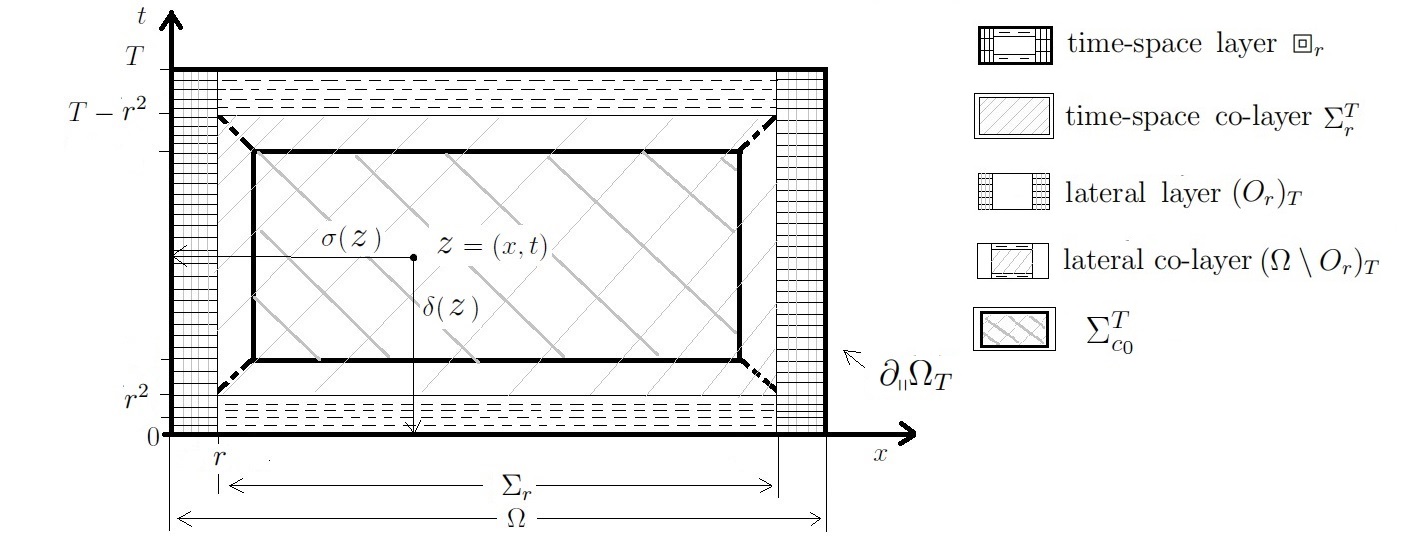}
\makeatletter\def\@captype{figure}\makeatother
\caption{sectional view of $\Omega_T$}\label{pic:1.2}

\item Notation for functions.
\begin{enumerate}
    \item The function $\textbf{1}_{E}$ is the indicator function of $E$.

\item $a\vee b := \max\{a,b\}$; $a\wedge b:=\min\{a,b\}$.

\item We denote $F(\cdot/\varepsilon,\cdot/\varepsilon^2)$ by $F^{\varepsilon}(\cdot,\cdot)$ for simplicity, and $(f)_r = \dashint_{Q_r} f  = \frac{1}{|Q_r|}\int_{Q_r} f$.
\item We denote the support of $f$ by $\text{supp}(f)$.

\item Let $\text{d}(z,z_0):=|x-x_0|+|t-t_0|^{\frac{1}{2}}$
be the parabolic distance between any $z,z_0\in\mathbb{R}^{d+1}$.
For any
$z=(x,t)\in \Omega_T$, the distance between $z$ and $\partial\Omega_T$ is denoted by
\begin{equation}\label{eq:2.4}
\delta(z)=\delta(x,t) =  \min\Big\{\text{dist}(x,\partial\Omega),t^{\frac{1}{2}},
(T-t)^{\frac{1}{2}}\Big\}.
\end{equation}
and the distance between $z$ and $\partial_{\shortparallel}\Omega_T$ is written by
\begin{equation}\label{eq:2.5}
 \sigma(z) =\sigma(x,t)= \text{dist}(x,\partial\Omega).
\end{equation}

\item The nontangential maximal function of $u$ is defined by
\begin{equation}\label{def:nontangential}
(u)^{*}(z_0):=\sup\big\{|u(z)|:z\in\Gamma_{N_0}(z_0)\big\}
\qquad\forall z_0\in\partial_{\shortparallel}\Omega_T,
\end{equation}
where $\Gamma_{N_0}(z_0):=\big\{z=(x,t)\in\Omega_T:\text{d}(z,z_0)\leq N_0\text{dist}(x,\partial\Omega)\big\}$ is the cone with vertex $z_0$
and the aperture $N_0$, and $N_0>1$ may be chosen sufficiently large.
(See Fig.\ref{pic:1.3})

  \end{enumerate}


\includegraphics[width=5.5in]{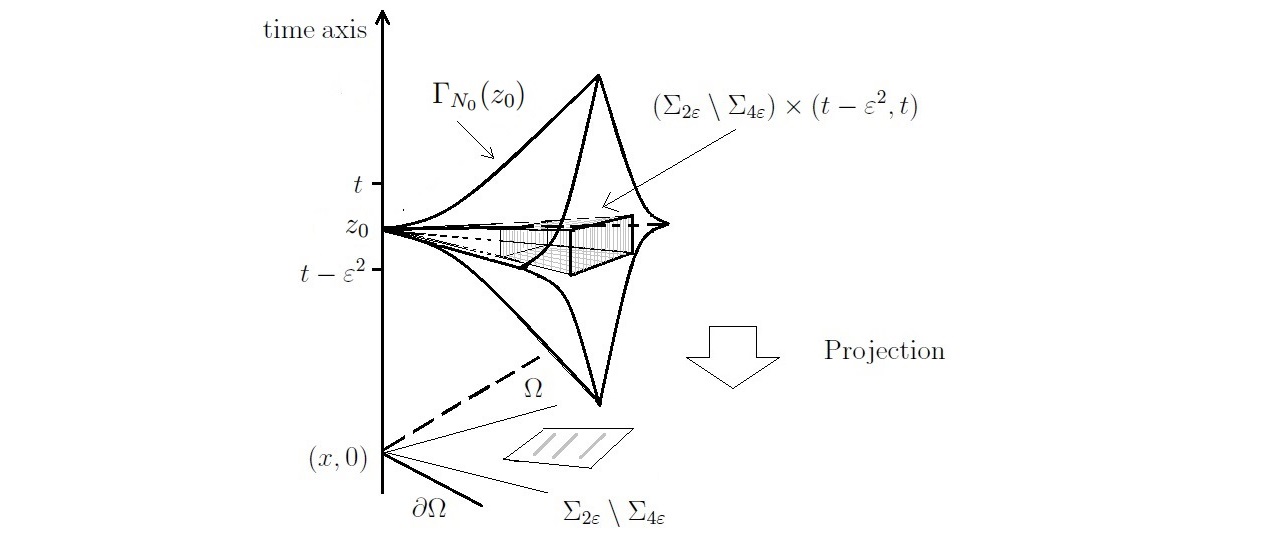}
\makeatletter\def\@captype{figure}\makeatother
\caption{Non-tangential region}\label{pic:1.3}

\item Notation for function spaces.
\begin{enumerate}
  \item $C_0^1(Q_{R})$ denotes the Banach space of functions with continuous one-order derivative with respect to spatial or time variables, requiring each element to vanish near $\partial Q_{R}$.
  \item $L^2(0,T;X)$ represents the Banach space of
$X$-valued functions owning $L^2$-integrability with respect to the time variable (see e.g., \cite[Subsection 5.9.2]{LCE}).

\item $W^{2,1}_q(\Omega_T)$ with $q\in(1,\infty)$ denotes the Sobolev space of functions with second-order spatial derivatives and first-order temporal derivatives in $L^q(\Omega_T)$.

  \item ${_0H^{1,1/2}_q(\partial_{\shortparallel}\Omega_T)}$
  with $q\in(1,\infty)$ denotes
the Sobolev space of functions with one spatial derivative and half
of a time derivative, denoted by $D_t^{1/2}$, in $L^q(\partial_{\shortparallel}\Omega_T)$, requiring each element to vanish near $\partial\Omega\times\{t=0\}$
(see e.g., \cite{Brown-89,Fabes-Riviere-79}).

\end{enumerate}

\end{enumerate}

\begin{remark}
\emph{Finally, we mention that:
(1)
When we say that the multiplicative constant depends on $\partial\Omega$,
it means that the constant relies on the geometric character of the boundary of the domain $\Omega$ (but independent of its diameter $R_0$);
(2) If a zero extension of the solutions $u_\varepsilon$ and $u_0$ of $\eqref{pde:1.1}$ and $\eqref{pde:1.2}$ is needed in the direction
$t<0$ (i.e., a zero-extension), or a natural extension in the direction of $t>T$ (simply extending the solution by increasing the evolution time of the corresponding partial differential equations\footnote{We mention that the parabolic equations
$\eqref{pde:1.1}$ and $\eqref{pde:1.2}$ possess the existence and uniqueness of the solutions, where we usually take $F=0$ in $\Omega_*\setminus\Omega_T$ in accordance with this extension.}), such
operations are not specifically indicated, and the extended functions are still represented using the original notations.}
\end{remark}

\section{Sensitive estimates}\label{section:2}

\noindent
This section is primarily concerned with establishing fluctuation estimates for correctors and flux-correctors by combining the spectral gap condition $\eqref{c:3}$ with sensitive estimate techniques. The key to the proof lies in the existence of stationary correctors, which is based on uniform sensitive estimates for the massive correctors. The conclusion is drawn from the following propositions.

\begin{proposition}[Massive correctors]\label{P:5.1*}
Let $d\geq 2$ and $1\leq p<\infty$. Let $\beta\in(0,1]$. Suppose that the ensemble $\langle\cdot\rangle$ satisfies the stationarity condition
with respect to $\eqref{c:1}$ and the spectral gap condition $\eqref{c:3}$. Let $\phi_{\beta}$ be the
massive corrector, whose component $\phi_{\beta,j}$
satisfies
\begin{equation}\label{massive-corrector}
  \mathcal{A}_{\beta}(\phi_{\beta,i}) :=
  \beta\phi_{\beta,i} + \partial_t \phi_{\beta,i} - \nabla\cdot a\nabla \phi_{\beta,i}
  =\nabla \cdot ae_i\quad\text{in}\quad \mathbb{R}^{d+1}.
\end{equation}
Then, it holds that
\begin{equation}\label{pri:5.2*}
 \bigg\langle\Big(\dashint_{Q_1}|\nabla \phi_{\beta}|^{2}\Big)^p\bigg\rangle^{1/p}\lesssim_{\mu,\lambda_1,d,p} 1.
\end{equation}
Moreover, for any deterministic fields $g\in L^2(\mathbb{R}^{d+1};\mathbb{R}^d)$ and $G\in L^2(\mathbb{R}^{d+1})$, there holds
a stochastic cancellation property:
\begin{equation}\label{pri:5.3*}
  \bigg\langle\Big|\int_{\mathbb{R}^{d+1}}
  \big(\nabla\phi_{\beta}\cdot g +\beta G\phi_{\beta}\big)\Big|^{2p}\bigg\rangle^{\frac{1}{p}}
  \lesssim_{\mu,\lambda_1,d,p} \int_{\mathbb{R}} dt\int_{\mathbb{R}^{d}}\big|(g,\sqrt{\beta} G)(\cdot,t)\big|^2.
\end{equation}
As a result, for any $z\in\mathbb{R}^{d+1}$, the following
fluctuation estimate holds:
\begin{equation}\label{pri:5.4*}
\bigg\langle \Big|\dashint_{Q_1(z)}\phi_{\beta} - \dashint_{Q_1(0)}\phi_{\beta}\Big|^{2p}\bigg\rangle^{1/p}
\lesssim_{\mu,\lambda_1,d,p}1,
\end{equation}
where the multiplicative constant is independent of $\beta$.
\end{proposition}

\begin{proposition}[Flux correctors]\label{P:5.3}
Let $d\geq 2$ and $1\leq p<\infty$. Suppose that the ensemble $\langle\cdot\rangle$ satisfies the stationarity and ergodicity conditions.
One can introduce the following quantities:
\begin{equation}\label{f:5.29}
q_{ij}
:= \bar{a}_{ij}
- a_{ij}
- a_{ik}\partial_k\phi_j,
\qquad
q_{(d+1)j} :=\phi_j-\langle\phi_j\rangle,
\end{equation}
where $1\leq i,j,k\leq d$ are integers. Then, we can verify that
\begin{equation}\label{eq:5.2}
 \big\langle q_{i'j}\big\rangle = 0;
 \qquad \sum_{i'=1}^{d+1}\partial_{i'} q_{i'j} = 0.
\end{equation}
where we denote the time derivative by $\partial_{d+1}$ throughout the lemma and $i'=1,\cdots,d+1$. Moreover,
there exists the so-called flux corrector, denoted by $\sigma_j$, such that
$\sigma_{k'ij}=-\sigma_{ik'j}$ with $k'=1,\cdots,d+1$, and
$\overline{\nabla}\sigma_{kij}$ is stationary, while $\partial_i\sigma_{(d+1)ij}$
(further $\nabla\sigma_{(d+1)ij}$) merely admits the sublinear growth:
\begin{equation}\label{pri:2.20}
\limsup_{R\to\infty}\frac{1}{R^2}\dashint_{Q_R}
 \big|\nabla\sigma_{(d+1)} - \dashint_{Q_{R}}\nabla\sigma_{(d+1)}\big|^2 = 0;
\end{equation}
and the flux correctors satisfy the following equations:
\begin{subequations}
\begin{align}
 &\Delta_{d+1}\sigma_{k'ij}
 =\partial_{k'}q_{ij}-\partial_i q_{k'j};
 \label{eq:5.1a}\\
 &\sum_{k'=1}^{d+1}\partial_{k'}\sigma_{k'ij}
 = q_{ij};
 \label{eq:5.1b}\\
 &
 \sum_{i=1}^{d}\partial_i\sigma_{(d+1)ij}
 = M_j-\phi_j \quad\forall M_j\in\mathbb{R},
 \label{eq:5.1c}
\end{align}
\end{subequations}

Suppose that the ensemble
$\langle\cdot\rangle$ additionally satisfies the spectral gap condition $\eqref{c:3}$.
Let $g\in L^2(\mathbb{R}^{d+1};\mathbb{R}^{d+1})$ be
any deterministic field.
Then, it is obtained that
\begin{equation}\label{pri:5.6a}
\bigg\langle\Big(\dashint_{Q_1(z)}|\overline{\nabla} \sigma_{k'ij}|^{2}\Big)^p\bigg\rangle^{1/p}\lesssim_{\mu,\lambda_1,d,p} 1
\qquad \forall z\in\mathbb{R}^{d+1}\text{~and~}k'=1,\cdots,d+1,
\end{equation}
where it is recalled that $\overline{\nabla}$ represents the $(d+1)$-dimensional
gradient operator.
Moreover, in the case of $k'=1,\cdots,d$, the following estimates hold:
\begin{subequations}
\begin{align}
&\bigg\langle\Big|\int_{\mathbb{R}^{d+1}}
\overline\nabla\sigma_{k'ij}\cdot g\Big|^{2p}\bigg\rangle^{\frac{1}{p}}
\lesssim_{\mu,\lambda_1,d,p} \int_{\mathbb{R}^{d+1}}|g|^2;
\label{pri:5.6b}\\
&\Big\langle \big|\dashint_{Q_1(z)}\sigma_{k'ij} - \dashint_{Q_1(0)}\sigma_{k'ij}\big|^{2p}\Big\rangle^{1/p}
\lesssim_{\mu,\lambda_1,d,p}1
\qquad \forall z\in\mathbb{R}^{d+1},
\label{pri:5.6c}
\end{align}
\end{subequations}
In the case $k'=d+1$, by setting $q:=\frac{2(d+2)}{d+4}$, there correspondingly hold
\begin{subequations}
\begin{align}
&\bigg\langle\Big|\int_{\mathbb{R}^{d+1}}\nabla\sigma_{(d+1)ij}\cdot g\Big|^{2p}\bigg\rangle^{\frac{1}{p}}
\lesssim_{\mu,\lambda_1,p} \int_{\mathbb{R}^{d+1}}|g|^2 
+\Big(\int_{\mathbb{R}^{d+1}}|g|^{q}\Big)^{\frac{2}{q}};
\label{pri:5.7b}\\
&\bigg\langle \big|\dashint_{Q_1(z)}\sigma_{(d+1)ij} - \dashint_{Q_1(0)}\sigma_{(d+1)ij}\big|^{2p}\bigg\rangle^{1/p}
\lesssim_{\mu,\lambda_1,p} \bar{\mu}_d^2(|z|)
\qquad \forall z=(x,t)\in\mathbb{R}^{d+1},
\label{pri:5.7c}
\end{align}
\end{subequations}
where the weight function $\bar{\mu}_d$ is given by $\eqref{pri:5.5}$.
In such a case, one can choose $M_j=\langle\phi_j\rangle$
in the identity $\eqref{eq:5.1c}$.
\end{proposition}

\begin{proposition}[Annealed Calder\'on-Zygmund estimates of Meyer type]\label{P:5.1-ap}
Let $\beta\in[0,1]$.
Let $v_\beta$ and the square-integrable $f$ and $g$ be associated with the
following equations:
\begin{equation}\label{}
  \mathcal{A}_\beta(v_\beta) = \nabla\cdot f + \beta g
  \qquad\text{in}\quad\mathbb{R}^{d+1}
\end{equation}
(or $\mathcal{A}_{\beta}^*(v_\beta) = \nabla\cdot f + \beta g$ in $\mathbb{R}^{d+1}$),
where the massive operator $\mathcal{A}_\beta$ is given in
$\eqref{massive-corrector}$, and its adjoint operator $\mathcal{A}_\beta^*$ is given in $\eqref{pde:5.3-ap}$.
Then, for any ensemble $\langle\cdot\rangle$ of time-dependent coefficient fields satisfying
the ellipticity condition $\eqref{c:1}$, it holds that
\begin{equation}\label{pri:5.1-ap}
\int_{\mathbb{R}^{d+1}}
 \big\langle|(\sqrt{\beta} v_\beta,\nabla v_\beta)|^{p}\big\rangle^{\frac{2}{p}}
 \lesssim \int_{\mathbb{R}^{d+1}}
 \big\langle|(f,\sqrt{\beta} g)|^{p}\big\rangle^{\frac{2}{p}},
\end{equation}
whenever $|p-2|\ll 1$, where the multiplicative constant is independent of
$\beta$.
\end{proposition}

\begin{lemma}[Existence of massive correctors]\label{lemma:2.1}
Assume that the ensemble $\langle\cdot\rangle$ is stationary and ergodic. Then,
for any $\beta\in(0,1]$ and $i\in\{1,\cdots,d\}$, there exists a unique stationary solution $\phi_{\beta}$ satisfying the following equations
$\mathcal{A}_\beta(\phi_{\beta,i}) =\nabla\cdot ae_i$ in $\mathbb{R}^{d+1}$
with the properties $\langle\phi_{\beta}\rangle =
\langle\nabla\phi_{\beta}\rangle= 0$, where the massive operator
$\mathcal{A}_\beta$ is defined in $\eqref{massive-corrector}$.
\end{lemma}

\begin{proof}
The basic idea is the so-called ``elliptie regularization method,''
which was originally developed by J.-L. Lions and E. Magenes \cite{Lions-Magenes-72}.
Since we need to ``lift'' this equation to the probability space to solve it, we mainly adopt the terminology and notations in \cite{Bella-Chiarini-Fehrman-19}.

Let $\{D_i\}_{i=0}^{d}$ be the horizontal derivatives, where $D_0$ represents the horizontal time derivative. Let $\mathcal{D}(D_i)$ denote
the domain of $D_i$. Then, one can introduce the Hilbert space, denoted by
$\mathcal{H}^1$, with the inner product
\begin{equation*}
\langle f,g\rangle:= \langle fg\rangle
+ \langle D_0fD_0g\rangle + \langle Df\cdot Dg\rangle
\qquad\forall f,g\in \mathcal{H}^1,
\end{equation*}
where $D:=(D_1,\cdots,D_d)$ denotes the horizontal spatial gradient.
Then, the dual space of $\mathcal{H}^1$ is denoted by $\mathcal{H}^{-1}$.
By virtue of the above notions, for arbitrarily given
$\beta\in(0,1]$ and $i\in\{1,\cdots,d\}$,
it suffices to find a unique solution
$\phi_{\beta,i}\in \mathcal{H}^1$ such that
\begin{equation}\label{massive-corrector*}
 \beta\langle\phi_{\beta,i}h\rangle
 +\langle D_0\phi_{\beta,i}h\rangle + \big\langle
  Dh\cdot a(D\phi_{\beta,i}+e_i)\big\rangle = 0
\end{equation}
holds for any $h\in\mathcal{H}^1$.

Now, for any $\gamma\in(0,1]$, consider the approximating correctors,
denoted by $\phi_{\beta,i,\gamma}$, satisfying the following equations:
\begin{equation}\label{ap-corrector}
\beta\phi_{\beta,i,\gamma}-\gamma D_0^2\phi_{\beta,i,\gamma}
+ D_0\phi_{\beta,i,\gamma} - D\cdot a(D\phi_{\beta,i,\gamma}+e_i) = 0.
\end{equation}
Let $\mathcal{B}_{\beta,\gamma}$ be the bilinear form, given by
\begin{equation*}
 \mathcal{B}_{\beta,\gamma}(f,g)
 :=\beta\langle fg\rangle + \gamma\langle D_0fD_0g\rangle
 + \langle D_0f g\rangle + \big\langle D g\cdot a Df\big\rangle
 \qquad \forall f,g\in \mathcal{H}^1.
\end{equation*}
It is not hard to verify that
\begin{subequations}
\begin{align}
& |\mathcal{B}_{\beta,\gamma}(f,g)| \lesssim \|f\|_{\mathcal{H}^1} \|g\|_{\mathcal{H}^1}; \label{f:5.45}\\
& \mathcal{B}_{\beta,\gamma}(f,f)
\geq \beta\langle f^2 \rangle + \gamma\langle (D_0f)^2 \rangle
+\mu\langle|Df|^2\rangle \qquad \forall f,g\in\mathcal{H}^1,
\label{f:5.46}
\end{align}
\end{subequations}
where we employ the fact that $\langle D_0f f\rangle = 0$. By noting that $D\cdot ae_i\in \mathcal{H}^{-1}$, it follows from Lax-Milgram theorem that
there exists a unique solution $\phi_{\beta,i,\gamma}\in\mathcal{H}^1$ such that the equation $\eqref{ap-corrector}$ holds for any fixed $\beta,\gamma\in(0,1]$. Moreover, one can obtain the following estimates:
\begin{subequations}
\begin{align}
& \langle\phi_{\beta,i,\gamma}^2 \rangle \lesssim_{\mu,d} 1/\beta;
\label{f:5.47a}\\
& \gamma\langle (D_0\phi_{\beta,i,\gamma})^2 \rangle \lesssim_{\mu,d} 1;
\label{f:5.47b}\\
& \langle |D\phi_{\beta,i,\gamma}|^2 \rangle \lesssim_{\mu,d} 1.
\label{f:5.47c}
\end{align}
\end{subequations}
Moreover, for any $g\in\mathcal{H}^1$, it can be derived that
\begin{equation*}
\begin{aligned}
\langle D_0\phi_{\beta,i,\gamma}g\rangle
&= -\langle Dg\cdot a(D\phi_{\beta,i,\gamma}+e_i)\rangle
-\gamma\langle D_0\phi_{\beta,i,\gamma} D_0g\rangle -\beta\langle\phi_{\beta,i,\gamma}g\rangle\\
&\lesssim_{\mu,d}^{\eqref{f:5.47a},\eqref{f:5.47b},\eqref{f:5.47c}}
\big(1+\sqrt{\gamma}+\sqrt{\beta}\big)\|g\|_{\mathcal{H}^1},
\end{aligned}
\end{equation*}
and this implies that
\begin{equation}\label{f:5.47d}
  \|D_0\phi_{\beta,j,\gamma}\|_{\mathcal{H}^{-1}}\lesssim_{\mu,d}1.
\end{equation}

For any fixed $\beta\in(0,1]$, it follows from the estimates
$\eqref{f:5.47a}$, $\eqref{f:5.47c}$, and $\eqref{f:5.47d}$
that
\begin{equation}\label{f:5.48}
\begin{aligned}
&\phi_{\beta,i,\gamma} \rightharpoonup \phi_{\beta,i} &\quad&\text{weakly~in}\quad L^2(\Sigma); \\
&D\phi_{\beta,i,\gamma} \rightharpoonup \Phi_{\beta,i}
&\quad&\text{weakly~in}\quad L^2(\Sigma;\mathbb{R}^d);\\
&D_0\phi_{\beta,i,\gamma} \rightharpoonup \Xi_{\beta,i}
&\quad&\text{weakly~in}\quad \mathcal{H}^{-1}.
\end{aligned}
\end{equation}
This implies that
\begin{equation*}
\mathcal{B}_{\beta,\gamma}(\phi_{\beta,i,\gamma},g)
\to \beta\langle\phi_{\beta,i}g\rangle
 + \langle \Xi_{\beta,i} g\rangle + \big\langle D g\cdot a \Phi_{\beta,i}\big\rangle,
 \qquad \text{as} \quad \gamma\to 0,
\end{equation*}
holds for any $g\in\mathcal{H}^1$. In view of $\mathcal{B}_{\beta,\gamma}(\phi_{\beta,i,\gamma},g) = D\cdot a e_i$, it
can be concluded that
\begin{equation}\label{pde:5.10}
 \beta\phi_{\beta,i} + \Xi_{\beta,i} - D\cdot a (\Phi_{\beta,i}+e_i) = 0.
\end{equation}
On the other hand, the convergence in
$\eqref{f:5.48}$ also leads to
the equalities $\Phi_{\beta,i} = D\phi_{\beta,i}$ and $\Xi_{\beta,i} = D_0\phi_{\beta,i}$. This, together with $\eqref{pde:5.10}$,
leads to the desired equation $\eqref{massive-corrector*}$. Due the
the linearity of $\eqref{massive-corrector*}$, it may be assumed that
$\phi_{\beta,i}'$ satisfies the same equation, and then by setting
$\bar{\phi}_{\beta,i}:=\phi_{\beta,i}'-\phi_{\beta,i}$, it holds that
\begin{equation*}
 \beta\bar{\phi}_{\beta,i} + D_0\bar{\phi}_{\beta,i}
 - D\cdot a D\bar{\phi}_{\beta,i} = 0.
\end{equation*}
As a result, one can verify that $\bar{\phi}_{\beta,i} = 0$.
The uniqueness of $\phi_{\beta,i}$ to the equation $\eqref{massive-corrector*}$ is proved. Taking expectation on
the both sides of $\eqref{pde:5.10}$, it is not hard to see that
$\langle\phi_{\beta}\rangle = 0$ since $(\Phi_\beta,\Xi_\beta)$ is the curl-free
vector which consequently leads to $\langle\Phi_\beta\rangle =
\langle \Xi_\beta\rangle = 0$ (we refer the reader to \cite[Lemma 3]{Bella-Chiarini-Fehrman-19}
for the details). We have completed the whole proof.
\end{proof}

\begin{corollary}\label{corollary:2.1}
Assume the ensemble $\langle\cdot\rangle$
satisfies the same conditions as in Theorem $\ref{thm:1.0}$.
Let $d\geq 2$, $0<q-2\ll 1$, and $1\leq p<\infty$. Then, it holds that
\begin{subequations}
\begin{align}
&\bigg\langle\Big(\dashint_{Q_1}\big|(\phi,\sigma_{k},\nabla \phi,\nabla\sigma_{k},\nabla\sigma_{d+1})\big|^{q}\Big)^{\frac{p}{q}}
\bigg\rangle^{1/p}\lesssim_{\mu,\lambda_1,d,p,q} 1; \label{pri:2.18} \\
&\bigg\langle\Big(\dashint_{Q_1(z)}|\sigma_{d+1}|^{q}
\Big)^{\frac{p}{q}}\bigg\rangle^{1/p}\lesssim_{\mu,\lambda_1,d,p,q} \bar{\mu}_d(z)
\qquad \forall z\in\mathbb{R}^{d+1}.\label{pri:2.19}
\end{align}
\end{subequations}
\end{corollary}

\begin{proof}
Under the condition that no regularity assumptions are made on the coefficients, the energy estimate can still be lifted to a Meyer-type estimate. To see this, the verification begins with showing that
$\eqref{pri:2.18}$ holds for
$\nabla\phi$, namely
\begin{equation}\label{f:2.20}
\Big(\dashint_{Q_1}|\nabla \phi|^{q}\Big)^{\frac{1}{q}}
\lesssim^{\eqref{pri:5.9-ap}} \Big(\dashint_{Q_2}|\nabla \phi|^{2}\Big)^{\frac{1}{2}} + 1
\lesssim
\sum_{i=1}^N\Big(\dashint_{Q_1(z_i)}|\nabla \phi|^{2}\Big)^{\frac{1}{2}} + 1,
\end{equation}
where $N$ is finite such that $Q_2\subset\cup_{i=1}^NQ_1(z_i)$. Therefore,
we have
\begin{equation}\label{f:2.21}
\bigg\langle \Big(\dashint_{Q_1}|\nabla \phi|^{q}\Big)^{\frac{p}{q}}
\bigg\rangle^{\frac{1}{p}}
\lesssim^{\eqref{f:2.20}}
\sum_{i=1}^N
\bigg\langle\Big(\dashint_{Q_1(z_i)}|\nabla \phi|^{2}\Big)^{\frac{p}{2}}
\bigg\rangle^{\frac{1}{p}} + 1
\lesssim^{\eqref{pri:5.2}} 1.
\end{equation}
Then, one can verify $\eqref{pri:2.18}$ for
$\nabla\sigma_k$. By virtue of the equation $\eqref{eq:5.1a}$ (taking $k'=k$ therein), it holds that
\begin{equation*}
\Big(\dashint_{Q_1}|\nabla \sigma_k|^{q}\Big)^{\frac{1}{q}}
\lesssim \Big(\dashint_{Q_2}|\nabla \sigma_k|^{2}\Big)^{\frac{1}{2}} +
\Big(\dashint_{Q_2}|\nabla\phi|^{q}\Big)^{\frac{1}{q}} + 1
\lesssim^{\eqref{f:2.20}}
\sum_{i=1}^{N'}\Big(\dashint_{Q_1(z_i)}|(\nabla\sigma_k,\nabla \phi)|^{2}\Big)^{\frac{1}{2}} + 1,
\end{equation*}
where $N'$ is finite such that $Q_4\subset\cup_{i=1}^NQ_1(z_i)$.
By taking $\langle(\cdot)^p\rangle^{1/p}$ on the both sides above,
we have
\begin{equation}\label{f:2.22}
\bigg\langle \Big(\dashint_{Q_1}|\nabla \sigma_k|^{q}\Big)^{\frac{p}{q}}
\bigg\rangle^{\frac{1}{p}}
\lesssim
\sum_{i=1}^{N'}
\bigg\langle\Big(\dashint_{Q_1(z_i)}|(\nabla \phi,\nabla\sigma_k)|^{2}\Big)^{\frac{p}{2}}
\bigg\rangle^{\frac{1}{p}} + 1
\lesssim^{\eqref{pri:5.2}} 1,
\end{equation}
where the stationary property of $(\nabla\phi,\nabla\sigma_k)$ is also employed in the last inequality. Moreover, we can verify $\eqref{pri:2.18}$ for
$(\phi,\sigma_k)$. It follows from the triangle inequality and Poincar\'e's inequality that
\begin{equation}\label{f:2.23}
\begin{aligned}
\bigg\langle\Big(\dashint_{Q_1(z)}|(\phi,
&\sigma_k)|^q\Big)^{\frac{p}{q}}
\bigg\rangle^{\frac{1}{p}}
\leq \bigg\langle\Big(\dashint_{Q_1(z)}\big|(\phi,\sigma_k)
-\dashint_{Q_1(z)}(\phi,\sigma_k)\big|^q\Big)^{\frac{p}{q}}
\bigg\rangle^{\frac{1}{p}}
+\Big\langle\Big|\dashint_{Q_1(z)}(\phi,\sigma_k)\Big|^{p}
\Big\rangle^{\frac{1}{p}}\\
&\lesssim
\bigg\langle\Big(\dashint_{Q_1(z)}
\big|(\nabla\phi,\nabla\sigma_k)\big|^q\Big)^{\frac{p}{q}}
\bigg\rangle^{\frac{1}{p}}
+\Big\langle\big|\dashint_{Q_1(z)}(\phi,\sigma_k)\big|^{p}
\Big\rangle^{\frac{1}{p}}
\lesssim^{\eqref{f:2.21},\eqref{f:2.22},\eqref{pri:5.4}}1.
\end{aligned}
\end{equation}
Now, attention is turned to showing the estimate $\eqref{pri:2.18}$ for
$\nabla\sigma_{d+1}$. By using the equation $\eqref{eq:5.1a}$ again (taking $k'=d+1$ therein), we can derive that
\begin{equation}\label{f:2.24}
\begin{aligned}
&\bigg\langle\Big(\dashint_{Q_1}|\nabla \sigma_{d+1}|^{q}\Big)^{\frac{p}{q}}
\bigg\rangle^{\frac{1}{p}}
\lesssim
\bigg\langle\Big(\dashint_{Q_2}|\nabla \sigma_{d+1}|^{2}\Big)^{\frac{p}{2}}\bigg\rangle^{\frac{1}{p}} +
\bigg\langle\Big(\dashint_{Q_2}|(\nabla\phi,\phi)|^{q}\Big)^{\frac{p}{q}}
\bigg\rangle^{\frac{1}{p}}\\
&\lesssim
\sum_{i=1}^{N'}\bigg\langle
\Big(\dashint_{Q_1(z_i)}|\nabla\sigma_{d+1}|^{2}\Big)^{\frac{p}{2}}
\bigg\rangle^{\frac{1}{p}}
+\sum_{i=1}^{N'}\bigg\langle\Big(\dashint_{Q_1(z_i)}
\big|(\nabla\phi,\phi)\big|^{q}\Big)^{\frac{p}{q}}\bigg\rangle^{\frac{1}{p}}
\lesssim^{\eqref{pri:5.2},\eqref{f:2.21},\eqref{f:2.23}} 1.
\end{aligned}
\end{equation}
This, together with $\eqref{f:2.21}$, $\eqref{f:2.22}$, and $\eqref{f:2.23}$,
leads to the desired estimate $\eqref{pri:2.18}$. By using the same argument as that given for $\eqref{f:2.23}$, it is obtained that
\begin{equation*}
\begin{aligned}
\bigg\langle\Big(\dashint_{Q_1(z)}|
&\sigma_{d+1}|^q\Big)^{\frac{p}{q}}
\bigg\rangle^{\frac{1}{p}}
\leq \bigg\langle\Big(\dashint_{Q_1(z)}\big|\sigma_{d+1}
-\dashint_{Q_1(z)}\sigma_{d+1}\big|^q\Big)^{\frac{p}{q}}
\bigg\rangle^{\frac{1}{p}}
+\Big\langle\big|\dashint_{Q_1(z)}\sigma_{d+1}\big|^{p}
\Big\rangle^{\frac{1}{p}}\\
&\lesssim
\bigg\langle\Big(\dashint_{Q_1(z)}
\big|\nabla\sigma_{d+1}\big|^q\Big)^{\frac{p}{q}}
\bigg\rangle^{\frac{1}{p}}
+\Big\langle\big|\dashint_{Q_1(z)}\sigma_{d+1}\big|^{p}
\Big\rangle^{\frac{1}{p}}
\lesssim^{\eqref{f:2.24},\eqref{pri:5.4b}}\bar{\mu}_d(|z|)
\quad \forall z\in\mathbb{R}^{d+1}.
\end{aligned}
\end{equation*}
This yields the stated estimate $\eqref{pri:2.19}$, and the proof
is thereby completed.
\end{proof}

\subsection{Proof of Proposition $\ref{P:5.1*}$}
\noindent
The basic idea and key steps of the proof are explained in its first step    and are therefore not discussed separately. The novelty of the proof is also primarily reflected in the technical aspects of the detailed demonstration. The entire proof is completed in seven steps.

\textbf{Step 1.} Reduction and outline the proof. The key idea is a buckling argument (see e.g., \cite{Gloria-Neukamm-Otto-20,Gloria-Neukamm-Otto-21,Josien-Otto-22}
for the prototype in the elliptic counterparts), and therefore one can derive the stated estimates
$\eqref{pri:5.2*}$ and $\eqref{pri:5.3*}$ simultaneously. Let $\beta\in(0,1]$ be arbitrary.
For any deterministic fields $g\in L^2(\mathbb{R}^{d+1};\mathbb{R}^d)$
and $\sqrt{\beta} G\in L^2(\mathbb{R}^{d+1})$, we start from defining
the associated functional as follows:
\begin{equation}\label{f:5.21-ap}
  F_{\beta,j,g,G}(a):
  = \int_{\mathbb{R}^{d+1}}
  \big(\nabla\phi_{\beta,j}(a) \cdot g
  +\beta\phi_{\beta,j}G\big).
\end{equation}
It is known from \cite[Lemma 1]{Bella-Chiarini-Fehrman-19} and Lemma $\ref{lemma:2.1}$ that
$\langle F_{\beta,j,g} \rangle =0$. In view of the spectral gap condition of
$L^p$-version\footnote{How to derive
the corresponding $L^p$-version from $\eqref{c:3}$,
we refer the reader to \cite[Lemma 3.1]{Josien-Otto-22} for the details.}, it holds that
\begin{equation}\label{f:5.20-ap}
 \big\langle F_{\beta,j,g,G}^{2p}\big\rangle^{1/p}
 \lesssim_{\lambda_1,p}
 \bigg\langle\bigg(\int_{\mathbb{R}^{d+1}}dz\Big(\dashint_{Q_1(z)}
 \frac{\partial F_{\beta,j,g,G}}{\partial a}\Big)^2\bigg)^p\bigg\rangle^{1/p}.
\end{equation}
Thus, to obtain the desired estimate $\eqref{pri:5.3*}$, it suffices to show
\begin{equation}\label{f:5.19-ap}
 \frac{\partial F_{\beta,j,g,G}(a)}{\partial a(x,t)} = -\nabla v\otimes(\nabla\phi_{\beta,j}+e_j)(x,t),
\end{equation}
where $v$, as an auxiliary function, satisfies the following equation:
\begin{equation}\label{pde:5.3-ap}
 \mathcal{A}_{\beta}^*(v):=\beta v
 -\partial_t v - \nabla\cdot a^*\nabla v =
 \beta G -\nabla\cdot g
 \quad\text{in}\quad\mathbb{R}^{d+1}.
\end{equation}

For ease of later reference, let
\begin{equation}\label{f:5.14-ap}
   u_{\beta,j} = \phi_{\beta,j} + x_j.
\end{equation}

Thus, by plugging the estimate $\eqref{f:5.19-ap}$ back into $\eqref{f:5.20-ap}$,
one of two crucial ingredients for $\eqref{pri:5.3*}$ can be derived, i.e.,
\begin{equation}\label{f:5.15-ap}
\bigg\langle
 \Big|\int_{\mathbb{R}^{d+1}}
 \big(\nabla\phi_{\beta,j} \cdot g + \beta G\phi_{\beta,j}\big)
 \Big|^{2p}\bigg\rangle^{\frac{1}{p}}
\lesssim
\Big\langle\Big(\dashint_{Q_1}|\nabla u_{\beta,j}|^{2}\Big)^{p}\Big\rangle^{\frac{1}{p}}
\int_{\mathbb{R}^{d+1}} |(g,\sqrt{\beta} G)|^2.
\end{equation}

Moreover, the other crucial ingredient is
the ``reverse H\"older inequality,'' given by
\begin{equation}\label{f:5.18-ap}
\begin{aligned}
\Big\langle\Big(\dashint_{Q_1}|\nabla u_{\beta,j}|^{2}\Big)^{p}\Big\rangle^{\frac{1}{p}}
&\lesssim R^{(d+2)(1-1/p)}\bigg\langle \Big(\dashint_{-(4R)^2}^{0}
dt|\nabla (\phi_{\beta,j})_{\vartheta R}(0,t)|^{2}\Big)^p\bigg\rangle^{1/p}
+ R^{(d+2)(1-1/p)}.
\end{aligned}
\end{equation}

To carry out the buckling argument, we claim a technical result:
\begin{equation}\label{f:5.17-ap}
\bigg\langle \Big(\dashint_{-(4R)^2}^{0}
dt|\nabla (\phi_{\beta,j})_{\vartheta R}(0,t)|^{2}\Big)^p\bigg\rangle^{1/p}
\lesssim R^{-d-2}
\Big\langle\Big(\dashint_{Q_1}|\nabla u_{\beta,j}|^{2}\Big)^{p}\Big\rangle^{\frac{1}{p}},
\end{equation}
where the stated estimate $\eqref{f:5.15-ap}$ has already been employed for this result (see Step 5). Now, inserting the estimate $\eqref{f:5.17-ap}$ back into $\eqref{f:5.18-ap}$, one can immediately acquire that
\begin{equation*}
\Big\langle\Big(\dashint_{Q_1}|\nabla u_{\beta,j}|^{2}\Big)^{p}\Big\rangle^{\frac{1}{p}}
\lesssim R^{-(d+2)/p}\Big\langle\Big(\dashint_{Q_1}|\nabla u_{\beta,j}|^{2}\Big)^{p}\Big\rangle^{\frac{1}{p}} + R^{(d+2)(1-1/p)},
\end{equation*}
and by choosing $R>1$ to be sufficiently large, we can obtain the desired estimate $\eqref{pri:5.2*}$. Furthermore, this together with $\eqref{f:5.15-ap}$ consequently leads to the stated estimate $\eqref{pri:5.3*}$.

\textbf{Step 2.} Arguments for $\eqref{f:5.19-ap}$.
Taking the functional derivatives $\frac{\partial}{\partial a_{ik}(x,t)}$
on the both sides of $\eqref{f:5.21-ap}$, we have
\begin{equation*}
 \frac{\partial F_{\beta,j,g,G}(a)}{\partial a_{ik}(x,t)} =
 \int_{\mathbb{R}^{d+1}}\Big(\nabla\frac{\partial \phi_{\beta,j}}{\partial a_{ik}(x,t)} \cdot g + \beta G\frac{\partial \phi_{\beta,j}}{\partial a_{ik}(x,t)}\Big)
 = \int_{\mathbb{R}^{d+1}}\frac{\partial \phi_{\beta,j}}{\partial a_{ik}(x,t)}\big(\beta G -\nabla \cdot g\big).
\end{equation*}
Recalling the auxiliary equation $\eqref{pde:5.3-ap}$, it is found that
\begin{equation}\label{f:5.22-ap}
\begin{aligned}
\frac{\partial F_{j,g}(a)}{\partial a_{ik}(x,t)}
=^{\eqref{pde:5.3-ap}}  \int_{\mathbb{R}^{d+1}}\frac{\partial \phi_{\beta,j}}{\partial a_{ik}(x,t)}
 \mathcal{A}_{\beta}^*(v)
= \int_{\mathbb{R}^{d+1}}
\mathcal{A}_\beta\big(\frac{\partial \phi_{\beta,j}}{\partial a_{ik}(x,t)}\big)v.
\end{aligned}
\end{equation}
By taking $\frac{\partial}{\partial a_{ik}(x,t)}$ on the both sides of the
approximating corrector's equation, it is obtained that
\begin{equation}\label{f:5.23-ap}
\mathcal{A}_{\beta}\big(\frac{\partial \phi_{\beta,j}}{\partial a_{ik}(x,t)}\big)
= \nabla\cdot \delta(x,t)e_ie_k(\nabla\phi_{\beta,j}+e_j),
\end{equation}
where $\delta(x,t)$ is the Dirac delta function with pole at $(x,t)\in\mathbb{R}^{d+1}$.
Therefore, plugging the estimate $\eqref{f:5.23-ap}$ back into
$\eqref{f:5.22-ap}$, we derive the stated estimate $\eqref{f:5.19-ap}$.

\textbf{Step 3.} Arguments for $\eqref{f:5.15-ap}$.
It follows from the spectral gap inequality of the $L^p$-version that
\begin{equation}\label{f:5.7-ap}
\begin{aligned}
 \Big\langle
 F_{\beta,j,g,G}^{2p}\Big\rangle^{\frac{1}{p}}
& \lesssim^{\eqref{f:5.20-ap},\eqref{f:5.19-ap}} \bigg\langle \Big( \int_{\mathbb{R}^{d+1}}dz\big|\dashint_{Q_1(z)}\nabla v\otimes(\nabla\phi_{\beta,j}+e_j)\big|^2\Big)^p\bigg\rangle^{\frac{1}{p}} \\
& = \sup_{\langle H^{2p'}\rangle =1}
\bigg\langle H^2\Big(\int_{\mathbb{R}^{d+1}}dz\big|\dashint_{Q_1(x,t)}\nabla v\otimes(\nabla\phi_{\beta,j}+e_j)\big|^2\Big)
\bigg\rangle,
\end{aligned}
\end{equation}
where $1/p+1/p' =1$ with $0<p'-1\ll 1$. Therefore, for any $H$ with $\langle H^{2p'}\rangle =1$,
the following computations are performed:
\begin{equation}\label{f:5.6-ap}
\begin{aligned}
&\bigg\langle \int_{\mathbb{R}^{d+1}}
\big|\dashint_{Q_1(x,t)}\nabla (vH)\otimes(\nabla\phi_{\beta,j}+e_j)\big|^2 dxdt\bigg\rangle\\
&\leq^{\eqref{f:5.14-ap}} \int_{\mathbb{R}^{d+1}} \Big\langle\Big(\dashint_{Q_1(x,t)}
|\nabla(vH)|^{2}\Big)^{p'}\Big\rangle^{\frac{1}{p'}}
\Big\langle\Big(\dashint_{Q_1(x,t)}|\nabla u_{\beta,j}|^{2}\Big)^p\Big\rangle^{\frac{1}{p}} dxdt\\
&= \Big\langle\Big(\dashint_{Q_1}|\nabla u_{\beta,j}|^{2}\Big)^{p}\Big\rangle^{\frac{1}{p}}
\int_{\mathbb{R}^{d+1}}\Big\langle\Big(\dashint_{Q_1(x,t)}|\nabla(vH)|^{2}
\Big)^{p'}\Big\rangle^{\frac{1}{p'}}dxdt\\
&\leq
\Big\langle\Big(\dashint_{Q_1}|\nabla u_{\beta,j}|^{2}\Big)^{p}\Big\rangle^{\frac{1}{p}}
\int_{\mathbb{R}^{d+1}}\big\langle|\nabla(vH)|^{2p'}
\big\rangle^{\frac{1}{p'}}
\lesssim^{\eqref{pri:5.1-ap}}
\Big\langle\Big(\dashint_{Q_1}|\nabla u_{\beta,j}|^{2}\Big)^{p}\Big\rangle^{\frac{1}{p}}
\int_{\mathbb{R}^{d+1}} \big|(g,\sqrt{\beta} G)\big|^2.
\end{aligned}
\end{equation}
where we employ the stationarity property of $\nabla u_{\beta,j}$ in the equality, and Minkowski's inequality and Fubini's theorem in the second inequality.
For the last inequality above, the fact that $\langle|H|^{2p'}\rangle = 1$
is also used. Plugging the estimate $\eqref{f:5.6-ap}$
back into $\eqref{f:5.7-ap}$, one can derive the stated estimate $\eqref{f:5.15-ap}$.

\textbf{Step 4.}  Arguments for $\eqref{f:5.18-ap}$.
Due to the stationarity and H\"older's inequality in $l^p$-spaces, the following starting point is taken:
\begin{equation}\label{f:5.12-ap}
\begin{aligned}
&\Big\langle\Big(\dashint_{Q_1}|\nabla u_{\beta,j}|^{2}\Big)^{p}\Big\rangle^{\frac{1}{p}}
\leq \Big\langle\dashint_{Q_R}dz\Big(\dashint_{Q_1(z)}|\nabla u_{\beta,j}|^{2}\Big)^{p}\Big\rangle^{\frac{1}{p}}
\lesssim_d R^{-(d+2)/p}\bigg\langle
\sum_{i}\Big(\dashint_{Q_1(z_i)}
|\nabla u_{\beta,j}|^{2}\Big)^p\bigg\rangle^{1/p}\\
&\lesssim R^{-(d+2)/p}
\bigg\langle\Big(\sum_{i}\dashint_{Q_1(z_i)}
|\nabla u_{\beta,j}|^{2}\Big)^{p}\bigg\rangle^{1/p}
\lesssim R^{(d+2)(1-1/p)}
\bigg\langle\Big(\dashint_{Q_{R}}dz
\dashint_{Q_1(z)}|\nabla u_{\beta,j}|^{2}\Big)^{p}\bigg\rangle^{1/p}\\
&\lesssim
R^{(d+2)(1-1/p)}
\bigg\langle\Big(\dashint_{Q_{2R}}
|\nabla u_{\beta,j}|^{2}\Big)^{p}\bigg\rangle^{1/p}.
\end{aligned}
\end{equation}

In view of the massive equation $\eqref{massive-corrector}$
and the estimate $\eqref{pri:6.3-ap}$, it is found that
\begin{equation}\label{f:5.10-ap}
\begin{aligned}
\bigg\langle\Big(\dashint_{Q_{2R}}
|\nabla \phi_{\beta,j}|^{2}
&\Big)^{p}\bigg\rangle^{1/p}
\leq
\theta \bigg\langle\bigg(\dashint_{I_{4R}}dt\dashint_{\partial B_{4R}}dx\dashint_{B_r(x)}|\nabla \phi_{\beta,j}(\cdot,t)|^2\bigg)^p
\bigg\rangle^{1/p} \\
&+ C_{\theta}\Big(\frac{r}{R}\Big)^2
\bigg\langle\Big(\dashint_{Q_{6R}}|\nabla \phi_{\beta,j}|^2\Big)^{p}\bigg\rangle^{1/p} +
 C_{\theta}\bigg\langle\Big(\dashint_{Q_{4R}}|\nabla (\phi_{\beta,j})_r|^2\Big)^p\bigg\rangle^{1/p} + C_{d,\mu}.
\end{aligned}
\end{equation}
By duality, we have
\begin{equation*}
\begin{aligned}
\bigg\langle\bigg(\dashint_{I_{4R}}dt\dashint_{\partial B_{4R}}dx\dashint_{B_r(x)}|\nabla \phi_{\beta,j}(\cdot,t)|^2\bigg)^p
\bigg\rangle^{1/p}
=\sup_{\langle H^{2p'}\rangle =1}
\bigg\langle
\dashint_{I_{4R}}dt\dashint_{\partial B_{4R}}dx\dashint_{B_r(x)}|H\nabla \phi_{\beta,j}(\cdot,t)|^2\bigg\rangle,
\end{aligned}
\end{equation*}
and it follows from the stationarity of $\nabla \phi_{\beta,j}$ that
\begin{equation*}
\begin{aligned}
\bigg\langle
\dashint_{I_{4R}}dt\dashint_{\partial B_{4R}}dx\dashint_{B_r(x)}|H\nabla \phi_{\beta,j}(\cdot,t)|^2\bigg\rangle
&= \dashint_{I_{4R}}dt\dashint_{\partial B_{4R}}dx
\Big\langle \dashint_{B_r(x)}|H\nabla \phi_{\beta,j}(\cdot,t)|^2 \Big\rangle \\
&= \Big\langle |H\nabla \phi_{\beta,j}(0,0)|^2 \Big\rangle
=\dashint_{Q_{2R}}\big\langle|H\nabla \phi_{\beta,j}|^2\big\rangle\\
&\leq \bigg\langle\Big(\dashint_{Q_{2R}}|\nabla \phi_{\beta,j}|^2\Big)^p
\bigg\rangle^{1/p},
\end{aligned}
\end{equation*}
where we employ the fact that $\langle H^{2p'}\rangle = 1$ in the last step above. This implies that
\begin{equation}\label{f:5.8-ap}
\bigg\langle\bigg(\dashint_{I_{4R}}dt\dashint_{\partial B_{4R}}dx\dashint_{B_r(x)}|\nabla \phi_{\beta,j}(\cdot,t)|^2\bigg)^p
\bigg\rangle^{1/p}
\leq \bigg\langle\Big(\dashint_{Q_{2R}}|\nabla \phi_{\beta,j}|^2\Big)^p
\bigg\rangle^{1/p}.
\end{equation}

Moreover, it is noticed that there exist $\{z_i\}_{i=1}^N\subset Q_{6R}$ such that
\begin{equation*}
 \dashint_{Q_{6R}}|\nabla \phi_{\beta,j}|^2
 \lesssim_d \sum_{i=1}^{N}\dashint_{Q_{2R}(z_i)}|\nabla \phi_{\beta,j}|^2,
\end{equation*}
and this provides us with the following:
\begin{equation}\label{f:5.9-ap}
 \bigg\langle\Big(\dashint_{Q_{6R}}|\nabla \phi_{\beta,j}|^2\Big)^p\bigg\rangle^{1/p}
 \lesssim \sum_{i=1}^{N}
 \bigg\langle\Big(\dashint_{Q_{2R}(z_i)}|\nabla \phi_{\beta,j}|^2\Big)^p\bigg\rangle^{1/p}
 \lesssim \bigg\langle\Big(\dashint_{Q_{2R}}|\nabla \phi_{\beta,j}|^2\Big)^p\bigg\rangle^{1/p},
\end{equation}
where the last step is also due to the stationarity of $\nabla \phi_{\beta,j}$.

Thus, plugging the estimates $\eqref{f:5.8-ap}$ and $\eqref{f:5.9-ap}$ back into
$\eqref{f:5.10-ap}$, we now arrive at
\begin{equation*}
\begin{aligned}
\bigg\langle\Big(\dashint_{Q_{2R}}
|\nabla \phi_{\beta,j}|^{2}\Big)^{p}\bigg\rangle^{1/p}
&\leq
\theta \bigg\langle\Big(\dashint_{Q_{2R}}|\nabla \phi_{\beta,j}|^2\Big)^p
\bigg\rangle^{1/p} + C_{\mu,d}\\
&+ C_{\theta}\Big(\frac{r}{R}\Big)^2
\bigg\langle\Big(\dashint_{Q_{2R}}|\nabla \phi_{\beta,j}|^2\Big)^{p}\bigg\rangle^{1/p} +
 C_{\theta}\bigg\langle\Big(\dashint_{Q_{4R}}|\nabla (\phi_{\beta,j})_r|^2\Big)^p\bigg\rangle^{1/p}.
\end{aligned}
\end{equation*}
By choosing $r=\vartheta R$ with $\vartheta\in(0,1)$ sufficiently small,
it consequently follows that
\begin{equation}\label{f:5.11-ap}
\begin{aligned}
\bigg\langle\Big(\dashint_{Q_{2R}}
|\nabla \phi_{\beta,j}|^{2}\Big)^{p}\bigg\rangle^{1/p}
&\lesssim \bigg\langle\Big(\dashint_{Q_{4R}}|\nabla (\phi_{\beta,j})_{\vartheta R}|^2\Big)^p\bigg\rangle^{1/p} + 1\\
&\lesssim \bigg\langle \Big(\dashint_{I_{4R}}
dt|\nabla (\phi_{\beta,j})_{\vartheta R}(0,t)|^{2}\Big)^p\bigg\rangle^{1/p} + 1,
\end{aligned}
\end{equation}
where the stationarity of $\nabla \phi_{\beta,j}$ with respect to the space variable is also used
in the last step.  Inserting $\eqref{f:5.11-ap}$ into $\eqref{f:5.12-ap}$, we obtain the stated estimate $\eqref{f:5.18-ap}$.

\textbf{Step 5.} Arguments for $\eqref{f:5.17-ap}$.
Note that, for any $k=1,\cdots,d$, it holds that
\begin{equation}\label{f:5.16-ap}
\begin{aligned}
\bigg\langle \Big(\dashint_{I_{4R}}
dt|\partial_k (\phi_{\beta,j})_{\vartheta R}(0,t)|^{2}\Big)^p\bigg\rangle^{1/p}
&\lesssim_d\frac{1}{R^2}\sup_{\langle H^{2p'}\rangle=1\atop
\|\zeta\|_{L^2(\mathbb{R})}=1}\bigg\langle
\Big|\int_{\mathbb{R}}\zeta
\int_{\mathbb{R}^d}\eta_{\vartheta R}\partial_k \phi_{\beta,j}\Big| H^2\bigg\rangle\\
&\lesssim
\frac{1}{R^2}\sup_{\|\zeta\|_{L^2(\mathbb{R})}=1}\bigg\langle
\Big|\int_{\mathbb{R}}\zeta
\int_{\mathbb{R}^d}\eta_{\vartheta R}\partial_k \phi_{\beta,j}\Big|^{2p}\bigg\rangle.
\end{aligned}
\end{equation}
One can fix $\tilde{g}=(0,\cdots,\overbrace{\eta_{\vartheta R}}^{k^{\text{th}}},\cdots,0)\zeta$ to be a vector form with respect to the spacial variable,
and by rewriting the right-hand side above, it is obtained that
\begin{equation}\label{f:5.13-ap}
\begin{aligned}
\bigg\langle
\Big|\int_{\mathbb{R}}\zeta
\int_{\mathbb{R}^d}\eta_{\vartheta R}\partial_k \phi_{\beta,j}\Big|^{2p}\bigg\rangle
&=\bigg\langle
\Big|\int_{\mathbb{R}^{d+1}}
\nabla \phi_{\beta,j}\cdot \tilde{g}\Big|^{2p}\bigg\rangle  \\
&\lesssim^{\eqref{f:5.15-ap}}
\Big\langle\Big(\dashint_{Q_1}|\nabla u_{\beta,j}|^{2}\Big)^{p}\Big\rangle^{\frac{1}{p}}
\int_{\mathbb{R}^{d+1}} |\tilde{g}|^2
\lesssim R^{-d}\Big\langle\Big(\dashint_{Q_1}|\nabla u_{\beta,j}|^{2}\Big)^{p}\Big\rangle^{\frac{1}{p}}.
\end{aligned}
\end{equation}
As a result, plugging the estimate $\eqref{f:5.13-ap}$ back into $\eqref{f:5.16-ap}$, we derive that $\eqref{f:5.17-ap}$.

\textbf{Step 6.} Arguments for $\eqref{pri:5.4*}$.
For any $z:=(x,t)\in\mathbb{R}^{d+1}$, recourse is made to constructing an auxiliary equation
\begin{equation}\label{pde:5.4-ap}
\beta v_z + \partial_t v_z -\Delta v_z = \frac{1}{|Q_1|}\big(\textbf{1}_{Q_1(z)} - \textbf{1}_{Q_1(0)}\big) \qquad \text{in}\quad \mathbb{R}^{d+1},
\end{equation}
with a claimed estimate (which will be demonstrated later):
\begin{equation}\label{f:5.24-ap}
\sup_{t\in\mathbb{R}}\int_{\mathbb{R}^d}|v_z(t)|^2
+\int_{\mathbb{R}^{d+1}}\big|(\sqrt{\beta} v_z,\nabla v_z)\big|^2 \lesssim 1.
\end{equation}

By multiplying $\phi_{\beta,j}$ on the both sides of $\eqref{pde:5.4-ap}$, it holds that
\begin{equation*}
\dashint_{Q_1(z)}\phi_{\beta,j} - \dashint_{Q_1(0)}\phi_{\beta,j}
=\int_{\mathbb{R}^{d+1}}\nabla\phi_{\beta,j}\cdot\nabla v_z
+ 2\beta\int_{\mathbb{R}^{d+1}}v_z\phi_{\beta,j}
+ \int_{\mathbb{R}^{d+1}}\nabla v_z\cdot a(\nabla\phi_{\beta,j}+e_j).
\end{equation*}
If there holds
\begin{equation}\label{f:5.49}
 \bigg\langle \Big|\int_{\mathbb{R}^{d+1}}\nabla v_z\cdot a(\nabla\phi_{\beta,j}+e_j)\Big|^{2p}\bigg\rangle^{1/p}
 \lesssim \int_{\mathbb{R}^{d+1}}|\nabla v_z|^2,
\end{equation}
then, we have
\begin{equation*}
\begin{aligned}
&\Big\langle \big|\dashint_{Q_1(z)}\phi_{\beta,j} - \dashint_{Q_1(0)}\phi_{\beta,j}\big|^{2p}\Big\rangle^{1/p}\\
&\lesssim \bigg\langle \Big|\int_{\mathbb{R}^{d+1}}\big(\nabla\phi_{\beta,j}\cdot\nabla v_z
+2\beta v_z\phi_{\beta,j}\big) \Big|^{2p}\bigg\rangle^{1/p}
+ \bigg\langle \Big|\int_{\mathbb{R}^{d+1}}\nabla v_z\cdot a(\nabla\phi_{\beta,j}+e_j)\Big|^{2p}\bigg\rangle^{1/p}\\
&\lesssim^{\eqref{pri:5.3*},\eqref{f:5.49}} \int_{\mathbb{R}^{d+1}}\big|(\sqrt{\beta} v_z,\nabla v_z)\big|^2
\lesssim^{\eqref{f:5.24-ap}} 1,
\end{aligned}
\end{equation*}
which therefore leads to the desired estimate $\eqref{pri:5.4*}$.

Arguments for $\eqref{f:5.24-ap}$. Multiplying $v_z$ on the both sides of
$\eqref{pde:5.4-ap}$, for any $t'\in\mathbb{R}$, it holds that
\begin{equation*}
\int_{\mathbb{R}^d}|v_z(\cdot,t')|^2
+\int_{-\infty}^{t'}\int_{\mathbb{R}^d}\big(\beta|v_z| + |\nabla v_z|^2\big)
\leq  \dashint_{Q_1(z)}|v_z| + \dashint_{Q_1(0)}|v_z|,
\end{equation*}
and this implies that (denoting the critical Sobolev embedding index by $2^*:=2+4/d$ and using Lemma $\ref{lemma:4.1-ap}$ in the following calculation):
\begin{equation}\label{f:5.2-n}
\begin{aligned}
&\sup_{t\in\mathbb{R}}\int_{\mathbb{R}^d}|v_z(\cdot,t)|^2
+\int_{\mathbb{R}^{d+1}}\big|(\sqrt{\beta} v_z,\nabla v_z)\big|^2
\leq \dashint_{Q_1(z)}|v_z| + \dashint_{Q_1(0)}|v_z|\\
&\leq \Big(\dashint_{Q_1(z)}|v_z|^{2^*} + \dashint_{Q_1(0)}|v_z|^{2^*}\Big)^{1/2^*}
\lesssim \Big(\int_{\mathbb{R}^{d+1}}|v_z|^{2^*}\Big)^{1/2^*}
\lesssim^{\eqref{pri:6.8-ap}}
\bigg(\sup_{t\in\mathbb{R}}\int_{\mathbb{R}^d}|v_z(\cdot,t)|^2
+\int_{\mathbb{R}^{d+1}}\big|\nabla v_z\big|^2
\bigg)^{\frac{1}{2}}\\
&\lesssim \bigg(\sup_{t\in\mathbb{R}}\int_{\mathbb{R}^d}|v_z(\cdot,t)|^2
+\int_{\mathbb{R}^{d+1}}\big|(\sqrt{\beta} v_z,\nabla v_z)\big|^2
\bigg)^{\frac{1}{2}},
\end{aligned}
\end{equation}
which immediately gives the stated estimate $\eqref{f:5.24-ap}$.

\textbf{Step 7.} Arguments for $\eqref{f:5.49}$. Due to the finite energy condition
$\eqref{f:5.24-ap}$, we have $\int_{\mathbb{R}^{d+1}}\nabla v_z = 0$. Therefore, the following decomposition can be carried out:
\begin{equation}\label{f:5.51}
  F:=\int_{\mathbb{R}^{d+1}}\nabla v_z\cdot a(\nabla\phi_{\beta,i}+e_i)
  = \int_{\mathbb{R}^{d+1}}\nabla v_z\cdot a\nabla\phi_{\beta,i}
  + \int_{\mathbb{R}^{d+1}}\nabla v_z\cdot a e_i := F_1 + F_2.
\end{equation}
It is not hard to see that $\langle F_1\rangle = \langle F_2 \rangle = 0$.
One may rewrite $F_1$ as follows:
\begin{equation*}
 F_1 = \int_{\mathbb{R}^{d+1}}\nabla v_z\cdot a\nabla\phi_{\beta,i}
 = \int_{\mathbb{R}^{d+1}}\nabla\phi_{\beta,i}\cdot \underbrace{a^{*}\nabla v_z}_{f_a}
 = \int_{\mathbb{R}^{d+1}}\nabla\phi_{\beta,i}\cdot f_a.
\end{equation*}
Then, one can construct the auxiliary equation $\mathcal{A}_\beta^{*}(w)=\nabla\cdot f_a$ in $\mathbb{R}^{d+1}$
as in $\eqref{pde:5.3-ap}$, and by the same computations as those given in Step 2, we can derive that
\begin{equation*}
 \frac{\partial F_{1}(a)}{\partial a(x,t)} = \big[-\nabla w\otimes(\nabla\phi_{\beta,i}+e_i) + \nabla\phi_{\beta,i}\otimes\nabla v_z\big](x,t).
\end{equation*}
It consequently follows from a routine computation as in Step 3 that
\begin{equation}\label{f:5.50}
 \big\langle F_1^{2p}\big\rangle^{1/p} \lesssim \int_{\mathbb{R}^{d+1}}
 |\nabla v_z|^2,
\end{equation}
where it noted that $|f_a|\lesssim_{\mu} |\nabla v_z|$. We mention that
the above estimate $\eqref{f:5.50}$ can be roughly regarded as an application of $\eqref{pri:5.3*}$ with the choices $g=f_a$ and $G=0$.
The computations for $F_2$ are much easier, and they reduce to showing
\begin{equation*}
\frac{\partial F_{2}(a)}{\partial a(x,t)} = \big[\nabla v_z\otimes e_i\big](x,t).
\end{equation*}
By applying the spectral gap of the $L^p$-version, one can immediately obtain that
\begin{equation}\label{f:5.52}
 \big\langle F_2^{2p}\big\rangle^{1/p}
\lesssim
\bigg\langle \Big( \int_{\mathbb{R}^{d+1}}\big|\dashint_{Q_1(x,t)}\nabla v_z\otimes e_i\big|^2
 dxdt\Big)^p\bigg\rangle^{\frac{1}{p}}
 \lesssim \int_{\mathbb{R}^{d+1}}|\nabla v_z|^2,
\end{equation}
using Minkowski's inequality.  Thus, combining the estimates
$\eqref{f:5.51}, \eqref{f:5.50}$, and $\eqref{f:5.52}$ leads to
the stated estimate $\eqref{f:5.49}$.
\qed

\subsection{Proof of Proposition $\ref{P:5.3}$}
%
\noindent
The proof is divided into qualitative and quantitative parts. We begin with the qualitative analysis, employing spectral representations and functional calculus (see the pioneering work of Papanicolaou and Varadhan \cite{Papanicolaou-Varadhan-79}), as well as applying the established results of Bella et al. \cite{Bella-Chiarini-Fehrman-19}. Subsequently, we utilize the spectral gap condition $\eqref{c:3}$ to investigate the fluctuation properties of the flux correctors.

\noindent
\textbf{Part I.} Arguments for $\eqref{eq:5.1a}$, $\eqref{eq:5.1b}$, and
$\eqref{eq:5.1c}$.
In view of the equations $\eqref{corrector}$ and $\eqref{eq:1.1}$, it is not hard to derive the identities $\eqref{eq:5.2}$. Let $j=1,\cdots,d$
and $i'=1,\cdots,d+1$.
The flux corrector $\sigma=(\sigma_{k'i'j})$ with $k'=1,\cdots,d+1$ can be constructed as follows. Consider the equation:  $\overline{\nabla}\cdot b_{i'j} = q_{i'j}$ in $\mathbb{R}^{d+1}$. Based on the spectral calculus of the operator,
it can be found that
\begin{equation}\label{eq:5.3}
b_{k'i'j}(a,z) = \big(e^{z\cdot D}-I\big)(D^*\cdot D)^{-1}D_{k'}q_{i'j},
\end{equation}
where the ``horizontal derivatives'' $D=(D_1,\cdots,D_{d+1})$ are defined as the generators of $(d+1)$-shift $L^2$-unitary groups $e^{z\cdot D}$, and $D_j:=\sqrt{-1}\int_{\mathbb{R}^{d+1}}\lambda_jdE_{\lambda}$,
with $E_{\lambda}$ denoting the spectral measure of a self-adjoint operator. Also, in this regard, it is mentioned that $e^{z\cdot D}:=\int_{\mathbb{R}^{d+1}}
\exp\{\sqrt{-1}z\cdot\lambda\}dE_{\lambda}$, and $e^{z\cdot D}f(a)=
f(\tau_za)$, where $\{\tau_z\}_{z\in\mathbb{R}^{d+1}}$ represents a family of shift operators acting on the configuration space governed by the stationary and ergodic ensemble $\langle\cdot\rangle$. Due to the second identity
in $\eqref{eq:5.2}$, it follows from the representation in $\eqref{eq:5.3}$ that
\begin{equation}\label{eq:5.4}
  \sum_{k'=1}^{d+1} b_{k'k'j} = 0.
\end{equation}

Let
\begin{equation}\label{eq:5.6}
 \sigma_{k'i'j}(a,z) := b_{k'i'j}(a,z)-b_{i'k'j}(a,z)
 =\big(e^{z\cdot D}-I\big)(D^*\cdot D)^{-1}\big(D_{k'}q_{i'j}-D_{i'}q_{k'j}\big),
\end{equation}
which immediately leads to the antisymmetry of $\sigma_j$ (i.e., $\sigma_{k'i'j}=
-\sigma_{i'k'j}$). Moreover, it is found that
\begin{equation*}
\begin{aligned}
\sum_{k'=1}^{d+1}\partial_{k'}\sigma_{k'ij}(a,z)
&=\sum_{k'=1}^{d+1}\Big(\partial_{k'} b_{k'ij}(a,z) - \partial_{k'} b_{ik'j}(a,z)\Big)\\
&=e^{z\cdot D}q_{ij} - \sum_{k'=1}^{d+1}\partial_i b_{k'k'j}
=^{\eqref{eq:5.4}}e^{z\cdot D}q_{ij} = q_{ij}(\tau_za),
\end{aligned}
\end{equation*}
which gives the identity $\eqref{eq:5.1b}$.
To see the equation $\eqref{eq:5.1a}$, it follows from $\eqref{eq:5.6}$ that
\begin{equation}\label{eq:5.5}
\Delta_{d+1}\sigma_{k'ij}(a,z)
= e^{z\cdot D}\big(D_{k'}q_{ij}-D_iq_{k'j}\big)
= \big(D_{k'}q_{ij}-D_iq_{k'j}\big)(\tau_{z}a)
= (\partial_{k'}q_{ij}-\partial_iq_{k'j}\big)(a,z),
\end{equation}
which gives the stated identity $\eqref{eq:5.1a}$.
For the case $k'=d+1$, taking the partial derivative $\partial_i$ on
the both sides of $\eqref{eq:5.5}$ and summing from $i=1$ to $i=d$, we obtain that
\begin{equation*}
\Delta_{d+1}\partial_i\sigma_{(d+1)ij}
=\partial_{d+1}\partial_{i}q_{ij}
-\Delta \phi_j
=^{\eqref{eq:5.2}}-\partial_{d+1}^2\phi_j
-\Delta \phi_j
=-\Delta_{d+1}\phi_j.
\end{equation*}
This implies that
\begin{equation}\label{pde:5.9}
\Delta_{d+1} \big(\partial_i\sigma_{(d+1)ij} + \phi_j\big)=0.
\end{equation}
Also, it is already known that $\phi_j$ possesses a sublinear growth property (see \cite[Proposition 1]{Bella-Chiarini-Fehrman-19}). By
the representation in $\eqref{eq:5.6}$, one can derive that
\begin{equation*}
\partial_i \sigma_{(d+1)ij}(a,z)= e^{z\cdot D}(D^*\cdot D)^{-1}
 D_i\big(D_{d+1}q_{ij}(a)-D_i\phi_j(a,z)\big)
 =^{\eqref{corrector}}e^{z\cdot D}\phi(a,z)=\phi(a(z),z),
\end{equation*}
where it is noted that $q_{ij}$ is stationary while $\phi_j$ is not.
Thus, $\partial_i \sigma_{(d+1)ij}$ also satisfies the sublinear growth
property, and therefore it is known from Liouville's theorem that
the equation $\eqref{pde:5.9}$ implies $\eqref{eq:5.1c}$.
Moreover, if $\phi_j$ is stationary, it follows from the ergodic condition and the equation $\eqref{pde:5.9}$ that
\begin{equation*}
\partial_i\sigma_{(d+1)ij} + \phi_j = \langle \partial_i\sigma_{(d+1)ij} + \phi_j \rangle = \langle \phi_j \rangle,
\end{equation*}
which consequently modifies $\eqref{eq:5.1c}$ to $\partial_i\sigma_{(d+1)ij} =
\langle \phi_j \rangle-\phi_j$.

Now, we turn to discussing the sublinearity of $\nabla \sigma_{(d+1)ij}$
instead of $\partial_i \sigma_{(d+1)ij}$ since the latter has already followed from the equation $\eqref{eq:5.1c}$ and the sublinearity of $\phi_i$. Let $\omega_{lij}:=\partial_l\sigma_{(d+1)ij}$ and $l=1,\cdots,d$. To obtain $\eqref{pri:2.20}$, it suffices to show
\begin{equation}\label{f:2.25}
\limsup_{R\to\infty}\frac{1}{R^2} \dashint_{Q_R}\big| \omega_{lij}- \dashint_{Q_R}\omega_{lij}\big|^2 = 0;
\quad \text{and}\quad
\limsup_{R\to\infty} \frac{1}{R^2}\dashint_{Q_R}|\omega_{lij}|^2 = 0,
\end{equation}
where an ``anchor point'' condition $\dashint_{Q_1}\omega_{lij}=0$ can be introduced by default to ensure the validity of the second limit in the above equation.
By the equation $\eqref{eq:5.1a}$, one can derive that
\begin{equation*}
  \Delta_{d+1}\omega_{lij} = \partial_{d+1}\partial_{l}q_{ij}-\partial_i \partial_l\phi_j,
\end{equation*}
which implies the stationarity of $\overline{\nabla} \omega_{lij}$, and we have $\langle\nabla \omega_{lij}\rangle = 0$ and $\langle|
\nabla \omega_{lij}|^2\rangle<\infty$. Moreover, by virtue of $\eqref{eq:5.1b}$ and $\eqref{eq:5.3}$, it can be derived that
\begin{equation*}
\begin{aligned}
\partial_{d+1}\omega_{lij} = \partial_{d+1}\partial_l\sigma_{(d+1)ij}
&=^{\eqref{eq:5.1b}}\partial_l\big(q_{ij}-\partial_k\sigma_{kij}\big)\\
&=^{\eqref{eq:5.3}}\partial_l\big(\partial_k b_{kij} -\partial_k\sigma_{kij}\big)
=^{\eqref{eq:5.6}}\partial_k\partial_l b_{ikj}=:\nabla\cdot F_{lij},
\end{aligned}
\end{equation*}
where it can be verified that $F_{lij}$ is a stationary vector field, since it can be shown that $F_{likj}(a,z)=\partial_l (b_{ikj}(a,z))
= e^{z\cdot D}(D^*\cdot D)^{-1}D_lD_iq_{kj} = \Pi_{likj}(\tau_za)$ and
$\Pi_{likj}=:(D^*\cdot D)^{-1}D_lD_iq_{kj}=-\int_{\mathbb{R}^{d+1}}\frac{\lambda_l\lambda_i}{|\lambda|^2}
dE_{\lambda}q_{kj}$. This also implies $\langle |F_{lij}|^2\rangle
\lesssim_d \langle|q_{ij}|^2\rangle\lesssim 1$. Now,
according to \cite[Lemma 6]{Bella-Chiarini-Fehrman-19}, one can derive
the first limit stated in $\eqref{f:2.25}$.
The second limit therein is not difficult to obtain, and the related details are left to the reader.
This completes the first part of the proof.

\medskip
\noindent
\textbf{Part II.} In this part,  the spectral gap
condition $\eqref{c:3}$ is additionally assumed, by which the
stationary solution can be recovered. Here, by skipping the step from the corresponding ``massive correctors'' to the correctors, the computation is made more straightforward without losing the key message. 

\textbf{Step 1.} Arguments for $\eqref{pri:5.6b}$.
In view of the equation $\eqref{eq:5.1a}$, consider the following equation:
\begin{equation}\label{f:5.8-n}
 \Delta_{d+1}\sigma_{kij} = \partial_{k}q_{ij} -\partial_{i}q_{kj} \qquad\text{in}\quad\mathbb{R}^{d+1}.
\end{equation}
For any deterministic function $g$, construct the associated functional
\begin{equation*}
  F_{kij,g}(a):
  = \int_{\mathbb{R}^{d+1}}\overline{\nabla}\sigma_{kij}(a) \cdot g,
\end{equation*}
where we recall that $\overline{\nabla}$ denotes the $(d+1)$-dimensional
gradient operator. To achieve the goal, the auxiliary equation can be constructed as follows:
\begin{equation}\label{pde:5.1-n}
 -\Delta_{d+1}v =\overline{\nabla}\cdot g \qquad\text{in}\quad\mathbb{R}^{d+1}.
\end{equation}
Thus, integration by parts leads to
\begin{equation}
\begin{aligned}
&F_{kij,g}(a):
= \int_{\mathbb{R}^{d+1}}\overline{\nabla}\sigma_{kij}(a) \cdot g
&=\int_{\mathbb{R}^{d+1}}q_{kj}\partial_{i}v
-\int_{\mathbb{R}^{d+1}}q_{ij}\partial_{k}v.
\end{aligned}
\end{equation}

Starting from taking the functional derivative $\frac{\partial}{\partial a(x,t)}$ on $F_{kij,g}$, it is obtained that
\begin{equation}\label{f:5.30}
\begin{aligned}
\frac{\partial F_{kij,g}}{\partial a_{mn}(x,t)}
=\int_{\mathbb{R}^{d+1}}\frac{\partial q_{kj}}{\partial a_{mn}(x,t)}\partial_{i}v
-\int_{\mathbb{R}^{d+1}}\frac{\partial q_{ij}}{\partial a_{mn}(x,t)}\partial_{k}v,
\end{aligned}
\end{equation}
and
\begin{equation}\label{f:5.31}
\frac{\partial q_{kj}}{\partial a_{mn}(x,t)}
=-e_ke_me_n\delta_{(x,t)}\big(e_j+\nabla\phi_j\big)
-e_ka\nabla\frac{\partial\phi_j}{\partial a_{mn}(x,t)},
\end{equation}
where $\delta_{(x,t)}$ is the Dirac delta function with pole at $(x,t)\in\mathbb{R}^{d+1}$. 
Plugging the estimate $\eqref{f:5.31}$ (and its counterpart by exchanging $i$ and $k$) back into $\eqref{f:5.30}$,
we arrive at
\begin{equation*}
\begin{aligned}
\frac{\partial F_{kij,g}}{\partial a_{mn}(x,t)}
&=-e_ke_me_n\big[(e_j+\nabla\phi_j)\partial_{i}v\big](x,t)
-\int_{\mathbb{R}^{d+1}}
e_ka\nabla\frac{\partial\phi_j}{\partial a_{mn}(x,t)}\partial_{i}v\\
&\qquad\qquad +e_ie_me_n\big[(e_j+\nabla\phi_j)\partial_{k}v\big](x,t)
+\int_{\mathbb{R}^{d+1}}
e_ia\nabla\frac{\partial\phi_j}{\partial a_{mn}(x,t)}\partial_{k}v\\
&=-e_ie_me_n\big[(e_j+\nabla\phi_j)\partial_{d+1}v\big](x,t)
-\int_{\mathbb{R}^{d+1}}
\frac{\partial\phi_j}{\partial a_{mn}(x,t)}
\nabla\cdot (e_ka)^*\partial_{i}v\\
&\qquad\qquad +e_ie_me_n\big[(e_j+\nabla\phi_j)\partial_{k}v\big](x,t)
+\int_{\mathbb{R}^{d+1}}
\frac{\partial\phi_j}{\partial a_{mn}(x,t)}
\nabla\cdot e_ia^*\partial_{k}v
\end{aligned}
\end{equation*}


To continue the idea, we construct the auxiliary equation:
\begin{equation}\label{pde:5.7}
-(\partial_t +\nabla\cdot a^*\nabla) w_{ki} = \nabla\cdot
(e_ka)^*\partial_{i}v
\qquad\text{in}\quad\mathbb{R}^{d+1}.
\end{equation}
It can be derived that
\begin{equation*}
\begin{aligned}
\frac{\partial F_{kij,g}}{\partial a_{mn}(x,t)}
&=e_m\cdot\big[(e_k\partial_{i}v-\nabla w_{ki})\otimes
(e_j+\nabla\phi_j)\big](x,t)e_n \\
&\qquad - e_m\cdot\big[(e_i\partial_{k}v-\nabla w_{ik})\otimes
(e_j+\nabla\phi_j)\big](x,t)e_n, 
\end{aligned}
\end{equation*}
which is equivalent to saying 
\begin{equation*}
\frac{\partial F_{kij,g}}{\partial a(x,t)}
=\big[(e_k\partial_{i}v-\nabla w_{ki})\otimes
(e_j+\nabla\phi_j)\big](x,t)
-\big[(e_i\partial_{k}v-\nabla w_{ik})\otimes
(e_j+\nabla\phi_j)\big](x,t).
\end{equation*}

Since $\langle F_{kij,g}\rangle = 0$, 
it follows from from the spectral gap inequality of the $L^p$-version that 
\begin{equation}\label{f:5.33}
\begin{aligned}
\big\langle F_{kij,g}^{2p}\big\rangle^{1/p}
&\lesssim^{\eqref{c:3}}
\bigg\langle\Big(\int_{\mathbb{R}^{d+1}}dz\Big(\dashint_{Q_1(z)}
\big|\frac{\partial F_{kij,g}}{\partial a}\big|\Big)^2\Big)^{p}\bigg\rangle^{1/p}\\
&\lesssim\bigg\langle\bigg(
\int_{\mathbb{R}^{d+1}}dz
\Big(\dashint_{Q_1(z)}
\big|(\nabla v-\nabla w)\otimes
(e+\nabla\phi)\big|\Big)^2\bigg)^p\bigg\rangle^{1/p}\\
&\lesssim
\bigg\langle\Big(\dashint_{Q_1}|e+\nabla\phi|^2\Big)^{p}
\bigg\rangle^{\frac{1}{p}}\int_{\mathbb{R}^{d+1}}
\bigg\langle\Big(\dashint_{Q_1(z)}H^2(\nabla v+\nabla w)^2\Big)^{p'}
\bigg\rangle^{\frac{1}{p'}}\\
&\lesssim^{\eqref{pri:5.2}}
\int_{\mathbb{R}^{d+1}} \big\langle(H\nabla v+H\nabla w)^{2p'}\big\rangle^{\frac{1}{p'}},
\end{aligned}
\end{equation}
where Minkowski's inequality and Fubini's theorem are also employed in the last step, and $H$ is arbitrary, satisfying $\langle H^{2p'}\rangle =1$
with $1/p+1/p'=1$ and $p\gg 1$.

Applying Proposition $\ref{P:5.1-ap}$ (in the case $\beta=0$) to the equation $\eqref{pde:5.7}$
after multiplying $H$ in $\eqref{f:5.33}$ on the both sides of $\eqref{pde:5.7}$,
we have
\begin{equation}\label{f:5.6-n}
\int_{\mathbb{R}^{d+1}} \big\langle(H\nabla w)^{2p'}\big\rangle^{\frac{1}{p'}}
\lesssim \int_{\mathbb{R}^{d+1}} \big\langle(H\nabla v)^{2p'}\big\rangle^{\frac{1}{p'}}
\lesssim \int_{\mathbb{R}^{d+1}}|\overline{\nabla}v|^2
\lesssim \int_{\mathbb{R}^{d+1}}|g|^2,
\end{equation}
where the energy estimate related to $\eqref{pde:5.1-n}$ is employed for the last inequality. By combining the estimates $\eqref{f:5.6-n}$ and $\eqref{f:5.33}$, it is obtained that
\begin{equation*}
  \big\langle F_{kij,g}^{2p}\big\rangle^{1/p}
  \lesssim \int_{\mathbb{R}^{d+1}}|g|^2,
\end{equation*}
which leads to the desired estimate $\eqref{pri:5.6b}$. 

\textbf{Step 2.} Arguments for $\eqref{pri:5.6c}$. We begin from a preparation by using the same idea given in the previous step. 
Let $v$ and $h$ be associated with by the equation $-\Delta_{d+1}v = h$ in $\mathbb{R}^{d+1}$, and we consider the functional 
\begin{equation}
 G_{kij,h}(a):=\int_{\mathbb{R}^{d+1}}
 \big(\sigma_{kij}-\big\langle\sigma_{kij}\big\rangle\big)h.
\end{equation}
Integration by parts and the equation $\eqref{f:5.8-n}$ lead to     
\begin{equation}
	G_{kij,h} = \int_{\mathbb{R}^{d+1}}q_{ij}\partial_kv-\int_{\mathbb{R}^{d+1}}q_{kj}\partial_i v.
\end{equation}
Noting that $\langle G_{kij,h}\rangle = 0$, one can employ the spectral gap inequality as before to derive that 
\begin{equation}\label{f:5.5-n}
  \big\langle G_{kij,h}^{2p}\big\rangle^{1/p}
  \lesssim \int_{\mathbb{R}^{d+1}}|\nabla v|^2
  \lesssim \Big(\int_{\mathbb{R}^{d+1}}|h|^{2_*}\Big)^{\frac{2}{2_*}},
\end{equation}
where we employ Sobolev's inequality on $\mathbb{R}^{d+1}$ in the last step and $2_*:=\frac{2(d+1)}{d+3}$. Obviously, as a by-product, one can derive that 
\begin{equation}\label{f:5.7-n}
\Big\langle\big|\int_{\mathbb{R}^{d+1}}q_{ij}h\big|^{2p}\Big\rangle^{1/p}
\lesssim \int_{\mathbb{R}^{d+1}}|h|^2.
\end{equation}

We are now in the position to handle the fluctuation estimate $\eqref{pri:5.6c}$. Construct the auxiliary equation:
for any $z_0\in\mathbb{R}^{d+1}$,
\begin{equation}\label{pde:5.2-n}
 \partial_t v_{z_0}-\Delta v_{z_0} = \textbf{1}_{Q_1(z_0)} - \textbf{1}_{Q_1(0)}
 \qquad\text{in}\quad\mathbb{R}^{d+1}.
\end{equation}
Then, it holds that
\begin{equation*}
\begin{aligned}
 \dashint_{Q_1(z)}\sigma_{kij}
 &-
 \dashint_{Q_1(0)}\sigma_{kij} 
 =\int_{\mathbb{R}^{d+1}}
 \nabla v_{z_0}\cdot \nabla\sigma_{kij} +\int_{\mathbb{R}^{d+1}}\partial_t v_{z_0}\big(\sigma_{kij}-\langle\sigma_{kij}\rangle\big).
 \end{aligned}
\end{equation*}
It follows from the previous estimates $\eqref{pri:5.6b}$ and $\eqref{f:5.5-n}$  that 
\begin{equation}
\begin{aligned}
\bigg\langle\Big(\dashint_{Q_1(z)}\sigma_{kij}-
 \dashint_{Q_1(0)}\sigma_{kij}\Big)^{2p}
 \bigg\rangle^{\frac{1}{p}}
\lesssim
\int_{\mathbb{R}^{d+1}}|\nabla v_{z_0}|^2+\Big(\int_{\mathbb{R}^{d+1}}|\partial_t v_{z_0}|^{2_*}\Big)^{\frac{2}{2_*}}
\lesssim^{\eqref{f:5.24-ap}} 1,
\end{aligned}
\end{equation}
where we also employ the Calder\'on-Zygmund estimate for the equation 
$\eqref{pde:5.2-n}$ in the last step\footnote{The requirement $2_*>1$ is satisfied by the condition $d\geq 2$.}. This gives 
the proof of the desired estimate $\eqref{pri:5.6c}$. 

\textbf{Step 3.} Arguments for $\eqref{pri:5.7b}$. 
In view of $\eqref{eq:5.1a}$ and $\eqref{eq:5.1b}$, one can find the following equation:
\begin{equation*}
\begin{aligned}
\Delta \sigma_{(d+1)ij} &= \partial_{d+1}q_{ij} -\partial_{i}q_{(d+1)j} -\partial_{d+1}^2\sigma_{(d+1)ij} \\
&=^{\eqref{eq:5.1b}}
\partial_{d+1}q_{ij}-\partial_i\phi_j-\partial_{d+1}\big(q_{ij}-\partial_{k}\sigma_{kij}\big)=-\partial_i\phi_j+\partial_{d+1}\partial_{k}\sigma_{kij}.
\end{aligned} 
\end{equation*}
By treating $\partial_{d+1}$ as $\partial_t$, one can further derive from 
$\eqref{eq:5.1b}$ and the above equation that  
\begin{equation}\label{pde:5.6}
(\partial_t+\Delta)\sigma_{(d+1)ij}= q_{ij}-\partial_k\sigma_{kij}
-\partial_i\phi_j+\partial_{d+1}\partial_{k}\sigma_{kij},
\end{equation}
where we employ the antisymmetry property of $\sigma_{(d+1)ij}$ (i.e., $\sigma_{k'ij} = -\sigma_{ik'j}$) to 
make the left-hand side be a parabolic operator (in its dual form). 

Now, we can construct the associated functional
\begin{equation*}
  \bar{F}_{ij,g}(a):
  = \int_{\mathbb{R}^{d+1}}\nabla\sigma_{(d+1)ij}(a) \cdot g.
\end{equation*}
Meanwhile, to achieve the goal, the auxiliary equation can be constructed as follows:
\begin{equation}\label{pde:5.8}
 \big(\partial_t -\Delta\big) v =\nabla\cdot g \qquad\text{in}\quad\mathbb{R}^{d+1}.
\end{equation}
Thus, integration by parts leads to
\begin{equation}\label{f:5.35}
\begin{aligned}
&\bar{F}_{ij,g}(a):
= \int_{\mathbb{R}^{d+1}}\nabla\sigma_{(d+1)ij} \cdot g
=-\int_{\mathbb{R}^{d+1}}\sigma_{(d+1)ij}\nabla\cdot g
=\int_{\mathbb{R}^{d+1}}\sigma_{(d+1)ij}\big(\partial_t-\Delta\big)v \\
&= -\int_{\mathbb{R}^{d+1}}\big(\partial_t+\Delta\big)\sigma_{(d+1)ij}v
=^{\eqref{pde:5.6}} \int_{\mathbb{R}^{d+1}}
\big(q_{ij}-\partial_k\sigma_{kij}
-\partial_i\phi_j\big)v
-\int_{\mathbb{R}^{d+1}}\partial_{d+1}\sigma_{kij}\partial_k v\\
&=\int_{\mathbb{R}^{d+1}}q_{ij}v-\int_{\mathbb{R}^{d+1}}
\overline{\nabla}\sigma_{kij}\cdot(e_kv+e_{d+1}\partial_k v) 
-\int_{\mathbb{R}^{d+1}}\nabla\phi_j\cdot e_iv.
\end{aligned}
\end{equation}

Combining the estimates $\eqref{f:5.7-n}$, $\eqref{pri:5.6b}$, $\eqref{pri:5.3*}$, and $\eqref{f:5.35}$,  we obtain that
\begin{equation*}
  \big\langle\bar{F}_{ij,g}^{2p}\big\rangle^{1/p}
  \lesssim \int_{\mathbb{R}^{d+1}}\big(|v|^2+|\nabla v|^2\big)
  \lesssim^{\eqref{pri:6.1-n}} \int_{\mathbb{R}^{d+1}}|g|^2+\Big(\int_{\mathbb{R}^{d+1}}|g|^{2_{**}}\Big)^{\frac{2}{2_{**}}}, 
\end{equation*}
where $2_{**}:=\frac{2(d+2)}{d+4}$ is due to 
the anisotropic fractional integration estimate (see Lemma \ref{lemma:6.1-n}). This gives the desired estimate $\eqref{pri:5.7b}$.

\textbf{Step 4.} Arguments for $\eqref{pri:5.7c}$. 
To do so, construct the auxiliary equation:
for any $z_0\in\mathbb{R}^{d+1}$,
\begin{equation}\label{pde:5.11}
 \partial_t v_{z_0}-\Delta v_{z_0} = \textbf{1}_{Q_1(z_0)} - \textbf{1}_{Q_1(0)}
 \qquad\text{in}\quad\mathbb{R}^{d+1}.
\end{equation}
Then, it holds that
\begin{equation*}
\begin{aligned}
 \dashint_{Q_1(z)}\sigma_{(d+1)ij}
 &-
 \dashint_{Q_1(0)}\sigma_{(d+1)ij} 
 =\int_{\mathbb{R}^{d+1}}
 \nabla v_z\cdot \nabla\sigma_{(d+1)ij} -\int_{\mathbb{R}^{d+1}}v_{z_0}\partial_{t}\sigma_{(d+1)ij}\\
 &=^{\eqref{eq:5.1b}}\int_{\mathbb{R}^{d+1}}
 \nabla v_{z_0}\cdot \nabla\sigma_{(d+1)ij} -\int_{\mathbb{R}^{d+1}}v_{z_0} q_{ij}
  -\int_{\mathbb{R}^{d+1}}v_{z_0}e_k\cdot\nabla\sigma_{kij}.
\end{aligned}
\end{equation*}
It follows the estimate $\eqref{pri:5.7b}$, $\eqref{f:5.7-n}$, and $\eqref{pri:5.6b}$ that
\begin{equation}\label{f:5.53}
\begin{aligned}
\bigg\langle\Big(\dashint_{Q_1(z)}\sigma_{(d+1)ij}-
 \dashint_{Q_1(0)}\sigma_{(d+1)ij}\Big)^{2p}
 \bigg\rangle^{\frac{1}{p}}
\lesssim
\int_{\mathbb{R}^{d+1}}|\nabla v_{z_0}|^2
+ \Big(\int_{\mathbb{R}^{d+1}}|\nabla v_{z_0}|^{2_{**}}\Big)^{\frac{2}{2_{**}}}
+ \int_{\mathbb{R}^{d+1}}|v_{z_0}|^2,
\end{aligned}
\end{equation}
where we recall that $2_{**}=\frac{2(d+2)}{d+4}$. Due to the estimate $\eqref{f:5.24-ap}$, to estimate the right-hand side of $\eqref{f:5.53}$, it suffices to estimate the following quantities: 
\begin{equation}\label{q:5.1}
\Big(\int_{\mathbb{R}^{d+1}}|\nabla v_{z_0}(z)|^{2_{**}}
\Big)^{\frac{2}{2_{**}}},
\qquad \int_{\mathbb{R}^{d+1}}|v_{z_0}|^2.
\end{equation}


We appeal to the fundamental solution of $\partial_t-\Delta$,
denoted by $\Gamma$, and then the solution of $\eqref{pde:5.11}$ is given by
\begin{equation*}
\begin{aligned}
& v_{z_0}(z) = \dashint_{Q_1(0)}dy\big(\Gamma(z-z_0-y)-\Gamma(z-y)\big);\\
&\nabla v_{z_0}(z) = \dashint_{Q_1(0)}dy\big(\nabla\Gamma(z-z_0-y)-\nabla\Gamma(z-y)\big)
\qquad \forall z\in\mathbb{R}^{d+1}.
\end{aligned}
\end{equation*}
One can set $r:=\text{d}(z_0,0)/4$. Let $A:=Q_1(0)\cup Q_1(z_0)$, $B:=\big(Q_r(0)\setminus Q_1(0)\big)\cup \big(Q_r(z_0)\setminus Q_1(z_{0})\big)$,  and  $C:=\mathbb{R}^{d+1}\setminus \big(Q_r(0)\cup Q_r(z_0)\big)$. It is not hard to observe that
$\mathbb{R}^{d+1}\subset A\cup B\cup C$. Therefore, it suffices to show
the following estimates:
\begin{equation}\label{f:5.54}
\begin{aligned}
\int_{\mathbb{R}^{d+1}}|v_{z_0}|^2
&+\Big(\int_{\mathbb{R}^{d+1}}|\nabla v_{z_0}|^{2_{**}}\Big)^{\frac{2}{2_{**}}}\\
&\leq \bigg\{\int_{A}+\int_{B}+\int_{C}\bigg\}|v_{z_0}|^2 
+ \bigg\{\Big(\int_{A}+\int_{B}+\int_{C}\Big)|\nabla v_{z_0}|^{2_{**}}\bigg\}^{\frac{2}{2_{**}}}
\lesssim 1 + \bar{\mu}_d^2(z_0) 
\end{aligned}
\end{equation}
which, together with $\eqref{f:5.53}$, consequently leads to the stated
estimate $\eqref{pri:5.7c}$.

\textbf{Step 5.} Arguments for $\eqref{f:5.54}$. We start from 
estimating the quantities in $\eqref{q:5.1}$ on the integral region $A$ that 
\begin{equation}\label{f:2.12}
\begin{aligned}
\Big(\int_{A}|\nabla v_{z_0}|^{2_{**}}\Big)^{\frac{2}{2_{**}}}
=
\Big(\int_{Q_1(0)\cup Q_1(z_0)}|\nabla v_{z_0}|^{\frac{2(d+2)}{d+4}}\Big)^{\frac{d+4}{d+2}}
\lesssim_d \int_{\mathbb{R}^{d+1}}
|\nabla v_{z_0}|^2 \lesssim^{\eqref{f:5.24-ap}}  1,
\end{aligned}
\end{equation}
where we merely employ H\"older's inequality in the first inequality. Let $q=\frac{2(d+2)}{d}$. By using H\"older's inequality again, we obtain  that 
\begin{equation}
\int_{A}|v_{z_0}|^2 \lesssim_d 
\Big(\int_{\mathbb{R}^{d+1}}|v_{z_0}|^q\Big)^{\frac{2}{q}}
\lesssim^{\eqref{pri:6.8-ap}} \sup_{t\in\mathbb{R}}\|v_{z_0}(\cdot,t)\|_{L^2(\mathbb{R}^d)}^2
+\int_{\mathbb{R}^{d+1}}|\nabla v_{z_0}|^2 
\lesssim^{\eqref{f:5.24-ap}} 1.
\end{equation}

For the region $C$ (which is known as the far-field region), one can 
 employ the smoothness of the fundamental solution, namely, 
\begin{equation*}
\begin{aligned}
 &|v_{z_0}(z)| 
  \lesssim  \text{d}(z_0,0)\big[\text{d}(z,0)\big]^{-d-1};\\
 &|\nabla v_{z_0}(z)| 
  \lesssim  \text{d}(z_0,0)\big[\text{d}(z,0)\big]^{-d-2}
  \qquad \forall z\in C.
\end{aligned}
\end{equation*}
Thus, we obtain that 
\begin{equation}\label{f:2.10}
\begin{aligned}
&\Big(\int_{C}|\nabla v_{z_0}|^{\frac{2(d+2)}{d+4}}\Big)^{\frac{4+d}{d+2}}
\lesssim r^2\Big(\int_{
\{\text{d}(z,0)\geq r/4\}}\big[\text{d}(z,0)\big]^{-\frac{2(d+2)^2}{d+4}}
\Big)^{\frac{d+4}{d+2}}
\lesssim  r^{2-d};\\
&\int_{C}|v_{z_0}|^{2}
\lesssim r^2\int_{\{\text{d}(z,0)\geq r/4\}}\big[\text{d}(z,0)\big]^{-2d-2}
\lesssim  r^{2-d}.
\end{aligned}
\end{equation}

For any $z\in \big(Q_r(0)\setminus Q_1(0)\big)$ (or $z\in (Q_r(z_0)\setminus Q_1(z_0))$), 
it is not hard to observe that $\text{d}(z,z_0)\geq r$ (or $\text{d}(z,0)\geq r$), 
and therefore one can derive that
\begin{equation}\label{f:2.11}
\begin{aligned}
&\Big(\int_{B}|\nabla v_{z_0}|^{2_{**}}\Big)^{\frac{2}{2_{**}}}
\lesssim \Big(\int_{B}dz|\nabla\Gamma(z-z_0)|^{\frac{2(d+2)}{d+4}}\Big)^{\frac{d+4}{d+2}} + \Big(\int_{B}dz|\nabla\Gamma(z)|^{\frac{2(d+2)}{d+4}}\Big)^{\frac{d+4}{d+2}}\\
&\lesssim \Big(r^{d+2}r^{-\frac{2(d+2)(d+1)}{d+4}}\Big)^{\frac{d+4}{d+2}}
+\Big(\int_{1}^{r^2} t^{-\frac{(d+1)(d+2)}{d+4}+\frac{d}{2}}dt\Big)^{\frac{d+4}{d+2}}
\lesssim r^{2-d}+\bar{\mu}_d^{2}(z_0)\lesssim \bar{\mu}_d^2(z_0).
\end{aligned}
\end{equation}
By the same token, one can derive that 
\begin{equation}\label{f:5.8-n}
\begin{aligned}
\int_{B}|v_{z_0}|^{2} 
\lesssim  \int_{B}dz|\Gamma(z-z_0)|^{2}  + \int_{B}dz|\Gamma(z)|^{2}
\lesssim r^{2-d}
+ \int_{1}^{r^2} t^{-\frac{d}{2}}dt 
\lesssim \bar{\mu}_d^2(z_0).
\end{aligned}
\end{equation}

As a result, combining the estimates $\eqref{f:2.10}$, $\eqref{f:2.11}$, 
$\eqref{f:5.8-n}$, 
and $\eqref{f:2.12}$ gives the estimate $\eqref{f:5.54}$.

\textbf{Step 6.}
Arguments for $\eqref{pri:5.6a}$. By translation it is fine to assume
$z=0$. Without a proof (whose strategy is similar to that given
for $\sigma_{(d+1)ij}$ later),
one can derive that
\begin{equation}\label{f:5.55}
\begin{aligned}
\bigg\langle\Big(\dashint_{Q_1(0)}|\overline{\nabla} \sigma_{k'ij}|^{2}\Big)^p\bigg\rangle^{1/p}\lesssim_{\mu,\lambda_1,d,p} 1
\qquad k'=1,\cdots,d.
\end{aligned}
\end{equation}

In view of the equation $\eqref{eq:5.1a}$ with $k'=d+1$, i.e., 
$-\Delta_{d+1}\sigma_{(d+1)ij} = \partial_{d+1}q_{ij}-\partial_i\phi_j$,
it follows from local $L^p$ estimates for this equation that
\begin{equation*}
 \dashint_{\hat{B}_{R}}\Big(\dashint_{\hat{B}_1(z)}
 |\overline{\nabla}\sigma_{(d+1)ij}|^2\Big)^p
 \lesssim \bigg(\dashint_{\hat{B}_{2R}}\dashint_{\hat{B}_1(z)}\big|\overline{\nabla}\sigma_{(d+1)ij}\big|^2\bigg)^{p}
 + \dashint_{\hat{B}_{2R}}\Big(\dashint_{\hat{B}_1(z)}\big(|q_{ij}|^2
 +|\tilde{\phi}_j|^2\big)\Big)^p,
\end{equation*}
where we mention that $\hat{B}_{R}:=\{z=(x,t)\in\mathbb{R}^{d+1}:|x|^2+|t|^2\leq R^2\}$ is the $(d+1)$-dimensional ball centered at zero with radius $R$, and $\tilde{\phi}_j(z) =
\phi_j(z) - \dashint_{Q_1(0)}\phi_j$.
In view of $\eqref{eq:5.1b}$, this implies that
\begin{equation}\label{f:5.36}
\begin{aligned}
& \dashint_{\hat{B}_R}\Big(\dashint_{\hat{B}_1(z)}|\overline{\nabla}\sigma_{(d+1)ij}|^2\Big)^p \\
& \lesssim \bigg(\dashint_{\hat{B}_{2R}}\dashint_{\hat{B}_1(z)}
 \big(|\nabla\sigma_{(d+1)ij}|^2+\sum_{k=1}^{d}
 |\nabla\sigma_{kij}|^2\big)\bigg)^{p}
 + \dashint_{\hat{B}_{2R}}\Big(\dashint_{\hat{B}_1(z)}\big(|\nabla\phi_j|^2
 +|\tilde{\phi}_j|^2\big)\Big)^p\\
 & \lesssim_d \bigg(\dashint_{\hat{B}_{2R}}\dashint_{Q_1(z)}
 \big(|\nabla\sigma_{(d+1)ij}|^2+\sum_{k=1}^{d}
 |\nabla\sigma_{kij}|^2\big)\bigg)^{p}
 + \dashint_{\hat{B}_{2R}}\Big(\dashint_{Q_1(z)}\big(|\nabla\phi_j|^2
 +|\tilde{\phi}_j|^2\big)\Big)^p,
\end{aligned}
\end{equation}
where we employ the simply geometry fact: $\hat{B}_1\subset Q_1$ in the last inequality. 


By letting $R\to\infty$ on the both sides above, it follows from
the ergodic theorem that
\begin{equation*}
\begin{aligned}
\bigg\langle\Big(\dashint_{\hat{B}_1}|\overline{\nabla}\sigma_{(d+1)ij}|^2\Big)^p
\bigg\rangle
&\lesssim \bigg\langle\dashint_{Q_1}
 \big(|\nabla\sigma_{(d+1)ij}|^2+\sum_{k=1}^{d}|\nabla\sigma_{kij}|^2\big)
 \bigg\rangle^{p} \\
&+\bigg\langle\Big(\dashint_{Q_1}|\nabla\phi|^2\Big)^p\bigg\rangle
+\bigg\langle\Big(\dashint_{Q_1}\big|\tilde{\phi}_j\big|^2\Big)^{p} \bigg\rangle
\lesssim^{\eqref{pri:5.2},\eqref{pri:5.4}} 1.
\end{aligned}
\end{equation*}
Based on the stationarity and covering method, it is not difficult to replace the integral region $\hat{B}_1$ on the left-hand side of the above estimate with $Q_1$.  This, together with the estimate $\eqref{f:5.55}$, leads to the stated estimate $\eqref{pri:5.6a}$. 
The proof is thereby completed.
\qed

\subsection{Proof of Proposition $\ref{P:5.1-ap}$}
\noindent
The idea of the proof is to repeatedly use a real method developed by  Shen \cite{Shen-05}, and we refer the reader to Subsection $\ref{subsec:6.1}$ for the main idea of utilizing Shen's real methods to address the annealed Calder\'on-Zygmund estimates in a random setting, originally developed in
\cite{Duerinckx-Otto-20}.
Therefore, it is suggested that readers who are not familiar with this method skip this proof during their first reading.
We first address the estimate $\eqref{pri:5.1-ap}$ in the case
$0<2-p\ll 1$, and the counterpart case $0<p-2\ll 1$ will then follow from a dual argument. For the ease of the statement,
the following notation are introduced throughout the proof: $0<2-\underline{p},~\overline{p}-2\ll1$ and $\underline{p}\leq q<\overline{p}$,
\begin{equation*}
 V := \big\langle|(\sqrt{\beta} v,\nabla v)|^{\underline{p}}\big\rangle^{\frac{1}{\underline{p}}};\qquad F := \big\langle|(f,\sqrt{\beta} g)|^{\underline{p}}\big\rangle^{\frac{1}{\underline{p}}}.
\end{equation*}
Given any parabolic cube $Q_R(z)\subset\mathbb{R}^{d+1}$ (abbreviated as $Q$), we also introduce
\begin{equation*}
\begin{aligned}
  W_{Q} := \big\langle|(\sqrt{\beta} w,\nabla w)|^{\underline{p}}\big\rangle^{\frac{1}{\underline{p}}};\qquad
  U_{Q} := \big\langle|(\sqrt{\beta} u, \nabla u)|^{\underline{p}}\big\rangle^{\frac{1}{\underline{p}}}.
\end{aligned}
\end{equation*}

\textbf{Step 1.} Reduction and outline of the main ingredients. For any $Q$,
let $w$ and $u:=v-w$ satisfy the following equations:
\begin{subequations}
\begin{align}
&\partial_t w -\nabla \cdot \big(a\nabla w + \mathbf{1}_{Q} f\big)
+\beta(w-\mathbf{1}_{Q}g) = 0    \qquad\text{in}\quad \mathbb{R}^{d+1};
\label{pde:5.1-ap}\\
&\beta u + (\partial_t -\nabla\cdot a \nabla) u = 0
\qquad \text{in} \quad Q,
\label{pde:5.2-ap}
\end{align}
\end{subequations}
where $\mathbf{1}_{Q}$ is the indicator function of $Q$. We first observe that $V\leq W_{Q} + U_{Q}$ on $Q$. According to Shen's real method (see \cite[Theorem 4.23]{Shen18}), it suffices to establish
\begin{subequations}
\begin{align}
&  \dashint_{\frac{1}{2}Q} W_{Q}^{\underline{p}}
\lesssim_{\mu,d,\underline{p}} \dashint_{Q} F^{\underline{p}};
\label{f:5.1a-ap}\\
& \bigg(\dashint_{\frac{1}{2}Q} U_{Q}^{\overline{p}}\bigg)^{\underline{p}/\overline{p}}
\lesssim_{\mu,d,\underline{p},\overline{p}}
\dashint_{2Q} \big(V^{\underline{p}}
+ F^{\underline{p}}\big). \label{f:5.1b-ap}
\end{align}
\end{subequations}

Admitting the above estimates provisionally, it follows from Shen's real argument that
\begin{equation}\label{f:5.2-ap}
 \int_{\mathbb{R}^{d+1}} V^q \lesssim_{\mu,d,q,\underline{p},\overline{p}}
 \int_{\mathbb{R}^{d+1}} F^q
\end{equation}
holds for any $\underline{p}\leq q<\overline{p}$. Now, one may fix $q=2$ and the stated estimate $\eqref{f:5.2-ap}$ remains true when   $\underline{p}$ is chosen arbitrarily in the range $0<2-\underline{p}\ll 1$, here denoted by $p$. In other words, the desired estimate $\eqref{pri:5.1-ap}$ holds for $0<2-p\ll 1$. For the case $0<p-2\ll 1$, it follows from a duality argument. To this end, for any $(h,H)$, consider the adjoint problem:
\begin{equation*}
\beta z -\big(\partial_t+\nabla\cdot a^*\nabla\big) z = \beta H-\nabla\cdot h
 \qquad\text{in}\quad\mathbb{R}^{d+1},
\end{equation*}
and we have the formula
\begin{equation*}
\begin{aligned}
\Big\langle\int_{\mathbb{R}^{d+1}}\big(\nabla v\cdot h
+\beta Hz\big)\Big\rangle
=\Big\langle\int_{\mathbb{R}^{d+1}}\big(f\cdot\nabla z-\beta gz\big)\Big\rangle.
\end{aligned}
\end{equation*}
This implies
\begin{equation*}
\begin{aligned}
\Big\langle\int_{\mathbb{R}^{d+1}}\big(\nabla v\cdot h +\beta vH\big)\Big\rangle
&\leq \int_{\mathbb{R}^{d+1}}\langle|(f,\sqrt{\beta} g)|^p\rangle^{\frac{1}{p}}
\langle|(\sqrt{\beta} z,\nabla z)|^{p'}\rangle^{\frac{1}{p'}}\\
&\lesssim^{\eqref{f:5.2-ap}}  \Big(\int_{\mathbb{R}^{d+1}}\big\langle|(f,\sqrt{\beta} g)|^p\big\rangle^{\frac{2}{p}}\Big)^{\frac{1}{2}}
\Big(\int_{\mathbb{R}^{d+1}}\big\langle|(h,\sqrt{\beta} H)|^{p'}\big\rangle^{\frac{2}{p'}}
\Big)^{\frac{1}{2}},
\end{aligned}
\end{equation*}
where $1/p+1/p' = 1$, and this together with the case $0<2-p\ll 1$
consequently leads to the desired estimate
$\eqref{pri:5.1-ap}$.

\textbf{Step 2.} Arguments for $\eqref{f:5.1a-ap}$. We now study the equation $\eqref{pde:5.1-ap}$. The key ingredient is
Meyer-type estimate $\eqref{pri:5.10-ap}$.
In this regard, it is noticed that
\begin{equation}\label{f:5.5-ap}
\dashint_{\frac{1}{2}Q} W_{Q}^{\underline{p}}
\lesssim \frac{1}{|Q|}\Big\langle\int_{\mathbb{R}^{d+1}}|\nabla w|^{\underline{p}}\Big\rangle
\lesssim^{\eqref{pri:5.10-ap}} \frac{1}{|Q|}\Big\langle\int_{\mathbb{R}^{d+1}}
|\textbf{1}_{Q} f|^{\underline{p}}\Big\rangle
= \dashint_{Q} F^{\underline{p}},
\end{equation}
which gives the stated estimate $\eqref{f:5.1a-ap}$.

\textbf{Step 3.} Arguments for $\eqref{f:5.1b-ap}$. We now turn to
the equation $\eqref{pde:5.2-ap}$, and the crucial ingredient is
the local version of Meyer-type estimate, namely
\begin{equation}\label{f:5.4-ap}
 \Big(\dashint_{\frac{1}{2}Q} |(\sqrt{\beta} u,\nabla u)|^{\overline{p}}\Big)^{1/\overline{p}}
 \lesssim \Big(\dashint_{Q} |
 (\sqrt{\beta} u, \nabla u)|^{\underline{p}}\Big)^{1/\underline{p}},
\end{equation}
which follows from $\eqref{pri:5.9-ap}$ by setting $p_0=\underline{p}$ therein.
Thus, starting from Minkowski's inequality, it is found that
\begin{equation*}
\begin{aligned}
\bigg(\dashint_{\frac{1}{2}Q} U_Q^{\overline{p}}\bigg)^{\underline{p}/\overline{p}}
&\leq \bigg\langle\Big(\dashint_{\frac{1}{2}Q}
|\nabla u|^{\overline{p}}\Big)^{\underline{p}/\overline{p}}\bigg\rangle\\
&\lesssim^{\eqref{f:5.4-ap}}
\Big\langle\dashint_{Q}
|\nabla u|^{\underline{p}}\Big\rangle
\lesssim^{\eqref{f:5.5-ap}} \dashint_{2Q}
\big(V^{\underline{p}}+F^{\underline{p}}\big),
\end{aligned}
\end{equation*}
where we also use the fact that $u = v-w$. This yields the
stated estimate $\eqref{f:5.1b-ap}$, and we have completed the whole proof.
\qed

\subsection{Proof of Theorem $\ref{thm:1.0}$}
\noindent
The part of the conclusions concerning the flux corrector in Theorem $\ref{thm:1.0}$ is already included in Proposition $\ref{P:5.3}$. Here,   only the proof of the existence of the stationary solution to the equation $\eqref{corrector}$ is provided, along with the proof of the related estimates. The structure of the proof is presented in Fig. $\ref{pic:2.1}$.
\begin{figure}[htbp]
\centering 
\resizebox{0.85\textwidth}{!}{
\begin{tikzpicture}[
    node distance=0.7cm and 1cm,
    >=Stealth, 
    box/.style={
        rectangle, draw,
        minimum width=3.2cm, minimum height=0.9cm,
        align=center, text width=3cm
    },
    smallbox/.style={
        rectangle, draw,
        minimum width=3.5cm, minimum height=0.7cm,
        align=center, text width=3cm
    },
]

\node[box] (T1) {\textbf{Theorem \ref{thm:1.0}}};
\node[smallbox, right=of T1,draw=white, text opacity=0] (C25) {};
\node[box, right=of C25] (C25*) {Corollary \ref{corollary:2.1}};
\node[box, below left=of T1] (P22) {Proposition \ref{P:5.3}\\
\small(flux correctors)};
\node[box, below right=of T1] (P21) {Proposition \ref{P:5.1*}\\\small(correctors)};
\node[box, below=of P22] (P23) {Proposition $\ref{P:5.1-ap}$\\ \small(Annealed C-Z of\\Meyer type)};
\node[box, right=of P23] (L64) {Lemma \ref{lemma:5.2}};
\node[box, below=of P21] (L24) {Lemma \ref{lemma:2.1}};
\node[smallbox, right=of L24] (L63) {Lemma \ref{lemma:5.1}};
\node[smallbox, above=of L63] (L62) {Lemma \ref{lemma:4.1-ap}};

\node[draw, dashed, fit=(L62)(L63), inner sep=5pt] (dashedbox) {};

\draw[->] (T1) -- node[midway, above] {$+$~\boxed{~~\text{Lemma} ~$\ref{lemma:5.2}$~~}} (C25*);
\draw[->] (P22) -- (T1);       
\draw[->,thick] (P21) -- (T1);       
\draw[->] (P23) -- (P22);      
\draw[->] (P23) -- (P21);      
\draw[->] (L64) -- (P23);      
\draw[->,thick] (L24) -- (P21);      
\draw[->,dashed] (P21) -- (P22);      

\node at ($(L24.east)!0.5!(L63.west)$) {$+$};

\end{tikzpicture} %
}
\caption{On the proof structure of Theorem $\ref{thm:1.0}$} 
\label{pic:2.1}
\end{figure}

In view of Lemma $\ref{lemma:2.1}$, the existence of
the stationary solution of massive correctors $\phi_\beta$ to the equations $\eqref{massive-corrector*}$ is established. By virtue of Proposition $\ref{P:5.1*}$, an estimate much stronger than that stated in $\eqref{f:5.47a}$ can be derived.
To this end, removing the second term on the right-hand side of $\eqref{pde:5.4-ap}$, we obtain that
\begin{equation*}\label{}
\Big\langle \big|\dashint_{Q_1(z)}\phi_{\beta}\big|^{2p}\Big\rangle^{1/p}
\lesssim_{\mu,\lambda_1,d,p}1
\qquad\forall z\in\mathbb{R}^{d+1}.
\end{equation*}
This implies that
\begin{equation*}
\begin{aligned}
\Big\langle \dashint_{Q_1(z)}|\phi_\beta|^2 \Big\rangle
\lesssim \Big\langle \dashint_{Q_1(z)}|\nabla\phi_\beta|^2 \Big\rangle
+ \Big\langle \big|\dashint_{Q_1(z)}\phi_\beta \big|^2 \Big\rangle
\lesssim^{\eqref{pri:5.2*}} 1,
\end{aligned}
\end{equation*}
which, by the stationarity of $\phi_\beta$, further leads to
\begin{equation}\label{f:2.6}
 \big\langle|\phi_\beta|^2\big\rangle\lesssim 1.
\end{equation}

Moreover, in view of the estimates $\eqref{f:5.47c}$ and $\eqref{f:5.47d}$, one can continue to derive that
\begin{equation}\label{f:2.7}
\begin{aligned}
& \|D_0\phi_{\beta}\|_{\mathcal{H}^{-1}}\lesssim_{\mu,d}1; \\
& \langle |D\phi_{\beta}|^2 \rangle \lesssim_{\mu,d} 1.
\end{aligned}
\end{equation}
Therefore, as $\beta\to 0$, it follows from the estimates
$\eqref{f:2.6}$ and $\eqref{f:2.7}$ that
\begin{equation*}
\begin{aligned}
&\phi_{\beta,i} \rightharpoonup \phi_{i} &\quad&\text{weakly~in}\quad L^2(\Sigma); \\
&D\phi_{\beta,i} \rightharpoonup \Phi_{i}
&\quad&\text{weakly~in}\quad L^2(\Sigma;\mathbb{R}^d);\\
&D_0\phi_{\beta,i} \rightharpoonup \Xi_{i}
&\quad&\text{weakly~in}\quad \mathcal{H}^{-1}.
\end{aligned}
\end{equation*}
Then, the limit with respect to $\beta$ can be taken in the equation
$\eqref{massive-corrector*}$, and
\begin{equation*}
 \langle D_0\phi_{i}h\rangle + \big\langle
  Dh\cdot a(D\phi_{i}+e_i)\big\rangle = 0
\end{equation*}
holds for any $h\in\mathcal{H}^1$, where we employ the fact that
$\Phi_i = D\phi_i$ and $\Xi_{i} = D_0\phi_i$. This yields
the stationary solution to the equations $\eqref{corrector}$.
Also, the stated estimates $\eqref{pri:5.2}$ and $\eqref{pri:5.4}$ for $\phi$ follow from $\eqref{pri:5.2*}$ and $\eqref{pri:5.4*}$, respectively, by taking the limit with respect to $\beta$. We have completed the whole proof.
\qed

\section{Two-scale expansions}\label{section:4}
\noindent
Recall that $L^2(0,T;X)$ denotes the Banach space of
$X$-valued functions possessing $L^2$-integrability with respect to the time variable, and in the following $X$ will be taken as either  $H^1_0(\Omega)$ or $H^{-1}(\Omega)$.

\begin{lemma}[Error of the two-scale expansion]\label{lemma:3.4}
Suppose that $u_\varepsilon,u_0\in L^2(0,T;H^1(\Omega))$ with
$\partial_t u_\varepsilon,\partial_t u_0\in L^2(0,T;H^{-1}(\Omega))$ satisfy
\begin{equation}\label{eq:3.4}
\left\{\begin{aligned}
\partial_t u_\varepsilon + \mathcal{L}_\varepsilon(u_\varepsilon)
& = \partial_t u_0 + \mathcal{L}_0(u_0) &\quad&\emph{in}\quad ~\Omega_T;\\
u_\varepsilon & = u_0 &\quad&\emph{on}\quad \partial_p\Omega_T.
\end{aligned}\right.
\end{equation}
Let the error of the two-scale expansion be given by
\begin{equation}\label{eq:3.1}
w_\varepsilon := u_\varepsilon  - u_0
-\varepsilon\bar{\phi}_j^\varepsilon\varphi_j
-\varepsilon^2 \tilde{\sigma}_{l(d+1)j}^\varepsilon\partial_l\varphi_j,
\end{equation}
where $\varphi_j\in W_{2}^{2,1}(\Omega_T)$ vanishes near $\partial_p\Omega_T$, and we define the shifted quantities:
\begin{equation}\label{eq:3.5}
\begin{aligned}
&\bar{\phi}_j^\varepsilon:=\phi_j(\cdot/\varepsilon,\cdot/\varepsilon^2)-\alpha; \\
& \bar{\sigma}_{ikj}^\varepsilon
:=\sigma_{ikj}(\cdot/\varepsilon,\cdot/\varepsilon^2)-
\beta(\cdot/\varepsilon^2);\\
&\tilde{\sigma}_{l(d+1)j}^\varepsilon:=
\sigma_{l(d+1)j}(\cdot/\varepsilon,\cdot/\varepsilon^2)-\pi_{ilj} y_i - \gamma,
\end{aligned}
\end{equation}
in which $\alpha,\gamma,\pi_{klj}\in\mathbb{R}$ are arbitrary constants  determined respectively by $\phi_j$, $\sigma_{l(d+1)j}$; $\beta$ is
a single-variable function determined by $\sigma_{ikj}$\footnote{The explicit expressions for these shifted quantities are given in $\eqref{min_rad*}$. A crucial aspect of this definition is the preservation of the antisymmetry of the flux correctors under the ``shifting'' operation.};
and $y=x/\varepsilon$. Then, it can be derived that
\begin{equation}\label{eq:3.2}
\left\{\begin{aligned}
\partial_t w_\varepsilon + \mathcal{L}_\varepsilon(w_\varepsilon)
& = \nabla\cdot f_\varepsilon - \varepsilon\sum_{l=1}^{d}\pi_{llj}\partial_t\varphi_j&\quad&\emph{in}\quad ~\Omega_T;\\
w_\varepsilon & = 0 &\quad&\emph{on}\quad \partial_p\Omega_T,
\end{aligned}\right.
\end{equation}
where the right-hand side of the above equation is formulated as follows:
\begin{equation}\label{eq:3.3}
\begin{aligned}
f_\varepsilon
&:= (a^\varepsilon-\bar{a})
\big(\nabla u_0-\varphi\big)
 + \varepsilon \big[a^{\varepsilon}\bar{\phi}_j^{\varepsilon}
-\bar{\sigma}_{j}^{\varepsilon}
+ a^{\varepsilon}(\partial \tilde{\sigma}_{(d+1)j})^{\varepsilon}
\big]\nabla\varphi_j \\
& \qquad\quad+ \varepsilon^2\Big[
a^{\varepsilon}\otimes \tilde{\sigma}_{(d+1)j}^{\varepsilon}: \partial^2 \varphi_j
- \tilde{\sigma}_{(d+1)j}^{\varepsilon}\partial_t\varphi_j\Big],
\end{aligned}
\end{equation}
where $1\leq i,j,l,k\leq d$, and it is recalled that
$a^\varepsilon:=a(\cdot/\varepsilon,\cdot/\varepsilon^2)$.
\end{lemma}

\begin{proof}
The proof is inspired by \cite{Geng-Shen-17,Xu-Zhou-17}, and we provide a proof for the sake of
the completeness. Observing the equation $\eqref{eq:3.4}$, we have
\begin{equation}\label{f:3.25}
\begin{aligned}
\frac{\partial w_\varepsilon}{\partial t} + \mathcal{L}_\varepsilon(w_\varepsilon)
&= \mathcal{L}_0(u_0) - \mathcal{L}_\varepsilon(u_0)
-\Big(\frac{\partial}{\partial t}+\mathcal{L}_\varepsilon\Big)
\big[\varepsilon\bar{\phi}_j(y,\tau)\varphi_j\big]\\
&-\Big(\frac{\partial}{\partial t}+\mathcal{L}_\varepsilon\Big)
\Big[\varepsilon^2\tilde{\sigma}_{l(d+1)j}(y,\tau)
\frac{\partial}{\partial x_l}(\varphi_j)\Big] \\
& = -\underbrace{\frac{\partial}{\partial x_i}
\bigg\{ q_{ij}(y,\tau)\varphi_j\bigg\}}_{I_1}
-\underbrace{\frac{\partial}{\partial t}\big[\varepsilon\bar{\phi}_j(y,\tau)\varphi_j\big]}_{I_2}
-\underbrace{\Big(\frac{\partial}{\partial t}+\mathcal{L}_\varepsilon\Big)
\big[\varepsilon^2 \tilde{\sigma}_{l(d+1)j}(y,\tau)
\frac{\partial\varphi_j}{\partial x_l}\big]}_{I_3}\\
&-\frac{\partial}{\partial x_i}\bigg\{[\bar{a}_{ij} - a_{ij}(y,\tau)]\big[\frac{\partial u_0}{\partial x_j}
- \varphi_j\big]
-\varepsilon a_{ik}(y,\tau)\bar{\phi}_j(y,\tau)
\frac{\partial\varphi_j}{\partial x_k}\bigg\},
\end{aligned}
\end{equation}
where $q_{ij}$ is given as in $\eqref{f:5.29}$, and
$y=x/\varepsilon$ with $\tau = t/\varepsilon^2$. The last line of $\eqref{f:3.25}$ is
a favorable term, and the task is reduced to computing $I_1$, $I_2$, and $I_3$. Recalling
Proposition $\ref{P:5.3}$, it is not hard to see that
\begin{equation}\label{f:3.26}
\begin{aligned}
~I_1 + I_2
=^{\eqref{eq:5.2}} \varepsilon\bar{\phi}_j(y,\tau)\frac{\partial\varphi_j}{\partial t}
+ q_{ij}(y,\tau)\frac{\partial\varphi_j}{\partial x_i}.
\end{aligned}
\end{equation}
Then, in view of
the equations $\eqref{eq:5.1b}$ and $\eqref{eq:5.1c}$, the right-hand side
of $\eqref{f:3.26}$ can be rewritten as
\begin{equation}\label{f:3.27}
\begin{aligned}
&\qquad\qquad
\varepsilon\frac{\partial}{\partial y_{k}}\big[\tilde{\sigma}_{k(d+1)j}\big]
\frac{\partial\varphi_j}{\partial t}
+ \varepsilon\sum_{k=1}^{d}\pi_{kkj}\frac{\partial\varphi_j}{\partial t}
+ \frac{\partial}{\partial y_{k}}\big[\bar{\sigma}_{k ij}\big]
\frac{\partial\varphi_j}{\partial x_i}
+ \frac{\partial}{\partial \tau}\big[\tilde{\sigma}_{(d+1)ij}\big]
\frac{\partial\varphi_j}{\partial x_i}\\
& =\varepsilon^2\frac{\partial}{\partial x_{k}}\big[\tilde{\sigma}_{k(d+1)j}(y,\tau)\big]
\frac{\partial\varphi_j}{\partial t}
+\varepsilon\frac{\partial}{\partial x_{k}}
\big[\bar{\sigma}_{k ij}(y,\tau)\big]
\frac{\partial\varphi_j}{\partial x_i}
+ \varepsilon^2\frac{\partial}{\partial t}\big[\tilde{\sigma}_{(d+1)ij}(y,\tau)\big]
\frac{\partial\varphi_j}{\partial x_i}
+ \varepsilon\sum_{k=1}^{d}\pi_{kkj}\partial_t\varphi_j,
\end{aligned}
\end{equation}
where it is noted that $\frac{\partial \bar{\sigma}_{k ij}}{\partial y_{k}}=\frac{\partial \sigma_{k ij}}{\partial y_{k}}$
and $\frac{\partial\tilde{\sigma}_{(d+1)ij}}{\partial \tau}
=\frac{\partial \sigma_{(d+1)ij}}{\partial \tau}$; also recall
$\bar{\sigma}_{k ij}(y,\tau) := \sigma_{k ij}(y,\tau)-\beta(\tau)$,
and therefore $\partial_{y_k}\beta=0$. Again, by employing the antisymmetry of $\sigma$ (i.e., $\eqref{skew-symmetric}$), the second line of $\eqref{f:3.27}$
is equal to
\begin{equation*}
\begin{aligned}
\varepsilon^2\frac{\partial}{\partial x_{k }}\bigg\{\tilde{\sigma}_{k (d+1)j}(y,\tau)
\frac{\partial\varphi_j}{\partial t}\bigg\}
+\varepsilon\frac{\partial}{\partial x_{k}}
\bigg\{\bar{\sigma}_{k ij}(y,\tau)
\frac{\partial\varphi_j}{\partial x_i}\bigg\}
+\varepsilon^2\frac{\partial}{\partial t}
\bigg\{\tilde{\sigma}_{(d+1)ij}(y,\tau)\frac{\partial\varphi_j}{\partial x_i}\bigg\}
+ \varepsilon\sum_{k=1}^{d}\pi_{kkj}\partial_t\varphi_j.
\end{aligned}
\end{equation*}
Thus, it is obtained that
\begin{equation*}
\begin{aligned}
&I_1 + I_2 + I_3\\
&:= \varepsilon^2\frac{\partial}{\partial x_{k}}\bigg\{\tilde{\sigma}_{k(d+1)j}(y,\tau)
\frac{\partial\varphi_j}{\partial t}\bigg\}
+\varepsilon\frac{\partial}{\partial x_{k}}
\bigg\{\bar{\sigma}_{kij}(y,\tau)
\frac{\partial\varphi_j}{\partial x_i}\bigg\}
+\varepsilon^2\mathcal{L}_\varepsilon\big[\tilde{\sigma}_{l(d+1)j}(y,\tau)
\frac{\partial\varphi_j}{\partial x_l}\big]
+ \varepsilon\sum_{k=1}^{d}\pi_{kkj}\partial_t\varphi_j,
\end{aligned}
\end{equation*}
and this, together with the last line of $\eqref{f:3.25}$, implies the desired formula
$\eqref{eq:3.3}$.
\end{proof}

\subsection{Smoothing operators and distance weights}

\begin{definition}\label{def:1}
\emph{Fix $\eta\in C_0^\infty(B_{\frac{1}{2}}(0))$ with $\int_{\mathbb{R}^d}\eta = 1$. A smoothing
operator associated with the spatial variable is defined as
\begin{equation}\label{def:K}
 K_\varepsilon(f)(x,t) =
 \int_{\mathbb{R}^d}dy\eta_\varepsilon(x-y)f(y,t),
\end{equation}
where $\eta_\varepsilon(x)=\varepsilon^{-d}\eta(x/\varepsilon)$.
Let $\zeta\in C_0^\infty(Q_{\frac{1}{2}}(0))$ satisfy $\int_{\mathbb{R}^{d+1}}\zeta  = 1$.
A parabolic smoothing operator is defined as
\begin{equation}\label{def:S}
 S_\varepsilon(f)(x,t) = \int_{\mathbb{R}^{d+1}}
 dyds\zeta_\varepsilon(x-y,t-s)f(y,s),
\end{equation}
where $\zeta_\varepsilon(x,t)=\varepsilon^{-d-2}
\zeta(x/\varepsilon,t/\varepsilon^2)$.}
\end{definition}

\begin{lemma}\label{lemma:2.4}
Let $1\leq p<\infty$.
Assume $f\in C^{1}_{0}(\mathbb{R}^{d+1})$.
Then it holds that
\begin{equation}\label{pri:2.17}
\begin{aligned}
\big\|K_\varepsilon(f)\big\|_{L^p(\mathbb{R}^{d+1})}
&\leq C\big\|f\big\|_{L^p(\mathbb{R}^{d+1})};\\
\big\|\nabla K_\varepsilon(f)\big\|_{L^p(\mathbb{R}^{d+1})}
&\leq C\varepsilon^{-1}\big\|f\big\|_{L^p(\mathbb{R}^{d+1})};\\
\big\|\partial_t S_\varepsilon(f)\big\|_{L^p(\mathbb{R}^{d+1})}
&\leq C\varepsilon^{-2}\big\|f\big\|_{L^p(\mathbb{R}^{d+1})},
\end{aligned}
\end{equation}
as well as,
\begin{subequations}
\begin{align}
&\big\|f-K_\varepsilon(f)\big\|_{L^p(\mathbb{R}^{d+1})}
\leq C\varepsilon \big\|\nabla f\big\|_{L^p(\mathbb{R}^{d+1})};
\label{pri:2.13}\\
& \big\|f-S_\varepsilon(f)\big\|_{L^p(\mathbb{R}^{d+1})}
\leq C\varepsilon \big\|\nabla f\big\|_{L^p(\mathbb{R}^{d+1})}
+ C\varepsilon^2\big\|\partial_t f\big\|_{L^p(\mathbb{R}^{d+1})},
\label{pri:2.14}
\end{align}
\end{subequations}
where $C$ depends at most on $d$, $\zeta$, and $\eta$.
\end{lemma}

\begin{proof}
The proof is analogous to that given in Lemma $\ref{lemma:2.5}$, and is therefore omitted. Also, we refer the reader to the literature \cite{Geng-Shen-17,Shen18,Xu-Zhou-17} for the basic idea.
\end{proof}

\begin{lemma}[Weighted-type estimates]\label{lemma:2.5}
Let $1< p<\infty$, and $f\in C^{1}_{0}(\Omega_T)$ be supported in $\Sigma_{2\varepsilon}^T$.
Then, for any $\alpha\in\mathbb{R}$, it holds that
\begin{subequations}
\begin{align}
&\|\delta^{\frac{\alpha}{p}} K_\varepsilon(f)\|_{L^p(\Sigma_{2\varepsilon}^T)}
 \lesssim\|\delta^{\frac{\alpha}{p}} f\|_{L^p(\Sigma_{2\varepsilon}^T)};
 \label{pri:2.8}\\
&\|\delta^{\frac{\alpha}{p}} \nabla K_\varepsilon(f)\|_{L^p(\Sigma_{2\varepsilon}^T)}
 \lesssim \varepsilon^{-1}\|\delta^{\frac{\alpha}{p}}f\|_{L^p(\Sigma_{2\varepsilon}^T)},
 \label{pri:2.9}
\end{align}
\end{subequations}
where $\delta(z):= \text{dist}(z,\partial_p\Omega_T)$ or $[\mu_d(z)]^{\frac{p}{\alpha}}$, as given in $\eqref{pri:5.5**}$, and
\begin{subequations}
\begin{align}
&\|\delta^{\frac{\alpha}{p}}(K_\varepsilon(f)-f)
\|_{L^p(\Sigma_{2\varepsilon}^T)}
\lesssim \varepsilon
\|\delta^{\frac{\alpha}{p}}\nabla f\|_{L^p(\Sigma_{\varepsilon}^T)};
\label{pri:2.10}\\
&\|\delta^{\frac{\alpha}{p}}(S_\varepsilon(f)-f)
\|_{L^p(\Sigma_{2\varepsilon}^T)}
\lesssim \varepsilon
\|\delta^{\frac{\alpha}{p}}\nabla f\|_{L^p(\Sigma_{\varepsilon}^T)}
+ \varepsilon^2
\|\delta^{\frac{\alpha}{p}}\partial_t f\|_{L^p(\Sigma_{\varepsilon}^T)},
\label{pri:2.12}
\end{align}
\end{subequations}
where the multiplicative constant depends at most on $d$, $\eta$, $\alpha$, $p$, $\eta$, and $\zeta$.
\end{lemma}

\begin{proof}
The above results generalize those stated in \cite[Lemmas 3.2, 3.3]{Xu-16}.
Since the estimates $\eqref{pri:2.9}$ and $\eqref{pri:2.10}$
can be proved in a manner similar to the proofs of $\eqref{pri:2.8}$ and $\eqref{pri:2.12}$, respectively, only
the corresponding proofs are shown, organized in two steps.

\textbf{Step 1.} Arguments for $\eqref{pri:2.8}$.
For any $X=(x,t)\in\Sigma_{2\varepsilon}^T$,
the following assertion is stated as a starting point:
for any $\alpha\in\mathbb{R}$, there exist $C_1$ and $C_2$, depending
only on $\alpha$, such that
\begin{equation}\label{f:2.13}
C_1\delta^\alpha(X)\leq
\dashint_{B_\varepsilon(x)}\delta^\alpha(\cdot,t)
\leq C_2 \delta^{\alpha}(X),
\end{equation}
due to the Lipschitz continuity of the distance function $\delta$ and the fact that $\delta(X)\geq 2\varepsilon$. Then,
it follows from
H\"older's inequality that
\begin{equation}\label{f:2.14}
\begin{aligned}
\big|K_{\varepsilon}(f)(X)\big|^p
&\leq \int_{\mathbb{R}^d}\eta_{\varepsilon}(x-\cdot)|f(\cdot,t)|^{p}
\delta^{\alpha}(\cdot,t)
\Big(\int_{\mathbb{R}^d}\eta_{\varepsilon}(x-\cdot)
\delta^{-\frac{\alpha}{p-1}}(\cdot,t)\Big)^{p-1}\\
&\lesssim  \int_{\mathbb{R}^d}\eta_{\varepsilon}(x-\cdot)|f(\cdot,t)|^{p}
\delta^{\alpha}(\cdot,t)
\Big(\dashint_{B_\varepsilon(x)}
\delta^{-\frac{\alpha}{p-1}}(\cdot,t)\Big)^{p-1} \\
&\lesssim^{\eqref{f:2.13}} \delta^{-\alpha}(x,t)
\int_{\mathbb{R}^d}\eta_{\varepsilon}(x-\cdot)|f(\cdot,t)|^{p}
\delta^{\alpha}(\cdot,t).
\end{aligned}
\end{equation}
This provides with the following:
\begin{equation*}
\begin{aligned}
\int_{\Sigma_{2\varepsilon}^T}dX
\big|K_{\varepsilon}(f)(X)\big|^p\delta^{\alpha}(X)
&\lesssim^{\eqref{f:2.14}}
\int_{\Sigma_{2\varepsilon}^T}dX
\int_{\mathbb{R}^d}\eta_{\varepsilon}(x-\cdot)|f(\cdot,t)|^{p}
\delta^{\alpha}(\cdot,t)\\
&\lesssim \int_{4\varepsilon^2}^{T-4\varepsilon^2}dt
\int_{\Sigma_\varepsilon}|f(\cdot,t)|^{p}
\delta^{\alpha}(\cdot,t)
= \int_{\Sigma_{\varepsilon}^T}dX|f(X)|^p\delta^{\alpha}(X),
\end{aligned}
\end{equation*}
where we recall $\Sigma_\varepsilon:=\{x\in\Omega:\text{dist}(x,\partial\Omega)\geq \varepsilon\}$, and this gives the stated estimate $\eqref{pri:2.8}$.

\textbf{Step 2.}
Arguments for $\eqref{pri:2.12}$.
For any $Y=(y,\tau)\in\mathbb{R}^{d+1}$ satisfying $|y|\leq 1$
and $|\tau|\leq 1$,
let $X=(x,t)\in\Sigma_{2\varepsilon}^{T}$, and set
$\varepsilon Y:=(\varepsilon y,\varepsilon^2\tau)$.
It is not hard to see that
\begin{equation*}
\begin{aligned}
&\big|g(X)-g(X-\varepsilon Y)\big|^p
\leq \varepsilon^p\int_0^1|\nabla g(x+(s-1)\varepsilon y,t)|^p ds
+ \varepsilon^{2p}
\int_0^1|\partial_t g(x-\varepsilon y,t+(\theta-1)\varepsilon^2\tau)|^p d\theta.
\end{aligned}
\end{equation*}

Then, integrating with respect to $X$,
we obtain that
\begin{equation}\label{f:2.5}
\begin{aligned}
&\int_{\Sigma_{2\varepsilon}^T}dX\big|g(X)-g(X-\varepsilon Y)\big|^p
\delta^{\alpha}(X-\varepsilon Y) \\
&\lesssim_\alpha \varepsilon^p
\int_{\Sigma_{2\varepsilon}^T}dX \delta^{\alpha}(X) \int_0^1 ds\big|\nabla g(X+\beta_s \varepsilon Y)
\big|^p
 + \varepsilon^{2p}
\int_{\Sigma_{2\varepsilon}^T}dX\delta^{\alpha}(X)
\int_0^1d\theta\big|\partial_t g(X-\alpha_\theta\varepsilon Y)\big|^p  \\
&\lesssim_\alpha
\varepsilon^p\int_{\Sigma_{\varepsilon}^T}dZ\delta^{\alpha}(Z)\big|\nabla g(Z)
\big|^p
+\varepsilon^{2p}\int_{\Sigma_{\varepsilon}^T}dZ
\delta^{\alpha}(Z)\big|\partial_t g(Z)
\big|^p,
\end{aligned}
\end{equation}
where $\beta_s \varepsilon Y=:((s-1)\varepsilon y,0)$
and $\alpha_\theta \varepsilon Y=:(\varepsilon y,(1-\theta)\varepsilon^2\tau)$, and it is not hard to
observe that
$X+\beta_s \varepsilon Y, X-\alpha_\theta\varepsilon Y\in \Sigma_\varepsilon^T$. In the estimate $\eqref{f:2.5}$, we employ
the fact that $\delta^\alpha(X)\sim \delta^{\alpha}(X-\varepsilon \gamma Y)$ provided $|\gamma|\leq 1$ and $\text{d}(Y,0)\leq 1$.
Before continuing the proof, let us make a preparation similar to $\eqref{f:2.13}$.
For any $\alpha\in\mathbb{R}$,
there exist $C_1$ and $C_2$, depending only on
$\alpha$, such that
\begin{equation}\label{f:3.41-ap}
 C_1\delta^{\alpha}(X)
 \leq \dashint_{Q_\varepsilon(X)}\delta^{\alpha}
 \leq C_2\delta^{\alpha}(X)
\end{equation}
holds for any $X\in \Omega_T$ satisfying $\delta(X)\geq 2\varepsilon$.

Now, it follows from H\"older's inequality and
Fubini's  theorem that
\begin{equation*}
\begin{aligned}
&\int_{\Sigma_{2\varepsilon}^T}dX\big|g(X)-S_\varepsilon(g)(X)\big|^p\delta^{\alpha}(X)
= \int_{\Sigma_{2\varepsilon}^T}dX
\Big|\int_{\mathbb{R}^{d+1}}dY\zeta_\varepsilon(X-Y)\big(g(X)-g(Y)\big)\Big|^p
\delta^{\alpha}(X)\\
&\lesssim \int_{\Sigma_{2\varepsilon}^T}dX
\Big(\dashint_{Q_\varepsilon(X)}\delta^{-\frac{\alpha}{p-1}}\Big)^{p-1}
\delta^{\alpha}(X)
\int_{\mathbb{R}^{d+1}}dY \zeta_\varepsilon(X-Y)\big|g(X)-g(Y)\big|^p
\delta^{\alpha}(Y)  \\
& \lesssim^{\eqref{f:3.41-ap}} \int_{Q_1}dY \zeta(Y)\int_{\Sigma_{2\varepsilon}^T}dX\big|g(X)-g(X-\varepsilon Y)\big|^p
\delta^{\alpha}(X-\varepsilon Y)\\
&\lesssim^{\eqref{f:2.5}}
\varepsilon^p\int_{\Sigma_\varepsilon^T}dZ\big|\nabla g(Z)
\big|^p\delta^\alpha(Z)
+\varepsilon^{2p}\int_{\Sigma_\varepsilon^T}dZ\big|\partial_t g(Z)
\big|^p\delta^{\alpha}(Z) ,
\end{aligned}
\end{equation*}
which implies the desired estimate $\eqref{pri:2.12}$,
and we have completed the whole proof.
\end{proof}

\begin{lemma}[Random cancellation]\label{lemma:5}
Let $\Omega_{T}\subset\mathbb{R}^{d+1}$ be a subset
with $d\geq 2$ and $T>0$ (not necessary to be a bounded parabolic cylinder). Let $p\in[1,\infty)$ and $r\in[2,\infty)$.
For any $\beta\in [(p\vee r),\infty)$, assume that the random field $\varpi$ satisfies
\begin{equation}\label{pri:5.8}
 \Big\langle\big(\dashint_{Q_1(z)}|\varpi|^r\big)^{\frac{\beta}{r}}
 \Big\rangle^{\frac{1}{\beta}}
 \leq c_0 \theta(z)
 \qquad\forall z\in\mathbb{R}^{d+1}.
\end{equation}
Then, for any $r\leq q<\infty$, $\alpha\in\mathbb{R}$, and $f\in C^{1}_{0}(\Omega_T)$ supported in $\Sigma_{2\varepsilon}^T$, it holds that
\begin{subequations}
\begin{align}
& \int_{\Omega_T}dz \Big\langle\big(\dashint_{U_\varepsilon(z)}
 |\varpi^\varepsilon S_\varepsilon(f)|^r\big)^{\frac{p}{r}}\Big\rangle^{\frac{q}{p}}
 \delta^{\alpha}(z)
  \lesssim_{d,c_0}
 \int_{\Omega_T} |f|^q\theta^q(\cdot/\varepsilon)
 \delta^{\alpha};
 \label{pri:3.42a}\\
& \int_{\Omega_T}dz
\Big\langle\big(\dashint_{U_\varepsilon(z)}
 |\varpi^\varepsilon \nabla  S_\varepsilon(f)|^r\big)^{\frac{p}{r}}\Big\rangle^{\frac{q}{p}}
 \delta^{\alpha}(z)
  \lesssim_{d,c_0}
 \varepsilon^{-q}\int_{\Omega_T} |f|^q\theta^q(\cdot/\varepsilon)
 \delta^{\alpha};
 \label{pri:3.42b}\\
& \Big\langle\Big(\int_{\Omega_T} dz
 |\varpi^\varepsilon(z)S_\varepsilon(f)(z)|^2
 \delta^{\alpha}(z)\Big)^{\frac{p}{2}}\Big\rangle^{\frac{2}{p}}
 \lesssim_{d,c_0}
 \int_{\Omega_T} |f|^2\theta^2(\cdot/\varepsilon)\delta^{\alpha};
 \label{pri:3.42c} \\
& \Big\langle\Big(\int_{\Omega_T} dz
 |\varpi^\varepsilon(z)\nabla S_\varepsilon(f)(z)|^2
 \delta^{\alpha}(z)\Big)^{\frac{p}{2}}\Big\rangle^{\frac{2}{p}}
 \lesssim_{d,c_0}
 \varepsilon^{-2}
 \int_{\Omega_T} |f|^2\theta^2(\cdot/\varepsilon)\delta^{\alpha},
 \label{pri:3.42d}
\end{align}
\end{subequations}
where it is recalled that $\delta(z):=\text{dist}(z,\partial_p\Omega_T)$, $U_\varepsilon(z)
:=Q_\varepsilon(z)\cap\Omega_T$, and $S_\varepsilon$ is given in
Definition $\ref{def:1}$.

\end{lemma}

\begin{proof}
The proof is totally deterministic, and the basic idea is similar
to that shown in Lemma $\ref{lemma:2.5}$. To strike a balance between brevity and completeness, we only provide the proofs of estimates $\eqref{pri:3.42b}$ and $\eqref{pri:3.42d}$ correspondingly.

\textbf{Step 1.} Arguments for $\eqref{pri:3.42b}$.
Let $\zeta$ be given as in Definition $\ref{def:1}$.
If we replace $\zeta$ with $|\nabla\zeta|$ in
the formula $\eqref{def:S}$, we denote
the associated smoothing operator by $\tilde{S}_\varepsilon$.
For any $z\in\Omega_T$, it follows from H\"older's inequality
and Fubini's theorem that
\begin{equation}\label{f:3.40-ap}
\begin{aligned}
\dashint_{U_\varepsilon(x)}
|\varpi^\varepsilon
 \nabla S_\varepsilon(f)|^r
\lesssim \varepsilon^{-r} \dashint_{U_\varepsilon(x)}
|\varpi^\varepsilon|^r
\tilde{S}_\varepsilon(|f|^r)
\lesssim \varepsilon^{-r} \dashint_{U_\varepsilon(x)}
\tilde{S}_\varepsilon\big(|\varpi^\varepsilon|^r\big)
|f|^r.
\end{aligned}
\end{equation}
Let $d\mu(z):=\delta^{\alpha}(z)dz$, and we further derive that
\begin{equation*}
\begin{aligned}
&\int_{\Omega_T}d\mu(z)
\Big\langle\big(\dashint_{U_\varepsilon(z)}
|\varpi^\varepsilon \nabla
 S_\varepsilon(f)|^r\big)^{\frac{p}{r}}
\Big\rangle^{\frac{q}{p}}
\lesssim^{\eqref{f:3.40-ap}}
\varepsilon^{-q}\int_{\Omega_T}d\mu(z)\bigg\langle\Big(
 \dashint_{U_\varepsilon(z)}
\tilde{S}_\varepsilon\big(|\varpi^\varepsilon|^r\big)|f|^r\Big)^{\frac{p}{r}}
\bigg\rangle^{\frac{q}{p}}\\
&\lesssim \varepsilon^{-q}
\int_{\Omega_T}d\mu(z)
\bigg(\dashint_{U_\varepsilon(z)}
|f|^2
\Big\langle\big(\tilde{S}_\varepsilon(
|\varpi^\varepsilon|^r)\big)^{\frac{\beta}{r}}
\Big\rangle^{\frac{r}{\beta}}\bigg)^{\frac{q}{r}}\\
&\lesssim^{\eqref{pri:5.8}}
\varepsilon^{-q}
\int_{\Omega_T}d\mu(z)
\bigg(\dashint_{U_\varepsilon(z)}
|f|^r\theta^r(\cdot/\varepsilon)\bigg)^{\frac{q}{r}}
\lesssim
\varepsilon^{-q}
\int_{\Omega_T}dz
|f(z)|^q\theta^q(\cdot/\varepsilon) \dashint_{Q_\varepsilon(z)}\delta^{\alpha}\\
&\lesssim^{\eqref{f:3.41-ap}}
\varepsilon^{-q}
\int_{\Omega_T}
|f|^q \theta^q(\cdot/\varepsilon)\delta^{\alpha},
\end{aligned}
\end{equation*}
where we employ Minkowski's inequality in the second inequality,
and H\"older's inequality with Fubini's theorem in the fourth inequality.
By the same token, one can establish the estimate
$\eqref{pri:3.42a}$.

\textbf{Step 2.}  Arguments for $\eqref{pri:3.42c}$. By the definition of $S_\varepsilon$ in $\eqref{def:S}$,
it can be obtained that
\begin{equation*}
\begin{aligned}
\int_{\Omega_T}dz|\varpi^\varepsilon(z)|^2
&|S_\varepsilon(f)(z)|^2 \delta^\alpha(z)
\leq  \int_{\Omega_T}dz
|\varpi^\varepsilon(z)|^2 S_{\varepsilon}(|f|^2\delta^{\alpha})(z)
S_{\varepsilon}(\delta^{-\alpha})(z) \delta^\alpha(z)\\
&\lesssim^{\eqref{f:3.41-ap}}
\int_{\Omega_T}dz
|\varpi^\varepsilon(z)|^2 S_{\varepsilon}(|f|^2\delta^{\alpha})(z)
\lesssim \int_{\Omega_T}dz |f(z)|^2\delta^{\alpha}(z) \dashint_{Q_1(z/\varepsilon)}|\varpi|^2.
\end{aligned}
\end{equation*}
Then, by taking $\langle|\cdot|^{\frac{p}{2}}\rangle^{\frac{2}{p}}$ on
the both sides above and then applying Minkowski's inequality
to its right-hand side, it is derived that
\begin{equation*}
\begin{aligned}
\Big\langle\Big(\int_{\Omega_T}dz|\varpi^\varepsilon(z)|^2
&|S_\varepsilon(f)(z)|^2 \delta^\alpha(z)\Big)^{\frac{p}{2}}
\Big\rangle^{\frac{2}{p}}
\lesssim \Big\langle\Big(\int_{\Omega_T}dz |f(z)|^2\delta^{\alpha}(z) \dashint_{Q_1(z/\varepsilon)}|\varpi|^2\Big)^{\frac{\beta}{2}}
\Big\rangle^{\frac{2}{\beta}}\\
&\lesssim \int_{\Omega_T}dz |f(z)|^2\delta^{\alpha}(z) \Big\langle\Big(\dashint_{Q_1(z/\varepsilon)}|\varpi|^2
\Big)^{\frac{\beta}{2}}
\Big\rangle^{\frac{2}{\beta}}
\lesssim^{\eqref{pri:5.8}}
\int_{\Omega_T}dz |f(z)|^2\delta^{\alpha}(z)\theta^2(z/\varepsilon),
\end{aligned}
\end{equation*}
where the fact that $\beta\geq p$ is employed in the first inequality, and this yields the stated estimate $\eqref{pri:3.42c}$. We have completed
the whole proof.
\end{proof}

\subsection{Proof of Lemma $\ref{lemma:approxi}$}

\noindent
By a rescaling argument, it may be assumed without loss of generality that  $r=1$.
Let $\Psi_{\delta}\in C_0^{2}(D_2)$ with $0< \delta\leq  1/2$ (to be fixed later) be a cut-off function satisfying $\Psi_{\delta} = 1$ on $D_{2-4\delta}$,
$\Psi_{\delta} = 0$ outside $D_{2-\delta}$ and
\begin{equation*}
(1/\delta)|\nabla\Psi_{\delta}|+|\nabla^2\Psi_{\delta}|+|\partial_t\Psi_{\delta}|
\lesssim_d 1/\delta^2
\qquad\text{on}
\quad D_{2-\delta}\setminus D_{2-2\delta}.
\end{equation*}
Let $\bar{u}_1$
be given by the equations: $\big(\partial_t+\mathcal{L}_0\big)(\bar{u}_1) = 0$ in $D_{2}$ with $\bar{u}_1 = u_\varepsilon$ on $\partial_p  D_{2}$. Appealing to Lemma $\ref{lemma:3.4}$, the error expansion
associated with $u_\varepsilon$ and $\bar{u}$
is similarly denoted by $w_\varepsilon$ as in $\eqref{eq:3.1}$ (with  $\varphi_j\in W^{2,1}_2(D_2)$ supported in $D_2$), and
it holds that
\begin{equation*}
\left\{\begin{aligned}
\partial_t w_\varepsilon + \mathcal{L}_\varepsilon(w_\varepsilon)
& = \nabla\cdot\tilde{f}_\varepsilon
-\varepsilon\sum_{l=1}^{d}\pi_{llj}\partial_t\varphi_j &\quad&\text{in}\quad ~D_2;\\
w_\varepsilon & = 0 &\quad&\text{on}\quad \partial_p D_{2},
\end{aligned}\right.
\end{equation*}
where $\tilde{f}_\varepsilon$ is given as in $\eqref{eq:3.3}$.

Noting that one may choose $\varphi_{j} =
S_\delta (\Psi_{\delta}\partial_j\bar{u}_1)$
in $\eqref{eq:3.1}$,
it follows from
the energy estimate that
\begin{equation}\label{f:6.13-ap}
\begin{aligned}
&\|\nabla w_\varepsilon\|_{L^2(D_2)}
\lesssim \big\|\tilde{f}_\varepsilon\big\|_{L^2(D_2)}
+ \varepsilon|\sum_{l=1}^{d}\pi_{llj}|\big\|\partial_t\varphi_j
\big\|_{L^2(D_2)}\\
&\lesssim
\varepsilon|\sum_{l=1}^{d}\pi_{llj}|\big\|\partial_t\varphi_j
\big\|_{L^2(D_2)}
+
\big\|\nabla \bar{u}_1
- S_\delta (\Psi_{\delta}\nabla \bar{u}_1)\big\|_{L^2(D_2)} \\
& +\varepsilon\big\|\varpi_j^{\varepsilon}
\nabla S_\delta (\Psi_{\delta}\partial_j\bar{u}_1)\big\|_{L^2(D_2)}
 +\varepsilon^2 \big\|\tilde{\sigma}_{(d+1)j}^\varepsilon\big(\nabla^2
+\partial_t\big)S_\delta
(\Psi_{\delta}\partial_j \bar{u}_1)\big\|_{L^2(D_2)}
=:I_0+ I_1 + I_2 + I_3,
\end{aligned}
\end{equation}
where $\varpi_j := (\bar{\phi}_j,\bar{\sigma}_j,\nabla\tilde{\sigma}_{(d+1)j})$ are referred to as
the extended correctors, and we recall the notation convention
$\varpi^\varepsilon:=\varpi(\cdot/\varepsilon,\cdot/\varepsilon^2)$.

To obtain  $I_0$, one may further set $\pi_{klj}:=\dashint_{Q_{2/\varepsilon}}\partial_k\sigma_{l(d+1)j}$ in $\eqref{eq:3.1}$, and observe that
\begin{equation}\label{f:6.17-ap}
\sum_{k=1}^{d}\pi_{kkj} =\sum_{k=1}^{d}\dashint_{Q_{2/\varepsilon}}\partial_k\sigma_{k(d+1)j}
=^{\eqref{eq:5.1c}} \dashint_{Q_{2/\varepsilon}}\bar{\phi}_j.
\end{equation}
Therefore, it follows from Lemma $\ref{lemma:2.4}$ that
\begin{equation}\label{f:6.18-ap}
  I_0 \lesssim^{\eqref{f:6.17-ap},\eqref{pri:2.17}} \frac{\varepsilon}{\delta^2}
  \Big|\dashint_{Q_{2/\varepsilon}}\bar{\phi}_j\Big|
  \|\partial_j\bar{u}_1\|_{L^2(D_2)}
  \lesssim \frac{\varepsilon}{\delta^2}
  \Big(\dashint_{Q_{2/\varepsilon}}|\bar{\phi}|^2\Big)^{\frac{1}{2}}
  \|\nabla u_\varepsilon\|_{L^2(D_2)}.
\end{equation}

To estimate $I_1$, we first note
the following fact
\begin{equation}\label{eq:6.2}
\begin{aligned}
\nabla \bar{u}_1
- S_\delta (\Psi_{\delta}\nabla \bar{u}_1)
&= \Psi_{\delta}\nabla \bar{u}_1
-  S_\delta(\Psi_{\delta}\nabla \bar{u}_1)
+ (\Psi_{\delta}-1)\nabla \bar{u}_1.
\end{aligned}
\end{equation}
Hence, it follows from Lemma $\ref{lemma:2.4}$  that
\begin{equation}\label{f:6.6-ap}
\begin{aligned}
I_1
&\lesssim^{\eqref{pri:2.14}} \delta
\bigg\{\|\nabla(\Psi_{\delta}\bar{u}_1)\|_{L^2(\mathbb{R}^{d+1})}
+ \delta\|\partial_t(\Psi_{\delta}\nabla \bar{u}_1)\|_{L^2(\mathbb{R}^{d+1})}
\bigg\}
+\|\nabla \bar{u}_1\|_{L^2(\boxbox_{4\delta})}\\
&\lesssim \|\nabla \bar{u}_1\|_{L^2(\boxbox_{4\delta})}
+ \delta\|(\nabla^2\bar{u}_1,\delta\partial_t
\nabla\bar{u}_1)\|_{L^2(Q_{2-2\delta})},
\end{aligned}
\end{equation}
where
$\boxbox_{4\delta}=Q_2\setminus Q_{2-4\delta}$.
For ease of presentation, one has the decomposition
$\boxbox_{4\delta} \subset H_{\delta} \cup T_{\delta}$,
in which we set $H_{\delta}:=\{x\in B_2^{+}: \text{dist}(x,\partial\Omega)\leq 4\delta\}\times (-4,4)$,
and
\begin{equation*}
T_{\delta}:= B_{2}^{+}\times\big\{(-4,-4+16\delta^2] \cup
[4-16\delta^2,4)\big\}.
\end{equation*}
We then obtain that
\begin{equation}\label{f:6.7-ap}
\begin{aligned}
\|\nabla \bar{u}_1\|_{L^2(\boxbox_{4\delta})} &\leq
\|\nabla \bar{u}_1\|_{L^2(H_{\delta})}
+  \|\nabla \bar{u}_1\|_{L^2(T_{\delta})}\\
&\lesssim \delta^{\frac{1}{2}-\frac{1}{p}}
\|\nabla \bar{u}_1\|_{L^p(D_2)}
+ \delta^{\frac{1}{2}}\|\nabla \bar{u}_1\|_{L^2(D_2)}
\lesssim \delta^{\frac{1}{2}-\frac{1}{p}}
\|\nabla u_\varepsilon\|_{L^2(D_4)},
\end{aligned}
\end{equation}
where $0<p-2\ll 1$ is referred to as the Meyer-type exponent; the mean value theorem for integrals and the definition of Riemann integral are mainly employed in the second inequality, while the third one primarily follows from Meyer-type estimates\footnote{Here,
we indeed employ a global Meyer-type estimate
which follows from the interior and boundary estimates $\eqref{pri:5.9-ap}$, $\eqref{pri:5.12-ap}$, respectively.
Then we use $\eqref{pri:5.12-ap}$ again
to get the last inequality stated in
$\eqref{f:6.7-ap}$.} (see Lemma \ref{lemma:5.2}). Moreover, for any
$z\in D_{2-2\delta}$, let $\rho(z):= \text{dist}(z,\partial_{\shortparallel}D_2)$ be the distance function,
where it is recalled  that $\partial_{\shortparallel}D_2:=\partial_p D_2\cap \partial_{\shortparallel}\Omega_{*}
=\big(B_2\cap\partial\Omega\big)\times(-4,4)$.
It holds that
\begin{equation}\label{f:6.8-ap}
\begin{aligned}
\delta^2\int_{D_{2-2\delta}}|\nabla^2 \bar{u}_1|^2
&\lesssim \delta^2\int_{D_{2}\setminus H_{\delta/2}}|\nabla^2 \bar{u}_1|^2
\lesssim \delta^2 \int_{D_{2}\setminus H_{\delta/2}} dz\rho^{-2}(z)
\dashint_{Q_{\rho}(z)}
|\nabla\bar{u}_1|^2 \\
&\lesssim \delta^2\Big(\int_{D_{2}\setminus H_{\delta/2}} dz [\rho(z)]^{-\frac{2p}{p-2}}\Big)^{1-\frac{2}{p}}
\Big(\int_{D_{2}\setminus H_{\delta/2}}\dashint_{Q_{\rho}(z)}|\nabla \bar{u}_1|^p\Big)^{\frac{2}{p}}\\
&\lesssim \delta^{1-\frac{2}{p}}\Big(\int_{D_2}|\nabla \bar{u}_1|^p\Big)^{\frac{2}{p}}
\lesssim \delta^{1-\frac{2}{p}}\Big(\int_{D_2}|\nabla u_\varepsilon|^p\Big)^{\frac{2}{p}}
\lesssim \delta^{1-\frac{2}{p}}
\int_{D_4}|\nabla u_\varepsilon|^2,
\end{aligned}
\end{equation}
where we use interior regularity estimates in the first step,
and H\"older's inequality with $0<p-2\ll 1$ in the second, and
the last line is primarily due to Meyer-type estimates as before.
By the same token, it may be derived that
\begin{equation}\label{f:6.9-ap}
\delta^4\int_{D_{2-2\delta}}|\partial_t\nabla\bar{u}_1|^2
\lesssim \delta^4\int_{D_{2-2\delta}}|\nabla^3\bar{u}_1|^2
\lesssim \delta^{1-\frac{2}{p}}
\int_{D_4}|\nabla u_\varepsilon|^2
\end{equation}
without any real difficulty. Plugging the estimates $\eqref{f:6.7-ap}$,
$\eqref{f:6.8-ap}$, and $\eqref{f:6.9-ap}$ back into $\eqref{f:6.6-ap}$, one has that
\begin{equation}\label{f:6.10-ap}
I_1
\lesssim \delta^{\beta}
\|\nabla u_\varepsilon\|_{L^2(D_4)}
\quad \text{with~~} \beta =1/2-1/p.
\end{equation}

We now proceed to study $I_2$, and it follows from Fubini's theorem
and energy estimates that
\begin{equation}\label{f:6.11-ap}
\begin{aligned}
I_2
&\lesssim \frac{\varepsilon}{\delta}\sup_{z\in Q_2}
\Big(\dashint_{Q_\delta(z)}
|\varpi_j(\cdot/\varepsilon,\cdot/\varepsilon^2)|^2\Big)^{\frac{1}{2}}
\|\partial_j\bar{u}_1\|_{L^2(D_2)}\\
&\lesssim  \frac{\varepsilon}{\delta^{d+3}}
\Big(\dashint_{Q_{2/\varepsilon}}|\varpi|^2\Big)^{1/2}
\|\nabla\bar{u}_1\|_{L^2(D_2)}
\lesssim \frac{\varepsilon}{\delta^{d+3}}
\Big(\dashint_{Q_{2/\varepsilon}}|\varpi|^2\Big)^{1/2}
\|\nabla u_\varepsilon\|_{L^2(D_2)},
\end{aligned}
\end{equation}
where it is recalled that $\varpi_j = (\bar{\phi}_j,\bar{\sigma}_j,\nabla\tilde{\sigma}_{(d+1)j})$ for
the reader's convenience.
By a similar computation,
it is not hard to derive that
\begin{equation}\label{f:6.12-ap}
I_3 \lesssim \frac{\varepsilon^2}{\delta^{d+4}}
\Big(\dashint_{Q_{2/\varepsilon}}|\tilde{\sigma}_{(d+1)j}|^2\Big)^{1/2}
\|\partial_j\bar{u}_1\|_{L^2(D_2)}
\lesssim \frac{\varepsilon^2}{\delta^{d+4}}
\Big(\dashint_{Q_{2/\varepsilon}}|\tilde{\sigma}_{(d+1)}|^2\Big)^{1/2}
\|\nabla u_\varepsilon\|_{L^2(D_2)}.
\end{equation}

Consequently, by plugging $\eqref{f:6.18-ap}$, $\eqref{f:6.10-ap}$, $\eqref{f:6.11-ap}$, and $\eqref{f:6.12-ap}$ back into $\eqref{f:6.13-ap}$,  it follows that
\begin{equation}\label{f:6.15-ap}
\begin{aligned}
\|\nabla w_\varepsilon\|_{L^2(D_2)}
&\lesssim \bigg\{\delta^{\beta}
+\frac{\varepsilon}{\delta^{d+3}}\Big(\dashint_{Q_{2/\varepsilon}} |\varpi|^{2}
\Big)^{\frac{1}{2}}
+\frac{\varepsilon^2}{\delta^{d+4}}
\Big(\dashint_{Q_{2/\varepsilon}}|\tilde{\sigma}_{(d+1)}|^2\Big)^{\frac{1}{2}}
\bigg\}\|\nabla u_\varepsilon\|_{L^2(D_4)}.
\end{aligned}
\end{equation}

In view of $\eqref{eq:3.1}$ with $\varphi_{j} =
S_\delta (\Psi_{\delta}\partial_j\bar{u}_1)$, it is derived that
\begin{equation}\label{f:6.14-ap}
\begin{aligned}
&\quad\Big(\int_{D_1}|\nabla u_\varepsilon - (e_j+(\nabla\phi_j)^\varepsilon)\varphi_j|^2\Big)^{\frac{1}{2}}
\lesssim
\|\nabla \bar{u}_1\|_{L^2(\boxbox_{4\delta})} +
\Big(\int_{D_2} |\nabla w_\varepsilon|^2\Big)^{\frac{1}{2}}\\
&
+ \Big(\int_{D_2}
\big|\varepsilon\big(\bar{\phi}_j^\varepsilon
+(\nabla\tilde{\sigma}_{l(d+1)j})^\varepsilon e_l \big) \nabla \varphi_j\big|^2
\Big)^{\frac{1}{2}}
+ \Big(\int_{D_2}
\big|\varepsilon^2\tilde{\sigma}_{l(d+1)j}^\varepsilon
\nabla\partial_l\varphi_j\big|^2
\Big)^{\frac{1}{2}}\\
&\lesssim^{\eqref{f:6.7-ap},\eqref{f:6.15-ap},\eqref{f:6.11-ap},\eqref{f:6.12-ap}}
\bigg\{\delta^{\beta}
+\frac{\varepsilon}{\delta^{d+3}}\Big(\dashint_{Q_{2/\varepsilon}} |\varpi|^{2}
\Big)^{\frac{1}{2}}
+\frac{\varepsilon^2}{\delta^{d+4}}
\Big(\dashint_{Q_{2/\varepsilon}}|\tilde{\sigma}_{(d+1)}|^2\Big)^{\frac{1}{2}}
\bigg\}\|\nabla u_\varepsilon\|_{L^2(D_4)}.
\end{aligned}
\end{equation}

Next, by setting $0<\alpha-1\ll 1$ and $1/\alpha+1/\alpha'=1$, consider
\begin{equation}\label{f:7.11}
\Big(\int_{D_1}
|(e_j+(\nabla\phi_j)^\varepsilon)\big(\varphi_{j}
-\partial_j\bar{u}_1\big)|^2\Big)^{\frac{1}{2}}
\leq \underbrace{\Big(\int_{Q_1}
|(e_j+(\nabla\phi_j)^\varepsilon)|^{2\alpha}\Big)^{\frac{1}{2\alpha}}}_{J_1}
\underbrace{\Big(\int_{D_1}
|\varphi_{j}
-\partial_j\bar{u}_1|^{2\alpha'}\Big)^{\frac{1}{2\alpha'}}}_{J_2}.
\end{equation}
In view of the identity $\eqref{eq:6.2}$,
a routine computation analogous to that for $I_1$ yields
\begin{equation*}
\begin{aligned}
J_2 &\leq \|\Psi_{\delta}\nabla \bar{u}_1
-  S_\delta(\Psi_{\delta}\nabla \bar{u}_1)\|_{L^{2\alpha'}(D_1)}
+\|(\Psi_{\delta}-1)\nabla \bar{u}_1\|_{L^{2\alpha'}(D_1)}\\
&\lesssim^{\eqref{pri:2.12}} \delta\Big\{\|\nabla(\Psi_{\delta}\nabla \bar{u}_1)\|_{L^{2\alpha'}(D_{1+\delta})}
+\delta \|\partial_t(\Psi_{\delta}\nabla \bar{u}_1)\|_{L^{2\alpha'}(D_{1+\delta})}\Big\}
+ \|\nabla \bar{u}_1\|_{L^{2\alpha'}(H_{\delta}\cap D_1)}\\
&\lesssim
\|\nabla \bar{u}_1\|_{L^{2\alpha'}(H_{\delta}\cap D_{3/2})}
+\delta\|\nabla^2 \bar{u}_1\|_{L^{2\alpha'}(D_{3/2}\setminus H_{\delta/2})}
+\delta^2\|\partial_t\nabla \bar{u}_1\|_{L^{2\alpha'}(D_{3/2}\setminus H_{\delta/2})}  \\
&\lesssim
\delta^{\frac{1}{s}}\|\nabla \bar{u}_1\|_{L^{2\alpha's'}(D_{3/2})}
+\delta^{\frac{1}{2\alpha'q}}
\|\nabla \bar{u}_1\|_{L^{2\alpha'q'}(D_{3/2})}
\end{aligned}
\end{equation*}
(where we employ the similar computation given in $\eqref{f:6.8-ap}$
for the last step), and this leads to
\begin{equation}\label{f:7.9}
\begin{aligned}
J_2
&\lesssim
\delta^{\frac{1}{s}}\|\nabla \bar{u}_1\|_{L^{2\alpha's'}(D_{3/2})}
+\delta^{\frac{1}{2\alpha'q}}
\|\nabla \bar{u}_1\|_{L^{2\alpha'q'}(D_{3/2})}\\
&\lesssim \delta^{\beta}
\|\nabla \bar{u}_1\|_{L^{4\alpha'}(D_{3/2})}
\lesssim^{\eqref{pri:5.11-ap}} \delta^{\beta}
\|\nabla \bar{u}_1\|_{L^{2}(D_{2})}
\lesssim
\delta^{\beta}
\|\nabla u_\varepsilon\|_{L^{2}(D_{2})},
\end{aligned}
\end{equation}
in which we take $q=2$, $\frac{1}{4\alpha'}=\beta$ (therefore $\alpha=\frac{p}{4-p}$ with $0<p-2\ll 1$ is the Meyer-type exponent), and $s=4\alpha'$ (so that $2\alpha's'\leq 4\alpha'$)  in the first inequality.
Proceeding to handle the term $J_1$,
it follows from the Meyer-type estimate and
Caccioppoli's inequality that
\begin{equation}\label{f:7.10}
 J_1 \lesssim 1 +
 \Big(\dashint_{Q_{1/\varepsilon}}|\nabla \phi|^{2\alpha}\Big)^{\frac{1}{2\alpha}}
 \lesssim^{\eqref{pri:5.9-ap},\eqref{pri:6.1-ap}} 1 + \varepsilon
 \Big(\dashint_{Q_{2/\varepsilon}}
 |\bar{\phi}|^2\Big)^{\frac{1}{2}},
\end{equation}
where we merely employ $\eqref{pri:5.9-ap}$ and $\eqref{pri:6.1-ap}$ in the case $\beta=0$ therein.

Plugging the estimates $\eqref{f:7.9}$ and $\eqref{f:7.10}$ back into $\eqref{f:7.11}$, we have
\begin{equation}\label{f:7.12}
\Big(\int_{D_1}
|(e_j+(\nabla\phi_j)^\varepsilon)\big(\varphi_{j}
-\partial_j\bar{u}_1\big)|^2\Big)^{\frac{1}{2}}
\lesssim \delta^{\beta}\Big(1+
\varepsilon\big(\dashint_{Q_{2/\varepsilon}}
 |\bar{\phi}|^2\big)^{\frac{1}{2}}\Big)\|\nabla u_\varepsilon\|_{L^{2}(D_{2})}.
\end{equation}
Therefore, it is obtained that
\begin{equation*}
\begin{aligned}
&\Big(\int_{D_1}|\nabla u_\varepsilon - (e_j+(\nabla\phi_j)^\varepsilon)\partial_j\bar{u}_1|^2\Big)^{\frac{1}{2}}\\
&\leq \Big(\int_{D_1}|\nabla u_\varepsilon - (e_j+(\nabla\phi_j)^\varepsilon)\varphi_{j}|^2\Big)^{\frac{1}{2}}
+ \Big(\int_{D_1}
|(e_j+(\nabla\phi_j)^\varepsilon)\big(\varphi_{j}
-\partial_j\bar{u}_1\big)|^2\Big)^{\frac{1}{2}}\\
&\lesssim^{\eqref{f:6.14-ap},\eqref{f:7.12}}
\bigg\{\delta^{\beta}
+\frac{\varepsilon}{\delta^{d+3}}\Big(\dashint_{Q_{2/\varepsilon}} |\varpi|^{2}
\Big)^{\frac{1}{2}}
+\frac{\varepsilon^2}{\delta^{d+4}}
\Big(\dashint_{Q_{2/\varepsilon}}|\tilde{\sigma}_{(d+1)}|^2\Big)^{\frac{1}{2}}
\bigg\}\|\nabla u_\varepsilon\|_{L^2(D_4)}.
\end{aligned}
\end{equation*}

Upon rescaling back (and therefore setting $\varepsilon':=\varepsilon/r$), the stated estimates $\eqref{f:6.14-ap}$ can be rewritten as
\begin{equation}\label{f:6.16-ap}
\begin{aligned}
&\Big(\dashint_{D_r}|\nabla u_\varepsilon - (e_j+(\nabla\phi_j)^\varepsilon)\partial_j\bar{u}_r|^2\Big)^{\frac{1}{2}}\\
&\lesssim \bigg\{{\delta}^{\beta}+\frac{\varepsilon'}{{\delta}^{d+3}}
\Big(\dashint_{Q_{2/\varepsilon'}} \big|
\varpi\big|^{2}\Big)^{\frac{1}{2}}
+\frac{{\varepsilon'}^2}{{\delta}^{d+4}}
\Big(\dashint_{Q_{2/\varepsilon'}}
|\tilde{\sigma}_{(d+1)}|^2\Big)^{\frac{1}{2}}\bigg\}
\Big(\dashint_{D_{4r}}|\nabla u_\varepsilon|^2\Big)^{\frac{1}{2}},
\end{aligned}
\end{equation}
where we recall that $\bar{u}_r$ satisfies
$\big(\partial_t+\mathcal{L}_0\big)(\bar{u}_r) = 0$ in $D_{2r}$ with $\bar{u}_r = u_\varepsilon$ on $\partial_p  D_{2r}$. Also, it is recalled that $\alpha,\beta,\gamma$ in $\eqref{eq:3.5}$ are chosen such that
\begin{equation*}
\begin{aligned}
&\bar{\phi}_j:=\phi_j-\dashint_{Q_{2R}}\phi_j;
\qquad \bar{\sigma}_{ikj}(x,t)
:=\sigma_{ikj}(x,t)-\dashint_{B_{2R}}\sigma_{ikj}(\cdot,t);\\
&\tilde{\sigma}_{(d+1)}(x,t)
:= \sigma_{(d+1)}(x,t) - \dashint_{Q_{2R}}\sigma_{(d+1)}
-x\cdot\dashint_{Q_{2R}}\nabla\sigma_{(d+1)}
\quad\text{with}\quad R=1/\varepsilon'.
\end{aligned}
\end{equation*}

In this regard, for any $\theta>0$ (to be determined later), we can
define the stationary field $\chi_*$ (referred to as the minimal radius) as follows:
\begin{equation}\label{min_rad-ap}
\chi_{*}(0):=
 \inf\Bigg\{l>0:
 \frac{1}{R}
  \Big(\dashint_{Q_{2R}}
 \big|(\bar{\phi},\bar{\sigma},\nabla\tilde{\sigma}_{(d+1)})\big|^{2}\Big)^{\frac{1}{2}}
+ \frac{1}{R^2}\Big(\dashint_{Q_{2R}}
 |\tilde{\sigma}_{(d+1)}|^{2}\Big)^{\frac{1}{2}}
 \leq \theta;~\forall R\geq l\Bigg\}\vee 1
\end{equation}
(This is the same as stated in $\eqref{min_rad*}$;
the center point is usually omitted, and it is denoted by $\chi_*$).
By the definition of $\chi_*$, it can be straightforwardly derived that
\begin{equation*}
\begin{aligned}
\Big(\dashint_{D_r}|\nabla u_\varepsilon - (e_j+(\nabla\phi_j)^\varepsilon)\partial_j\bar{u}_r|^2\Big)^{\frac{1}{2}}
&\lesssim \theta^{\frac{\beta}{d+4+\beta}}
\Big(\dashint_{D_{4r}}|\nabla u_\varepsilon|^2\Big)^{\frac{1}{2}}
\leq \nu \Big(\dashint_{D_{4r}}|\nabla u_\varepsilon|^2\Big)^{\frac{1}{2}}
\qquad \forall r\geq \varepsilon\chi_{*},
\end{aligned}
\end{equation*}
which proves the stated result $\eqref{pri:1.5}$ upon choosing
$\theta = (\nu/ C(\mu,d,\partial\Omega))^{\frac{d+4+\beta}{\beta}}$, where it is fixed that $\delta = \theta^{\frac{1}{d+4+\beta}}$
by optimizing the size of
the right-hand side of $\eqref{f:6.16-ap}$ according to the definition of $\chi_*$ given in $\eqref{min_rad-ap}$.

Finally, the arguments for $\eqref{m-1}$ are presented; these are
mainly based on a dyadic decomposition of scales
and Theorem $\ref{thm:1.0}$. To this end, for any $\alpha\in(0,1)$,
it suffices to define stationary
random fields as follows:
\begin{equation*}
 \chi_{**}(0):=
 \Bigg\{\frac{1}{\theta}\sup_{R\geq 1}\frac{1}{R^{\alpha}}
 \bigg[
 \Big(\dashint_{Q_{2R}}|(\bar{\phi},\bar{\sigma},\nabla\tilde{\sigma}_{(d+1)})|^{2}
 \Big)^{\frac{1}{2}}
+
\frac{1}{R}
\Big(\dashint_{Q_{2R}}|\tilde{\sigma}_{(d+1)}|^{2}\Big)^{\frac{1}{2}}
 \bigg]\Bigg\}^{\frac{1}{1-\alpha}}\vee 1.
\end{equation*}
Thus, it can be verified that for any
$R\geq\chi_{**}(0)$ there holds
\begin{equation*}
\frac{1}{R}
  \Big(\dashint_{Q_{2R}}
 \big|(\bar{\phi},\bar{\sigma},\nabla\tilde{\sigma}_{(d+1)})\big|^{2}\Big)^{\frac{1}{2}}
+ \frac{1}{R^2}\Big(\dashint_{Q_{2R}}
 |\tilde{\sigma}_{(d+1)}|^{2}\Big)^{\frac{1}{2}}
 \leq \theta.
\end{equation*}
Recalling the definition of $\chi_*$, we have
$\chi_{**}(0)\geq \chi_{*}(0)$. For any $\beta\in[1,\infty)$,
it can be obtained that
\begin{equation*}
\begin{aligned}
\langle|\chi_{**}|^\beta\rangle^{\frac{1}{\beta}}
&\lesssim_{\theta} \Bigg\langle\Bigg(\sup_{R\geq 1}
\bigg[\frac{1}{R^{\alpha}}
\Big(\dashint_{Q_{2R}}|(\bar{\phi},\bar{\sigma},\nabla\tilde{\sigma}_{(d+1)})|^{2}
\Big)^{\frac{1}{2}}+
\frac{1}{R^{1+\alpha}}
\Big(\dashint_{Q_{2R}}|\tilde{\sigma}_{(d+1)}|^{2}
\Big)^{\frac{1}{2}}\bigg]\Bigg)^{\beta}\Bigg\rangle^{\frac{1}{\beta}}\\
&\lesssim
\Bigg\langle\Bigg(\sum_{R\geq 1\atop\text{dyadic}}
\bigg[\frac{1}{R^{\alpha}}
\Big(\dashint_{Q_{2R}}|(\bar{\phi},\bar{\sigma},\nabla\tilde{\sigma}_{(d+1)})|^{2}
\Big)^{\frac{1}{2}}+
\frac{1}{R^{1+\alpha}}
\Big(\dashint_{Q_{2R}}|\tilde{\sigma}_{(d+1)}|^{2}
\Big)^{\frac{1}{2}}\bigg]\Bigg)^{\beta}\Bigg\rangle^{\frac{1}{\beta}}\\
&\lesssim^{\eqref{pri:5.2},\eqref{pri:5.4},\eqref{pri:5.4b}} \sum_{R\geq1\atop\text{dyadic}}\Big(\frac{1}{R^{\alpha}}
+\frac{1}{R^{1+\alpha}}\mu_d(R)\Big)
\lesssim_{\mu,\lambda_1,d,\beta,\theta} 1,
\end{aligned}
\end{equation*}
which, together with $\chi_{**}\geq \chi_{*}$, leads to the desired
estimate $\eqref{m-1}$. This ends the whole proof.
\qed

\begin{remark}\label{remark:3.2}
\emph{To see that the minimum radius $\chi_*$ is well-defined,
it only remains to verify the following limit:
\begin{equation}\label{pri:3.7}
\limsup_{R\to\infty}\frac{1}{R^4}\dashint_{Q_R}
 \big|\sigma_{(d+1)} - \dashint_{Q_{R}}\sigma_{(d+1)}
-x\cdot\dashint_{Q_{R}}\nabla\sigma_{(d+1)}\big|^2 = 0.
\end{equation}
Let $\zeta(t):=\dashint_{B_R}\sigma_{(d+1)}(\cdot,t)$, and we start from
\begin{equation}\label{f:3.60}
\begin{aligned}
&\dashint_{Q_R}
 \big|\sigma_{(d+1)} - \dashint_{Q_{R}}\sigma_{(d+1)}
-x\cdot\dashint_{Q_{R}}\nabla\sigma_{(d+1)}\big|^2\\
&\lesssim \dashint_{Q_R}
 \big|\sigma_{(d+1)} -\sigma_{(d+1)}(0,\cdot)
-x\cdot\dashint_{Q_{R}}\nabla\sigma_{(d+1)}\big|^2
+\dashint_{Q_R}
\big|\sigma_{(d+1)}(0,\cdot)- \dashint_{Q_{R}}\sigma_{(d+1)}\big|^2\\
&\lesssim R^2\dashint_{Q_R}
 \big|\nabla\sigma_{(d+1)}
-\dashint_{Q_{R}}\nabla\sigma_{(d+1)}\big|^2
+\dashint_{I_{R}}\big|\sigma_{(d+1)}(0,\cdot)- \dashint_{Q_{R}}\sigma_{(d+1)}\big|^2 =: J_1 + J_2.
\end{aligned}
\end{equation}
In view of the result $\eqref{pri:2.20}$ stated in Proposition $\ref{P:5.3}$, $J_1$ is a term that can be handled easily, while $J_2$
is addressed as follows:
\begin{equation}\label{f:3.61}
J_{2}\lesssim  \dashint_{I_R}|\sigma_{(d+1)}(0,\cdot)-\zeta|^2
+\dashint_{I_R}|\zeta-\dashint_{I_R}\zeta|^2
\lesssim R^2\dashint_{Q_R}|\nabla\sigma_{(d+1)}|^2
+ \dashint_{I_R}|\zeta-\dashint_{I_R}\zeta|^2.
\end{equation}
Combining the estimates $\eqref{f:3.60}$ and $\eqref{f:3.61}$ leads to
\begin{equation*}
\begin{aligned}
&\limsup_{R\to\infty}\frac{1}{R^4}\dashint_{Q_R}
 \big|\sigma_{(d+1)} - \dashint_{Q_{R}}\sigma_{(d+1)}
-x\cdot\dashint_{Q_{R}}\nabla\sigma_{(d+1)}\big|^2\\
&\lesssim \limsup_{R\to\infty}\bigg(\frac{1}{R^2}\dashint_{Q_R}
 \big|\nabla\sigma_{(d+1)}
-\dashint_{Q_{R}}\nabla\sigma_{(d+1)}\big|^2
+\frac{1}{R^2}\dashint_{Q_R}|\nabla\sigma_{(d+1)}|^2
+ \frac{1}{R^{4}}\dashint_{I_R}|\zeta-\dashint_{I_R}\zeta|^2\bigg)=0,
\end{aligned}
\end{equation*}
where the results $\eqref{f:2.25}$ and \cite[Subsection 4.3]{Bella-Chiarini-Fehrman-19} are employed. This proves
the stated limit $\eqref{pri:3.7}$. Also, $\partial_{d+1}\sigma_{(d+1)}$
and $\overline{\nabla}(\nabla\sigma_{(d+1)})$ are
stationary, which together the known results leads to
that the minimum radius $\chi_*$ is well-defined and stationary.}
\end{remark}

\section{Calder\'on-Zygmund estimates}\label{section:5}

\begin{definition}[$A_q$ weights]\label{def:weight}
Let $1\leq q<\infty$. A non-negative, locally integrable function $\omega$
is said to be an $A_q$ weight if there exists a positive constant
$A$ such that, for every parabolic cube $Q\subset\mathbb{R}^{d+1}$,
\begin{equation}\label{f:w-1} \bigg(\dashint_{Q}\omega\bigg)\bigg(\dashint_{Q}\omega^{-1/(q-1)}\bigg)^{q-1}
  \leq A,
\end{equation}
if $q>1$, where
the average integral symbol $\dashint_{Q}$ is defined by $\frac{1}{|Q|}\int_Q$,
with $|Q|$ denoting  the Lebesgue measure of $Q$,  or
\begin{equation}\label{f:w-2}
 \bigg(\dashint_{Q}\omega\bigg)\esssup_{z\in Q}\frac{1}{\omega(z)}\leq A,
\end{equation}
if $q=1$. The infimum over all such constants $A$ is called
the $A_q$ Muckenhoupt characteristic constant of $\omega$, denoted by
$[\omega]_{A_q}$. Moreover, $A_q$ denotes
the set of all $A_q$ weights.
\end{definition}

For the reader's convenience, some geometric notation
on integral regions is recalled. For any $z\in\Omega_T$, we set
$U_{*,\varepsilon}(z):=\Omega_T\cap Q_{*,\varepsilon}(z)$
with $Q_{*,\varepsilon}(z)
:=Q_{\varepsilon\chi_{*}(z/\varepsilon)}(z)$
and, $U_\varepsilon(z):=Q_\varepsilon(z)\cap\Omega_T$, where
$\chi_{*}$ is known as the
minimum radius given in Lemma $\ref{lemma:approxi}$.

\begin{proposition}\label{P:7.1}
Let $\Omega\subset\mathbb{R}^d$ be a $C^1$ domain with
$d\geq 2$, $\varepsilon\in(0,1]$, and $T>0$.
Suppose that the ensemble $\langle\cdot\rangle$ is stationary and ergodic.
Let $u_\varepsilon$ and $f$ be associated with the equations $\eqref{pde:A}$.
Then, there exists a stationary random field $\chi_*$, as
given in Lemma $\ref{lemma:approxi}$ such that,
for any $1<p<\infty$ and  $\omega\in A_p$,
the following weighted quenched Calder\'on-Zygmund estimate holds:
\begin{equation}\label{pri:7.1}
\bigg(\int_{\Omega_T}dz\Big(
\dashint_{U_{*,\varepsilon}(z)}|\nabla u_\varepsilon|^2
\Big)^{\frac{p}{2}}\omega(z)\bigg)^{\frac{1}{p}}
\lesssim_{\mu,d,\partial\Omega,p,[\omega]_{A_p}}
\bigg(
\int_{\Omega_T}dz\Big(\dashint_{U_{*,\varepsilon}(z)}|f|^2\Big)^{\frac{p}{2}}
\omega(z)\bigg)^{\frac{1}{p}}.
\end{equation}
Moreover, if the ensemble $\langle\cdot\rangle$ additionally satisfies
$\eqref{c:3}$, then, for any $1<p,q<\infty$ and $\omega\in A_q$, we have
the weighted annealed Calder\'on-Zygmund estimate
\begin{equation}\label{pri:7.2}
\bigg(\int_{\Omega_T}dz
\Big\langle\Big(\dashint_{
U_{*,\varepsilon}(z)}|\nabla u_\varepsilon|^2 \Big)^{\frac{p}{2}}
\Big\rangle^{\frac{q}{p}}\omega(z)\bigg)^{\frac{1}{q}}
\lesssim_{\mu,d,\partial\Omega,p,q,[\omega]_{A_q}}
\bigg(\int_{\Omega_T}dz
\Big\langle\Big(\dashint_{U_{*,\varepsilon}(z)}
|f|^2\Big)^{\frac{p}{2}}\Big\rangle^{\frac{q}{p}}\omega(z)\bigg)^{\frac{1}{q}}.
\end{equation}
In particular, the estimates $\eqref{pri:7.1}$ and $\eqref{pri:7.2}$ also hold in the case $\Omega_T=\mathbb{R}^{d+1}$.
\end{proposition}

\begin{remark}\label{remark:3}
\emph{To establish Proposition $\ref{P:7.1}$, the Lipschitz continuity of
the minimum radius is required, as originally developed
by Gloria et al. \cite[pp.137-138]{Gloria-Neukamm-Otto-20}.
In terms of $\chi_*$ introduced in Lemma $\ref{lemma:approxi}$, one can similarly construct
$\frac{1}{L}$-Lipschitz continuous random fields, where we can fix
$L:=8$ for convenience.
Let $\theta_0$ be such that $\eqref{pri:1.5}$ holds
for any $r\geq \varepsilon\chi_*(0;\theta)$ with $0<\theta\leq\theta_0$.
Define $\underline{\chi}_{*}(0)$ as
the largest function with Lipschitz
constant $\frac{1}{L}$ less than
$\chi_*(0;\theta)$ for some $\theta$ chosen later, i.e.,
\begin{equation*}
  \underline{\chi}_*(0):=\inf_{y\in\mathbb{R}^{d+1}}
  \big(\chi_*(y;\theta)+\text{d}(y,0)/L\big).
\end{equation*}
By setting $\theta_1:= (L+1)^{-\frac{d}{2}-2}\theta_0$, it follows that
$\underline{\chi}_{*}$ is $\frac{1}{L}$-Lipschitz continuous and satisfies
$\chi_*(0;\theta_0)\leq \underline{\chi}_{*}(0)\leq \chi_*(0;\theta_1)$.
Obviously, this definition is independent of the origin and
can be translated to any point $z\in\mathbb{R}^{d+1}$. For ease of presentation,  the original notation is retained to denote the minimal radius with the $\frac{1}{L}$-Lipschitz continuity.}
\end{remark}

Let $R_*:= R_0\vee T$, where $R_0$ is the diameter of $\Omega$.
Recall that $\Omega_*:=\Omega\times I_{R_*}$ is an extension cylinder of
$\Omega_T$, and $\partial_{\shortparallel}\Omega_*:=\partial\Omega\times I_{R_*}$ represents the lateral boundary of $\Omega_*$. Below, we state the main tools used in this section.

\begin{lemma}[Weighted version of Shen's real argument  \cite{Shen-20}]\label{shen's lemma2}
Let $0<p_0<p<p_1$, $\omega\in A_{\frac{p}{p_0}}$,
and $\Omega_*\subset\mathbb{R}^{d+1}$ is a bounded Lipschitz domain. Let
$F\in L^{p_0}(\Omega_*)$ and $g\in L^{p_1}(\Omega)$.
Let $Q_0:=Q_{r_0}(x_0)$, where $x_0\in\partial_{\shortparallel}\Omega_{*}$ and
$0<r_0<c_0\text{diam}(\Omega_*)$. Suppose that for each parabolic cube $Q$
with the property that $|Q|\leq  c_1|Q_0|$
and either $4Q\subset 2Q_0\cap\Omega_*$ or $z_Q\in\partial_{\shortparallel}\Omega_*\cap 2Q_0$, there exist two measurable functions
$F_Q$ and $R_Q$ on $2Q\cap \Omega_*$,
such that $|F|\leq |W_Q| +
|V_Q|$ on $2Q\cap\Omega_*$, and
\begin{subequations}
\begin{align}
\Big(\dashint_{2Q\cap\Omega_{*}}|V_Q|^{p_1}\omega
\Big)^{\frac{1}{p_1}}
&\lesssim_{N_1} \bigg\{
\Big(\dashint_{\alpha Q\cap\Omega_{*}}|F|^{p_0}\Big)^{\frac{1}{p_0}}
+\sup_{4Q_0\supseteq Q^\prime\supseteq Q}
\Big(\dashint_{Q^\prime\cap\Omega_{*}}|g|^{p_0}
\Big)^{\frac{1}{p_0}}
\bigg\}\Big(\dashint_{Q}\omega\Big)^{1/{p_1}};
\label{pri:7.4a}\\
\Big(\dashint_{2Q\cap\Omega_{*}}|W_Q|^{p_0}
\Big)^{\frac{1}{p_0}}
&\lesssim_{N_2}
\sup_{4Q_0\supseteq Q^\prime\supseteq Q}
\Big(\dashint_{Q^\prime\cap\Omega_{*}}|g|^{p_0}
\Big)^{\frac{1}{p_0}}
+\eta\Big(\dashint_{\alpha Q\cap\Omega_{*}}
|F|^{p_0}\Big)^{\frac{1}{p_0}},
\label{pri:7.4b}
\end{align}
\end{subequations}
where $N_1, N_2>0$, $\eta\geq0$, $0<c_1<1$, and $\alpha>2$. Then there exists $\theta_{0}>0$, depending on
$d,p_0,p_1,p,N_1,[\omega]_{A_{p/p_0}}$, and the
Lipschitz character of $\Omega_*$, such that if $0\leq \eta\leq\eta_0$, it is obtained that
\begin{equation}\label{pri:7.5}
\Big(\dashint_{Q_0\cap\Omega_*}|F|^p\omega\Big)^{\frac{1}{p}}
\lesssim \bigg\{\Big(\dashint_{4Q_0\cap\Omega_*}|F|^{p_0}\Big)^{\frac{1}{p_0}}
\Big(\dashint_{Q_0}\omega\Big)^{1/{p}}
+\Big(\dashint_{4Q_0\cap\Omega_*}|\textcolor[rgb]{0.00,0.00,1.00}{g}|^p\omega\Big)^{\frac{1}{p}}\bigg\},
\end{equation}
where the multiplicative constant depends on
$d,p_0,p_1,p,c_1,\alpha,N_1,N_2,[\omega]_{A_{p/p_0}}$, and the Lipschitz
character of $\Omega_*$.
\end{lemma}

\begin{lemma}[Geometric properties of integrals]\label{lemma:integral geometry}
Let $\chi_{*}$ be given as in Lemma $\ref{lemma:approxi}$ with the $\frac{1}{L}$-Lipschitz continuity, and $\varepsilon\in(0,1]$.
Let $f\in L^1_{\emph{loc}}(\mathbb{R}^{d+1})$, and
$\Omega_*\subset\mathbb{R}^{d+1}$ be a Lipschitz domain with
$X\in\partial_{\shortparallel}\Omega_*$. Then
the following inequalities hold:
\begin{itemize}
  \item For all $Q_r(X)
  \subset\mathbb{R}^{d+1}$ and $x_0\in Q_r(X)\cap\Omega_*$ with
   $r<\frac{\varepsilon}{4}\chi_*(x_0/\varepsilon)$, we have
  \begin{equation}\label{pri:7.6}
  \Big(\dashint_{U_{*,\varepsilon}(x_0)}|f|\Big)^{\frac{1}{s}}
  \lesssim
  \dashint_{D_{5r}(X)}dz
  \Big(
  \dashint_{U_{*,\varepsilon}(z)}
  |f|\Big)^{\frac{1}{s}},
  \end{equation}
  where $s>0$, and we recall $D_r(z):= Q_r(z)\cap\Omega_*$ whenever $z\in\partial_{\shortparallel}\Omega_*$.
  \item For all $Q_r(X)
  \subset\mathbb{R}^{d+1}$ with $r\geq \frac{\varepsilon}{4}
  \chi_*(X/\varepsilon)$,
  we have
\begin{subequations}
\begin{align}
  &\dashint_{D_r(X)} |f|
  \lesssim \dashint_{D_{2r}(X)}
  \Big(\dashint_{U_{*,\varepsilon}(x)}|f|\Big)dx
  \lesssim \dashint_{D_{7r}(X)} |f|;
  \label{pri:7.7a}\\
  &
\int_{\Omega_*}|f|
\sim
\int_{\Omega_*}
  \Big(\dashint_{U_{*,\varepsilon}(x)}|f|\Big)dx,
  \label{pri:7.7b*}
\end{align}
\end{subequations}
\end{itemize}
where the multiplicative constant depends only on $d$ and the
Lipschitz character of $\Omega_*$.
\end{lemma}

\begin{proof}
Apart from the parabolic scaling and the parabolic distance, the definition of a parabolic cube is analogous to that of a standard cube in $\mathbb{R}^{d+1}$. Consequently, the proof of this lemma can be entirely derived by replicating the proof of
\cite[Lemma 2.12]{Wang-Xu-25}, and it is omitted here. The idea can also be found in
\cite{Duerinckx-Otto-20}.
\end{proof}

\subsection{Outline some basic ingredients}
\label{subsec:6.1}

\noindent
Let $\chi_{*}$ be given as in Lemma $\ref{lemma:approxi}$
with the $\frac{1}{L}$-Lipschitz continuity.
Then, in order to be consistent with the notation in Lemma $\ref{shen's lemma2}$, define
\begin{equation*}
F(z):=\Big(\dashint_{U_{*,\varepsilon}(z)}
 |\nabla u_\varepsilon|^2\Big)^{\frac{1}{2}};
 \qquad
 g(z):=\Big(\dashint_{U_{*,\varepsilon}(z)}
 |f|^2\Big)^{\frac{1}{2}},
\end{equation*}
where $u_\varepsilon$ and $f$ are associated with the equations $\eqref{pde:7.1}$.

\medskip
\noindent
\textbf{Basic ideas:} The repeated application of Shen's lemma (i.e., Lemma $\ref{shen's lemma2}$) is reflected in two aspects.
\begin{itemize}
  \item[a.] \textbf{From unweighted estimates to weighted ones.} Roughly speaking, the proof begins by establishing
  the estimate $\eqref{pri:7.1}$ in the case $\omega=1$ (i.e., the estimate $\eqref{pri:A}$) through the use of Shen's lemma, where the energy estimate plays a fundamental role
  as one of ``inputs'' and accordingly $p_0=2$ is taken therein.
  Then, the desired estimate $\eqref{pri:7.1}$ for $A_p$ weights follows from reusing Shen's lemma, where the role played by the energy estimate is replaced by the estimate obtained in the previous step,
  and accordingly the range of $p_0$ will be greatly relaxed.
  \item[b.] \textbf{From quenched estimates to annealed ones} (as shown in the flow of the proofs in Fig.$\ref{pic:4.1}$).
\end{itemize}

Understanding how Shen’s lemma is concretely utilized to address the Calder\'on-Zygmund estimates is crucial for the proof of
Theorem $\ref{thm:C-Z}$ (or Proposition $\ref{P:7.1}$).
The estimate $\eqref{pri:A}$ is taken
as an example to illustrate it.

\textbf{Step A.} \emph{Decompose the domain $\Omega_T$ and reduce
the desired estimate $\eqref{pri:A}$ to local estimates.}
Let $B_0\subset\mathbb{R}^d$ be a fixed ball centered at the origin
with radius satisfying $r_{B_0}\sim R_0$ if $R_0<\infty$;
or $r_{B_0}\sim 1$ if $R_0=\infty$. The basic parabolic cube
is denoted by $Q_0:=B_0\times I_{-r_{B_0}}$. For any $z\in\mathbb{R}^{d+1}$ one can set $Q_0(z):=Q_0+z$. Then, two
families of center points are introduced: $\mathcal{I}$ for
interior points and $\mathcal{B}$ for boundary points. They satisfy the following properties:
\begin{itemize}
  \item[(1)] for any $z\in\mathcal{I}$, it holds that $Q_0(z)\subset \Omega_T$;
  \item[(2)] $\mathcal{B}\subset \partial_{\shortparallel} \Omega_T$, and
\begin{equation}\label{decompsition}
\Omega_T\subset \underbrace{\Big(\cup_{z\in\mathcal{B}}Q_0(z)\Big)}_{\text{boundary~ part}}
  \bigcup \underbrace{\Big(\cup_{z\in\mathcal{I}}(1/4)Q_0(z)\Big)}_{
  \text{interior~part}}.
\end{equation}
\end{itemize}

For any $2\leq p<\infty$, the following local estimates are established:
\begin{equation}\label{pri:7.9}
\dashint_{\tilde{Q}_0(z)\cap\Omega_T} F^p
\lesssim  \Big(\dashint_{4\tilde{Q}_0(z)\cap\Omega_T} F^2\Big)^{\frac{p}{2}}
+ \dashint_{4\tilde{Q}_0(z)\cap\Omega_T} g^p,
\end{equation}
where
$\tilde{Q}_0(z) := Q_0(z)$ if $z\in\mathcal{B}$;
and $\tilde{Q}_0(z) := (1/4)Q_0(z)$ if $z\in\mathcal{I}$.
A covering argument then yields
\begin{equation*}
\int_{\Omega_T}F^p\lesssim \int_{\Omega_T}g^p.
\end{equation*}
Then, by a duality argument, the above estimate
also holds for $1<p<2$, which finally leads to $\eqref{pri:A}$.
Thus, the desired estimate
$\eqref{pri:A}$ has been reduced to $\eqref{pri:7.9}$.

\textbf{Step B.} \emph{Shen's lemma
(with $p_1\in(2,\infty)$, $p_0=2$, and $\omega =1$) is employed
to obtain the stated local estimates $\eqref{pri:7.9}$.} By the decomposition $\eqref{decompsition}$,
the proof of $\eqref{pri:7.9}$ is divided into boundary
and interior parts.
Here, only the arguments for the boundary part are introduced,
since the interior case can be treated similarly.

In view of Lemma $\ref{shen's lemma2}$,
for the given $Q_0(z)$ with $z\in\mathcal{B}$, consider any cube
$Q$ with the properties: $|Q|\leq c_1|Q_0|$ and either
$z_Q\in 2Q_0(z)\cap\partial_{\shortparallel}\Omega_T$ or $4Q\subset 2Q_0(z)\cap\Omega_T$,
where $c_1\in(0,1)$.
Then, $F$ must be split according to $Q$.
Recalling the error expansion in Lemma $\ref{lemma:approxi}$,
it heuristically suggests that
\begin{equation}\label{eq:7.1}
\nabla u_\varepsilon = \underbrace{\nabla u_\varepsilon -
(e_j+(\nabla\phi_j)^\varepsilon)\varphi_j}_{W_\varepsilon} + \underbrace{(e_j+(\nabla\phi_j)^\varepsilon)\varphi_j}_{V}
\end{equation}
and the right-hand side of $\eqref{eq:7.1}$
provides a concrete form of $W_Q$ and $V_Q$
stated in Lemma $\ref{shen's lemma2}$:
\begin{equation}\label{eq:7.2}
W_Q(z):= \Big(\dashint_{U_{*,\varepsilon}(z)}|W_\varepsilon|^2\Big)^{\frac{1}{2}}
\quad \text{and} \quad
V_Q(z):= \Big(\dashint_{U_{*,\varepsilon}(z)}|V|^2\Big)^{\frac{1}{2}},
\end{equation}
by which it is immediately obtained that $F\leq W_Q + V_Q$ on $2Q\cap\Omega_T$.

However, the decomposition of $F$ is subtle because
geometric properties of integrals render
the cases complicated under the multi-scale action
(see Lemma $\ref{lemma:integral geometry}$). Therefore, the decomposition of $F$
is based on the scale of the cube $Q$, and
the following cases are considered\footnote{We
impose the notation $W_Q$ (and $V_B$) to stress that
the decomposition of $Q$ is indeed subjected to $Q$.}:
\begin{itemize}
  \item [(I).] $0<r_Q<\frac{\varepsilon}{4}\chi_{*}(z_Q/\varepsilon)$;
  in this case, set
  $W_Q = 0$ and $V_Q = F$;
  \item [(II).] $r_Q\geq \frac{\varepsilon}{4}\chi_{*}(z_Q/\varepsilon)$;
  in this case, $W_Q$ and $V_Q$ are exactly given by $\eqref{eq:7.2}$.
\end{itemize}

Then,
the remaining work is to verify
the estimates $\eqref{pri:7.4a}$
and $\eqref{pri:7.4b}$
satisfied by $W_B$ and $V_B$,
respectively, for each of the two cases above. These verification are addressed in detail in the next subsection, and
the desired estimate $\eqref{pri:7.9}$ is finally obtained.

\subsection{Quenched Calder\'on-Zygmund estimates}
\label{subsec:6.2}

\noindent
First, the ranges of values of parameters $p_1,p_0$ in Lemma $\ref{shen's lemma2}$ are fixed as follows:
\begin{itemize}
\item [(a)] To handle the estimate $\eqref{pri:A}$ (or $\eqref{pri:7.1}$ for $\omega=1$), take $p_1\in(2,\infty)$ and $p_0=2$;
\item [(b)] To handle the estimate $\eqref{pri:7.1}$ for any
$\omega\in A_p$, take $p_1\in(p_0,\infty)$ and $p_0\in(1,2]$.
\end{itemize}

\textbf{Step 1.} Outline of the proof of
$\eqref{pri:7.1}$.
Let $\omega\in A_{p/p_0}$ with $p\in(p_0,p_1)$.
As in \textbf{Step A},
the desired estimate $\eqref{pri:7.1}$ is based on the following local estimates:
\begin{equation}\label{f:7.1}
\Big(\dashint_{\tilde{Q}_0(z)\cap\Omega}F^p\omega\Big)^{\frac{1}{p}}
\lesssim
\Big(\dashint_{4\tilde{Q}_0(z)\cap\Omega}F^{p_0}\Big)^{\frac{1}{p_0}}
\Big(\dashint_{\tilde{Q}_0(z)}\omega\Big)^{1/{p}}
+\Big(\dashint_{4\tilde{Q}_0(z)\cap\Omega}|g|^p\omega\Big)^{\frac{1}{p}},
\end{equation}
where
$\tilde{Q}_0(z) := Q_0(z)$ if $z\in\mathcal{B}$,
and $\tilde{Q}_0(z) := (1/4)Q_0(z)$ if $z\in\mathcal{I}$.

Admitting the estimates $\eqref{f:7.1}$ provisionally, it can be further derived that
\begin{equation}\label{f:7.2}
\begin{aligned}
\dashint_{\tilde{Q}_0(z)\cap\Omega}F^p\omega
\lesssim
\Big(\dashint_{4\tilde{Q}_0(z)\cap\Omega}F^{p_0}\Big)^{\frac{p}{p_0}}
\Big(\dashint_{\tilde{Q}_0(z)}\omega\Big)
+ \dashint_{4\tilde{Q}_0(z)\cap\Omega}|g|^p\omega.
\end{aligned}
\end{equation}
Two cases must be considered:
(1) bounded parabolic cylinder; (2) unbounded parabolic cylinder.

For the case (1), the number of elements in $\mathcal{I}\cup\mathcal{B}$ is finite, and the covering argument yields
\begin{equation}\label{f:7.3}
\begin{aligned}
\dashint_{\Omega}F^p\omega
&\lesssim^{\eqref{f:7.2}}
\Big(\dashint_{\Omega}F^{p_0}\Big)^{\frac{p}{p_0}}
\max_{z\in\mathcal{I}\cup\mathcal{B}}\Big(\dashint_{\tilde{Q}_0(z)}\omega\Big)
+\dashint_{\Omega}g^p\omega\\
&\lesssim^{\eqref{pri:A}} \Big(\dashint_{\Omega}g^{p_0}\Big)^{\frac{p}{p_0}}
\Big(\dashint_{2\Omega}\omega\Big)
+\dashint_{\Omega}g^p\omega
\lesssim^{\eqref{f:apw-2}}
\Big(\dashint_{\Omega}g^{p_0}\Big)^{\frac{p}{p_0}}
\Big(\dashint_{\Omega}\omega\Big)
+\dashint_{\Omega}g^p\omega.
\end{aligned}
\end{equation}
To proceed with the computation, further set
$0<p_0-1\ll1$. By using H\"older's inequality and the property of the
$A_{p/p_0}$ class (it is known from $\eqref{f:apw-5}$ that
$\omega\in A_p$ implies $\omega\in A_{p/p_0}$), it can be derived that
\begin{equation}\label{f:7.4}
\Big(\dashint_{\Omega}g^{p_0}\Big)^{\frac{1}{p_0}}
\leq \Big(\dashint_{\Omega}g^{p}\omega\Big)^{\frac{1}{p}}
\Big(\dashint_{\Omega}\omega^{-\frac{p_0}{p-p_0}}\Big)^{\frac{p-p_0}{p_0p}}
\lesssim^{\eqref{f:w-1}}_{d,p,[\omega]_{A_p}} \Big(\dashint_{\Omega}g^{p}\omega\Big)^{\frac{1}{p}}
\Big(\dashint_{\Omega}\omega\Big)^{-\frac{1}{p}}.
\end{equation}
Thus, by plugging the estimate $\eqref{f:7.4}$ back into $\eqref{f:7.3}$,
the desired estimate $\eqref{pri:7.1}$ is established for the case (1).


For the case (2), the number of elements in $\mathcal{I}\cup\mathcal{B}$ is infinite, and an approximating argument is applied. Instead, fix any $z\in\mathcal{B}$, and denote it by $\hat{z}$. It follows from
the estimate $\eqref{f:7.2}$ that
\begin{equation}\label{f:7.5}
\begin{aligned}
\int_{Q_0(\hat{z})\cap\Omega}F^p\omega
\lesssim
|Q_0|^{1-\frac{p}{p_0}}\Big(\int_{\Omega_T}F^{p_0}\Big)^{\frac{p}{p_0}}
\Big(\dashint_{\tilde{Q}_0(\hat{z})}\omega\Big)
+ \int_{\Omega_T}|g|^p\omega,
\end{aligned}
\end{equation}
Using the property $\eqref{f:apw-5}$ again, one can infer that $\omega\in A_{p/p_0}$ implies $\omega\in A_q$ for some $q\in(1,p/p_0)$.
Therefore, it holds that
\begin{equation}\label{f:7.6}
\Big(\dashint_{Q_0(\hat{z})}\omega\Big)
\lesssim |Q_0|^{q-1}\Big(\int_{Q_0(\hat{z})}\omega^{-\frac{1}{q-1}}\Big)^{1-q}
\lesssim |Q_0|^{q-1},
\end{equation}
where it is noted that $\omega\geq 0$ in the last step above. Plugging
$\eqref{f:7.5}$ back into $\eqref{f:7.6}$, we have
\begin{equation*}
\int_{Q_0(\hat{z})\cap\Omega}F^p\omega
\lesssim
|Q_0|^{q-\frac{p}{p_0}}\Big(\int_{\Omega_T}F^{p_0}\Big)^{\frac{p}{p_0}}
+ \int_{\Omega_T}|g|^p\omega,
\end{equation*}
The desired estimate $\eqref{pri:7.1}$ for the case (2) consequently follows by letting $|Q_0|\to\infty$ above.

\textbf{Step 2.} Further reduction of
$\eqref{f:7.1}$. As mentioned in \textbf{Step B} of
Subsection $\ref{subsec:6.1}$,
it suffices to show $\eqref{f:7.1}$ for the case of $z\in\mathcal{B}$.
Moreover, it may be assumed without loss of generality that $z=0$.
Given $\tilde{Q}_0$ (which is equal to
$Q_0$ by definition, and whose radius is $r_0$), let $Q$ be arbitrary cube
with the properties: $r_Q\leq (1/100)r_{0}$ and
either $z_Q\in 2Q_0\cap\partial_{\shortparallel}\Omega_T$ or
$4Q\subset 2Q_0\cap\Omega_T$.
Here, only the case $z_Q\in 2Q_0\cap \partial_{\shortparallel}\Omega_T$ is considered, and $D:=Q\cap\Omega_T$.
Thus, according to the guideline in \textbf{Step B}, the next task is to determine the concrete
forms of $W_B$ and $V_B$ for the case
$r_Q\geq \frac{\varepsilon}{4}\chi_{*}(z_Q/\varepsilon)$ and
$z_Q\in 2Q_0\cap\partial_{\shortparallel}\Omega_T$.

\textbf{Step 3.} Decomposition of the field $\nabla u_\varepsilon$ and
final reduction of $\eqref{f:7.1}$ based on Step 2.
Let $Q$ be given as in the previous step.
Suppose that $w_\varepsilon\in L^2(0,T;H_0^1(\Omega))$ and
$v_\varepsilon:=u_\varepsilon-w_\varepsilon$ satisfy the following equations:
\begin{equation}\label{pde:7.1}
\left\{\begin{aligned}
 \partial_t w_\varepsilon +\mathcal{L}_\varepsilon(w_\varepsilon)
 &= \nabla \cdot f\textbf{1}_{30D}
 &\quad&\text{in}~~\Omega_T;\\
 w_\varepsilon
 &= 0
 &\quad&\text{on}~~\partial_p\Omega_T,
\end{aligned}\right.
 \quad\text{and}\quad
 \partial_t v_\varepsilon +\mathcal{L}_\varepsilon(v_\varepsilon)
 =0 \quad\text{in}\quad 30D.
\end{equation}
Then, consider the approximating function $\bar{v}$ satisfying
$\big(\partial_t+\mathcal{L}_0\big)(\bar{v}_{r_Q}) = 0$ in $9D$
with $\bar{v}_{r_Q} = v_\varepsilon$ on $\partial_{p}(9D)$.
The following notation is also introduced:
\begin{equation*}
\begin{aligned}
&V_{Q}(z)
:=\Big(\dashint_{U_{*,\varepsilon}(z)}
|(e_j+(\nabla\phi_j)^\varepsilon)
\varphi_{j,r_{Q}}|^{2}\Big)^{1/2};
\qquad
W_{Q}(z)
:=\Big(\dashint_{U_{*,\varepsilon}(z)}
|\nabla u_\varepsilon-(e_j+(\nabla\phi_j)^\varepsilon)
\varphi_j|^{2}\Big)^{1/2};\\
&W^{(1)}_{Q}(z)
:=\Big(\dashint_{U_{*,\varepsilon}(z)}
|\nabla v_\varepsilon - (e_j+(\nabla\phi_j)^\varepsilon)
\varphi_{j,r_{Q}}|^{2}\Big)^{1/2};\qquad
W^{(2)}_{Q}(z)
:=\Big(\dashint_{U_{*,\varepsilon}(z)}|\nabla w_\varepsilon|^{2}\Big)^{1/2},
\end{aligned}
\end{equation*}
where $\varphi_{j,r_Q}=\partial_j\bar{v}_{r_Q}$ is similar to the function  defined in the proof of Lemma $\ref{lemma:approxi}$.

Then, combining the cases (I) and (II),
one can observe that
$F\leq W_{Q}+V_{Q}$ and $W_{Q}\leq W_{Q}^{(1)} + W_{Q}^{(2)}$
on $2D$.
Let $p_0$ and $p_1$ be fixed as in (a) and (b) above, and let $p_0<p<p_1$ with $\omega\in A_{p/p_0}$ being a weight function. We claim that:
\begin{subequations}
\begin{align}
&\Big(\dashint_{2D}
V_Q^{p_1}\omega
\Big)^{\frac{1}{p_1}} \lesssim
\Big(\dashint_{60D}
(F^{p_0}+g^{p_0})\Big)^{\frac{1}{p_0}}
\Big(\dashint_{Q}\omega\Big)^{\frac{1}{p_1}};
\label{f:7.7a} \\
& \dashint_{\frac{8}{7}D}
W_Q^{p_0}
\lesssim \dashint_{60D}
(g^{p_0}
+\nu^{p_0}
F^{p_0}).
\label{f:7.7b}
\end{align}
\end{subequations}

By Lemma $\ref{shen's lemma2}$, it is derived that
\begin{equation*}
\Big(\dashint_{Q_0\cap\Omega_T}F^p\omega\Big)^{\frac{1}{p}}
\lesssim^{\eqref{pri:7.5}}\bigg\{
\Big(\dashint_{4Q_0\cap\Omega_T}F^{p_0}\Big)^{\frac{1}{p_0}}
\Big(\dashint_{Q_0}\omega\Big)^{1/{p}}
+\Big(\dashint_{4Q_0\cap\Omega_T}g^p\omega\Big)^{\frac{1}{p}}\bigg\},
\end{equation*}
which indeed leads to the desired estimate $\eqref{f:7.1}$
by an analogous computation for the cases $z\in\mathcal{B}\cup\mathcal{I}$.

Thus, the crucial task is to verify the claims $\eqref{f:7.7a}$
and $\eqref{f:7.7b}$.
As mentioned in \textbf{Step B},
the proof of $\eqref{f:7.7a}$
and $\eqref{f:7.7b}$ is divided into
two cases: (I) $0<r_Q<\frac{\varepsilon}{4}\chi_{*}(z_Q/\varepsilon)$;
(II) $r_Q\geq \frac{\varepsilon}{4}\chi_{*}(z_Q/\varepsilon)$.
We plan to address the case (I) in Step 6,
and the case (II) in Steps 7a and 7b, respectively.

\textbf{Step 4.} Proof of the estimate $\eqref{pri:A}$ (i.e., the estimate $\eqref{pri:7.1}$ in the case $\omega=1$).
Let $p_0$ and $p_1$ be given as in (a). It can be verified that  $\eqref{f:7.7a}$ and $\eqref{f:7.7b}$ hold
in the cases (I) and (II), with the verifications for the case (II)
postponed to Steps 5a and 5b.
Then, by Lemma $\ref{shen's lemma2}$, it is obtained that
\begin{equation*}
\Big(\dashint_{Q_0\cap\Omega_T}F^p\Big)^{\frac{1}{p}}
\lesssim^{\eqref{pri:7.5}}\bigg\{
\Big(\dashint_{4Q_0\cap\Omega_T}F^{2}\Big)^{\frac{1}{2}}
+\Big(\dashint_{4Q_0\cap\Omega_T}g^p\Big)^{\frac{1}{p}}\bigg\},
\end{equation*}
which yields the stated estimate $\eqref{pri:7.9}$
by following similar calculations for the cases  $z\in\mathcal{B}\cup\mathcal{I}$.
(Recall that, by previous reductions,
the present case is $z\in\mathcal{B}$ with $z=0$).
Therefore, by
\textbf{Step A} of Subsection $\ref{subsec:6.1}$,
the estimate $\eqref{pri:A}$ has been proved. As a byproduct,
if $f\equiv0$ in $\eqref{pde:7.1}$, then we have $u_\varepsilon=v_\varepsilon$, and the above estimate becomes
the reverse H\"older inequality (which will also be useful later
in establishing some annealed Calder\'on-Zygmund estimates):
\begin{equation*}
\Big(\dashint_{Q_0\cap\Omega_T}F^p\Big)^{\frac{1}{p}}
\lesssim
\Big(\dashint_{4Q_0\cap\Omega_T}F^{2}\Big)^{\frac{1}{2}}
\lesssim
\Big(\dashint_{5Q_0\cap\Omega_T}F^{s}\Big)^{\frac{1}{s}},
\end{equation*}
where $s>0$, and last inequality follows
from a convexity argument (\cite[pp.173]{Fefferman-Stein72}).

The verification of $\eqref{f:7.7a}$ and $\eqref{f:7.7b}$ for the
case (I) now begins. In this case,
set $V_B = F$ and $W_B = 0$. Then
the estimate $\eqref{f:7.7b}$ is trivial, while the
estimate $\eqref{f:7.7a}$ is reduced to showing
\begin{equation}\label{f:7.8}
\Big(\dashint_{2D}
F^{p_1}\Big)^{\frac{1}{p_1}}
\lesssim \Big(\dashint_{15D}F^{p_0}\Big)^{\frac{1}{p_0}}.
\end{equation}
To obtain the estimate $\eqref{f:7.8}$,
on account of the $\frac{1}{L}$-Lipschitz continuity of
$\chi_{*}$, it follows that for any
$z_0\in 3D$, there holds $0< \vartheta
r_Q<\frac{\varepsilon}{4}\chi_{*}(z_0/\varepsilon)$ with
$\vartheta:=(1-\frac{3}{4L})\in[7/8,1)$ by the definition of
$L$ in Remark $\ref{remark:3}$.
Let $\tilde{Q}$ be the concentric cube with
the radius $r_{\tilde{Q}}=\vartheta r_Q$.
It can be inferred that
$2Q\subset(5/2)\tilde{Q}\subset 3Q$.
Then, for any $z_0\in (5/2)\tilde{Q}$, by
appealing to Lemma $\ref{lemma:integral geometry}$, it can be derived that
\begin{equation*}
 \Big(\dashint_{U_{*,\varepsilon}(z_0)}
 |\nabla u_\varepsilon|^2\Big)^{\frac{p_0}{2}}
 \lesssim^{\eqref{pri:7.6}}
 \dashint_{15D}dz\Big(
 \dashint_{U_{*,\varepsilon}(z)}|\nabla u_\varepsilon|^2
 \Big)^{\frac{p_0}{2}},
\end{equation*}
which further implies that for any $p_1\geq p_0$,
\begin{equation*}
\Big(\dashint_{U_{*,\varepsilon}(z_0)}
|\nabla u_\varepsilon|^2\Big)^{\frac{p_1}{2}}
\lesssim \bigg(\dashint_{15D}dz\Big(
 \dashint_{U_{*,\varepsilon}(z)}
 |\nabla u_\varepsilon|^2
 \Big)^{\frac{p_0}{2}}\bigg)^{\frac{p_1}{p_0}}.
\end{equation*}
Integrating both sides above with respect to
$z_0\in 2D$, we arrive at
\begin{equation*}
\bigg(\dashint_{2D}dz
\Big(\dashint_{U_{*,\varepsilon}(z)}
|\nabla u_\varepsilon|^2\Big)^{\frac{p_1}{2}}
\bigg)^{\frac{1}{p_1}}
\lesssim \bigg(\dashint_{15D}dz\Big(
 \dashint_{U_{*,\varepsilon}(z)}|\nabla u_\varepsilon|^2
 \Big)^{\frac{p_0}{2}}\bigg)^{\frac{1}{p_0}},
\end{equation*}
which indicates the stated estimate $\eqref{f:7.8}$ and therefore
verifies $\eqref{f:7.7a}$ in the case (I). It is noted that
the restriction $p_0=2$ is not necessary at this stage.

\textbf{Step 5a.}
Show the estimate
$\eqref{f:7.7a}$ in the case (II) (under the restrictions on $p_1$ and $p_0$ as in (a)).
Recall that $\varphi_{j,r_Q} =\partial_j\bar{v}_{r_Q}$,
$Q_{*,\varepsilon}(z)
:=Q_{\varepsilon\chi_{*}(z/\varepsilon)}(z)$, and
$D_{*,\varepsilon}(z)
:= D_{\varepsilon\chi_*(z/\varepsilon)}(z)$.
For any $z\in 2D$,
it follows from H\"older's inequality ($1/\alpha+1/\alpha'=1$
with $0<\alpha-1\ll 1$) that
\begin{equation*}
\begin{aligned}
\Big(\dashint_{U_{*,\varepsilon}(z)} \big|(e_j+(\nabla\phi_j)^\varepsilon)
\varphi_{j,r_Q}\big|^2\Big)^{\frac{1}{2}}
&\lesssim \Big(\dashint_{Q_{*,\varepsilon}(z)} |(e_j+(\nabla\phi_j)^\varepsilon)|^{2\alpha}\Big)^{\frac{1}{2\alpha}}
\Big(\dashint_{D_{*,\varepsilon}(z)}  |\partial_j\bar{v}_{r_Q}|^{2\alpha'}\Big)^{\frac{1}{2\alpha'}}\\
&\lesssim \Big\{1
+\chi_{*}(z/\varepsilon)
\big(\dashint_{Q_{2\chi_{*}(z/\varepsilon)}(z/\varepsilon)}
|\bar{\phi}_j|^2\big)^{\frac{1}{2}}\Big\} \Big(\dashint_{D_{*,\varepsilon}(z)}
|\partial_j\bar{v}_{r_Q}|^{2\alpha'}\Big)^{\frac{1}{2\alpha'}}\\
&\lesssim \Big(\dashint_{D_{*,\varepsilon}(z)}
|\nabla\bar{v}_{r_Q}|^{2\alpha'}\Big)^{\frac{1}{2\alpha'}},
\end{aligned}
\end{equation*}
where the computation for the second inequality is analogous to
that given for $\eqref{f:7.10}$, and the last inequality follows from
the definition of $\chi_*(z/\varepsilon)$. Thus, by taking $p_1>2\alpha'$, we have
\begin{equation}\label{f:7.13}
\begin{aligned}
\Big(\dashint_{2D}
 V_Q^{p_1}\Big)^{\frac{1}{p_1}}
&\lesssim  \Bigg(\dashint_{2D}dz
\Big(\dashint_{\mathcal{Q}_{*,\varepsilon}(z)}
|\nabla\bar{v}_{r_{Q}}|^{2\alpha'}\Big)^{\frac{p_1}{2\alpha'}}
\Bigg)^{\frac{1}{p_1}}\\
&\lesssim^{\eqref{pri:7.7a}}
\bigg(\dashint_{7D}
|\nabla\bar{v}_{r_{Q}}|^{p_1}
\bigg)^{\frac{1}{p_1}}
\lesssim^{\eqref{pri:5.11-ap}}\bigg(\dashint_{9D}
|\nabla\bar{v}_{r_{Q}}|^{2}
\bigg)^{\frac{1}{2}}
\lesssim \bigg(\dashint_{9D}
|\nabla v_\varepsilon|^{2}
\bigg)^{\frac{1}{2}}\\
&\lesssim \bigg(\dashint_{9D}
|\nabla u_\varepsilon|^{2}
\bigg)^{\frac{1}{2}}
+\bigg(\dashint_{30D}
|f|^{2}
\bigg)^{\frac{1}{2}}
\lesssim^{\eqref{pri:7.7a}}
\bigg(\dashint_{18D}
V_Q^{2}
\bigg)^{\frac{1}{2}}
+\bigg(\dashint_{60D}
g^{2}
\bigg)^{\frac{1}{2}},
\end{aligned}
\end{equation}
where
the third line follows from the energy estimates for $w_\varepsilon$ defined
in $\eqref{pde:7.1}$. This yields the desired estimate
$\eqref{f:7.7a}$.

\textbf{Step 5b.}
Show the estimate
$\eqref{f:7.7b}$ in the case (II) (under the restrictions on $p_1$ and $p_0$ as in (a)). It follows from Lemma \ref{lemma:approxi} and
energy estimates that
\begin{equation*}
\begin{aligned}
\dashint_{\frac{8}{7}D}
W_Q^{2}
&\lesssim^{\eqref{pri:7.7a}} \dashint_{4D}
\big|\nabla v_\varepsilon - (e_j+(\nabla\phi_j)^\varepsilon)
\partial_j\bar{v}_{r_Q}\big|^{2}
+ \dashint_{4D}
\big|\nabla w_\varepsilon\big|^{2}\\
&\lesssim^{\eqref{pri:1.5}} \nu^2 \dashint_{16D}|\nabla v_\varepsilon|^2
+\dashint_{30D}|f|^2
\lesssim \nu^2 \dashint_{16D}|\nabla u_\varepsilon|^2
+\dashint_{30D}|f|^2\\
&\lesssim^{\eqref{pri:7.7a}}
\nu^2 \dashint_{32D}F^2
+\dashint_{60D}g^2,
\end{aligned}
\end{equation*}
where it is recalled that $v_\varepsilon:=u_\varepsilon-w_\varepsilon$
as defined in $\eqref{pde:7.1}$. This implies the stated estimate
$\eqref{f:7.7b}$.

\textbf{Step 6.}
Show the estimates $\eqref{f:7.7a}$ and
$\eqref{f:7.7b}$ in the case (I) (under the restrictions on $p_1$ and $p_0$ as in (b)). Here, it is recalled that
$V_Q = F$ and $W_Q = 0$ in this case,
and therefore attention is merely focused on the estimate $\eqref{f:7.7a}$.
It follows from H\"older's inequality ($1/\alpha+1/\alpha'=1$
with $0<\alpha-1\ll 1$) and the
reverse H\"older  property of the $A_p$-class that
\begin{equation*}
\begin{aligned}
\bigg(\dashint_{2D}dz
F^{p_1}
\omega(z)\bigg)^{\frac{1}{p_1}}
&\leq
\bigg(\dashint_{2D}
F^{p_1\alpha'} \bigg)^{\frac{1}{p_1\alpha'}}
\Big(\dashint_{2Q}\omega^{\alpha}\Big)^{\frac{1}{p_1\alpha}}\\
&\lesssim^{\eqref{f:7.8},\eqref{f:apw-3}}
\bigg(\dashint_{15D}
F^{p_0} \bigg)^{\frac{1}{p_0}}
\Big(\dashint_{2Q}\omega\Big)^{\frac{1}{p_1}}
\lesssim^{\eqref{f:apw-2}}
\bigg(\dashint_{15D}
F^{p_0} \bigg)^{\frac{1}{p_0}}
\Big(\dashint_{Q}\omega\Big)^{\frac{1}{p_1}},
\end{aligned}
\end{equation*}
which in fact yields the stated estimate $\eqref{f:7.7a}$.

\textbf{Step 7a.} Show the estimate
$\eqref{f:7.7a}$ in the case (II) (under the restrictions on $p_1$ and $p_0$ as in (b)). Based on the established estimate $\eqref{pri:A}$, begin by modifying the estimate $\eqref{f:7.13}$:
\begin{equation}\label{f:7.14}
\begin{aligned}
\Big(\dashint_{2D}
 V_Q^{p_1}\Big)^{\frac{1}{p_1}}
&\lesssim
\bigg(\dashint_{9D}
|\nabla v_\varepsilon|^{2}
\bigg)^{\frac{1}{2}}
\lesssim
\bigg(\dashint_{10D}
|\nabla v_\varepsilon|^{p_0}
\bigg)^{\frac{1}{p_0}}
\lesssim
\bigg(\dashint_{10D}
|\nabla u_\varepsilon|^{p_0}
\bigg)^{\frac{1}{p_0}}
+
\bigg(\dashint_{10D}
|\nabla w_\varepsilon|^{p_0}
\bigg)^{\frac{1}{p_0}}\\
&\lesssim
\bigg(\dashint_{20D}
F^{p_0}
+\big(W_Q^{(2)}\big)^{p_0}
\bigg)^{\frac{1}{p_0}}
\lesssim^{\eqref{pri:A}}
\bigg(\dashint_{45D}
F^{p_0}
+g^{p_0}
\bigg)^{\frac{1}{p_0}},
\end{aligned}
\end{equation}
where we also employ the fact that $\text{supp}(\tilde{g})\subset 44Q$
with $\tilde{g}(z):=\big(\dashint_{U_{*,\varepsilon}(z)}
 |f\textbf{1}_{30D}|^2\big)^{\frac{1}{2}}$ in the last step,
which follows from the $(1/8)$-Lipschitz continuity of $\chi_*$.
Then, by H\"older's inequality, the reverse H\"older inequality
and the doubling property
of the $A_p$-class, it is obtained that
\begin{equation*}
\begin{aligned}
\Big(\dashint_{2D}
 V_{Q}^{p_1}\omega\Big)^{\frac{1}{p_1}}
&\leq
\bigg(\dashint_{2D}
V_{Q}^{p_1\alpha'}
\bigg)^{\frac{1}{p_1\alpha'}}
\Big(\dashint_{2Q}
\omega^{\alpha}\Big)^{\frac{1}{p_1\alpha}}\\
&\lesssim^{\eqref{f:7.14},\eqref{f:apw-3},\eqref{f:apw-2}}
\bigg(\dashint_{45D}
F^{p_0}
+g^{p_0}
\bigg)^{\frac{1}{p_0}}
\Big(\dashint_{Q}
\omega \Big)^{\frac{1}{p_1}},
\end{aligned}
\end{equation*}
which consequently leads to the stated estimate $\eqref{f:7.7a}$.

\textbf{Step 7b.}
Show the estimate
$\eqref{f:7.7b}$ in the case (II) (under the restrictions on $p_1$ and $p_0$ as in (b)).
Recalling $v_\varepsilon=u_\varepsilon-w_\varepsilon$ in $\eqref{pde:7.1}$,
we first have
\begin{equation}\label{f:7.15}
\begin{aligned}
\dashint_{\frac{8}{7}D}
W_B^{p_0}
&\lesssim
\bigg(\dashint_{\frac{8}{7}D}dz
\dashint_{U_{*,\varepsilon}(z)}
|\nabla v_\varepsilon
-(e_j+(\nabla\phi_j)^\varepsilon)\partial_j\bar{v}_{r_Q}|^2\bigg)^{\frac{p_0}{2}}
 + \dashint_{2D}dz
 \Big(\dashint_{U_{*,\varepsilon}(z)}|\nabla w_\varepsilon
 |^2\Big)^{\frac{p_0}{2}}\\
&\lesssim^{\eqref{pri:7.7a},\eqref{pri:A}}
\Big(\dashint_{4D}
|\nabla v_\varepsilon
-(e_j+(\nabla\phi_j)^\varepsilon)\partial_j\bar{v}_{r_Q}|^2\Big)^{\frac{p_0}{2}}
 + \dashint_{45D}dz
 \Big(\dashint_{U_{*,\varepsilon}(z)}|f|^2\Big)^{\frac{p_0}{2}}\\
&\lesssim^{\eqref{pri:1.5}}
\Big(\dashint_{16D}
|\nabla v_\varepsilon|^2\Big)^{\frac{p_0}{2}}
+\dashint_{45D}
 g^{p_0}.
\end{aligned}
\end{equation}
Based on the Meyer-type estimates (see Lemma \ref{lemma:5.2}), a routine computation yields
\begin{equation*}
\begin{aligned}
\Big(\dashint_{16D}|\nabla v_\varepsilon|^2\Big)^{\frac{1}{2}}
&\lesssim^{\eqref{pri:5.12-ap}} \dashint_{18D}|\nabla v_\varepsilon|
\lesssim  \dashint_{18D}|\nabla u_\varepsilon| + |\nabla w_\varepsilon|\\
&\lesssim^{\eqref{pri:7.7a}}
 \Big(\dashint_{36D} \big(F^{p_0} + (W_Q^{(2)})^{p_0}\big)\Big)^{\frac{1}{p_0}}
 \lesssim^{\eqref{pri:A}}
 \Big(\dashint_{45D} \big(F^{p_0} + g^{p_0}\big)\Big)^{\frac{1}{p_0}}.
\end{aligned}
\end{equation*}
Substituting this back into $\eqref{f:7.15}$ we obtain the desired
estimate $\eqref{pri:7.7a}$. This completes the proof.
\qed

\subsection{Annealed Calder\'on-Zygmund estimates}

\noindent
As mentioned before, the idea is to repeatedly
use Shen's real argument, which was first applied
to the study of annealed Calder\'on-Zygmund estimates for
the Helmholtz projection
$\nabla(\nabla\cdot a\nabla)^{-1}
\nabla\cdot $ in \cite{Duerinckx-Otto-20}.
The original proof is modified for the present purpose, and the entire proof is completed in two parts (\textbf{Steps 1-3}). In the first part,
the following is
established:
\begin{equation}\label{pri:6.5}
\int_{\Omega_T}dz
\Big\langle\Big(\dashint_{
U_{*,\varepsilon}(z)}|\nabla u_\varepsilon|^2 \Big)^{\frac{p}{2}}
\Big\rangle^{\frac{q}{p}}
\lesssim_{\mu,d,\partial\Omega,p,q}
\int_{\Omega_T}dz
\Big\langle\Big(\dashint_{U_{*,\varepsilon}(z)}
|f|^2\Big)^{\frac{p}{2}}\Big\rangle^{\frac{q}{p}}
\end{equation}
for any $1<p,q<\infty$, by employing the quenched Calder\'on-Zygmund estimate $\eqref{pri:A}$. Then, the second part (\textbf{Steps 4-6}) is devoted to proving
the stated estimate $\eqref{pri:7.2}$.
Unlike the decomposition made in Subsection $\ref{subsec:6.2}$,
the concrete forms of $W_Q$ and $V_Q$ defined below are
independent of $Q$.
In addition,
the details of moving from a family of ``local estimates''
to global ones are omitted (as this has
been well presented in Subsection $\ref{subsec:6.1}$),
whereas attention is focused solely on verifying the preconditions
satisfied by $W_Q$ and $V_Q$, respectively,
in Lemma $\ref{shen's lemma2}$.

To this end, given $Q_0$ (with center $z_0\in\mathcal{B}$ and radius  $r_0$), let $Q$ be arbitrary cube
satisfying $r_Q\leq (1/100)r_{0}$ and
either $z_Q\in 2Q_0\cap\partial_{\shortparallel}\Omega_T$ or
$4Q\subset 2Q_0\cap\Omega_T$.
Here, only the case $z_Q\in 2Q_0\cap \partial_{\shortparallel}\Omega_T$ is considered, and $D:=Q\cap\Omega_T$.

\medskip
\textbf{Step 1.} Reduction of the proof of $\eqref{pri:6.5}$.
Let $p,q_1\in(1,\infty)$ be arbitrarily fixed.
The estimate $\eqref{pri:6.5}$ must be shown for the case $p\leq q<q_1$;  then the case $q<p$ follows from a duality argument.
For the cube $Q$ satisfying $z_Q\in 2Q_0\cap \partial_{\shortparallel}\Omega_T$, let $w_\varepsilon,v_\varepsilon\in L^2(0,T;H_0^1(\Omega))$
satisfy the following equations:
\begin{equation*}
\left\{\begin{aligned}
 \partial_t w_\varepsilon
 +\mathcal{L}_\varepsilon(w_\varepsilon) &=
 \nabla \cdot f\textbf{1}_{2D}
 &~\text{in}&~\Omega_T;\\
 w_\varepsilon &= 0
 &~\text{on}&~\partial_p\Omega_T,
\end{aligned}\right.
 \quad\text{and}\quad
\left\{\begin{aligned}
 \partial_t v_\varepsilon
 +\mathcal{L}_\varepsilon(v_\varepsilon)
 &=\nabla \cdot f\textbf{1}_{\Omega\setminus 2D}
 &~\text{in}&~ \Omega_T;\\
 v_\varepsilon &= 0
 &~\text{on}&~\partial_p\Omega_T.
\end{aligned}\right.
\end{equation*}
For the ease of statement,
the following notation is introduced:
\begin{equation*}
 F(z)
 :=\Big\langle\Big(\dashint_{U_{*,\varepsilon}(z)}|\nabla u_\varepsilon|^2\Big)^{\frac{p}{2}}\Big\rangle^{\frac{1}{p}};
 \qquad
 g(z)
 :=\Big\langle\Big(\dashint_{
 U_{*,\varepsilon}(z)}
 |f|^{2}\Big)^{\frac{p}{2}}\Big\rangle^{\frac{1}{p}},
\end{equation*}
and
\begin{equation*}
 W_Q(z):=\Big\langle
 \Big(\dashint_{U_{*,\varepsilon}(z)}
 |\nabla w_\varepsilon|^2\Big)^{\frac{p}{2}}\Big\rangle^{\frac{1}{p}};
 \qquad
 V_Q(z):=\Big\langle\Big(\dashint_{
 U_{*,\varepsilon}(z)}
 |\nabla v_\varepsilon|^2\Big)^{\frac{p}{2}}\Big\rangle^{\frac{1}{p}}.
\end{equation*}
Thus, it follows from the relationship
$u_\varepsilon = w_\varepsilon + v_\varepsilon$ in $\Omega_T$ that $F\leq W_B+V_B$.
According to the preconditions stated in
Lemma $\ref{shen's lemma2}$ (with $\omega=1$ here),
the following two estimates must be established:
\begin{subequations}
\begin{align}
& \dashint_{\frac{1}{2}D}W_B^{p}
 \lesssim \dashint_{7D}
 g^{p};
 \label{pri:6.6a}\\
& \Big(\dashint_{\frac{1}{8}D}V_B^{q_1}\Big)^{\frac{p}{q_1}}
 \lesssim \dashint_{7D}(F^{p}+g^{p}).
 \label{pri:6.6b}
\end{align}
\end{subequations}


\textbf{Step 2.} Show the estimate $\eqref{pri:6.6a}$.
It is reduced to showing
\begin{equation}\label{f:6.20}
\dashint_{\frac{1}{2}D}
\Big(\dashint_{U_{*,\varepsilon}(x)}
|\nabla w_\varepsilon|^2\Big)^{\frac{p}{2}}
 \lesssim \dashint_{7D}
\Big(\dashint_{U_{*,\varepsilon}(x)}
|f|^{2}\Big)^{\frac{p}{\bar{2}}}.
\end{equation}
The main work is to analyze the size of
support set of $\dashint_{U_{*,\varepsilon}(x)} f\textbf{1}_{2D}$,
and this is discussed in two cases: (1) $r_Q\geq \frac{\varepsilon}{4}\chi_{*}(z_Q/\varepsilon)$;
(2) $r_Q < \frac{\varepsilon}{4}\chi_{*}(z_Q/\varepsilon)$.
Starting from the first case, it is noted that if
$\dashint_{U_{*,\varepsilon}(z)}
f\textbf{1}_{2D}\not=0$, then it must
hold that $\text{d}(z,z_Q)-\varepsilon\chi_{*}(z/\varepsilon)<2r_Q$ due to
the geometric condiserations. In this respect, $z\in\mathbb{R}^{d+1}$ is required to satisfy
\begin{equation}\label{f:6.19}
 \text{d}(z,z_Q)\leq 2r_Q+ \varepsilon \chi_{*}(z/\varepsilon)
 \leq  2r_Q+ \varepsilon \chi_{*}(z_Q/\varepsilon)
 + \text{d}(z,z_Q)/8 \quad
 \Rightarrow \quad
 \text{d}(z,z_Q) \leq 7r_Q.
\end{equation}
Hence, with the help of the quenched Calder\'on-Zygmund estimate, it is found that
\begin{equation}\label{f:6.21}
\begin{aligned}
\dashint_{D}dz
\Big(\dashint_{U_{*,\varepsilon}(z)}
|\nabla w_\varepsilon|^2\Big)^{\frac{p}{2}}
&\leq \frac{1}{|Q|}
\int_{\Omega_T}dz
\Big(\dashint_{U_{*,\varepsilon}(z)}
|\nabla w_\varepsilon|^2\Big)^{\frac{p}{2}}\\
&\lesssim^{\eqref{pri:A}}
\frac{1}{|Q|}
\int_{\Omega_T}dz
\Big(\dashint_{U_{*,\varepsilon}(z)}
|f\textbf{1}_{2D}|^2\Big)^{\frac{p}{2}}
\lesssim^{\eqref{f:6.19}} \dashint_{7D}dz
\Big(\dashint_{U_{*,\varepsilon}(z)}
|f|^2\Big)^{\frac{p}{2}}.
\end{aligned}
\end{equation}

Now, we proceed to deal with the case (2)
$r_Q < \frac{\varepsilon}{4}\chi_{*}(z_Q/\varepsilon)$.
For any $y_0\in\frac{1}{2}D$,
by noting $2Q\subset Q_{*,\varepsilon}(y_0)$ and using
the energy estimate, it follows that
\begin{equation}\label{f:6.22}
 \dashint_{U_{*,\varepsilon}(y_0)} |\nabla w_\varepsilon|^2
 \lesssim
 \dashint_{U_{*,\varepsilon}(y_0)} |f|^{2}.
\end{equation}
Then, by appealing to Lemma $\ref{lemma:integral geometry}$ (and the fact that $r_Q<(\frac{8L}{8L-1})\frac{\varepsilon}{4}\chi_{*}(y_0/\varepsilon)$),
it follows that
\begin{equation*}
\Big(\dashint_{U_{*,\varepsilon}(x_0)} |\nabla w_\varepsilon|^2\Big)^{\frac{p}{2}}
\lesssim
\Big(\dashint_{U_{*,\varepsilon}(x_0)} |f|^{2}
\Big)^{\frac{p}{2}}
\lesssim^{\eqref{pri:7.6}}
\dashint_{5D}dz
\Big(\dashint_{U_{*,\varepsilon}(z)} |f|^{2}\Big)^{\frac{p}{2}}.
\end{equation*}
Integrating both sides above with respect to $y_0\in\frac{1}{2}D$ yields the estimate $\eqref{f:6.20}$ in this case. This, together with
$\eqref{f:6.21}$, completes the argument for the stated estimate $\eqref{f:6.20}$; finally, taking expectation on both
sides of $\eqref{f:6.20}$ leads to the desired estimate $\eqref{pri:6.6a}$.

\textbf{Step 3.} Show the estimate $\eqref{pri:6.6b}$.
By using Minkowski's inequality, H\"older's inequality,
and the reverse H\"older inequality (which follows
from $\eqref{f:7.1}$ by setting $\omega=1$
and $g=0$ therein), it is obtained that
\begin{equation}\label{f:6.25}
\begin{aligned}
\Big(\dashint_{\frac{1}{8}D}
V_B^{q_1}\Big)^{\frac{p}{q_1}}
&= \bigg(\dashint_{\frac{1}{8}D}dz
\Big\langle\Big(\dashint_{U_{*,\varepsilon}(z)}
|\nabla v_\varepsilon|^2\Big)^{\frac{p}{2}}\Big\rangle^{\frac{q_1}{p}}
\bigg)^{\frac{p}{q_1}}
\leq
\bigg\langle
\Big(\dashint_{\frac{1}{8}D}dz
\Big(\dashint_{
U_{*,\varepsilon}(z)}|\nabla v_\varepsilon|^2
\Big)^{\frac{q_1}{2}}\Big)^{\frac{p}{q_1}}\bigg\rangle\\
&\lesssim^{\eqref{f:7.1}}
\dashint_{\frac{1}{2}D}dz\Big\langle
\Big(\dashint_{U_{*,\varepsilon}(z)}
|\nabla v_\varepsilon|^2\Big)^{\frac{p}{2}}\Big\rangle.
\end{aligned}
\end{equation}
This, together with the fact that $u_\varepsilon = w_\varepsilon + v_\varepsilon$ in $\Omega$ and with the estimate $\eqref{f:6.20}$,
yields the desired estimate $\eqref{pri:6.6b}$. It concludes the first part of the proof.

\textbf{Step 4.} Show the estimate $\eqref{pri:7.2}$.
Let $0<q_0-1\ll 1$, and let $q_1\in(1,\infty)$ be chosen sufficiently large. On account of
Lemma $\ref{shen's lemma2}$, for any $q_0<q<q_1$ and $\omega\in A_{q/q_0}$, it suffices to establish
\begin{subequations}
\begin{align}
& \dashint_{\frac{1}{2}D}W_Q^{q_0}
 \lesssim \dashint_{7D}
 g^{q_0};
\label{pri:6.7a}\\
& \Big(\dashint_{\frac{1}{8}D}V_Q^{q_1}\omega\Big)^{\frac{q_0}{q_1}}
 \lesssim \Big(\dashint_{7D}(F^{q_0}+g^{q_0})\Big)
 \Big(\dashint_{Q}\omega\Big)^{\frac{q_0}{q_1}}.
 \label{pri:6.7b}
\end{align}
\end{subequations}
Then, the same arguments as those given for Step 1 in Subsection $\ref{subsec:6.2}$ lead to the stated estimate $\eqref{pri:7.2}$,
where the estimate $\eqref{pri:6.5}$ is also used.

\textbf{Step 5.} Show the estimate $\eqref{pri:6.7a}$, which is parallel to Step 2. For the case $r_Q\geq\frac{\varepsilon}{4}\chi_*(z_Q/\varepsilon)$, it follows that
\begin{equation}\label{f:6.23}
\begin{aligned}
\dashint_{D}dz
\Big\langle\Big(\dashint_{U_{*,\varepsilon}(z)}
|\nabla w_\varepsilon|^2
\Big)^{\frac{p}{2}}\Big\rangle^{\frac{q_0}{p}}
&\lesssim \frac{1}{|B|}
\int_{\Omega_T}dz
\Big\langle\Big(\dashint_{U_{*,\varepsilon}(z)}
|\nabla w_\varepsilon|^2\Big)^{\frac{p}{2}}\Big\rangle^{\frac{q_0}{p}}
\\
&\lesssim^{\eqref{pri:6.5}}
\frac{1}{|B|}
\int_{\Omega_T}dz
\Big\langle\Big(\dashint_{U_{*,\varepsilon}(z)}
|f\textbf{1}_{2D}|^2\Big)^{\frac{p}{2}}\Big\rangle^{\frac{q_0}{p}}
\lesssim^{\eqref{f:6.19}} \dashint_{7D}
F^{q_0}.
\end{aligned}
\end{equation}
Then, for the case $r_Q<\frac{\varepsilon}{4}\chi_*(z_Q/\varepsilon)$, consider any $y_0\in\frac{1}{2}Q\cap\Omega_{*}$.  In view of the estimate
$\eqref{f:6.22}$, it is obtained that
\begin{equation*}
\Big\langle\Big(\dashint_{U_{*,\varepsilon}(y_0)} |\nabla w_\varepsilon|^2
\Big)^{\frac{p}{2}}\Big\rangle^{\frac{q_0}{p}}
 \lesssim
\Big\langle\Big(\dashint_{U_{*,\varepsilon}(y_0)} |f|^{2}
\Big)^{\frac{p}{2}}\Big\rangle^{\frac{q_0}{p}}.
\end{equation*}
Integrating both sides above with respect to $y_0\in\frac{1}{2}D$ then yields \eqref{f:6.23}. This completes the proof of the estimate $\eqref{pri:6.7a}$.

\textbf{Step 6.} Show the estimate $\eqref{pri:6.7b}$,
which is parallel to Step 3.
Let $\gamma$ and $\gamma'$ be such that $1/\gamma+1/\gamma'=1$
with $0<\gamma'-1\ll 1$. By using H\"older's inequality,
the reverse H\"older property of the $A_q$ class, and Minkowski's inequality (in that
order), it is derived that
\begin{equation}\label{f:6.24}
\begin{aligned}
\Big(\dashint_{\frac{1}{8}D}V_Q^{q_1}\omega\Big)^{\frac{1}{q_1}}
&\leq
\bigg(\dashint_{\frac{1}{8}D}dz
\Big\langle\big(
\dashint_{U_{*,\varepsilon}(z)}|\nabla v_\varepsilon|^2
\big)^{\frac{p}{2}}\Big\rangle^{\frac{q_1\gamma}{p}}
\bigg)^{\frac{1}{q_1\gamma}}\Big(\dashint_{Q}
\omega^{\gamma'}\Big)^{\frac{1}{q_1\gamma'}}\\
&\lesssim^{\eqref{f:apw-3}}
\bigg(\dashint_{\frac{1}{8}D}dz
\Big\langle\big(
\dashint_{U_{*,\varepsilon}(z)}|\nabla v_\varepsilon|^2
\big)^{\frac{p}{2}}\Big\rangle^{q_1\gamma}
\bigg)^{\frac{1}{q_1\gamma p}}\Big(\dashint_{Q}
\omega\Big)^{\frac{1}{q_1}}\\
&\lesssim
\bigg\langle
\bigg(\dashint_{\frac{1}{8}D}dz
\Big(
\dashint_{U_{*,\varepsilon}(z)}
|\nabla v_\varepsilon|^2\Big)^{\frac{pq_1\gamma}{2}}
\bigg)^{\frac{1}{q_1\gamma}}\bigg\rangle^{\frac{1}{p}}
\Big(\dashint_{Q}
\omega\Big)^{\frac{1}{q_1}}.
\end{aligned}
\end{equation}

On account of the estimate $\eqref{f:6.25}$,
the right-hand side of $\eqref{f:6.24}$ is bounded by
\begin{equation}\label{f:6.26}
\bigg\langle
\bigg(\dashint_{\frac{1}{2}D}dz
\Big(
\dashint_{U_{*,\varepsilon}(z)}
|\nabla v_\varepsilon|^2\Big)^{\frac{1}{2}}\bigg)^{p}\bigg\rangle^{\frac{1}{p}}
\Big(\dashint_{Q}
\omega\Big)^{\frac{1}{q_1}}
\leq
\bigg(\dashint_{\frac{1}{2}D}dz
\Big\langle\Big(
\dashint_{U_{*,\varepsilon}(z)}
|\nabla v_\varepsilon|^2\Big)^{\frac{p}{2}}
\Big\rangle^{\frac{1}{p}}\bigg)
\Big(\dashint_{Q}
\omega\Big)^{\frac{1}{q_1}},
\end{equation}
where Minkowski's inequality is also employed for the inequality above.
By plugging $\eqref{f:6.26}$ back into $\eqref{f:6.24}$, it can be obtained  that
\begin{equation*}
\begin{aligned}
\Big(\dashint_{\frac{1}{8}D}V_Q^{q_1}\omega\Big)^{\frac{1}{q_1}}
&\lesssim
\bigg(\dashint_{\frac{1}{2}D}
V_Q
\bigg)
\Big(\dashint_{Q}
\omega\Big)^{\frac{1}{q_1}}
\lesssim
\Big(\dashint_{\frac{1}{2}D}
V_Q^{q_0}
\Big)^{\frac{1}{q_0}}
\Big(\dashint_{Q}
\omega\Big)^{\frac{1}{q_1}}\\
&\lesssim
\Big(\dashint_{\frac{1}{2}D}
(F^{q_0}+W_Q^{q_0})
\Big)^{\frac{1}{q_0}}
\Big(\dashint_{Q}
\omega\Big)^{\frac{1}{q_1}}
\lesssim^{\eqref{pri:6.6a}}
\Big(\dashint_{7D}
(F^{q_0}+g^{q_0})
\Big)^{\frac{1}{q_0}}
\Big(\dashint_{Q}
\omega\Big)^{\frac{1}{q_1}},
\end{aligned}
\end{equation*}
which consequently leads to $\eqref{pri:6.6b}$.
This concludes the proof.

\subsection{Proof of Theorem $\ref{thm:C-Z}$}

\noindent
Based on the results stated in Proposition $\ref{P:7.1}$,
to obtain Theorem $\ref{thm:C-Z}$, it suffices to transfer
the large-scale average ``$\dashint_{U_{*,\varepsilon}(z)}$''
in the estimates
to the small-scale counterpart ``$\dashint_{U_{\varepsilon}(z)}$'';
the price to pay is the need for higher integrability
with respect to the probability index.
The demonstration is indeed independent of partial differential equations,
and the related results are
stated below.

\begin{lemma}\label{lemma:7.1}
Let $\chi_{*}$ be given as in Lemma $\ref{lemma:approxi}$ with the  $\frac{1}{L}$-Lipschitz continuity.
Let
$1\leq p\leq q<\infty$ and $1\leq s<\infty$.
Then, for any $p_1>p$, the following hold:
\begin{subequations}
\begin{align}
&\bigg(\int_{\Omega_T}dz\Big\langle\big(
\dashint_{U_{*,\varepsilon}(x)}|F|^s\big)^{\frac{p}{s}}
\Big\rangle^{\frac{q}{p}}\omega(z) \bigg)^{\frac{1}{q}}
\lesssim
\bigg(\int_{\Omega_T}dz\Big\langle\big(
\dashint_{U_\varepsilon(z)}|F|^s\big)^{\frac{p_1}{s}}
\Big\rangle^{\frac{q}{p_1}}\dashint_{Q_{\varepsilon}(z)}\omega \bigg)^{\frac{1}{q}};
\label{pri:7.8a}\\
& \bigg(\int_{\Omega_T}dz\Big\langle\big(
\dashint_{U_{\varepsilon}(z)}|F|^s\big)^{\frac{p}{s}}
\Big\rangle^{\frac{q}{p}}
\dashint_{Q_{\varepsilon}(z)}\omega \bigg)^{\frac{1}{q}}
\lesssim
\bigg(\int_{\Omega_T}dz\Big\langle\big(
\dashint_{U_{*,\varepsilon}(z)}|F|^s\big)^{\frac{p_1}{s}}
\Big\rangle^{\frac{q}{p_1}}\omega(z) \bigg)^{\frac{1}{q}},
\label{pri:7.8b}
\end{align}
\end{subequations}
for any $\omega\in A_q$, where the multiplicative constant depends
on $\mu,\lambda_1,d,\partial\Omega,s,p_1,p,q$, and $[\omega]_{A_q}$.
\end{lemma}

\begin{proof}
See \cite[Lemma 3.10]{Wang-Xu-25} for the details.
\end{proof}

\medskip

The flow of the proofs is presented in Fig.$\ref{pic:4.1}$.

\begin{figure}[htbp]
\centering 
\resizebox{0.95\textwidth}{!}{
\begin{tikzpicture}[
    node distance=0.7cm and 1cm,
    >=Stealth, 
    box/.style={
        rectangle, draw,
        minimum width=3.2cm, minimum height=0.9cm,
        align=center, text width=3cm
    },
    smallbox/.style={
        rectangle, draw,
        minimum width=3.5cm, minimum height=0.7cm,
        align=center, text width=3cm
    },
    bigbox/.style={
    rectangle, draw,
        minimum width=6cm, minimum height=1.5cm,
        align=center, text width=5.5cm
    },
    cbigbox/.style={
    rectangle, draw, dashed,
    inner sep=10pt, 
    rounded corners=3pt 
    }
]

\node[smallbox] (N1) {Lemma \ref{lemma:6.1}};
\node[smallbox, left=of N1] (N2) {Lemma \ref{lemma:integral geometry}};
\node[smallbox, right=of N1] (N3) {Lemma \ref{lemma:approxi}};
\node[smallbox, right=of N3, draw=white, text opacity=0] (N4) {}; 
\node[smallbox, below=of N1, draw=white, text opacity=0] (B1) {}; 
\node[smallbox, right=of B1, draw=white, text opacity=0] (B2) {}; 
\node[smallbox, right=of B2, draw=white, text opacity=0] (B2*) {}; 
\node[smallbox, right=of B2*, draw=white, text opacity=0] (B3) {}; 
\node[bigbox, below =of B1] (N5) {Proposition \ref{P:7.1}\\\small(Quenched C-Z estimate \eqref{pri:7.1})};
\node[smallbox, right=of N5, draw=white, text opacity=0] (N6) {}; 
\node[smallbox, right=of N6, draw=white, text opacity=0] (N7) {}; 
\node[bigbox, below=of B3] (N8) {Proposition $\ref{P:7.1}$\\\small(Annealed C-Z estimate \eqref{pri:7.2})};
\node[bigbox, above=of B3] (N9) {\textbf{Theorem \ref{thm:C-Z}}\\\small (Calder\'on-Zygmund theory)};

\node[cbigbox,fit=(N1) (N2) (N3)] (N0){};


\draw[->,thick] (N0) -- node[midway, right, align=center, text width=4cm] {
    $+$~~Shen's Lemma\\
    Lemma \ref{lemma:weight}
} (N5);     

\draw[->,thick] (N5) -- node[midway, above, align=center, text width=4cm] {
    $+$~~Shen's Lemma
} (N8);     

\draw[->,thick] (N8) -- node[midway, right, align=center, text width=4cm] {
    $+$~~Lemma $\ref{lemma:7.1}$
} (N9);     

\draw[->] (N5) -- (N9);     


\end{tikzpicture} %
}
\caption{On the proof structure of Theorem $\ref{thm:C-Z}$} 
\label{pic:4.1}
\end{figure}
\noindent
\textbf{Proof of Theorem $\ref{thm:C-Z}$}.
The estimate $\eqref{pri:A}$ has already been proven in Subsection $\ref{subsec:6.2}$. Let $p_2>p_1>p>1$ and $\omega\in A_q$ be arbitrary.
In view of Lemma $\ref{lemma:7.1}$, it follows
from Proposition $\ref{P:7.1}$ that
\begin{equation}\label{f:6.28}
\begin{aligned}
&\Bigg(\int_{\Omega_T}dz\Big\langle
\Big(\dashint_{
U_\varepsilon(z)}
|\nabla u_\varepsilon|^2 \Big)^{\frac{p}{2}}
\Big\rangle^{\frac{q}{p}}
\dashint_{Q_{\varepsilon}(z)}\omega\Bigg)^{\frac{1}{q}}
\lesssim^{\eqref{pri:7.8b}}
\Bigg(\int_{\Omega_T}dz\Big\langle
\Big(\dashint_{
U_{*,\varepsilon}(z)}
|\nabla u_\varepsilon|^2 \Big)^{\frac{p_1}{2}}
\Big\rangle^{\frac{q}{p_1}}\omega(z)\Bigg)^{\frac{1}{q}}\\
&\lesssim^{\eqref{pri:7.2}}
\Bigg(\int_{\Omega_T}dz
\Big\langle\Big(\dashint_{U_{*,\varepsilon}(z)}
|f|^2\Big)^{\frac{p_1}{2}}\Big\rangle^{\frac{q}{p_1}}\omega(z)
\Bigg)^{\frac{1}{q}}
\lesssim^{\eqref{pri:7.8a}}
\Bigg(\int_{\Omega_T}dz
\Big\langle\Big(\dashint_{U_{\varepsilon}(z)}
|f|^2\Big)^{\frac{p_2}{2}}\Big\rangle^{\frac{q}{p_2}}
\dashint_{Q_\varepsilon(z)}\omega
\Bigg)^{\frac{1}{q}}.
\end{aligned}
\end{equation}

If the parabolic cylinder $\Omega_T$ is bounded, take
$\omega(z)=\omega_\sigma(z):=[\text{dist}(x,\partial\Omega_0)]^{\underline{q}-1}$
with $0<\underline{q}<q$ and
$z=(x,t)\in\Omega_T$. Then, for any $z\in\Omega_T$, similar to
$\eqref{f:3.41-ap}$,
it can be derived that
\begin{equation*}
\dashint_{Q_\varepsilon(z)}\omega_\sigma \sim \omega_\sigma(z).
\end{equation*}
If $\Omega_T=\mathbb{R}^{d+1}$, set
$\omega(z)=\omega_{\sigma}(z)=\text{d}(z+\varepsilon,0)^q$ with $z\in\mathbb{R}^{d+1}$; the above relationship
also holds. Then, by substituting this back into
$\eqref{f:6.28}$, the desired estimate $\eqref{pri:B}$ is obtained,
and the proof is hereby completed.
\qed

\section{Oscillation and fluctuation of homogenization errors}\label{section:3}

\noindent
This section is primarily devoted to presenting the proofs of Theorems $\ref{thm:1.1}$ and $\ref{thm:1.2}$. If readers want to have a prior understanding of the proof structure, they are advised to refer to Subsection $\ref{subsection:5.4}$.

\subsection{A duality argument}

\noindent
Recall that ${_0H^{1,1/2}(\partial_{\shortparallel}\Omega_T)}$ denotes
the Hilbert space of functions possessing one spatial derivative and half
of a time derivative in $L^2(\partial_{\shortparallel}\Omega_T)$,
with each element required to vanish on $\partial\Omega\times\{t=0\}$ (see
e.g., \cite[pp.353]{Brown-89}). It is noted that the following conclusion also holds for nonzero initial values (such as $H_0^1(\Omega)$ data), but requires a more delicate analysis.

\begin{proposition}\label{P:4.2}
Let $\Omega\subset\mathbb{R}^d$ be a bounded $C^1$ domain with
$d\geq 2$, $\varepsilon\in(0,1]$, and $T>0$.
Suppose that the ensemble $\langle\cdot\rangle$ is stationary
with respect to $\eqref{c:1}$ and satisfies
the spectral gap condition $\eqref{c:3}$.
Given $F\in L^2(\Omega_T)$ and  $g\in{_0H^{1,1/2}(\partial_{\shortparallel}\Omega_T)}$  satisfying $\|F\|_{L^2(\Omega_T)}
+\|g\|_{H^{1,1/2}(\partial_{\shortparallel}\Omega_T)}=1$, let
$u_\varepsilon$ and $u_0$ be the weak solutions of the following  initial-Dirichlet problems:
\begin{equation}\label{pde:3.1*}
(\emph{DP}_\varepsilon)\left\{\begin{aligned}
\big(\partial_t + \mathcal{L}_\varepsilon\big)(u_\varepsilon)
& = F &~&\text{in}~ \Omega_T;\\
u_\varepsilon & = g &~&\text{on}~ \partial_{\shortparallel}\Omega_T;\\
u_\varepsilon(\cdot,0) & = 0 &~&\text{on}~ \Omega,
\end{aligned}\right.
\qquad
(\emph{DP}_0)\left\{\begin{aligned}
\big(\partial_t + \mathcal{L}_0\big)(u_0)
& = F &~&\text{in}~ \Omega_T;\\
u_0 & = g &~&\text{on}~ \partial_{\shortparallel}\Omega_T;\\
u_0(\cdot,0) & = 0 &~&\text{on}~ \Omega,
\end{aligned}\right.
\end{equation}
respectively.
Then, for any $p\in[1,\infty)$, it holds that
\begin{equation}\label{pri:3.0}
\Big\langle\Big(\int_{\Omega_T}|u_\varepsilon - u_0|^2\Big)^p\Big\rangle^{\frac{1}{p}}
\lesssim  \varepsilon\Big[
\mu_d^2(R_0/\varepsilon)
+\mu_d(R_0/\varepsilon)\ln(R_0/\varepsilon)\Big],
\end{equation}
where $\mu_d$ is given by $\eqref{pri:5.5**}$,  
and the multiplicative constant depends only on $\mu, \lambda_1, d,T$, and $\Omega$.
\end{proposition}

Let
$\Psi_{r}\in C_0^{2}(\Omega_T)$ be a cut-off function satisfying
$\Psi_{r} = 1$ on $\Sigma_{4r}^T$,
$\Psi_{r} = 0$ outside $\Sigma_{2r}^T$, and
\begin{equation}\label{f:4.1*}
(1/r)|\nabla\Psi_{r}|
+|\nabla^2\Psi_{r}|+|\partial_t\Psi_{r}|
\lesssim_d 1/r^2
\qquad\text{in}
\quad \Omega_T.
\end{equation}

\begin{lemma}[Energy estimates]\label{lemma:3.1}
Let $\Omega\subset\mathbb{R}^d$ be a (bounded) Lipschitz domain with
$d\geq 2$, $\varepsilon\in(0,1]$, and $T>0$.
Suppose that the ensemble $\langle\cdot\rangle$ is stationary
with respect to $\eqref{c:1}$ and satisfies
the spectral gap condition $\eqref{c:3}$.
Given $u_0\in W_{2}^{2,1}(\Omega_T)$,
let $w_\varepsilon$ be given as in $\eqref{eq:3.1}$ and
satisfy the equations $\eqref{eq:3.2}$.
Then, by choosing $\varphi_j =
S_\varepsilon(\Psi_{\varepsilon}\nabla_ju_0)$
in $\eqref{eq:3.1}$ and setting $\pi_{ilj}=0$ in $\eqref{eq:3.5}$, for any $p\in[1,\infty)$, it can be obtained that
\begin{equation}\label{pri:3.1}
\begin{aligned}
&\bigg\langle\Big(\int_{\Omega_T}|\nabla w_\varepsilon|^2\Big)^{\frac{p}{2}}\bigg\rangle^{\frac{1}{p}}
+\bigg\langle\Big(\sup_{0\leq t\leq T}\big\|w_\varepsilon(\cdot ,t)\big\|_{L^2(\Omega)}\Big)^p\bigg\rangle^{\frac{1}{p}}\\
&\lesssim \mu_d(1/\varepsilon)\bigg\{
\|\bar{\mu}_d\nabla u_0\|_{L^2(\boxbox_{4\varepsilon})}
 + \varepsilon\Big(
\|\bar{\mu}_d\nabla^2 u_0 \|_{L^2((\Omega\setminus O_{2\varepsilon})_T)}
+\|\bar{\mu}_d\partial_t u_0 \|_{L^2((\Omega\setminus O_{2\varepsilon})_T)}\Big)\bigg\},
\end{aligned}
\end{equation}
where we recall that $\boxbox_{r}:=\Omega_T\setminus \Sigma_{r}^T$ 
and $(\Omega\setminus O_r)_T:=
(\Omega\setminus O_r)\times [0,T]$,
and the multiplicative constant depends only on $\mu,d,p$,
and $\lambda_1$.
\end{lemma}

\begin{proof}
It suffices to prove the statement for $p\geq 2$, since the probability measure is finite.

\textbf{Step 1.} Outline of the proof.
Starting from
the energy estimate and the expression $\eqref{eq:3.3}$, it follows that
\begin{equation*}
\begin{aligned}
&\|w_\varepsilon\|_{L^2(0,T;
H_0^1(\Omega))} + \sup_{0\leq t\leq T}\|w_\varepsilon(\cdot,t)\|_{L^2(\Omega)}
\lesssim\|\tilde{f}\|_{L^2(\Omega_T)}
\lesssim
\big\|\nabla u_0
- S_\varepsilon K_\varepsilon(\Psi_{\varepsilon}\nabla u_0 )\big\|_{L^2(\Omega_T)} \\
& +\varepsilon\big\|\varpi(\cdot/\varepsilon,\cdot/\varepsilon^2)\nabla S_\varepsilon K_\varepsilon(\Psi_{\varepsilon}\nabla u_0)\big\|_{L^2(\Omega_T)}
+\varepsilon^2 \big\|\tilde{\varpi}(\cdot/\varepsilon,\cdot/\varepsilon^2)
\big(\nabla^2+\partial_t\big)S_\varepsilon K_\varepsilon
(\Psi_{\varepsilon}\nabla u_0)\big\|_{L^2(\Omega_T)}\\
&=: I_1 + I_2 + I_3,
\end{aligned}
\end{equation*}
where $\varpi_j:= a\bar{\phi}_j
-\bar{\sigma}_{j}
+ a(\partial \tilde{\sigma}_{(d+1)j})$ and
$\tilde{\varpi}_j = \big(a \otimes \tilde{\sigma}_{(d+1)j},
\tilde{\sigma}_{(d+1)j}\big)$. Obviously, the first term $I_1$ is deterministic, while
the other two terms are random and must be estimated
in terms of their $p$-moments.
For the stated estimate $\eqref{pri:3.1}$, by taking
$\langle|\cdot|^p\rangle^{\frac{1}{p}}$ with $p\geq 2$
on the both sides of the above equality, it is obtained that
\begin{equation}\label{f:3.1}
\begin{aligned}
\big\langle\|\nabla w_\varepsilon\|_{L^2(\Omega_T)}^p\big\rangle^{\frac{1}{p}} +
\big\langle\big(\sup_{0\leq t\leq T}\|w_\varepsilon(\cdot,t)\|_{L^2(\Omega)}\big)^{p}\big\rangle^{\frac{1}{p}}
\lesssim I_1
+ \langle I_2^p\rangle^{\frac{1}{p}} + \langle I_3^p\rangle^{\frac{1}{p}}.
\end{aligned}
\end{equation}
The remainder of the proof concerning $I_1$, $I_2$, and $I_3$ is carried out in the following two steps.

\textbf{Step 2.} Arguments for $I_1$. The following fact is noted:
\begin{equation}\label{f:3.5}
\begin{aligned}
\nabla u_0
&- S_\varepsilon K_\varepsilon(\Psi_{\varepsilon}\nabla u_0)
= \Psi_{\varepsilon}\nabla u_0
-  S_\varepsilon K_\varepsilon(\Psi_{\varepsilon}\nabla u_0)
+ (1 -\Psi_{\varepsilon})\nabla u_0\\
&= \Psi_{\varepsilon}\nabla u_0
-  K_\varepsilon(\Psi_{\varepsilon}\nabla u_0)
+ K_\varepsilon(\Psi_{\varepsilon}\nabla u_0)
-S_\varepsilon K_\varepsilon(\Psi_{\varepsilon}\nabla u_0)
+ (1 -\Psi_{\varepsilon})\nabla u_0,
\end{aligned}
\end{equation}
and it therefore follows from Lemma $\ref{lemma:2.4}$ that
\begin{equation}\label{f:3.2}
\begin{aligned}
I_1&\leq \Bigg\{\|
\Psi_{\varepsilon}\nabla u_0
-  SK_\varepsilon(\Psi_{\varepsilon}\nabla u_0)\|_{L^2(\Omega_T)}
+\|(1-\Psi_{\varepsilon})\nabla u_0\|_{L^2(\Omega_T)}\Bigg\}\\
&\lesssim^{\eqref{f:3.5}}_{\eqref{pri:2.13},\eqref{pri:2.14}} \varepsilon
\bigg\{\|\nabla(\Psi_{\varepsilon}\nabla u_0)\|_{L^2(\Omega_T)}
+ \varepsilon\|\partial_tK_\varepsilon(\Psi_{\varepsilon}\nabla u_0)\|_{L^2(\Omega_T)}
\bigg\}
+\|\nabla u_0\|_{L^2(\boxbox_{4\varepsilon})}\\
&\lesssim^{\eqref{pri:2.17}}
\|\nabla u_0\|_{L^2(\boxbox_{4\varepsilon})}
+ \varepsilon\Big\{\|\nabla^2 u_0\|_{L^2((\Omega\setminus O_{2\varepsilon})_T)}
+\|\partial_t u_0\|_{L^2((\Omega\setminus O_{2\varepsilon})_T)}\Big\},
\end{aligned}
\end{equation}
where it is recalled that $\boxbox_{\varepsilon}:=\Omega_T\setminus \Sigma_{\varepsilon}^T$
and $(\Omega\setminus O_\varepsilon)_T:=
(\Omega\setminus O_\varepsilon)\times [0,T]$, and
the following identity is also used:
\begin{equation}\label{f:3.28}
\begin{aligned}
\partial_t K_\varepsilon(\Psi_{\varepsilon}\nabla u_0)
= K_\varepsilon(\partial_t\Psi_{\varepsilon}\nabla u_0)
+ \nabla K_\varepsilon(\Psi_{\varepsilon}\partial_t u_0)
- K_\varepsilon(\nabla\Psi_{\varepsilon}\partial_t u_0)
\end{aligned}
\end{equation}
in the last step.

\textbf{Step 3.} Arguments for $I_2$ and $I_3$. By appealing to
Lemma $\ref{lemma:5}$ (with $\delta=1$) and to
Theorem $\ref{thm:1.0}$, it is obtained that
\begin{equation}\label{f:3.3}
\begin{aligned}
\langle I_2^p\rangle^{\frac{1}{p}}
&\lesssim^{\eqref{pri:3.42c}}_{\eqref{pri:5.2},\eqref{pri:5.4}}
\varepsilon\bigg(
\int_{\Omega_T}dz|K_\varepsilon(\nabla(\Psi_{\varepsilon}\nabla u_0))(z)|^2
\bigg)^{\frac{1}{2}}\\
&
\lesssim^{\eqref{pri:2.17}}
\|\nabla u_0\|_{L^2(\boxbox_{4\varepsilon})}
+\varepsilon\|\nabla^2 u_0\|_{L^2((\Omega\setminus O_{2\varepsilon})_T)}.
\end{aligned}
\end{equation}

To estimate $I_3$, begin by observing that
\begin{equation}\label{f:4.19}
\begin{aligned}
&\nabla^2S_\varepsilon K_\varepsilon(\Psi_{\varepsilon}\nabla u_0)
= \nabla S_\varepsilon K_\varepsilon(\nabla(\Psi_{\varepsilon}\nabla u_0));\\
&\partial_t S_\varepsilon K_\varepsilon(\Psi_{\varepsilon}\nabla u_0)
= S_\varepsilon K_\varepsilon(\partial_t\Psi_{\varepsilon}\nabla u_0)
+ \nabla S_\varepsilon K_\varepsilon(\Psi_{\varepsilon}\partial_t u_0)
- S_\varepsilon K_\varepsilon(\nabla\Psi_{\varepsilon}\partial_t u_0).
\end{aligned}
\end{equation}
Then, it follows from Lemmas $\ref{lemma:5}$, $\ref{lemma:2.5}$, and Theorem $\ref{thm:1.0}$ that
\begin{equation}\label{f:3.4}
\begin{aligned}
\langle I_3^p\rangle^{\frac{1}{p}}
&\lesssim_{\eqref{pri:5.6a},\eqref{pri:5.7c}}^{\eqref{pri:3.42c},\eqref{pri:3.42d}}
\varepsilon\bigg(\int_{\Omega_T}
\bar{\mu}_d^2(\cdot/\varepsilon)
\Big|K_\varepsilon\big(|\nabla(\Psi_{\varepsilon}\nabla u_0)|
+|\Psi_{\varepsilon}\partial_t u_0|
+|\varepsilon(\partial_t\Psi_{\varepsilon}\nabla u_0)|
+|\varepsilon\nabla\Psi_{\varepsilon}\partial_t u_0|\big)\Big|^2
\bigg)^{\frac{1}{2}}\\
&\lesssim^{\eqref{pri:2.8}} \mu_d(1/\varepsilon)
\bigg(
\|\bar{\mu}_d\nabla u_0\|_{L^2(\boxbox_{4\varepsilon})}
+\varepsilon\|\bar{\mu}_d\nabla^2 u_0\|_{L^2((\Omega\setminus O_{2\varepsilon})_T)}
+\varepsilon\|\bar{\mu}_d\partial_t u_0\|_{L^2((\Omega\setminus O_{2\varepsilon})_T)}
\bigg),
\end{aligned}
\end{equation}
where the following fact is also employed:
\begin{equation}\label{f:4.28}
 \bar{\mu}_d(\cdot/\varepsilon)
 \leq^{\eqref{pri:5.5}} \mu_d(1/\varepsilon)\bar{\mu}_d(\cdot)
\end{equation}
in the last inequality.

Consequently, by substituting $\eqref{f:3.2}$, $\eqref{f:3.3}$, and $\eqref{f:3.4}$
back into $\eqref{f:3.1}$, the desired result is obtained, and
the proof is hereby completed.
\end{proof}

In order to accelerate the convergence rate, the so-called duality methods is employed. To this end, first consider the adjoint initial-Dirichlet problems: given $\Phi\in L^2(\Omega_T)$, let
$v_\varepsilon$ and $v_0$ be the weak solutions to
\begin{equation}\label{pde:3.1}
(\text{DP}_\varepsilon^*)\left\{\begin{aligned}
-\partial_t v_\varepsilon  + \mathcal{L}_\varepsilon^*(v_\varepsilon)
& = \Phi &~&\text{in}~ ~\Omega_T;\\
v_\varepsilon & = 0 &~&\text{on}~ \partial_{\shortparallel}\Omega_T;\\
v_\varepsilon(\cdot,T) & = 0 &~&\text{on}~ \Omega,
\end{aligned}\right.
\quad
(\text{DP}_0^*)\left\{\begin{aligned}
-\partial_t v_0 + \mathcal{L}_0^*(v_0)
& = \Phi &~&\text{in}~ ~\Omega_T;\\
v_0 & = 0 &~&\text{on}~ \partial_{\shortparallel}\Omega_T;\\
v_0(\cdot,T) & = 0 &~&\text{on}~ \Omega,
\end{aligned}\right.
\end{equation}
respectively. Here $\mathcal{L}_\varepsilon^*$ is the adjoint operator of $\mathcal{L}_\varepsilon$.
Let
$\tilde{v}_\varepsilon(x,t)=v_\varepsilon(x,T-t)$ and
$\tilde{v}_0(x,t)=v_0(x,T-t)$; these functions precisely solve the following initial-Dirichlet problems:
\begin{equation*}
\left\{\begin{aligned}
\frac{\partial \tilde{v}_\varepsilon^\alpha}{\partial t}
-\frac{\partial}{\partial x_i}\Big\{a_{ji}\Big(\frac{x}{\varepsilon},
\frac{T-t}{\varepsilon^2}\Big)\frac{\partial \tilde{v}_\varepsilon^\beta}{\partial x_j}\Big\}
& = \tilde{\Phi} &~&\text{in}~ ~\Omega_T;\\
\tilde{v}_\varepsilon & = 0 &~&\text{on}~ \partial_p\Omega_T,
\end{aligned}\right.
\quad
\text{and}
\quad
\left\{\begin{aligned}
\frac{\partial \tilde{v}_0}{\partial t} + \mathcal{L}_0(\tilde{v}_0)
& = \tilde{\Phi} &~&\text{in}~ ~\Omega_T;\\
\tilde{v}_0 & = 0 &~&\text{on}~ \partial_p\Omega_T,
\end{aligned}\right.
\end{equation*}
where $\tilde{\Phi}(x,t):=\Phi(x,T-t)$.
It is noted that $\phi^{*}$ and
$\sigma^{*}$ are the respective correctors and flux correctors
associated with $a^*(x,T-t)$. Therefore,
the same result as stated in Lemma $\ref{lemma:3.1}$ can be derived for
the adjoint problems in $\eqref{pde:3.1}$, because
the conclusions concerning
$\phi^{*}$ and $\sigma^{*}$ are consistent with those of Theorem $\ref{thm:1.0}$.

\begin{lemma}[Duality lemma]\label{lemma:4.3}
Let $\Omega\subset\mathbb{R}^d$ be a bounded Lipschitz domain with
$d\geq 2$, $\varepsilon\in(0,1]$, and $T>0$.
Suppose that the ensemble $\langle\cdot\rangle$ is stationary
with respect to $\eqref{c:1}$ and satisfies $\eqref{c:3}$.
Let $u_\varepsilon,u_0$ be the weak solutions of $(\emph{PD}_\varepsilon)$ and
$(\emph{PD}_0)$, respectively, in $\eqref{pde:3.1*}$.
Given any $\Phi\in L^2(\Omega_T)$,
let
$v_\varepsilon$ and $v_0$ be the weak solutions to
the adjoint problems $(\emph{PD}_\varepsilon^*)$ and $(\emph{PD}_0^*)$, respectively, in $\eqref{pde:3.1}$. Then,
by choosing $\varphi_j =
S_\varepsilon(\Psi_{2\varepsilon}\nabla_ju_0)$
in $\eqref{eq:3.1}$ and setting $\pi_{ilj}=0$ in $\eqref{eq:3.5}$, it holds that
\begin{equation}\label{eq:4.1}
\int_{\Omega_T} w_\varepsilon \Phi dx dt
= -\int_{\Omega_T} \tilde{f}\cdot\nabla v_\varepsilon dx dt,
\end{equation}
where $\tilde{f}$ is given in $\eqref{eq:3.3}$. Moreover, if it is assumed that
\begin{equation}\label{eq:4.2}
\begin{aligned}
\breve{w}_\varepsilon =
\tilde{v}_\varepsilon - \tilde{v}_0
-\varepsilon (\phi^{*}_j)_T^{\varepsilon}\tilde{\varphi}_j
-\varepsilon^2 (\sigma_{l(d+1)j}^{*})_{T}^{\varepsilon}
\partial_l\tilde{\varphi}_j,
\end{aligned}
\end{equation}
where
\begin{equation*}
\begin{aligned}
&\tilde{\varphi}_j:= S_\varepsilon K_\varepsilon(
\Psi_{4\varepsilon}\partial_j\tilde{v}_0);\\
&(\phi^{*}_{j})_T^{\varepsilon}(x,t)
:=\phi^{*}_j(x/\varepsilon,(T-t)/\varepsilon^2)
-\alpha;\\
&(\sigma^{*}_{l(d+1)j})^{\varepsilon}_{T}(x,t)
:=\sigma_{l(d+1)j}^{*}(x/\varepsilon,(T-t)/\varepsilon^2)
-\gamma
\quad\text{with}~ \alpha,\gamma\in\mathbb{R}.
\end{aligned}
\end{equation*}
Then, for any $p\in[1,\infty)$, there holds
\begin{equation}\label{pri:4.1}
\begin{aligned}
\bigg\langle\Big|\int_{\Omega_T} w_\varepsilon
\Phi\Big|^p\bigg\rangle^{\frac{1}{p}}
&\lesssim \mu_d^2\big(\frac{R_0}{\varepsilon}\big)
\Big\{
\|\nabla u_0\|_{L^2(\boxbox_{8\varepsilon})}
+\varepsilon\|\nabla^2 u_0\|_{L^2((\Omega\setminus O_{4\varepsilon})_T)}
+\varepsilon\|\partial_t u_0\|_{L^2((\Omega\setminus O_{4\varepsilon})_T)}
\Big\}\\
&\qquad\times\bigg\{
\|\nabla v_0\|_{L^2(\boxbox_{16\varepsilon})}
 + \varepsilon\|\nabla^2 v_0 \|_{L^2((\Omega\setminus O_{8\varepsilon})_T)}
+\varepsilon\|\partial_t v_0 \|_{L^2((\Omega\setminus O_{8\varepsilon})_T)}\bigg\}\\
&+\mu_d(\frac{R_0}{\varepsilon})
\bigg\{
\|\delta^{\frac{1}{2}}\nabla u_0\|_{L^2(\boxbox_{8\varepsilon})}
+\varepsilon\|\delta^{\frac{1}{2}}\nabla^2 u_0\|_{L^2((\Omega\setminus O_{4\varepsilon})_T)}\\
&\qquad\qquad\qquad\qquad\qquad\qquad
+\varepsilon\|\delta^{\frac{1}{2}}\partial_t u_0\|_{L^2((\Omega\setminus O_{4\varepsilon})_T)}
\bigg\}\|\delta^{-\frac{1}{2}}\nabla \tilde{v}_0\|_{L^2(\Sigma_{2\varepsilon}^T)},
\end{aligned}
\end{equation}
where the weight $\delta$ is defined in
$\eqref{eq:2.4}$, and the multiplicative constant depends on
$\mu,d,p$, and $\lambda_1$.
\end{lemma}

\begin{proof}
The basic idea can be found in \cite{Xu-16,Xu-Zhou-17}, originally inspired by \cite{Suslina-13,Kenig-Lin-Shen-12}.
First, it is not hard to see that the equality $\eqref{eq:4.1}$ follows from integration by parts
\begin{equation*}
\begin{aligned}
\int_{\Omega_T}w_\varepsilon \Phi dxdt
&= \int_0^T \big<w_\varepsilon,(-\partial_t+\mathcal{L}_\varepsilon^*)v_\varepsilon\big>dt \\
&= \int_0^T \big<(\partial_t+\mathcal{L}_\varepsilon)w_\varepsilon,v_\varepsilon\big>dt
+ \int_\Omega w_\varepsilon(x,T)v_\varepsilon(x,T)dx
- \int_\Omega w_\varepsilon(x,0)v_\varepsilon(x,0)dx \\
&= - \int_{\Omega_T}\tilde{f}\cdot\nabla v_\varepsilon dxdt,
\end{aligned}
\end{equation*}
where $w_\varepsilon$ and $v_\varepsilon$ are weak solutions of $\eqref{eq:3.2}$ and $(\text{DP}_\varepsilon^*)$, respectively, and the initial-boundary conditions
$w_\varepsilon = v_\varepsilon = 0$ on $\partial_{\shortparallel}\Omega_T$
are employed in the second step, while
$w_\varepsilon(x,0) = v_\varepsilon(x,T) = 0$ is used in the last step.

\textbf{Step 1.} Outline of the proof of $\eqref{pri:4.1}$ and the main idea. Let $\varpi$ and $\tilde{\varpi}$
be given as in the proof of Lemma $\ref{lemma:3.1}$.
Thus, by virtue of $\eqref{eq:4.1}$ and $\eqref{eq:3.3}$, it follows that
\begin{equation*}
\begin{aligned}
\Big|\int_{\Omega_T}w_\varepsilon \Phi \Big|
&\leq \int_{\Omega_T}
\big|\nabla u_0 - S_\varepsilon K_\varepsilon(\Psi_{2\varepsilon}\nabla u_0)\big|
\big|\nabla v_\varepsilon\big| \\
& +\varepsilon\int_{\Omega_T}
\big|\varpi^\varepsilon
\nabla S_\varepsilon K_\varepsilon(\Psi_{2\varepsilon}\nabla u_0)\big|
\big|\nabla v_\varepsilon\big| \\
&\quad +\varepsilon^2
\int_{\Omega_T}
\big|\tilde{\varpi}^\varepsilon
(\nabla^2 + \partial_t)S_\varepsilon K_\varepsilon(\Psi_{2\varepsilon}\nabla u_0)\big|
\big|\nabla v_\varepsilon\big|
=: I_1 + I_2 + I_3,
\end{aligned}
\end{equation*}
where it is recalled that
$\varpi^\varepsilon:=\varpi(\cdot/\varepsilon,\cdot/\varepsilon^2)$
and $\tilde{\varpi}^\varepsilon
:=\tilde{\varpi}(\cdot/\varepsilon,\cdot/\varepsilon^2)$.
Therefore, the desired estimate $\eqref{pri:4.1}$ follows from
estimating the right-hand side of the following inequality:
\begin{equation}\label{f:4.20}
\Big\langle\Big|\int_{\Omega_T}w_\varepsilon \Phi \Big|^p\Big\rangle^{\frac{1}{p}}
\leq \langle I_{1}^p\rangle^{\frac{1}{p}}
+ \langle I_{2}^p\rangle^{\frac{1}{p}}
+ \langle I_{3}^p\rangle^{\frac{1}{p}}.
\end{equation}

Before proceeding further, we plan to show the main ideas on accelerating the convergence rates. The key step is to replace $\nabla v_\varepsilon$ by
\begin{equation}\label{f:4.1**}
\begin{aligned}
&\underbrace{\big[\nabla \breve{w}_\varepsilon
+\varepsilon(\phi^{*}_j)_{T}^{\varepsilon}\nabla\tilde{\varphi}_j
 +\varepsilon (\nabla\sigma_{l(d+1)j}^{*})_{T}^{\varepsilon}
\partial_l\tilde{\varphi}_j
+ \varepsilon^2 (\sigma_{l(d+1)j}^{*})_{T}^{\varepsilon}
\nabla\partial_l\tilde{\varphi}_j\big](\cdot,T-\cdot)
}_{\mathcal{R}_{\varepsilon,1}}\\
&+ \underbrace{\big[\nabla \tilde{v}_0
+(\nabla\phi^{*}_j)_{T}^{\varepsilon}
\tilde{\varphi}_j\big](\cdot,T-\cdot)}_{\mathcal{R}_{\varepsilon,2}},
\end{aligned}
\end{equation}
where $\breve{w}$ is given by $\eqref{eq:4.2}$. It follows from Lemma  $\ref{lemma:3.1}$ (in the case where $\Omega$ is bounded) that
\begin{equation}\label{f:4.12}
\begin{aligned}
\bigg\langle\Big(\int_{\Omega_T}|\nabla \breve{w}_\varepsilon|^2\Big)^{\frac{p}{2}}\bigg\rangle^{\frac{1}{p}}
&\lesssim \mu_d\big(\frac{R_0}{\varepsilon}\big)\bigg\{
\|\nabla v_0\|_{L^2(\boxbox_{16\varepsilon})} \\
&\qquad\qquad + \varepsilon\Big(
\|\nabla^2 v_0 \|_{L^2((\Omega\setminus O_{8\varepsilon})_T)}
+\|\partial_t v_0 \|_{L^2((\Omega\setminus O_{8\varepsilon})_T)}\Big)\bigg\},
\end{aligned}
\end{equation}
where we recall $R_0:=\text{diam}(\Omega)$.
By repeatedly using the same arguments as those given for $I_2$ and $I_3$ in the proof of Lemma $\ref{lemma:3.1}$, the term
$\mathcal{R}_{\varepsilon,1}-\nabla \breve{w}_\varepsilon$ can be bounded by the same right-hand side above. Therefore,
the triangle inequality yields
\begin{equation}\label{pri:4.5}
\begin{aligned}
\bigg\langle\Big(\int_{\Omega_T}|\mathcal{R}_{\varepsilon,1}|^2\Big)^{\frac{p}{2}}\bigg\rangle^{\frac{1}{p}}
&\lesssim \mu_d\big(\frac{R_0}{\varepsilon}\big)\bigg\{
\|\nabla v_0\|_{L^2(\boxbox_{16\varepsilon})} \\
&\qquad\qquad + \varepsilon\Big(
\|\nabla^2 v_0 \|_{L^2((\Omega\setminus O_{8\varepsilon})_T)}
+\|\partial_t v_0 \|_{L^2((\Omega\setminus O_{8\varepsilon})_T)}\Big)\bigg\},
\end{aligned}
\end{equation}
which consequently produces $O(\mu_d(R_0/\varepsilon)\varepsilon^{1/2})$.
For $\mathcal{R}_{\varepsilon,2}$, a weighted estimate is appealed to
accelerating the convergence rate, and it is thereby obtained that
\begin{equation}\label{pri:4.6}
\begin{aligned}
\bigg\langle\Big(\int_{\Sigma_{2\varepsilon}^T}dz
&|\mathcal{R}_{\varepsilon,2}(z)|^2 \delta^{-1}(z)
\Big)^{\frac{p}{2}}\bigg\rangle^{\frac{1}{p}}\\
&\lesssim \|\delta^{-\frac{1}{2}}\nabla v_0\|_{L^2(\Sigma_{2\varepsilon}^T)}
+ \Big\langle\Big(\int_{\Omega_T}dz|(\nabla\phi^{*}_j)^{\varepsilon}(z)
(\tilde{\varphi}_j)_T(z)|^2 \delta^{-1}(z)
\Big)^{\frac{p}{2}}\Big\rangle^{\frac{1}{p}}\\
&\lesssim^{\eqref{pri:3.42c},\eqref{pri:5.2}}
\|\delta^{-\frac{1}{2}}\nabla \tilde{v}_0\|_{L^2(\Sigma_{2\varepsilon}^T)},
\end{aligned}
\end{equation}
which finally yields $O(\ln(1/\varepsilon))$; however,
its counterpart estimate will benefit from a weight factor $\delta^{1/2}$ later, where $\delta$ is defined in $\eqref{eq:2.4}$,
and $(\tilde{\varphi}_j)_T(z):=\tilde{\varphi}_j(x,T-t)$.


\textbf{Step 2.} Arguments for $I_1$. This is
divided into two parts:
\begin{equation*}
\begin{aligned}
\underbrace{\int_{\Omega_T}
\big|\Psi_{2\varepsilon}
\nabla u_0 - S_\varepsilon K_\varepsilon(\Psi_{2\varepsilon}\nabla u_0)\big|
\big|\nabla v_\varepsilon\big|}_{I_{11}}
\text{\quad and\quad}\underbrace{\int_{\Omega_T}
(1-\Psi_{2\varepsilon})\big|\nabla u_0\big|
\big|\nabla v_\varepsilon\big|}_{I_{12}}.
\end{aligned}
\end{equation*}
We first handle $I_{12}$ as follows:
\begin{equation*}\label{f:4.3}
\begin{aligned}
I_{12}
\lesssim \|\nabla u_0\|_{L^2(\boxbox_{8\varepsilon})}
\|\nabla v_\varepsilon\|_{L^2(\boxbox_{8\varepsilon})}
\lesssim \|\nabla u_0\|_{L^2(\boxbox_{8\varepsilon})}
\Big\{
\|\nabla \breve{w}_\varepsilon\|_{L^2(\Omega_T)}
+\|\nabla v_0\|_{L^2(\boxbox_{8\varepsilon})}
\Big\},
\end{aligned}
\end{equation*}
where we replace $\nabla v_\varepsilon$ with $\eqref{f:4.1**}$ in the last step, and use
the fact that $\text{supp}(\tilde{\varphi})\subset\Sigma_{4\varepsilon}^T$.
This further leads to
\begin{equation}\label{f:4.13}
\begin{aligned}
&\langle I_{12}^p\rangle^{\frac{1}{p}}
\lesssim \|\nabla u_0\|_{L^2(\boxbox_{8\varepsilon})}
\Big\{
\big\langle\|\nabla \breve{w}_\varepsilon\|_{L^2(\Omega_T)}^p\big\rangle^{\frac{1}{p}}
+\|\nabla v_0\|_{L^2(\boxbox_{8\varepsilon})}
\Big\}\\
&\lesssim^{\eqref{f:4.12}}
\mu_d\big(\frac{R_0}{\varepsilon}\big)
\|\nabla u_0\|_{L^2(\boxbox_{8\varepsilon})}\bigg\{
\|\nabla v_0\|_{L^2(\boxbox_{16\varepsilon})}
 + \varepsilon\Big(
\|\nabla^2 v_0 \|_{L^2((\Omega\setminus O_{8\varepsilon})_T)}
+\|\partial_t v_0 \|_{L^2((\Omega\setminus O_{8\varepsilon})_T)}\Big)\bigg\}.
\end{aligned}
\end{equation}
Attention is then turned to $I_{11}$. It can be further
decomposed into two parts:
\begin{equation*}
\begin{aligned}
I_{11} &\leq \int_{\Omega_T}
\big|\Psi_{2\varepsilon}
\nabla u_0 - S_\varepsilon K_\varepsilon(\Psi_{2\varepsilon}\nabla u_0)\big|
\big|\mathcal{R}_{\varepsilon,1}+\mathcal{R}_{\varepsilon,2}\big|\\
&\qquad\leq \Big(\int_{\Omega_T}\big|\Psi_{2\varepsilon}
\nabla u_0 - S_\varepsilon K_\varepsilon(\Psi_{2\varepsilon}\nabla u_0)\big|^2\Big)^{\frac{1}{2}}\|\mathcal{R}_{\varepsilon,1}\|_{L^2(\Omega_T)}\\
&\qquad\qquad\qquad+ \Big(\int_{\Omega_T}\big|\Psi_{2\varepsilon}
\nabla u_0 - S_\varepsilon K_\varepsilon(\Psi_{2\varepsilon}\nabla u_0)\big|^2\delta\Big)^{\frac{1}{2}}
\|\delta^{\frac{1}{2}}\mathcal{R}_{\varepsilon,1}\|_{L^2(\Sigma_{2\varepsilon}^T)}
:= I_{111}+I_{112}.
\end{aligned}
\end{equation*}
Taking the $p$-moment estimates into account, one can derive that
\begin{equation}\label{f:4.14}
\begin{aligned}
\langle I_{111}^p\rangle^{\frac{1}{p}}
&\leq \Big(\int_{\Omega_T}\big|\Psi_{2\varepsilon}
\nabla u_0 - S_\varepsilon K_\varepsilon(\Psi_{2\varepsilon}\nabla u_0)\big|^2\Big)^{\frac{1}{2}}
\Big\langle\Big(\int_{\Omega_T}
\big|\mathcal{R}_{\varepsilon,1}\big|^{2}\Big)^{\frac{p}{2}}
\Big\rangle^{\frac{1}{p}} \\
&\lesssim^{\eqref{f:3.2},\eqref{pri:4.5}}
\mu_d\big(\frac{R_0}{\varepsilon}\big)\Big\{\|\nabla u_0\|_{L^2(\boxbox_{8\varepsilon})}
+ \varepsilon\big(\|\nabla^2 u_0\|_{L^2((\Omega\setminus O_{4\varepsilon})_T)}
+\|\partial_t u_0\|_{L^2((\Omega\setminus O_{4\varepsilon})_T)}\big)\Big\}\\
&\qquad\quad \times
\Big\{
\|\nabla v_0\|_{L^2(\boxbox_{16\varepsilon})}
 + \varepsilon\big(
\|\nabla^2 v_0 \|_{L^2((\Omega\setminus O_{8\varepsilon})_T)}
+\|\partial_t v_0 \|_{L^2((\Omega\setminus O_{8\varepsilon})_T)}\big)\Big\}.
\end{aligned}
\end{equation}
Furthermore, a similar computation as that given for $\eqref{f:3.2}$ (where the original proof is merely modified to accommodate the weighted estimates) yields
\begin{equation}\label{f:4.15}
\begin{aligned}
\langle I_{112}^p\rangle^{\frac{1}{p}}
&\leq \Big(\int_{\Omega_T}\big|\Psi_{2\varepsilon}
\nabla u_0 - S_\varepsilon K_\varepsilon(\Psi_{2\varepsilon}\nabla u_0)\big|^2\delta\Big)^{\frac{1}{2}}
\Big\langle\Big(\int_{\Sigma_{2\varepsilon}^T}
\big|\mathcal{R}_{\varepsilon,2}\big|^{2}\delta^{-1}\Big)^{\frac{p}{2}}
\Big\rangle^{\frac{1}{p}} \\
&\lesssim^{\eqref{pri:2.10},\eqref{pri:2.12}}_{\eqref{pri:4.6}}
\Big\{\|\delta^{\frac{1}{2}}\nabla u_0\|_{L^2(\boxbox_{8\varepsilon})}
+ \varepsilon\big(\|\delta^{\frac{1}{2}}\nabla^2 u_0\|_{L^2((\Omega\setminus O_{4\varepsilon})_T)}
+\|\delta^{\frac{1}{2}}\partial_t u_0\|_{L^2((\Omega\setminus O_{4\varepsilon})_T)}\big)\Big\}\\
&\qquad\quad \times
\|\delta^{-\frac{1}{2}}\nabla \tilde{v}_0\|_{L^2(\Sigma_{2\varepsilon}^T)}.
\end{aligned}
\end{equation}
Combining the estimates $\eqref{f:4.13}$, $\eqref{f:4.14}$,
and $\eqref{f:4.15}$, one can acquire that
\begin{equation}\label{f:4.2}
\begin{aligned}
\langle I_{1}^p\rangle^{\frac{1}{p}}
&\leq \langle I_{11}^p\rangle^{\frac{1}{p}}
+\langle I_{12}^p\rangle^{\frac{1}{p}}
\leq \langle I_{11}^p\rangle^{\frac{1}{p}}
+\langle I_{111}^p\rangle^{\frac{1}{p}}
+\langle I_{112}^p\rangle^{\frac{1}{p}}\\
&\lesssim
\mu_d\big(\frac{R_0}{\varepsilon}\big)\Big\{\|\nabla u_0\|_{L^2(\boxbox_{8\varepsilon})}
+ \varepsilon\big(\|\nabla^2 u_0\|_{L^2((\Omega\setminus O_{4\varepsilon})_T)}
+\|\partial_t u_0\|_{L^2((\Omega\setminus O_{4\varepsilon})_T)}\big)\Big\}\\
&\qquad\qquad\qquad\qquad \times
\Big\{
\|\nabla v_0\|_{L^2(\boxbox_{16\varepsilon})}
 + \varepsilon\big(
\|\nabla^2 v_0 \|_{L^2((\Omega\setminus O_{8\varepsilon})_T)}
+\|\partial_t v_0 \|_{L^2((\Omega\setminus O_{8\varepsilon})_T)}\big)\Big\}\\
& +
\Big\{\|\delta^{\frac{1}{2}}\nabla u_0\|_{L^2(\boxbox_{8\varepsilon}^T)}
+ \varepsilon\big(\|\delta^{\frac{1}{2}}\nabla^2 u_0\|_{L^2((\Omega\setminus O_{4\varepsilon})_T)}
+\|\delta^{\frac{1}{2}}\partial_t u_0\|_{L^2((\Omega\setminus O_{4\varepsilon})_T)}\big)\Big\}\\
&\qquad\qquad\qquad\qquad \times
\|\delta^{-\frac{1}{2}}\nabla \tilde{v}_0\|_{L^2(\Sigma_{2\varepsilon}^T)}.
\end{aligned}
\end{equation}

\textbf{Step 3.}
Arguments for $I_2$.
Proceeding as in the proof for $I_1$, it is first found that
\begin{equation*}
\begin{aligned}
I_2&\leq \varepsilon\int_{\Omega_T}
\big|\varpi^\varepsilon
\nabla S_\varepsilon K_\varepsilon(\Psi_{2\varepsilon}\nabla u_0)\big|
\big|\mathcal{R}_{\varepsilon,1}+\mathcal{R}_{\varepsilon,2}\big| \\
&\qquad\qquad\leq \varepsilon\bigg\{
\big\|\varpi^\varepsilon\nabla S_\varepsilon K_\varepsilon(\Psi_{2\varepsilon}\nabla u_0)\big\|_{L^2(\Omega_T)}\big\|\mathcal{R}_{\varepsilon,1}
\big\|_{L^2(\Omega_T)}\\
&\qquad\qquad\qquad\qquad+ \big\|\delta^{\frac{1}{2}}\varpi^\varepsilon\nabla S_\varepsilon K_\varepsilon(\Psi_{2\varepsilon
}\nabla u_0)
\big\|_{L^2(\Omega_T)}
\big\|\delta^{-\frac{1}{2}}\mathcal{R}_{\varepsilon,2}
\big\|_{L^2(\Sigma_{2\varepsilon}^T)}
\bigg\}=:I_{21} + I_{22}.
\end{aligned}
\end{equation*}
Let $1/p = 1/p_1 + 1/p_2$ with $p_1,p_2\geq 1$.
By taking $\langle|\cdot|^p\rangle^{\frac{1}{p}}$ on both sides above,
it can be derived that
\begin{equation}\label{f:4.16}
\begin{aligned}
\langle |I_{21}|^p\rangle^{\frac{1}{p}}
&\leq \varepsilon
\big\langle\|\varpi^\varepsilon\nabla S_\varepsilon K_\varepsilon(\Psi_{2\varepsilon}\nabla u_0)\|_{L^2(\Omega_T)}^{p_1}\big\rangle^{\frac{1}{p_1}}
\big\langle\|\mathcal{R}_{\varepsilon,1}
\|_{L^2(\Omega_T)}^{p_2}\rangle^{\frac{1}{p_2}}\\
&\lesssim^{\eqref{f:3.3},\eqref{pri:4.5}}
\mu_d\big(\frac{R_0}{\varepsilon}\big)
\Big\{
\|\nabla u_0\|_{L^2(\boxbox_{8\varepsilon})}
+\varepsilon\|\nabla^2 u_0\|_{L^2((\Omega\setminus O_{4\varepsilon})_T)}
\Big\}\\
&\times
\bigg\{
\|\nabla v_0\|_{L^2(\boxbox_{16\varepsilon})}
 + \varepsilon\Big(
\|\nabla^2 v_0 \|_{L^2((\Omega\setminus O_{8\varepsilon})_T)}
+\|\partial_t v_0 \|_{L^2((\Omega\setminus O_{8\varepsilon})_T)}\Big)\bigg\}.
\end{aligned}
\end{equation}
By the same token, it is found that
\begin{equation}\label{f:4.17}
\begin{aligned}
\langle |I_{22}|^p\rangle^{\frac{1}{p}}
&\leq \varepsilon
\big\langle\|\delta^{\frac{1}{2}}\varpi^\varepsilon\nabla S_\varepsilon K_\varepsilon(\Psi_{2\varepsilon}\nabla u_0)\|_{L^2(\Omega_T)}^{p_1}\big\rangle^{\frac{1}{p_1}}
\big\langle\|\delta^{-\frac{1}{2}}\mathcal{R}_{\varepsilon,2}
\|_{L^2(\Sigma_{2\varepsilon}^{T})}^{p_2}\rangle^{\frac{1}{p_2}}\\
&\lesssim^{\eqref{pri:3.42c},\eqref{pri:4.6}}
\Big\{
\|\delta^{\frac{1}{2}}\nabla u_0\|_{L^2(\boxbox_{8\varepsilon})}
+\varepsilon\|\delta^{\frac{1}{2}}\nabla^2 u_0\|_{L^2((\Omega\setminus O_{4\varepsilon})_T)}
\Big\}\|\delta^{-\frac{1}{2}}\nabla \tilde{v}_0\|_{L^2(\Sigma_{2\varepsilon}^T)}.
\end{aligned}
\end{equation}
Thus, it follows that
\begin{equation}\label{f:4.21}
\begin{aligned}
\langle |I_{2}|^p\rangle^{\frac{1}{p}}
&\leq \langle |I_{21}|^p\rangle^{\frac{1}{p}} +
\langle |I_{22}|^p\rangle^{\frac{1}{p}} \\
&\lesssim^{\eqref{f:4.16},\eqref{f:4.17}}
\mu_d\big(\frac{R_0}{\varepsilon}\big)
\Big\{
\|\nabla u_0\|_{L^2(\boxbox_{8\varepsilon})}
+\varepsilon\|\nabla^2 u_0\|_{L^2((\Omega\setminus O_{4\varepsilon})_T)}
\Big\}\\
&\qquad\times
\Big\{
\|\nabla v_0\|_{L^2(\boxbox_{16\varepsilon})}
 + \varepsilon
\|\nabla^2 v_0 \|_{L^2((\Omega\setminus O_{8\varepsilon})_T)}
+\varepsilon\|\partial_t v_0 \|_{L^2((\Omega\setminus O_{8\varepsilon})_T)}\Big\}\\
&\qquad\qquad\qquad+\Big\{
\|\delta^{\frac{1}{2}}\nabla u_0\|_{L^2(\boxbox_{8\varepsilon})}
+\varepsilon\|\delta^{\frac{1}{2}}\nabla^2 u_0\|_{L^2((\Omega\setminus O_{4\varepsilon})_T)}
\Big\}\|\delta^{-\frac{1}{2}}\nabla \tilde{v}_0\|_{L^2(\Sigma_{2\varepsilon}^T)}.
\end{aligned}
\end{equation}

\textbf{Step 4.} Arguments for $I_3$.
By an argument analogous to that given for $I_2$, it suffices to estimate
\begin{equation*}
\begin{aligned}
I_3&\leq \varepsilon^2\int_{\Omega_T}
\big|\tilde{\varpi}^\varepsilon
(\nabla^2 + \partial_t)S_\varepsilon K_\varepsilon(\Psi_{2\varepsilon}\nabla u_0)\big|
\big|\mathcal{R}_{\varepsilon,1}+\mathcal{R}_{\varepsilon,2}\big| \\
&\qquad\leq \varepsilon^2\bigg\{
\big\|\tilde{\varpi}^\varepsilon(\nabla^2 + \partial_t) S_\varepsilon K_\varepsilon(\Psi_{2\varepsilon}\nabla u_0)\big\|_{L^2(\Omega_T)}\big\|\mathcal{R}_{\varepsilon,1}
\big\|_{L^2(\Omega_T)}\\
&\qquad\qquad+ \big\|\delta^{\frac{1}{2}}\tilde{\varpi}^\varepsilon
(\nabla^2 + \partial_t)S_\varepsilon K_\varepsilon(\Psi_{2\varepsilon
}\nabla u_0)
\big\|_{L^2(\Omega_T)}
\big\|\delta^{-\frac{1}{2}}\mathcal{R}_{\varepsilon,2}
\big\|_{L^2(\Sigma_{2\varepsilon}^T)}
\bigg\}=:I_{31} + I_{32}.
\end{aligned}
\end{equation*}
Based on the results developed for $I_3$ in the proof of
Lemma $\ref{lemma:3.1}$, it follows from that
\begin{equation}\label{f:4.18}
\begin{aligned}
&\langle |I_{31}|^p\rangle^{\frac{1}{p}}
\leq \varepsilon^2
\big\langle\|\tilde{\varpi}^\varepsilon
(\nabla^2 + \partial_t) S_\varepsilon K_\varepsilon(\Psi_{2\varepsilon}\nabla u_0)\|_{L^2(\Omega_T)}^{p_1}\big\rangle^{\frac{1}{p_1}}
\big\langle\|\mathcal{R}_{\varepsilon,1}
\|_{L^2(\Omega_T)}^{p_2}\rangle^{\frac{1}{p_2}}\\
&\lesssim^{\eqref{f:3.4},\eqref{pri:4.5}}
\mu_d^2\big(\frac{R_0}{\varepsilon}\big)
\Big\{
\|\nabla u_0\|_{L^2(\boxbox_{8\varepsilon})}
+\varepsilon\|\nabla^2 u_0\|_{L^2((\Omega\setminus O_{4\varepsilon})_T)}
+\varepsilon\|\partial_t u_0\|_{L^2((\Omega\setminus O_{4\varepsilon})_T)}
\Big\}\\
&\times
\bigg\{
\|\nabla v_0\|_{L^2(\boxbox_{16\varepsilon})}
 + \varepsilon\Big(
\|\nabla^2 v_0 \|_{L^2((\Omega\setminus O_{8\varepsilon})_T)}
+\|\partial_t v_0 \|_{L^2((\Omega\setminus O_{8\varepsilon})_T)}\Big)\bigg\}.
\end{aligned}
\end{equation}
For the term $\langle I_{32}^p\rangle^{\frac{1}{p}}$, appealing to
Lemmas $\ref{lemma:5}$, $\ref{lemma:2.5}$, and the relationship $\eqref{f:4.19}$,
one can derive that
\begin{equation*}
\begin{aligned}
&\varepsilon^2\big\langle\|\delta^{\frac{1}{2}}\tilde{\varpi}^\varepsilon
(\nabla^2 + \partial_t)S_\varepsilon K_\varepsilon(\Psi_{2\varepsilon
}\nabla u_0)
\|_{L^2(\Omega_T)}^{p_1}\big\rangle^{\frac{1}{p_1}}\\
&\lesssim^{\eqref{pri:3.42c},\eqref{pri:3.42d}}
\varepsilon\mu_d(\frac{R_0}{\varepsilon})\bigg(
\int_{\Omega_T}
\Big|K_\varepsilon\big(|\nabla(\Psi_{2\varepsilon}\nabla u_0)|
+|\Psi_{\varepsilon}\partial_t u_0|
+|\varepsilon(\partial_t\Psi_{2\varepsilon}\nabla u_0)|
+|\varepsilon\nabla\Psi_{2\varepsilon}\partial_t u_0|\big)\Big|^2
\delta
\bigg)^{\frac{1}{2}}\\
&\lesssim^{\eqref{pri:2.8}}
\mu_d(\frac{R_0}{\varepsilon})
\bigg(
\|\delta^{\frac{1}{2}}\nabla u_0\|_{L^2(\boxbox_{8\varepsilon})}
+\varepsilon\|\delta^{\frac{1}{2}}\nabla^2 u_0\|_{L^2((\Omega\setminus O_{4\varepsilon})_T)}
+\varepsilon\|\delta^{\frac{1}{2}}\partial_t u_0\|_{L^2((\Omega\setminus O_{4\varepsilon})_T)}
\bigg).
\end{aligned}
\end{equation*}
This, together with $\eqref{pri:4.6}$, consequently leads to
\begin{equation}\label{f:4.9}
\begin{aligned}
\langle |I_{32}|^p\rangle^{\frac{1}{p}}
&\leq \varepsilon^2
\big\langle\|\delta^{\frac{1}{2}}\tilde{\varpi}^\varepsilon
(\nabla^2 + \partial_t)S_\varepsilon K_\varepsilon(\Psi_{2\varepsilon
}\nabla u_0)\|_{L^2(\Omega_T)}^{p_1}\big\rangle^{\frac{1}{p_1}}
\big\langle\|\delta^{-\frac{1}{2}}\mathcal{R}_{\varepsilon,1}
\|_{L^2(\Sigma_{2\varepsilon}^T)}^{p_2}\rangle^{\frac{1}{p_2}}\\
&\lesssim
\mu_d(\frac{R_0}{\varepsilon})
\Big\{
\|\delta^{\frac{1}{2}}\nabla u_0\|_{L^2(\boxbox_{8\varepsilon})}
+\varepsilon\|\delta^{\frac{1}{2}}\nabla^2 u_0\|_{L^2((\Omega\setminus O_{4\varepsilon})_T)} \\
&\qquad\qquad\qquad+\varepsilon\|\delta^{\frac{1}{2}}\partial_t u_0\|_{L^2((\Omega\setminus O_{4\varepsilon})_T)}
\Big\}\|\delta^{-\frac{1}{2}}\nabla v_0\|_{L^2(\Sigma_{2\varepsilon}^T)}.
\end{aligned}
\end{equation}
Hence, one can acquire that
\begin{equation}\label{f:4.22}
\begin{aligned}
&\langle |I_{3}|^p\rangle^{\frac{1}{p}}
\leq \langle |I_{31}|^p\rangle^{\frac{1}{p}} +
\langle |I_{32}|^p\rangle^{\frac{1}{p}} \\
&\lesssim^{\eqref{f:4.18},\eqref{f:4.9}}
\mu_d^2\big(\frac{R_0}{\varepsilon}\big)
\Big\{
\|\nabla u_0\|_{L^2(\boxbox_{8\varepsilon})}
+\varepsilon\|\nabla^2 u_0\|_{L^2((\Omega\setminus O_{4\varepsilon})_T)}
+\varepsilon\|\partial_t u_0\|_{L^2((\Omega\setminus O_{4\varepsilon})_T)}
\Big\}\\
&\times
\bigg\{
\|\nabla v_0\|_{L^2(\boxbox_{16\varepsilon})}
 + \varepsilon\|\nabla^2 v_0 \|_{L^2((\Omega\setminus O_{8\varepsilon})_T)}
+\varepsilon\|\partial_t v_0 \|_{L^2((\Omega\setminus O_{8\varepsilon})_T)}\bigg\}\\
&+\mu_d(\frac{R_0}{\varepsilon})
\Big\{
\|\delta^{\frac{1}{2}}\nabla u_0\|_{L^2(\boxbox_{8\varepsilon})}
+\varepsilon\|\delta^{\frac{1}{2}}\nabla^2 u_0\|_{L^2((\Omega\setminus O_{4\varepsilon})_T)}
+\varepsilon\|\delta^{\frac{1}{2}}\partial_t u_0\|_{L^2((\Omega\setminus O_{4\varepsilon})_T)}
\Big\}\|\delta^{-\frac{1}{2}}\nabla v_0\|_{L^2(\Sigma_{2\varepsilon}^T)}.
\end{aligned}
\end{equation}

As a result,
combining the estimates $\eqref{f:4.20}$, $\eqref{f:4.2}$, $\eqref{f:4.21}$, and $\eqref{f:4.22}$
gives the desired estimate $\eqref{pri:4.1}$. We have completed
the whole proof.
\end{proof}

\medskip
\noindent
\textbf{Proof of Proposition $\ref{P:4.2}$.}
Based on Lemmas $\ref{lemma:4.3}$ and \ref{lemma:3.2-ap}, the desired estimate $\eqref{pri:3.0}$ follows essentially from the triangle inequality.
By the results of Lemma \ref{lemma:3.2-ap}, we can rewrite the result of Lemma $\ref{lemma:4.3}$ as follows:
\begin{equation}\label{f:4.5}
 \big\langle\|w_\varepsilon\|_{L^2(\Omega_T)}^{p}\big\rangle^{\frac{1}{p}}
 \lesssim^{\eqref{pri:3.2-ap},\eqref{pri:3.6-ap}} \varepsilon\mu_d^2(R_0/\varepsilon) +
 \varepsilon\mu_d(R_0/\varepsilon)\ln(R_0/\varepsilon).
\end{equation}
Recalling that $\varphi_j=S_\varepsilon(\Psi_{2\varepsilon}\nabla_ju_0)$,
it follows from Lemmas $\ref{lemma:5}$ and $\ref{lemma:2.5}$,
Theorem $\ref{thm:1.0}$, and Proposition $\ref{P:5.3}$ that
\begin{equation*}
\big\langle\|\varepsilon\bar{\phi}_j^\varepsilon\varphi_j
+\varepsilon^2 \tilde{\sigma}_{l(d+1)j}^\varepsilon\partial_l\varphi_j
\|^{p}_{L^2(\Omega_T)}\rangle^{\frac{1}{p}}
\lesssim^{\eqref{pri:3.42c},\eqref{pri:3.42d}
}_{\eqref{pri:5.4},\eqref{pri:5.7c},\eqref{pri:2.8}} \big(\varepsilon+\varepsilon\mu_d(R_0/\varepsilon)\big)\|\nabla u_0\|_{L^2(\Omega_T)}
\lesssim \varepsilon+\varepsilon\mu_d(R_0/\varepsilon).
\end{equation*}
This, together with $\eqref{f:4.5}$, gives the stated estimate
$\eqref{pri:3.0}$, and thereby concludes the proof.
\qed

\subsection{A weighted argument}

\begin{proposition}\label{P:4.1}
Let $2\leq q<\infty$, $\varepsilon\in(0,1]$, and $T>0$.
Suppose that the ensemble $\langle\cdot\rangle$ is stationary
with respect to $\eqref{c:1}$ and satisfies
the spectral gap condition $\eqref{c:3}$.
Given $F\in L^q(\Omega_T)$, let
$u_\varepsilon$ and $u_0$ be the weak solutions of the zero initial-Dirichlet problems
in $\eqref{pde:1.1}$ and $\eqref{pde:1.2}$, respectively.
Then, for any $p\in[1,\infty)$, the following results hold:
\begin{itemize}
  \item If $\Omega\subset\mathbb{R}^d$
   with $d\geq 2$ is a bounded $C^1$ domain, then
\begin{equation}\label{pri:4.9}
\Bigg\langle\bigg(\int_{\Omega_T}
\Big(\dashint_{U_\varepsilon(z)}|u_\varepsilon-u_0|^2
\Big)^{\frac{q}{2}}\bigg)^{\frac{p}{q}}\Bigg\rangle^{\frac{1}{p}}
\lesssim
\varepsilon^{1-}\mu_d(R_0/\varepsilon)\|F\|_{L^q(\Omega_T)},
\end{equation}
where $1-:=1-\sigma$ and $0<\sigma\ll 1$;


  \item For the whole time-spacial region (i.e., $\Omega_T=\mathbb{R}^{d+1}$ with $d\geq 2$), it can be acquired  that
\begin{equation}\label{pri:4.10}
\Bigg\langle\bigg(\int_{\mathbb{R}^{d+1}}
\Big(\dashint_{U_\varepsilon(z)}|u_\varepsilon-u_0|^2
\Big)^{\frac{q(d+2)}{2d}}\bigg)^{\frac{pd}{q(d+2)}}\Bigg\rangle^{\frac{1}{p}}
\lesssim
\varepsilon\mu_d(1/\varepsilon)
\Big\{\|F\|_{L^2(\mathbb{R}^{d+1})}
+\|\bar{\mu}_{d}^{\alpha}F\|_{L^q(\mathbb{R}^{d+1})}\Big\},
\end{equation}
\end{itemize}
in which $\alpha=1+\frac{2}{d}$. The multiplicative constant in $\eqref{pri:4.9}$ depends at most only on $\mu, \lambda_1, d, p, q, T$, and $\Omega$, while its counterpart in $\eqref{pri:4.10}$ is independent of $\Omega$ and $T$.
\end{proposition}

\begin{lemma}\label{lemma:4.4}
Let $2\leq q<\infty$.
Let $\Omega\subset\mathbb{R}^d$ be a bounded Lipschitz domain with
$d\geq 2$, $\varepsilon\in(0,1]$, and $T>0$.
Suppose that the ensemble $\langle\cdot\rangle$ is stationary
with $\eqref{c:1}$ and $\eqref{c:3}$.
Assume that $u_\varepsilon, u_0\in L^2(0,T;H^1_0(\Omega))
\cap W_p^{2,1}(\Omega_T)$ are the weak solutions of the zero initial-Dirichlet problems
in $\eqref{pde:1.1}$ and $\eqref{pde:1.2}$, respectively.
Let $w_\varepsilon$ be given as in $\eqref{eq:3.1}$.
Then, by suitably choosing $\varphi_j$
in $\eqref{eq:3.1}$, for any $p\in[1,\infty)$, it can be obtained that
\begin{equation}\label{pri:4.2}
\begin{aligned}
&\Bigg\langle
\bigg(\int_{\Omega_T}dz \Big(\dashint_{U_\varepsilon(z)}
|\nabla w_\varepsilon|^2\Big)^{\frac{q}{2}}\omega_\sigma(z)
\bigg)^{\frac{p}{q}}\Bigg\rangle^{\frac{1}{p}}\\
&\lesssim_{\mu,d,\lambda_1,p,q} \mu_d(R_0/\varepsilon)
\bigg\{
\Big(\int_{(O_{2\varepsilon})_T}|\nabla u_0|^q\omega_\sigma
\Big)^{\frac{1}{q}}
+ \varepsilon
\Big(\int_{(\Omega\setminus O_{\varepsilon})_T}\big(|\nabla^2 u_0|^q
+|\partial_t u_0|^q\big)\omega_\sigma
\Big)^{\frac{1}{q}}\bigg\},
\end{aligned}
\end{equation}
where it is recalled that $(O_{2\varepsilon})_T:=O_{2\varepsilon}\times(0,T]$
and $(\Omega\setminus O_\varepsilon)_T:=
(\Omega\setminus O_\varepsilon)\times (0,T]$,
and $\omega_\sigma$ is the weight function given as in Theorem $\ref{thm:C-Z}$.
\end{lemma}

\begin{proof}
The basic idea is to apply the weighted annealed Calder\'on-Zygmund estimates
stated in Theorem $\ref{thm:C-Z}$ to accelerate the convergence rates
near the lateral region. However, this method does not easily handle the loss in the time-layer caused by the truncated initial value, which is an aspect quite different from the duality argument. Therefore, the assumption of $u_\varepsilon(\cdot,0) =u_0(\cdot,0)=0$ is essential here\footnote{
This enables us to conveniently extend the solution in the time direction, thereby avoiding the discussion on time layers.}.
It suffices to prove the stated estimate $\eqref{pri:4.2}$ for
$p\geq q$; the case $p\in[1,q)$ then follows directly from H\"older's inequality. The proof is similarly divided into three steps.

\textbf{Step 1.} Outline of the proof.
Let
$\psi_{\varepsilon}\in C_0^{2}(\Omega)$ be a cut-off function satisfying
$\psi_{\varepsilon} = 1$ on $\Omega\setminus O_{4\varepsilon}$,
$\psi_{\varepsilon} = 0$ in $O_{3\varepsilon}$, and
\begin{equation}\label{cut-off}
(1/\varepsilon)|\nabla\psi_{\varepsilon}|
+|\nabla^2\psi_{\varepsilon}|
\lesssim_d 1/\varepsilon^2
\qquad\text{in}
\quad \Omega.
\end{equation}

Choose $\varphi_j =
S_\varepsilon K_\varepsilon(\psi_{\varepsilon}\partial_ju_0)$ in $\eqref{eq:3.1}$ so that $w_\varepsilon$ satisfies
\begin{equation*}
\left\{\begin{aligned}
\partial_t w_\varepsilon + \mathcal{L}_\varepsilon(w_\varepsilon)
& = \nabla\cdot f_\varepsilon &\quad&\emph{in}\quad ~\tilde{\Omega}_T;\\
w_\varepsilon & = 0 &\quad&\emph{on}\quad \partial_p\tilde{\Omega}_T,
\end{aligned}\right.
\end{equation*}
where $\tilde{\Omega}_T:=\Omega\times(-1,T]$, and
$u_\varepsilon$ and $u_0$ are extended by zero from $\Omega_T$ to $\tilde{\Omega}_T$, with the extended functions still denoted by $u_\varepsilon$ and $u_0$. It is recalled that
$\Omega_0\supseteq\Omega$
is a  $C^1$ domain such that $\partial\Omega_0$ is
the hypersurface at a distance of $4\varepsilon$ from $\partial\Omega$
and parallel to it.
Let $\omega_\sigma(z):=[\text{dist}(x,\partial\Omega_0)]^{\underline{q}-1}$ with $0<\underline{q}<q$ and $z=(x,t)\in\tilde{\Omega}_T$.
Thus, it follows from
Theorem $\ref{thm:C-Z}$ and the expression $\eqref{eq:3.3}$ that
\begin{equation}\label{f:4.24}
\begin{aligned}
& \bigg(\int_{\tilde{\Omega}_T}dz\Big\langle\big(\dashint_{U_\varepsilon(z)}
|\nabla w_\varepsilon|^2\big)^{\frac{p}{2}}
\Big\rangle^{\frac{q}{p}}\omega_\sigma(z)\bigg)^{\frac{1}{q}}
\lesssim^{\eqref{pri:B}}
\bigg(\int_{\tilde{\Omega}_T}dz\Big\langle\big(\dashint_{U_\varepsilon(z)}
|f_\varepsilon|^2
\big)^{\frac{\bar{p}}{2}}
\Big\rangle^{\frac{q}{\bar{p}}}\omega_\sigma(z)\bigg)^{\frac{1}{q}}\\
&\leq
\bigg(\int_{\tilde{\Omega}_T}dz\Big\langle\big(\dashint_{U_\varepsilon(z)}
|\nabla u_0
- S_\varepsilon K_\varepsilon(\psi_{\varepsilon}\nabla u_0)|^2
\big)^{\frac{\bar{p}}{2}}
\Big\rangle^{\frac{q}{\bar{p}}}\omega_\sigma(z)\bigg)^{\frac{1}{q}}
 \\
& +\varepsilon
\bigg(\int_{\tilde{\Omega}_T}dz\Big\langle\big(\dashint_{U_\varepsilon(z)}
|\varpi(\cdot/\varepsilon,\cdot/\varepsilon^2)\nabla  S_\varepsilon K_\varepsilon(\psi_{\varepsilon}\nabla u_0)|^2
\big)^{\frac{\bar{p}}{2}}
\Big\rangle^{\frac{q}{\bar{p}}}\omega_\sigma(z)\bigg)^{\frac{1}{q}} \\
&+\varepsilon^2
\bigg(\int_{\tilde{\Omega}_T}dz\Big\langle\big(\dashint_{U_\varepsilon(z)}
|\tilde{\varpi}(\cdot/\varepsilon,\cdot/\varepsilon^2)
\big(\nabla^2+\partial_t\big) S_\varepsilon K_\varepsilon
(\psi_{\varepsilon}\nabla u_0)|^2
\big)^{\frac{\bar{p}}{2}}
\Big\rangle^{\frac{q}{\bar{p}}}\omega_\sigma(z)\bigg)^{\frac{1}{q}}
=: I_1 + I_2 + I_3,
\end{aligned}
\end{equation}
where $\varpi_j:= a\bar{\phi}_j
-\bar{\sigma}_{j}
+ a(\partial \tilde{\sigma}_{(d+1)j})$ and
$\tilde{\varpi}_j = \big(a \otimes \tilde{\sigma}_{(d+1)j},
\tilde{\sigma}_{(d+1)j}\big)$. Therefore,
the desired estimate $\eqref{pri:4.2}$ is reduced to estimating the right-hand sides above.

\textbf{Step 2.} Arguments for $I_1$. In view of
the relationships $\eqref{f:3.41-ap}$ and $\eqref{f:3.5}$
(with $\Psi_\varepsilon$ replaced by $\psi_\varepsilon$),
it therefore follows from Fubini's theorem and
Lemma $\ref{lemma:2.5}$ that
\begin{equation}\label{f:4.25}
\begin{aligned}
I_1&\leq
\bigg(\int_{\tilde{\Omega}_T}dz
\Big(\dashint_{U_\varepsilon(z)}
|\psi_{\varepsilon}\nabla u_0
-  S_\varepsilon K_\varepsilon(\psi_{\varepsilon}\nabla u_0)|^2\Big)^{\frac{q}{2}}\omega_\sigma(z)\bigg)^{\frac{1}{q}}
+\bigg(\int_{\tilde{\Omega}_T}dz
\Big(\dashint_{U_\varepsilon(z)}
|\textbf{1}_{O_{2\varepsilon}}\nabla u_0|^2\Big)^{\frac{q}{2}}\omega_\sigma(z)\bigg)^{\frac{1}{q}}\\
&\lesssim
\bigg(\int_{\tilde{\Omega}_T}dz
|\psi_{\varepsilon}\nabla u_0
-  S_\varepsilon K_\varepsilon(\psi_{\varepsilon}\nabla u_0)|^q\dashint_{Q_\varepsilon(z)}\omega_\sigma\bigg)^{\frac{1}{q}}
+\bigg(\int_{\tilde{\Omega}_T}dz
|\textbf{1}_{O_{2\varepsilon}}\nabla u_0|^q\dashint_{Q_\varepsilon(z)}\omega_\sigma \bigg)^{\frac{1}{q}}\\
&\lesssim^{\eqref{f:3.41-ap}}
\bigg(\int_{\tilde{\Omega}_T}
|\psi_{\varepsilon}\nabla u_0
-  S_\varepsilon K_\varepsilon(\psi_{\varepsilon}\nabla u_0)|^q \omega_\sigma\bigg)^{\frac{1}{q}}
+ \bigg(\int_{(O_{2\varepsilon})_T}|\nabla u_0|^q\omega_\sigma
\bigg)^{\frac{1}{q}} \\
&\lesssim_{\eqref{pri:2.10},\eqref{pri:2.12}}^{\eqref{f:3.28}}
\bigg(\int_{(O_{2\varepsilon})_T}|\nabla u_0|^q\omega_\sigma
\bigg)^{\frac{1}{q}}
+ \varepsilon
\bigg(\int_{(\Omega\setminus O_{\varepsilon})_T}
\big(|\nabla^2 u_0|^q + |\partial_t u_0|^q
\big)\omega_\sigma
\bigg)^{\frac{1}{q}},
\end{aligned}
\end{equation}
where we recall that $(O_{2\varepsilon})_T:=O_{2\varepsilon}\times[0,T]$
and $(\Omega\setminus O_\varepsilon)_T:=
(\Omega\setminus O_\varepsilon)\times [0,T]$ in this proof.

\textbf{Step 3.} Arguments for $I_2$ and $I_3$. Beginning with the estimates for $I_2$,
it follows from
Lemmas $\ref{lemma:5}$, $\ref{lemma:2.5}$, and Corollary $\ref{corollary:2.1}$ that
\begin{equation}\label{f:4.26}
\begin{aligned}
I_2
&\lesssim^{\eqref{pri:3.42a},\eqref{pri:2.18}} \varepsilon
\Big(
\int_{\tilde{\Omega}_T} |\nabla K_\varepsilon(\psi_{\varepsilon}\nabla u_0)|^q
\omega_\sigma
\Big)^{\frac{1}{q}}\\
&\lesssim^{\eqref{pri:2.8}}
\varepsilon
\Big(
\int_{\tilde{\Omega}_T} |\nabla(\psi_{\varepsilon}\nabla u_0)|^q
\omega_\sigma
\Big)^{\frac{1}{q}}
\lesssim
\bigg(\int_{(O_{2\varepsilon})_T}|\nabla u_0|^q\omega_\sigma
\bigg)^{\frac{1}{q}}
+ \varepsilon
\bigg(\int_{(\Omega\setminus O_{\varepsilon})_T}|\nabla^2 u_0|^q\omega_\sigma
\bigg)^{\frac{1}{q}}.
\end{aligned}
\end{equation}

By $\eqref{f:4.19}$, we first observe that
\begin{equation}\label{f:4.23}
\begin{aligned}
&\nabla^2 S_\varepsilon K_\varepsilon(\psi_{\varepsilon}\nabla u_0)
= \nabla S_\varepsilon K_\varepsilon(\nabla(\psi_{\varepsilon}\nabla u_0));\\
&\partial_t S_\varepsilon K_\varepsilon(\psi_{\varepsilon}\nabla u_0)
=\nabla S_\varepsilon K_\varepsilon(\psi_{\varepsilon}\partial_t u_0)
- S_\varepsilon K_\varepsilon(\nabla\psi_{\varepsilon}\partial_t u_0).
\end{aligned}
\end{equation}
Then, by using Lemmas $\ref{lemma:5}$, $\ref{lemma:2.5}$, and Corollary $\ref{corollary:2.1}$ again,
it can be derived that
\begin{equation}\label{f:4.27}
\begin{aligned}
I_3
&\lesssim^{\eqref{f:4.23},\eqref{pri:3.42a}
}_{\eqref{pri:3.42b},\eqref{pri:2.19}}
\varepsilon
\Big(
\int_{\tilde{\Omega}_T} |
K_\varepsilon(\nabla(\psi_{\varepsilon}\nabla u_0)+\psi_\varepsilon\partial_tu_0)|^q
\bar{\mu}_d^q(\cdot/\varepsilon)\omega_\sigma
\Big)^{\frac{1}{q}}
+ \varepsilon^2
\Big(
\int_{\tilde{\Omega}_T} |K_\varepsilon(
\nabla\psi_\varepsilon\partial_t u_0)|^q
\bar{\mu}_d^q(\cdot/\varepsilon)
\omega_\sigma
\Big)^{\frac{1}{q}}\\
&\lesssim^{\eqref{f:4.28}}_{\eqref{pri:2.8}} \mu_d(R_0/\varepsilon)
\bigg\{
\Big(\int_{(O_{2\varepsilon})_T}|\nabla u_0|^q\omega_\sigma
\Big)^{\frac{1}{q}}
+ \varepsilon
\Big(\int_{(\Omega\setminus O_{\varepsilon})_T}
\big(|\partial_t u_0|^q+|\nabla^2 u_0|^q\big)\omega_\sigma
\Big)^{\frac{1}{q}}
\bigg\}.
\end{aligned}
\end{equation}

Consequently, plugging $\eqref{f:4.25}$, $\eqref{f:4.26}$, and $\eqref{f:4.27}$
back into $\eqref{f:4.24}$ leads to the desired result,
where we also employ Minkowski's inequality on
the left-hand side of $\eqref{f:4.24}$. This ends the proof.
\end{proof}

\begin{lemma}\label{lemma:4.2}
Let $2\leq q<\infty$.
Assume that $u_\varepsilon$ and $u_0$ are associated by the following equation:
\begin{equation*}
\partial_t u_\varepsilon + \mathcal{L}_\varepsilon(u_\varepsilon)
= \partial_t u_0 + \mathcal{L}_0(u_0) \quad \emph{in}\quad ~\mathbb{R}^{d+1}.
\end{equation*}
Suppose that the ensemble $\langle\cdot\rangle$ is stationary
with respect to $\eqref{c:1}$ and satisfies $\eqref{c:3}$.
Let $w_\varepsilon$ be given as in $\eqref{eq:3.1}$
by setting $\varphi_j = S_\varepsilon K_\varepsilon(\partial_ju_0)$
therein. Then, for any $p\in[1,\infty)$, it can be obtained that
\begin{subequations}
\begin{align}
\bigg\langle\Big(\int_{\mathbb{R}^{d+1}}|\nabla w_\varepsilon|^2\Big)^{\frac{p}{2}}\bigg\rangle^{\frac{1}{p}}
+\bigg\langle\Big(
&\sup_{t\in\mathbb{R}}\big\|w_\varepsilon(\cdot ,t)
\big\|_{L^2(\Omega)}\Big)^p\bigg\rangle^{\frac{1}{p}}
+\Big\langle\sup_{t\in\mathbb{R}}\Big(\int_{\mathbb{R}^d}
  \dashint_{Q_\varepsilon(\cdot,t)}|w_\varepsilon|^2
  \Big)^{\frac{p}{2}}\Big\rangle^{\frac{1}{p}}\nonumber\\
&\lesssim \varepsilon\mu_d(1/\varepsilon)
\Big\{\|\mu_d\nabla^2 u_0\|_{L^2(\mathbb{R}^{d+1})}
+ \|\mu_d \partial_t u_0\|_{L^2(\mathbb{R}^{d+1})}\Big\};
\label{pri:4.7}\\
\bigg\langle
\bigg(\int_{\mathbb{R}^{d+1}}dz \Big(\dashint_{Q_\varepsilon(z)}
|\nabla w_\varepsilon|^2\Big)^{\frac{q}{2}}
\bigg)^{\frac{p}{q}}\bigg\rangle^{\frac{1}{p}}
&\lesssim \varepsilon\mu_d(1/\varepsilon)
\Big\{\|\bar{\mu}_d\nabla^2 u_0\|_{L^q(\mathbb{R}^{d+1})}
+ \|\bar{\mu}_d \partial_t u_0\|_{L^q(\mathbb{R}^{d+1})}\Big\},
\label{pri:4.8}
\end{align}
\end{subequations}
where the multiplicative constants depend
at most on $\mu,\lambda_1,d,p$, and $q$.
\end{lemma}

\begin{proof}
The proof of $\eqref{pri:4.7}$ is analogous to that given for Lemma $\ref{lemma:3.1}$, while the proof
of $\eqref{pri:4.8}$ follows a computation similar to that shown in
Lemma $\ref{lemma:4.4}$. The proof is divided into two steps.

\textbf{Step 1.} Arguments for $\eqref{pri:4.7}$.
Starting from the energy estimate, it follows that
\begin{equation*}
\begin{aligned}
&\|\nabla w_\varepsilon\|_{L^2(\mathbb{R}^{d+1})}
+ \sup_{t\in\mathbb{R}}\|w_\varepsilon(\cdot,t)\|_{L^2(\mathbb{R}^{d})}
\lesssim\|\tilde{f}\|_{L^2(\mathbb{R}^{d+1})}
\lesssim
\big\|\nabla u_0
- S_\varepsilon K_\varepsilon(\nabla u_0)\big\|_{L^2(\mathbb{R}^{d+1})} \\
& +\varepsilon\big\|\varpi(\cdot/\varepsilon,\cdot/\varepsilon^2)\nabla S_\varepsilon K_\varepsilon(\nabla u_0)\big\|_{L^2(\mathbb{R}^{d+1})}
+\varepsilon^2 \big\|\tilde{\varpi}(\cdot/\varepsilon,\cdot/\varepsilon^2)
\big(\nabla^2+\partial_t\big)S_\varepsilon K_\varepsilon
(\nabla u_0)\big\|_{L^2(\mathbb{R}^{d+1})}\\
&=: I_1 + I_2 + I_3,
\end{aligned}
\end{equation*}
where $\varpi_j:= a\bar{\phi}_j
-\bar{\sigma}_{j}
+ a(\partial \tilde{\sigma}_{(d+1)j})$ and
$\tilde{\varpi}_j = \big(a \otimes \tilde{\sigma}_{(d+1)j},
\tilde{\sigma}_{(d+1)j}\big)$.
This further implies that
\begin{equation}\label{f:4.29}
\begin{aligned}
\bigg\langle\sup_{t\in\mathbb{R}}\Big(\int_{\mathbb{R}^d}
  \dashint_{Q_\varepsilon(\cdot,t)}|w_\varepsilon|^2
  \Big)^{\frac{p}{2}}\bigg\rangle^{\frac{1}{p}}
&+
\Big\langle\big(\sup_{t\in\mathbb{R}}\|w_\varepsilon(\cdot,t)\|_{L^2(\mathbb{R}^{d})}\big)^{p}\Big\rangle^{\frac{1}{p}}\\
&+
\Big\langle\|\nabla w_\varepsilon\|_{L^2(\mathbb{R}^{d+1})}^p\Big\rangle^{\frac{1}{p}}
\lesssim I_1
+ \langle I_2^p\rangle^{\frac{1}{p}} + \langle I_3^p\rangle^{\frac{1}{p}}.
\end{aligned}
\end{equation}
Similar to the computations given for $\eqref{f:3.2}$, it follows from
Lemma $\ref{lemma:2.4}$ that
\begin{equation}\label{f:4.30}
\begin{aligned}
I_1
&\lesssim^{\eqref{pri:2.14},\eqref{pri:2.13}} \varepsilon
\bigg\{\|\nabla^2 u_0\|_{L^2(\mathbb{R}^{d+1})}
+ \|\nabla K_\varepsilon(\nabla u_0)\|_{L^2(\mathbb{R}^{d+1})}
+ \varepsilon\|
\nabla K_\varepsilon(\partial_t u_0)\|_{L^2(\mathbb{R}^{d+1})}
\bigg\}\\
&\lesssim^{\eqref{pri:2.17}}
\varepsilon\Big\{\|\nabla^2 u_0\|_{L^2(\mathbb{R}^{d+1})}
+\|\partial_t u_0\|_{L^2(\mathbb{R}^{d+1})}\Big\}.
\end{aligned}
\end{equation}
In view of Lemmas $\ref{lemma:5}$, $\ref{lemma:2.5}$, $\ref{lemma:2.4}$,
Theorem $\ref{thm:1.0}$, and Proposition $\ref{P:5.3}$, it is obtained that
\begin{equation}\label{f:4.31}
\begin{aligned}
&\langle I_2^p\rangle^{\frac{1}{p}}
\lesssim^{\eqref{pri:3.42c},\eqref{pri:5.4},
\eqref{pri:5.6a},\eqref{pri:5.6c}}
\varepsilon \|\nabla K_\varepsilon(\nabla u_0)\|_{L^2(\mathbb{R}^{d+1})}
\lesssim^{\eqref{pri:2.17}}
\varepsilon \|\nabla^2 u_0\|_{L^2(\mathbb{R}^{d+1})};\\
&\langle I_3^p\rangle^{\frac{1}{p}}
\lesssim^{\eqref{pri:3.42d},\eqref{pri:5.7c},\eqref{f:4.28}}
\varepsilon\mu_d(1/\varepsilon)\Big\{\|\bar{\mu}_d\nabla K_\varepsilon(\nabla u_0)\|_{L^2(\mathbb{R}^{d+1})}
+ \|\bar{\mu}_dK_\varepsilon(\partial_t u_0)\|_{L^2(\mathbb{R}^{d+1})}\Big\}\\
&\qquad~\lesssim^{\eqref{pri:2.8}}
\varepsilon\mu_d(1/\varepsilon)\Big\{\|\bar{\mu}_d\nabla^2 u_0\|_{L^2(\mathbb{R}^{d+1})}
+ \|\bar{\mu}_d \partial_t u_0\|_{L^2(\mathbb{R}^{d+1})}\Big\}.
\end{aligned}
\end{equation}

Combining the estimates $\eqref{f:4.29}$, $\eqref{f:4.30}$,
and $\eqref{f:4.31}$ leads to the desired estimate $\eqref{pri:4.7}$.

\textbf{Step 2.} Arguments for $\eqref{pri:4.8}$.
In view of Theorem $\ref{thm:C-Z}$ with $\omega_\sigma=1$, we obtain that
\begin{equation}\label{f:4.32}
\begin{aligned}
&\Bigg\langle
\bigg(\int_{\mathbb{R}^{d+1}}dz \Big(\dashint_{Q_\varepsilon(z)}
|\nabla w_\varepsilon|^2\Big)^{\frac{q}{2}}
\bigg)^{\frac{p}{q}}\Bigg\rangle^{\frac{1}{p}}
\lesssim^{\eqref{pri:B}}
\Bigg\langle
\bigg(\int_{\mathbb{R}^{d+1}}dz \Big(\dashint_{Q_\varepsilon(z)}
|\tilde{f}|^2\Big)^{\frac{q}{2}}
\bigg)^{\frac{p}{q}}\Bigg\rangle^{\frac{1}{p}}\\
&\lesssim
\bigg(\int_{\mathbb{R}^{d+1}}dz\Big\langle\big(\dashint_{Q_\varepsilon(z)}
|\nabla u_0
- S_\varepsilon K_\varepsilon(\nabla u_0)|^2
\big)^{\frac{\bar{p}}{2}}
\Big\rangle^{\frac{q}{\bar{p}}}\bigg)^{\frac{1}{q}} \\
& +\varepsilon
\bigg(\int_{\mathbb{R}^{d+1}}dz\Big\langle\big(\dashint_{Q_\varepsilon(z)}
|\varpi(\cdot/\varepsilon,\cdot/\varepsilon^2)\nabla  S_\varepsilon K_\varepsilon(\nabla u_0)|^2
\big)^{\frac{\bar{p}}{2}}
\Big\rangle^{\frac{q}{\bar{p}}}\bigg)^{\frac{1}{q}} \\
&+\varepsilon^2
\bigg(\int_{\mathbb{R}^{d+1}}dz\Big\langle
\big(\dashint_{Q_\varepsilon(z)}
|\tilde{\varpi}(\cdot/\varepsilon,\cdot/\varepsilon^2)
(\nabla^2+\partial_t) S_\varepsilon K_\varepsilon
(\nabla u_0)|^2
\big)^{\frac{\bar{p}}{2}}
\Big\rangle^{\frac{q}{\bar{p}}}\bigg)^{\frac{1}{q}}
=: I_1 + I_2 + I_3.
\end{aligned}
\end{equation}
On account of Lemma $\ref{lemma:2.4}$, it follows from
H\"older's inequality and Fubini's theorem that
\begin{equation}\label{f:4.33}
\begin{aligned}
I_1&\leq
\bigg(\int_{\mathbb{R}^{d+1}}dz
\Big(\dashint_{Q_\varepsilon(z)}
|\nabla u_0
-  S_\varepsilon K_\varepsilon(\nabla u_0)|^2\Big)^{\frac{q}{2}}\bigg)^{\frac{1}{q}}\\
&\lesssim
\bigg(\int_{\mathbb{R}^{d+1}}
|\nabla u_0
-  S_\varepsilon K_\varepsilon(\nabla u_0)|^q\bigg)^{\frac{1}{q}}
\lesssim^{\eqref{pri:2.13},\eqref{pri:2.14}}
\varepsilon
\bigg(\int_{\mathbb{R}^{d+1}}
\big(|\nabla^2 u_0|^q + |\partial_t u_0|^q
\big)
\bigg)^{\frac{1}{q}}.
\end{aligned}
\end{equation}

Moreover, by using Lemmas $\ref{lemma:5}$, $\ref{lemma:2.5}$,
Theorem $\ref{thm:1.0}$, Proposition $\ref{P:5.3}$, one can derive that
\begin{equation}\label{f:4.34}
\begin{aligned}
I_2
&\lesssim^{\eqref{pri:3.42a},\eqref{pri:5.4},
\eqref{pri:5.6a},\eqref{pri:5.6c}} \varepsilon
\Big(
\int_{\mathbb{R}^{d+1}} |\nabla K_\varepsilon(\nabla u_0)|^q
\Big)^{\frac{1}{q}}
\lesssim^{\eqref{pri:2.17}}
\varepsilon
\Big(
\int_{\mathbb{R}^{d+1}} |\nabla^2 u_0|^q
\Big)^{\frac{1}{q}}; \\
I_3
&\lesssim^{\eqref{f:4.23},\eqref{pri:3.42a}
}_{\eqref{pri:3.42b},\eqref{pri:5.7c}}
\varepsilon
\Big(
\int_{\mathbb{R}^{d+1}} |
K_\varepsilon(\nabla^2 u_0)|^q
\bar{\mu}_d^q(\cdot/\varepsilon)
\Big)^{\frac{1}{q}}
+ \varepsilon
\Big(
\int_{\mathbb{R}^{d+1}} |K_\varepsilon(\partial_t u_0)|^q
\bar{\mu}_d^q(\cdot/\varepsilon)
\Big)^{\frac{1}{q}}\\
&\lesssim^{\eqref{f:4.28},\eqref{pri:2.8}} \varepsilon\mu_d(1/\varepsilon)
\Big(\int_{\mathbb{R}^{d+1}}
\big(|\partial_t u_0|^q+|\nabla^2 u_0|^q\big)\bar{\mu}_d^q
\Big)^{\frac{1}{q}}.
\end{aligned}
\end{equation}

Combining the estimates $\eqref{f:4.32}$, $\eqref{f:4.33}$, and
$\eqref{f:4.34}$ leads to the stated estimate $\eqref{pri:4.8}$.
We have completed the whole proof.
\end{proof}

\medskip
\noindent
\textbf{Proof of Proposition $\ref{P:4.1}$.}
Based on Lemmas $\ref{lemma:4.4}$ and $\ref{lemma:4.2}$,
the remainder of the proof consists mainly of transferring the estimates of the two-scale error at the gradient level to the estimates on the two-scale expansion error $w_\varepsilon$ itself. The main idea is to use the weighted Hardy inequality and the Sobolev inequality (i.e.,
Lemma \ref{lemma:4.1-ap}). The proof is therefore divided into two steps.

\textbf{Step 1.}  Arguments for $\eqref{pri:4.9}$.
It suffices to handle the left-hand side of $\eqref{pri:4.2}$ as stated in
Lemma $\ref{lemma:4.4}$.
Let $W_\varepsilon(z):=\big(\dashint_{U_\varepsilon(z)}
|w_\varepsilon|^2\big)^{\frac{1}{2}}$. A routine computation leads to
\begin{equation}\label{f:4.40}
  |\nabla W_\varepsilon(z)|\leq \Big(\dashint_{U_\varepsilon(z)}
|\nabla w_\varepsilon|^2\Big)^{\frac{1}{2}}.
\end{equation}
The left-hand side of $\eqref{pri:4.2}$ is then handled as follows:
\begin{equation}\label{f:4.36}
\Bigg\langle
\bigg(\int_{\Omega_T}dz \Big(\dashint_{U_\varepsilon(z)}
|\nabla w_\varepsilon|^2\Big)^{\frac{q}{2}}\omega_\sigma(z)
\bigg)^{\frac{p}{q}}\Bigg\rangle^{\frac{1}{p}}
\geq \bigg\langle
\Big(\int_{\Omega_T} |\nabla W_\varepsilon|^q\omega_\sigma
\Big)^{\frac{p}{q}}\bigg\rangle^{\frac{1}{p}}.
\end{equation}
Then, by $\eqref{eq:2.5}$, it is recalled that  $\omega_\sigma(z)=[\text{dist}(x,\partial\Omega_0)]^{\underline{q}-1}
=[\sigma(z)]^{\underline{q}-1}$
with $0<\underline{q}<q$ and
$z=(x,t)\in\Omega_T$. Using the weighted Hardy inequality (see e.g., \cite[Theorem 1.1]{Lehrback-14}), we obtain that
\begin{equation}\label{f:4.37}
\bigg\langle
\Big(\int_{\Omega_T} |\nabla W_\varepsilon|^q \sigma^{\underline{q}-1}
\Big)^{\frac{p}{q}}\bigg\rangle^{\frac{1}{p}}
\gtrsim \bigg\langle
\Big(\int_{\Omega_T} |W_\varepsilon|^q \sigma^{\underline{q}-1-q}
\Big)^{\frac{p}{q}}\bigg\rangle^{\frac{1}{p}}
\gtrsim_{R_0}
\bigg\langle
\Big(\int_{\Omega_T} |W_\varepsilon|^q
\Big)^{\frac{p}{q}}\bigg\rangle^{\frac{1}{p}},
\end{equation}
where the fact that $\max_{z\in\Omega_T}\sigma(z)\lesssim R_0$ is also
employed in the last inequality. Thus, it follows that
\begin{equation}\label{f:4.38}
\begin{aligned}
&\bigg\langle\bigg(\int_{\Omega_T}dz
\Big(\dashint_{U_\varepsilon(z)}|u_\varepsilon-u_0|^2
\Big)^{\frac{q}{2}}\bigg)^{\frac{p}{q}}\bigg\rangle^{\frac{1}{p}}\\
&\leq \bigg\langle
\Big(\int_{\Omega_T} |W_\varepsilon|^q
\Big)^{\frac{p}{q}}\bigg\rangle^{\frac{1}{p}}
+\bigg\langle\bigg(\int_{\Omega_T}
\Big(\dashint_{U_\varepsilon(\cdot)}|\varepsilon\bar{\phi}_j^\varepsilon\varphi_j
+\varepsilon^2 \tilde{\sigma}_{l(d+1)j}^\varepsilon\partial_l\varphi_j|^2
\Big)^{\frac{q}{2}}\bigg)^{\frac{p}{q}}\bigg\rangle^{\frac{1}{p}}\\
&\lesssim^{\eqref{f:4.36}}_{\eqref{f:4.37}}
\bigg\langle
\bigg(\int_{\Omega_T} \Big(\dashint_{U_\varepsilon(\cdot)}
|\nabla w_\varepsilon|^2\Big)^{\frac{q}{2}}\omega_\sigma
\bigg)^{\frac{p}{q}}\bigg\rangle^{\frac{1}{p}}
+\bigg(\int_{\Omega_T}\Big\langle
\Big(\dashint_{U_\varepsilon(\cdot)}|\varepsilon\bar{\phi}_j^\varepsilon\varphi_j
+\varepsilon^2 \tilde{\sigma}_{l(d+1)j}^\varepsilon\partial_l\varphi_j|^2
\Big)^{\frac{p}{2}}\Big\rangle^{\frac{q}{p}}\bigg)^{\frac{1}{q}},
\end{aligned}
\end{equation}
where the integration variable is omitted and ``$U_\varepsilon(\cdot)$'' is used instead, and it is recalled that
$\varphi_j =
S_\varepsilon K_\varepsilon(\psi_{\varepsilon}\partial_ju_0)$.

Furthermore, similarly to the computations given for $I_2$ and $I_3$ in the proof Lemma $\ref{lemma:4.4}$,
\begin{equation}\label{f:4.39}
\begin{aligned}
&\bigg(\int_{\Omega_T}\Big\langle
\Big(\dashint_{U_\varepsilon(\cdot)}|\varepsilon\bar{\phi}_j^\varepsilon\varphi_j
+\varepsilon^2 \tilde{\sigma}_{l(d+1)j}^\varepsilon\partial_l\varphi_j|^2
\Big)^{\frac{p}{2}}\Big\rangle^{\frac{q}{p}}\bigg)^{\frac{1}{q}}\\
&\lesssim^{\eqref{pri:3.42a},\eqref{pri:5.4*}}_{\eqref{pri:3.42b},
\eqref{pri:5.7c}} \varepsilon\mu_d(R_0/\varepsilon)
\Big(
\int_{\Omega_T} |K_\varepsilon(\psi_{\varepsilon}\nabla u_0)|^q
\Big)^{\frac{1}{q}}
\lesssim^{\eqref{pri:2.17}} \varepsilon\mu_d(R_0/\varepsilon)\|\nabla u_0\|_{L^q(
\Omega_T)}.
\end{aligned}
\end{equation}
Plugging the estimates $\eqref{f:4.39}$ and $\eqref{pri:4.2}$
back into $\eqref{f:4.38}$, we obtain that
\begin{equation*}
\begin{aligned}
&\bigg\langle\bigg(\int_{\Omega_T}dz
\Big(\dashint_{D_\varepsilon(z)}|u_\varepsilon-u_0|^2
\Big)^{\frac{q}{2}}\bigg)^{\frac{p}{q}}\bigg\rangle^{\frac{1}{p}}\\
&\lesssim \mu_d(R_0/\varepsilon)
\bigg\{
\Big(\int_{(O_{2\varepsilon})_T}|\nabla u_0|^q\omega_\sigma
\Big)^{\frac{1}{q}}
+ \varepsilon
\Big(\int_{(\Omega\setminus O_{\varepsilon})_T}|\nabla^2 u_0|^q\omega_\sigma
\Big)^{\frac{1}{q}}\bigg\}
+\varepsilon\mu_d(R_0/\varepsilon)\|\nabla u_0\|_{L^q(
\Omega_T)}.
\end{aligned}
\end{equation*}
This, together with
the estimates $\eqref{pri:3.4-ap}$ and  
$\eqref{pri:3.5-ap}$,
yields the desired estimate $\eqref{pri:4.9}$.

\textbf{Step 2.} Arguments for $\eqref{pri:4.10}$.
Let $q^*:=\frac{q(d+2)}{d}$,
based on Lemma $\ref{lemma:4.1-ap}$, and start from the following estimate:
\begin{equation*}
\begin{aligned}
 \|W_\varepsilon\|_{L^{q^{*}}(\mathbb{R}^{d+1})}
\lesssim_{q,d} \sup_{t\in\mathbb{R}}\|W_\varepsilon(\cdot,t)\|_{L^2(\mathbb{R}^d)}
  + \|\nabla W_\varepsilon\|_{L^q(\mathbb{R}^{d+1})}.
\end{aligned}
\end{equation*}
By taking $\langle|\cdot|^p\rangle^{\frac{1}{p}}$ on the both sides above,
it thereby follows from Lemma $\ref{lemma:4.2}$ that
\begin{equation}\label{f:4.41}
\begin{aligned}
&\big\langle\|W_\varepsilon
\|_{L^{q^{*}}(\mathbb{R}^{d+1})}^p\big\rangle^{\frac{1}{p}}
\lesssim^{\eqref{f:4.40}} \Big\langle\sup_{t\in\mathbb{R}}\Big(\int_{\mathbb{R}^d}
  \dashint_{Q_\varepsilon(\cdot,t)}|w_\varepsilon|^2
  \Big)^{\frac{p}{2}}\Big\rangle^{\frac{1}{p}}
  + \bigg\langle
\bigg(\int_{\mathbb{R}^{d+1}}dz \Big(\dashint_{Q_\varepsilon(z)}
|\nabla w_\varepsilon|^2\Big)^{\frac{q}{2}}
\bigg)^{\frac{p}{q}}\bigg\rangle^{\frac{1}{p}}\\
&\lesssim^{\eqref{pri:4.7},\eqref{pri:4.8}}
\varepsilon\mu_d(1/\varepsilon)
\Big\{\|\bar{\mu}_{d}(\nabla^2 u_0,\partial_t u_0)\|_{L^2(\mathbb{R}^{d+1})}
+ \|\bar{\mu}_{d} (\nabla^2 u_0,\partial_t u_0)\|_{L^q(\mathbb{R}^{d+1})}\Big\} \\
&\lesssim \varepsilon\mu_d(1/\varepsilon)
\Big\{\|\bar{\mu}_dF\|_{L^2(\mathbb{R}^{d+1})}
+\|\bar{\mu}_dF\|_{L^q(\mathbb{R}^{d+1})}\Big\},
\end{aligned}
\end{equation}
where the fact that $\bar{\mu}_{d}\in A_2\cap A_q$ is also employed, and the last inequality follows from the weighted
Calder\'on-Zygmund estimates\footnote{By Mihlin’ theorem,
one can derive the Calder\'on-Zygmund estimates (see
arguments given for $\eqref{f:5.27-ap}$ in Appendix), and then
a routine argument (see e.g. \cite[Theorem 7.11]{Duoandikoetxea01} or using Shen's lemma $\ref{shen's lemma2}$)
leads to the desired results.} for $\partial_t u_0
+\mathcal{L}_0(u_0) = F$ in $\mathbb{R}^{d+1}$.

Moreover, similarly to the computations given for $\eqref{f:4.39}$,
it is found that
\begin{equation*}
\begin{aligned}
&
\bigg(\int_{\mathbb{R}^{d+1}}dz
\bigg\langle\Big(\dashint_{Q_\varepsilon(z)}|\varepsilon\bar{\phi}_j^\varepsilon\varphi_j
+\varepsilon^2 \tilde{\sigma}_{l(d+1)j}^\varepsilon\partial_l\varphi_j|^2
\Big)^{\frac{p}{2}}\bigg\rangle^{\frac{q^{*}}{p}}\bigg)^{\frac{1}{q^{*}}}\\
&\lesssim^{\eqref{pri:3.42a},\eqref{pri:5.4*}}_{\eqref{pri:3.42b},
\eqref{pri:5.7c}}
\varepsilon\mu_d(1/\varepsilon)
\Big(
\int_{\mathbb{R}^{d+1}} |K_\varepsilon(\nabla u_0)|^{q^{*}}
(1+\bar{\mu}_d^{q^{*}})
\Big)^{\frac{1}{q^{*}}}
\lesssim^{\eqref{pri:2.17}}_{\eqref{pri:2.8}}
\varepsilon\mu_d(1/\varepsilon)\|\bar{\mu}_{d} \nabla u_0\|_{L^{q^{*}}(\mathbb{R}^{d+1})}\\
&\lesssim^{\eqref{pri:6.9-ap}}
\varepsilon\mu_d(1/\varepsilon)
\Big\{\sup_{t\in\mathbb{R}^d}\|\nabla u_0\|_{L^2(\mathbb{R}^d)}
+ \|\bar{\mu}_{d}^{\alpha} \nabla^2 u_0\|_{L^{q}(\mathbb{R}^{d+1})}\Big\}\\
&\lesssim
\varepsilon\mu_d(1/\varepsilon)
\Big\{\|\nabla^2 u_0\|_{L^{2}(\mathbb{R}^{d+1})} +
\|\partial_t u_0\|_{L^{2}(\mathbb{R}^{d+1})}
+ \|\bar{\mu}_{d}^{\alpha} \nabla^2 u_0\|_{L^{q}(\mathbb{R}^{d+1})}\Big\}
\lesssim \|F\|_{L^{2}(\mathbb{R}^{d+1})}
+\|\bar{\mu}_{d}^{\alpha} F\|_{L^{q}(\mathbb{R}^{d+1})},
\end{aligned}
\end{equation*}
where $\alpha = 1+\frac{2}{d}$, and the (weighted) Calder\'on-Zygmund estimates are also used in
the last inequality (in which one employed the fact that $\bar{\mu}_2^2\in A_q$ for any $q\in[2,\infty)$). 
This, together with $\eqref{f:4.41}$, yields the desired estimate
$\eqref{pri:4.10}$. We have completed the whole proof.
\qed

\subsection{Weak norm estimates}

\begin{proposition}\label{P:3.1}
Let $\Omega\subset\mathbb{R}^d$ be a bounded $C^1$ domain with
$d\geq 2$, $\varepsilon\in(0,1]$, and $T>0$.
Suppose that the ensemble $\langle\cdot\rangle$ is stationary
with respect to $\eqref{c:1}$, and satisfies the spectral gap condition $\eqref{c:3}$.
Given $f\in C_0^{1}(\Omega_T)$, let
$u_\varepsilon$ and $u_0$ be the weak solutions of the initial-Dirichlet problems:
\begin{equation}\label{pde:3-5-4}
\left\{\begin{aligned}
(\partial_t+\mathcal{L}_\varepsilon)(u_\varepsilon)
&= \nabla\cdot f &\quad&\text{in}\quad\Omega_T;\\
u_\varepsilon
&= 0 &\quad&\text{on}\quad\partial_p\Omega_T,
\end{aligned}\right.\quad
\text{and}
\quad
\left\{\begin{aligned}
(\partial_t+\mathcal{L}_0)(u_0)
&= \nabla\cdot f &\quad&\text{in}\quad\Omega_T;\\
u_0 &= 0 &\quad&\text{on}\quad\partial_p\Omega_T.
\end{aligned}\right.
\end{equation}
For any $h\in C_0^{2}(\Omega_T;\mathbb{R}^{d+1})$ adhering to
a parabolic scaling,
and for some $\varphi_i\in W_{2}^{2,1}(\Omega_T)$ vanishing near $\partial_p\Omega_T$ with $i=1,\cdots,d$, the random
variable $H^\varepsilon$ is defined as in $\eqref{eq:3-5-1}$.
Then, for any $p\in[1,\infty)$, it holds that
\begin{equation}\label{pri:3-5-1}
\begin{aligned}
&\varepsilon^{-\frac{d+2}{2}} \big\langle (H^\varepsilon
-\langle H^\varepsilon\rangle)^{2p} \big\rangle^{\frac{1}{2p}}\\
&\lesssim
\bigg[\varepsilon^{1-}\mu_{d}(R_0/\varepsilon)
\Big(\int_{\Omega_T}|\nabla h|^{2s}\Big)^{\frac{1}{s}}
+ \varepsilon^{2}\mu_{d}(R_0/\varepsilon)
\Big(\int_{\Omega_T}|(\partial_t,\partial^2) h|^{2s}\Big)^{\frac{1}{s}}
\bigg]
\Big(\int_{\Omega_T}
 |R_0\nabla f|^{2s'}\Big)^{\frac{1}{s'}},
\end{aligned}
\end{equation}
where $s$ and $s'$ are associated with $1/s'+1/s=1$ and $0<s'-1\ll 1$.
\end{proposition}

\begin{proof}
The main idea is to appeal to the sensitive arguments employed
in Propositions $\ref{P:5.1*}$ and $\ref{P:5.3}$, which are also
similar to that given in \cite[Theorem 1.5]{Wang-Xu-25}, originally inspired by \cite[Proposition 6.1]{Josien-Otto-22}. The main contribution here lies in the technical aspect of constructing auxiliary equations  suitable for parabolic systems to achieve the goal of the optimal estimates. Let $s,s'>1$ be given as in the proposition above, and the whole proof is divided into four steps.

\textbf{Step 1.}
Outline of the proof and reduction.
Let $\check{h},\hat{h},\tilde{h}$ represent the different rescalings
of $h$, as follows:
\begin{equation}\label{eq:3-5-2}
 \check{h}:=\varepsilon^{d+2} h(\varepsilon\cdot);
 \quad \hat{h}:= h(\varepsilon\cdot);
 \quad \tilde{h}:= \frac{1}{\varepsilon}h(\varepsilon\cdot).
\end{equation}
Throughout the proof, the corresponding derivative symbols will be distinguished, since the symbols of the functions under different scale transformations are marked as in $\eqref{eq:3-5-2}$.
Thus, in view of the equations $\eqref{pde:3-5-4}$, they can be rewritten as
\begin{equation*}
\left\{\begin{aligned}
\big(\partial_t+\mathcal{L}_1\big)(\tilde{u}_{\varepsilon})
&= \nabla\cdot \hat{f}
&\quad&\text{in}~~\tilde{\Omega}_T/\varepsilon;\\
\tilde{u}_{\varepsilon}
&= 0
&\quad&\text{on}~\partial_p(\tilde{\Omega}_T/\varepsilon),
\end{aligned}\right.
\qquad
\left\{\begin{aligned}
\big(\partial_t+\mathcal{L}_1\big)(\tilde{u}_0)
&= \nabla\cdot \hat{f}
&\quad&\text{in}~~\tilde{\Omega}_T/\varepsilon;\\
\tilde{u}_0
&= 0
&\quad&\text{on}~\partial_p(\tilde{\Omega}_T/\varepsilon),
\end{aligned}\right.
\end{equation*}
where $\tilde{\Omega}_T:=\Omega\times(-1,T]$; $\tilde{\Omega}_T/\varepsilon:=(\Omega/\varepsilon)\times (-1/\varepsilon^2,T/\varepsilon^2]$; $u_\varepsilon$ and $u_0$ have zero-extension from $\Omega_T$ to $\tilde{\Omega}_T$, still denoted by $u_\varepsilon$ and $u_0$.
Recall that $w_\varepsilon:= u_\varepsilon-u_0-\varepsilon\phi_j^\varepsilon\varphi_j-
\varepsilon^2\sigma_{l(d+1)j}^{\varepsilon}\partial_{l}\varphi_j$,
and we now fix $\varphi_j = S_\varepsilon K_\varepsilon(\eta\partial_ju_0)$, where $\eta\in C_0^1(\Omega)$ is a cut-off function satisfying
$\eqref{cut-off}$.
It follows that
\begin{equation}\label{eq:3-5-4}
\hat{\varphi}_j = S_1 K_1(\hat{\eta}\partial_j\tilde{u}_0),
\quad \tilde{w}_{\varepsilon} = \tilde{u}_{\varepsilon} - \tilde{u}_0-\phi_j\hat{\varphi}_j-\sigma_{l(d+1)j}\partial_{l}\hat{\varphi}_j,
\quad\text{and}\quad
\nabla\tilde{w}_{\varepsilon} = \nabla w_\varepsilon.
\end{equation}

Therefore, the error of two-scale expansions satisfies
\begin{equation*}
\begin{aligned}
\big(\partial_t+\mathcal{L}_1\big)(\tilde{w}_\varepsilon)
= \nabla\cdot \Big[(a-\bar{a})(1-\hat{\eta})\nabla\tilde{u}_0
&+\big(a\phi_j-\sigma_j+a\partial\sigma_{(d+1)j}\big)\nabla\hat{\varphi}_j \\
&+a\otimes\sigma_{(d+1)j}:\partial^2\hat{\varphi}_j-\sigma_{(d+1)j}
\partial_t\hat{\varphi}_j\Big]
\quad \text{in}\quad \tilde{\Omega}_T/\varepsilon,
\end{aligned}
\end{equation*}
with $\tilde{w}_{\varepsilon} = 0$ on $\partial_p(\tilde{\Omega}_T/\varepsilon)$. The random variable
is now denoted  by
\begin{equation*}
H_\varepsilon : = \int_{\tilde{\Omega}_T/\varepsilon}
\check{h} \cdot (a-\bar{a})\big(\nabla \tilde{u}_{\varepsilon}
-\nabla\tilde{u}_0
-\nabla\phi_i\hat{\varphi}_i\big),
\end{equation*}
and it follows from a change of variable in $\eqref{eq:3-5-1}$ that $H_{\varepsilon}=H^{\varepsilon}$.
The advantage of the expression for $H_\varepsilon$ is that one can
directly appeal to the $L^p$-version
spectral gap inequality\footnote{It can be derived from
the spectral gap condition $\eqref{c:3}$ (see e.g.
\cite[pp.17-18]{Josien-Otto-22}).}
\begin{equation}\label{pri:3-5-4}
\big\langle (H_\varepsilon -
\langle H_\varepsilon \rangle)^{2p} \big\rangle^{\frac{1}{p}}
\lesssim \Big\langle\Big(\int_{\mathbb{R}^{d+1}}dz
\big(\dashint_{Q_1(z)}|\frac{\partial H_\varepsilon}{\partial a}|\big)^2 \Big)^p\Big\rangle^{\frac{1}{p}}
\end{equation}
to obtain the desired estimate $\eqref{pri:3-5-1}$, provided that
the concrete expression for $\frac{\partial H_\varepsilon}{\partial a}$
is available, namely
\begin{equation}\label{eq:3-5-3}
\begin{aligned}
\frac{\partial H_\varepsilon}{\partial a}
&= \check{h}_j(e_j+\nabla\phi_j^{*})\otimes \big[\nabla\tilde{w}_\varepsilon+\phi_j\nabla\hat{\varphi}_j+
\nabla(\sigma_{l(d+1)j}\partial_l\hat{\varphi}_j)\big]
+ \nabla\phi_j^*\check{h}_j\otimes(\nabla\tilde{u}_0-\hat{\varphi})\\
&-\big[\nabla z_j^* + \phi_i^*\nabla(\check{h}_i\hat{\varphi}_j)
+\nabla(\sigma_{(d+1)lk}^{*}\partial_l(\check{h}_k\hat{\varphi}_j))\big]
\otimes (\nabla\phi_j+e_j) \\
&+ \textbf{1}_{\Omega_T/\varepsilon}\big[\nabla v^* + \phi^{*}_j\nabla\check{h}_j
 +\nabla(\sigma_{(d+1)lj}^{*}\partial_{l}\check{h}_j)\big]
 \otimes\nabla\tilde{u}_\varepsilon,
\end{aligned}
\end{equation}
where the auxiliary functions $v^{*}$ and $z_j^{*}$ satisfy the following equations:
\begin{equation}\label{pde:3-5-3}
\left\{\begin{aligned}
&\big(-\partial_t+\mathcal{L}_1^*\big)(v^*)
=\nabla\cdot\big[(a^{*}\phi_{j}^{*}-\sigma_{j}^{*}
-a^*\partial\sigma_{(d+1)j}^{*})\nabla\check{h}_j \\
&\qquad\qquad
\qquad\qquad\qquad\quad-a^{*}\otimes\sigma_{(d+1)j}^{*}:\partial^2 \check{h}_j-\sigma_{(d+1)j}^{*}\partial_t\check{h}_j\big]
\quad \text{in}~~\tilde{\Omega}_T/\varepsilon;\\
& v^* =0 \quad \text{on}~~\tilde{S}_T/\varepsilon;
\qquad
v^*(\cdot,T/\varepsilon^2) =0 \quad\text{on}~~\Omega/\varepsilon,
\end{aligned}\right.
\end{equation}
in which $-\partial_t+\mathcal{L}_1^{*}$ is the adjoint operator of
$\partial_t+\mathcal{L}_1$, and $\tilde{S}_T/\varepsilon:=(\partial\Omega/\varepsilon)\times
(-1/\varepsilon^2,T/\varepsilon^2]$;
\begin{equation}\label{pde:3-5-5}
\begin{aligned}
\big(-\partial_t+\mathcal{L}_1^{*}\big)(z_j^*)
=\nabla\cdot\Big[\big(a^{*}\phi_{k}^{*}
&-\sigma_{k}^{*}
-a^*\partial\sigma_{(d+1)k}^{*}\big)
\nabla(\check{h}_k\hat{\varphi}_j)\\
&- a^{*}\otimes\sigma_{(d+1)k}^{*}:\partial^2(\check{h}_k\hat{\varphi}_j)
-\sigma_{(d+1)k}^{*}\partial_t(\check{h}_k\hat{\varphi}_j)\Big]
\quad \text{in}~~\mathbb{R}^{d+1}.
\end{aligned}
\end{equation}
The notation $(\phi^{*},\sigma^{*})$ denotes the corresponding
extended corrector, satisfying the following equations:
\begin{equation*}
\begin{aligned}
\partial_t\phi_j^{*} + \nabla\cdot \big(a^*\nabla\phi_j^{*} + e_j\big)
&= 0; \\
\Delta_{d+1}\sigma_{kij}^{*}
 &=\partial_{k}q_{ij}^{*}-\partial_i q_{kj}^{*};\\
 \partial_{k'}\sigma_{k'ij}^{*}
 &= q_{ij}^{*};\\
 \partial_i\sigma_{(d+1)ij}^{*}
 &= \phi_j^{*} + \langle \phi_j^{*}\rangle,
\end{aligned}
\end{equation*}
where $q_{ij}^{*}
:= \bar{a}_{ij}^{*}
- a_{ij}^{*}
- a_{ik}^{*}\partial_k\phi_j^{*}$ and $q_{(d+1)j}^{*} :=\langle\phi_j^{*}\rangle-\phi_j^{*}$.

To complete the proof, the following estimates must be
established:
\begin{subequations}
\begin{align}
&\begin{aligned}
&\bigg\langle
 \bigg(\int_{\mathbb{R}^{d+1}}dz
 \Big(\dashint_{Q_1(z)}\big|
 \check{h}_j(e_j+\nabla\phi_j^{*})\otimes \big(\nabla\tilde{w}_\varepsilon+\phi_j\nabla\hat{\varphi}_j+
\nabla(\sigma_{l(d+1)j}\partial_l\hat{\varphi}_j)\big)\big|\Big)^2
 \bigg)^{p}\bigg\rangle^{\frac{1}{p}}\\
&\qquad\qquad\qquad \qquad\qquad
\lesssim \varepsilon^{d+4-}\mu_d^{2}(R_0/\varepsilon)
\Big(\int_{\Omega_T} |\nabla h|^{2s} dx\Big)^{\frac{1}{s}}
\Big(\int_{\Omega_T} |R_0\nabla f|^{2s'} dx\Big)^{\frac{1}{s'}};
\end{aligned} \label{pri:3-5-2a}\\
 & \begin{aligned}
\bigg\langle
 \bigg(\int_{\tilde{\Omega}_T/\varepsilon}dz
 \Big(\dashint_{U_1(z)}\big|
 \nabla\phi^{*}_j \check{h}_j \otimes
& (\nabla\tilde{u}_0-\hat{\varphi})\big|\Big)^2
 \bigg)^{p}\bigg\rangle^{\frac{1}{p}}\\
& \lesssim  
\varepsilon^{d+4-}\Big(\int_{\Omega_T}
|\nabla h|^{2s}\Big)^{\frac{1}{s}}
\Big(\int_{\Omega_T}
|R_0\nabla f|^{2s'}
\Big)^{\frac{1}{s'}},
\end{aligned} \label{pri:3-5-2b}
\end{align}
\end{subequations}
and
\begin{subequations}
\begin{align}
&\begin{aligned}
\bigg\langle
 \bigg(\int_{\Omega_T/\varepsilon}dz
& \Big(\dashint_{U_1(z)}\big|
 \big(\nabla v^* + \phi^{*}_j\nabla\check{h}_j
 +\nabla(\sigma_{(d+1)lj}^{*}\partial_{l}\check{h}_j)\big)
 \otimes\nabla\tilde{u}_\varepsilon\big|\Big)^2
 \bigg)^{p}\bigg\rangle^{\frac{1}{p}}\\
&\lesssim \bigg[\varepsilon^{d+4}
\Big(\int_{\Omega_T}|\nabla h|^{2s}\Big)^{\frac{1}{s}}
+ \varepsilon^{d+6}\mu_{d}^2(R_0/\varepsilon)
\Big(\int_{\Omega_T}|(\partial_t,\partial^2) h|^{2s}\Big)^{\frac{1}{s}}
\bigg]
\Big(\int_{\Omega_T}
 |f|^{2s'}\Big)^{\frac{1}{s'}};
\end{aligned}  \label{pri:3-5-3a}\\
&\begin{aligned}
&\bigg\langle
 \bigg(\int_{\mathbb{R}^{d+1}}dz
 \Big(\dashint_{Q_1(z)}
 \big|
\big(\nabla z_j^{*}+\phi_k^{*}\nabla(\check{h}_k\hat{\varphi}_j)
+\nabla(\sigma_{(d+1)lk}^{*}\partial_l(\check{h}_k\hat{\varphi}_j))\big)\otimes
 (\nabla\phi_j+e_j)\big|\Big)^2
 \bigg)^{p}\bigg\rangle^{\frac{1}{p}}\\
&\lesssim
\bigg[\varepsilon^{d+4-}\mu_{d}^2(R_0/\varepsilon)
\Big(\int_{\Omega_T}|\nabla h|^{2s}\Big)^{\frac{1}{s}}
+ \varepsilon^{d+6}\mu_{d}^2(R_0/\varepsilon)
\Big(\int_{\Omega_T}|(\partial_t,\partial^2) h|^{2s}\Big)^{\frac{1}{s}}
\bigg]
\Big(\int_{\Omega_T}
 |R_0\nabla f|^{2s'}\Big)^{\frac{1}{s'}}
\end{aligned}\label{pri:3-5-3b},
\end{align}
\end{subequations}
where the multiplicative constant depends on
$\mu,\lambda_1,d,s,p$, and $\partial\Omega$.

Admitting them for a while, combining the estimates $\eqref{pri:3-5-4}$, $\eqref{eq:3-5-3}$, $\eqref{pri:3-5-2a}$,
$\eqref{pri:3-5-2b}$, $\eqref{pri:3-5-3a}$ and $\eqref{pri:3-5-3b}$ yields  the desired estimate $\eqref{pri:3-5-1}$.

\textbf{Step 2.} Arguments for the equality $\eqref{eq:3-5-3}$.
To this end, using the auxiliary equations
$\eqref{pde:3-5-3}$ and $\eqref{pde:3-5-5}$, the auxiliary equations
can be constructed as follows:
\begin{equation}\label{f:3-5-1}
\begin{aligned}
&\big(-\partial_t+\mathcal{L}_1^{*}\big)(\underbrace{v^*+\phi_j^{*}\check{h}_j
+\sigma_{(d+1)lj}^{*}\partial_l\check{h}_j}_{V^{*}})
=\nabla\cdot \big[(a^*-\bar{a}^*)\check{h}\big]
\qquad\text{in}\quad  \tilde{\Omega}_T/\varepsilon; \\
&\big(-\partial_t+\mathcal{L}_1^{*}\big)\big(\underbrace{z_j^*
+\phi_k^{*}\check{h}_k\hat{\varphi}_j
+\sigma_{(d+1)lk}^{*}\partial_l(\check{h}_k\hat{\varphi}_j)}_{Z_j^{*}}\big)
=\nabla\cdot\big[(a^*-\bar{a}^*)\check{h}\hat{\varphi}_j\big]
\qquad \text{in}\quad \mathbb{R}^{d+1}.
\end{aligned}
\end{equation}

We now start from
\begin{equation}\label{f:3-5-4}
\begin{aligned}
\delta H_\varepsilon =
\int_{\Omega_T/\varepsilon}
\check{h} \cdot \delta a\big(\nabla \tilde{u}_{\varepsilon}
-\nabla\tilde{u}_0
-\nabla\phi_i\hat{\varphi}_i\big)
+
\int_{\Omega_T/\varepsilon}
\check{h} \cdot (a-\bar{a})\big[\nabla (\delta\tilde{u}_{\varepsilon})
-\nabla(\delta\phi_j)\hat{\varphi}_j\big]
\end{aligned}
\end{equation}
and the second term in the right-hand side above can be rewritten as
\begin{equation}\label{f:3-5-5}
\begin{aligned}
&\int_{\Omega_T/\varepsilon}
\check{h} \cdot (a-\bar{a})\big[\nabla (\delta\tilde{u}_{\varepsilon})
-\nabla(\delta\phi_j)\hat{\varphi}_j\big]\\
&= \int_{\Omega_T/\varepsilon}
\nabla (\delta\tilde{u}_{\varepsilon})\cdot (a^{*}-\bar{a}^{*})\check{h}
-\int_{\Omega_T/\varepsilon}
\nabla (\delta\phi_j)\cdot (a^{*}-\bar{a}^{*})\check{h}\hat{\varphi}_j\\
&= -\int_{\Omega_T/\varepsilon}
\delta\tilde{u}_{\varepsilon}\nabla\cdot \big[(a^{*}-\bar{a}^{*})\check{h}\big]
+\int_{\mathbb{R}^{d+1}}
 \delta\phi_j\nabla\cdot\big[(a^{*}-\bar{a}^{*})\check{h}\hat{\varphi}_j\big]\\
&=^{\eqref{f:3-5-1}}-\int_{\Omega_T/\varepsilon}
\delta\tilde{u}_{\varepsilon}\big(-\partial_t+\mathcal{L}_1^{*}\big)(V^{*})
+\int_{\mathbb{R}^{d+1}}
 \delta\phi_j
 \big(-\partial_t+\mathcal{L}_1^{*}\big)(Z_j^{*})=:I_1 + I_2.
\end{aligned}
\end{equation}
The terms $I_1$ and $I_2$ are dealt with separately. It is recalled that $V^{*} = 0$ on $\tilde{S}_T/\varepsilon$, $V^{*}(\cdot,T/\varepsilon^2) = 0$ on $\Omega/\varepsilon$, and $\delta\tilde{u}_{\varepsilon}$ satisfies
the following equations:
\begin{equation}\label{pde:3-5-1}
\left\{\begin{aligned}
\big(\partial_t+\mathcal{L}_1\big)(\delta\tilde{u}_{\varepsilon})
&= \nabla\cdot \delta a\nabla \tilde{u}_{\varepsilon}
&\quad &\text{in}\quad \tilde{\Omega}_T/\varepsilon;\\
\delta\tilde{u}_{\varepsilon} &= 0
&\quad &\text{on}\quad \partial_p\tilde{\Omega}_T/\varepsilon.
\end{aligned}\right.
\end{equation}
By integration by parts, it is obtained that
\begin{equation}\label{f:3-5-2}
\begin{aligned}
I_1 &= -\int_{\tilde{\Omega}_T/\varepsilon}
\big(\partial_t+\mathcal{L}_1\big)
(\delta\tilde{u}_{\varepsilon})V^{*}
=^{\eqref{pde:3-5-1}} -\int_{\tilde{\Omega}_T/\varepsilon}
\nabla\cdot \delta a\nabla \tilde{u}_{\varepsilon}V^{*}
= \int_{\tilde{\Omega}_T/\varepsilon}
\nabla V^{*}\cdot \delta a\nabla \tilde{u}_{\varepsilon}\\
&=\int_{\tilde{\Omega}_T/\varepsilon}
\big(\nabla v^*+\phi_j^{*}\nabla\check{h}_j+\nabla(\sigma_{(d+1)lj}^{*}\partial_{l}\check{h}_j)\big)
\cdot \delta a\nabla \tilde{u}_{\varepsilon}
+ \int_{\tilde{\Omega}_T/\varepsilon}
\nabla\phi_j^{*}\check{h}_j\cdot \delta a\nabla \tilde{u}_{\varepsilon}.
\end{aligned}
\end{equation}
By the same token, noting that
\begin{equation}\label{pde:3-5-2}
\big(\partial_t+\mathcal{L}_1\big)(\delta\phi_j)=\nabla\cdot \delta a(\nabla\phi_j+e_j)
\qquad \text{in}\quad\mathbb{R}^{d+1},
\end{equation}
there holds
\begin{equation}\label{f:3-5-3}
\begin{aligned}
I_2 &= \int_{\mathbb{R}^{d+1}}
\big(\partial_t+\mathcal{L}_1\big)(\delta\phi_j)Z_j^{*}
=^{\eqref{pde:3-5-2}}
\int_{\mathbb{R}^{d+1}}
\nabla\cdot \delta a(\nabla\phi_j+e_j)Z_j^{*}
=-\int_{\mathbb{R}^{d+1}}
\nabla Z_j^{*}\cdot \delta a(\nabla\phi_j+e_j)\\
&=-\int_{\mathbb{R}^{d+1}}
\big(\nabla z_j^{*}+\phi_k^{*}\nabla(\check{h}_k\hat{\varphi}_j)
+\nabla(\sigma_{(d+1)lk}^{*}\partial_l(\check{h}_k\hat{\varphi}_j))\big)\cdot \delta a(\nabla\phi_j+e_j)\\
& \qquad\qquad -
\int_{\mathbb{R}^{d+1}}
\nabla\phi_k^{*}\check{h}_k\cdot \delta a(\nabla\phi_j+e_j)\hat{\varphi}_j.
\end{aligned}
\end{equation}
Taking the boundary layer into account, and combining the last terms in the right-hand side of $\eqref{f:3-5-2}$ and $\eqref{f:3-5-3}$, we have
\begin{equation}\label{f:3-5-6}
\begin{aligned}
&\int_{\tilde{\Omega}_T/\varepsilon}
\nabla\phi_j^{*}\check{h}_j\cdot \delta a\nabla \tilde{u}_{\varepsilon}
- \int_{\mathbb{R}^{d+1}}
\nabla\phi_k^{*}\check{h}_k\cdot \delta a(\nabla\phi_j+e_j)\hat{\varphi}_j \\
&=\int_{\tilde{\Omega}_T/\varepsilon}
\nabla\phi_k^{*}\check{h}_k\cdot \delta a
\big(\nabla \tilde{u}_{\varepsilon} - \nabla\tilde{u}_0 - \nabla\phi_j\hat{\varphi}_j\big)
+ \int_{\tilde{\Omega}_T/\varepsilon}
\nabla\phi_k^{*}\check{h}_k\cdot \delta a \big(\nabla\tilde{u}_0-\hat{\varphi}\big).
\end{aligned}
\end{equation}

Plugging the equalities $\eqref{f:3-5-2}$, $\eqref{f:3-5-3}$, and $\eqref{f:3-5-6}$ back into $\eqref{f:3-5-5}$, and then combining
the equality $\eqref{f:3-5-4}$, it follows that
\begin{equation*}
\begin{aligned}
\delta H_\varepsilon &=
\int_{\tilde{\Omega}_T/\varepsilon} \big(\check{h} +
\nabla\phi_k^{*}\check{h}_k\big)\cdot \delta a
\big(\nabla \tilde{u}_{\varepsilon} - \nabla\tilde{u}_0 - \nabla\phi_j\hat{\varphi}_j\big)
+ \int_{\tilde{\Omega}_T/\varepsilon}
\nabla\phi_k^{*}\check{h}_k\cdot \delta a \big(\nabla\tilde{u}_0-\hat{\varphi}\big) \\
&-\int_{\mathbb{R}^{d+1}}
\big(\nabla z_j^{*}+\phi_k^{*}\nabla(\check{h}_k\hat{\varphi}_j)
+\nabla(\sigma_{(d+1)lk}^{*}\partial_l(\check{h}_k\hat{\varphi}_j))\big)\cdot \delta a(\nabla\phi_j+e_j)\\
&+\int_{\tilde{\Omega}_T/\varepsilon}
\big(\nabla v^*+\phi_j^{*}\nabla\check{h}_j
+\nabla(\sigma_{(d+1)lj}^{*}\partial_{l}\check{h}_j)\big)
\cdot \delta a\nabla \tilde{u}_{\varepsilon}.
\end{aligned}
\end{equation*}
This yields the desired equality $\eqref{eq:3-5-3}$.

\textbf{Step 3.} Arguments for the estimates $\eqref{pri:3-5-2a}$ and $\eqref{pri:3-5-2b}$.
By a change of variables, we have
\begin{equation}\label{f:3-5-9}
\begin{aligned}
& \bigg\langle
 \bigg(\int_{\mathbb{R}^{d+1}}dz
 \Big(\dashint_{Q_1(z)}\big|
 \check{h}_j(e_j+\nabla\phi_j^{*})\otimes
 \big(\nabla\tilde{w}_\varepsilon+\phi_j\nabla\hat{\varphi}_j+
\nabla(\sigma_{l(d+1)j}\partial_l\hat{\varphi}_j)\big)\big|\Big)^2
 \bigg)^{p}\bigg\rangle^{\frac{1}{p}}\\
&= \varepsilon^{d+2} \bigg\langle
 \bigg(\int_{\mathbb{R}^{d+1}}
 \Big(\dashint_{Q_\varepsilon(\cdot)}\big|
 h_j(e_j+\nabla\phi_j^{*\varepsilon})\otimes
 (\underbrace{\nabla w_\varepsilon+\varepsilon\phi_i^\varepsilon
 \nabla \varphi_i+\varepsilon^2\nabla(\sigma_{l(d+1)j}^{\varepsilon}
 \partial_l\varphi_j)}_{F_\varepsilon})\big|\Big)^2
 \bigg)^{p}\bigg\rangle^{\frac{1}{p}}.
\end{aligned}
\end{equation}
Then, by using H\"older's inequality, Minkowski's inequality,
and Corollary $\ref{corollary:2.1}$ (in that order), it is derived that
\begin{equation}\label{f:3-5-8}
\begin{aligned}
&\bigg\langle
 \bigg(\int_{\mathbb{R}^{d+1}}
 \Big(\dashint_{Q_\varepsilon(\cdot)}\big|
 h_j(e_j+\nabla\phi_j^{*\varepsilon})\otimes F_\varepsilon\big|\Big)^2
 \bigg)^{p}\bigg\rangle^{\frac{1}{p}}\\
&\leq
 \int_{\mathbb{R}^{d+1}}
 \Big\langle
 \Big(\dashint_{Q_\varepsilon(\cdot)}|
 (e_j+\nabla\phi_j^{*\varepsilon})|^{2s'}\Big)^{\frac{sp}{s'}}
 \Big\rangle^{\frac{1}{sp}}
 \Big\langle\Big(\dashint_{U_\varepsilon(\cdot)}|
 F_\varepsilon|^{2}\Big)^{s'p}\Big\rangle^{\frac{1}{s'p}}
 \Big(\dashint_{U_\varepsilon(\cdot)}|
 h_j|^{2s}\Big)^{\frac{1}{s}} dx\\
&\lesssim^{\eqref{pri:2.18}}
 \int_{\tilde{\Omega}_T}
 \Big\langle\Big(\dashint_{U_\varepsilon(\cdot)}|
 F_\varepsilon|^{2}\Big)^{s'p}\Big\rangle^{\frac{1}{s'p}}
 \Big(\dashint_{U_\varepsilon(\cdot)}|
 h_j|^{2s}\Big)^{\frac{1}{s}}\\
&\lesssim
\bigg(\int_{\tilde{\Omega}_T}
 \Big\langle\Big(\dashint_{U_\varepsilon(\cdot)}|
 F_\varepsilon|^{2}\Big)^{s'p}\Big\rangle^{\frac{1}{p}}
 \sigma^{2s'-1}\bigg)^{\frac{1}{s'}}
\bigg(\int_{\tilde{\Omega}_T}
\dashint_{U_\varepsilon(\cdot)}|
 h_j|^{2s}\sigma^{\frac{s}{s'}-2s}\bigg)^{\frac{1}{s}},
\end{aligned}
\end{equation}
where $\sigma(z):=\text{dist}(x,\partial\Omega_0)$ is defined similarly
as in $\eqref{eq:2.5}$, and
$\Omega_0\supseteq\Omega$ is a $C^1$ domain such that $\partial\Omega_0$ is
the hypersurface at a distance of $4\varepsilon$ from $\partial\Omega$
and parallel to it.

On the one hand, from  Fubini's theorem and the weighted Hardy inequality
(see e.g., \cite[Theorem 1.1]{Lehrback-14}), it follows that
\begin{equation}\label{f:3-5-7}
\begin{aligned}
\bigg(\int_{\tilde{\Omega}_T}
\dashint_{U_\varepsilon(\cdot)}|
 h|^{2s}
&{\sigma}^{\frac{s}{s'}-2s}\bigg)^{\frac{1}{s}}
\lesssim_d \bigg(\int_{\tilde{\Omega}_T}
|h|^{2s} \dashint_{Q_\varepsilon(\cdot)}{\sigma}^{\frac{s}{s'}-2s} \bigg)^{\frac{1}{s}}
\lesssim^{\eqref{f:3.41-ap}}_{d,s}
\bigg(\int_{\tilde{\Omega}_T}
|h|^{2s}{\sigma}^{\frac{s}{s'}-2s} dx\bigg)^{\frac{1}{s}}\\
&\lesssim \bigg(\int_{\tilde{\Omega}_{0T}}
|h|^{2s} {\sigma}^{\frac{s}{s'}-2s} \bigg)^{\frac{1}{s}}
 \lesssim \Big(\int_{\tilde{\Omega}_{0T}} |\nabla h|^{2s} {\sigma}^{\frac{s}{s'}}\Big)^{\frac{1}{s}}
 \lesssim R_0^{\frac{1}{s'}}\Big(\int_{\Omega_T} |\nabla h|^{2s} \Big)^{\frac{1}{s}},
\end{aligned}
\end{equation}
where $\tilde{\Omega}_{0T}:=\Omega_0\times(-1,T]$, and it is also noted that $\text{supp}(h)\subset\Omega_T$.
On the other hand, in view of Lemmas $\ref{lemma:4.4}$ and \ref{lemma:3.2-ap},
by setting the weight $\omega_{\sigma}:=\sigma^{2\underline{s}'-1}$ with $\underline{s}'$ being such that $0<s-\underline{s}'\ll 1$,
it is obtained that
\begin{equation*}
\begin{aligned}
&\bigg(\int_{\tilde{\Omega}_T}
 \Big\langle\Big(\dashint_{U_\varepsilon(\cdot)}|
 F_\varepsilon|^{2}\Big)^{s'p}\Big\rangle^{\frac{1}{p}}
{\sigma}^{2s'-1}\bigg)^{\frac{1}{s'}}
\lesssim \bigg(\int_{\tilde{\Omega}_T}
\Big\langle
\big(\dashint_{U_\varepsilon(\cdot)}|\nabla w_\varepsilon|^2\big)^{s'p}\Big\rangle^{\frac{1}{p}}
\omega_{\sigma}\bigg)^{\frac{1}{s'}}\\
&\qquad\qquad\qquad\qquad\qquad+
\bigg(\int_{\tilde{\Omega}_T}
 \Big\langle\Big(\dashint_{U_\varepsilon(\cdot)}|
 \varepsilon\phi_j^\varepsilon
 \nabla \varphi_j+\varepsilon^2\nabla(\sigma_{l(d+1)j}^{\varepsilon}
 \partial_l\varphi_j)|^{2}\Big)^{s'p}\Big\rangle^{\frac{1}{p}}
 \omega_{\sigma}^{2s'-1}\bigg)^{\frac{1}{s'}}\\
&\lesssim^{\eqref{pri:4.2}} \mu_d^2(R_0/\varepsilon)
\bigg\{
\Big(\int_{(O_{2\varepsilon})_T}|\nabla u_0|^{2s'}\omega_{\sigma}
\Big)^{\frac{1}{s'}}
+ \varepsilon^2
\Big(\int_{(\Omega\setminus O_{\varepsilon})_T}\big(|\nabla^2 u_0|^{2s'}
+|\partial_t u_0|^{2s'}\big)\sigma^{2s'-1}
\Big)^{\frac{1}{s'}}\bigg\}\\
&\lesssim^{\eqref{pri:3.4-ap},\eqref{pri:3.5-ap}}R_0^{2-\frac{1}{s'}}
\varepsilon^{2-}\mu_d^2(R_0/\varepsilon)\|\nabla f\|_{L^{2s'}(\Omega_T)}^2.
\end{aligned}
\end{equation*}
This, together with the estimates $\eqref{f:3-5-7}$, $\eqref{f:3-5-8}$, and $\eqref{f:3-5-9}$, yields
the stated estimate $\eqref{pri:3-5-2a}$.

Attention is now turned to the estimate $\eqref{pri:3-5-2b}$. In view of the notation in $\eqref{eq:3-5-2}$ and
$\eqref{eq:3-5-4}$, it follows that
\begin{equation}\label{f:3-5-12}
\begin{aligned}
&\bigg\langle
 \bigg(\int_{\tilde{\Omega}_T/\varepsilon}dz
 \Big(\dashint_{U_1(z)}\big|
 \nabla\phi^{*}_j \check{h}_j \otimes
 (\nabla\tilde{u}_0-\hat{\varphi})\big|\Big)^2
 \bigg)^{p}\bigg\rangle^{\frac{1}{p}} \\
&= \varepsilon^{d+2}\bigg\langle
 \bigg(\int_{\tilde{\Omega}_T}
 \Big(\dashint_{U_\varepsilon(\cdot)}\big|
 (\nabla\phi^{*\varepsilon}_j h_j)\otimes \big(
 \nabla u_0
- S_\varepsilon K_\varepsilon(\psi_{\varepsilon}\nabla u_0)\big)\big|\Big)^2
 \bigg)^{p}\bigg\rangle^{\frac{1}{p}},
\end{aligned}
\end{equation}
and it follows from Minkowski's inequality, H\"older's inequality, and Corollary $\ref{corollary:2.1}$ that
\begin{equation}\label{f:3-5-10}
\begin{aligned}
&\bigg\langle
 \bigg(\int_{\tilde{\Omega}_T}dz
 \Big(\dashint_{U_\varepsilon(z)}\big|
 (\nabla\phi^{*\varepsilon}_j h_j)\otimes \big(
 \nabla u_0
- S_\varepsilon K_\varepsilon(\psi_{\varepsilon}\nabla u_0)\big)\big|\Big)^2
 \bigg)^{p}\bigg\rangle^{\frac{1}{p}}\\
&\leq
\int_{\tilde{\Omega}_T}
\Big(\dashint_{U_\varepsilon(\cdot)}|h_j|^{2s}\Big)^{\frac{1}{s}}
\Big\langle\big(\dashint_{U_\varepsilon(\cdot)}\big|\nabla\phi_j^{*\varepsilon}\big|^{2s'}
\big)^{\frac{sp}{s'}}\Big\rangle^{\frac{1}{sp}}
\Big\langle\Big(\dashint_{U_\varepsilon(\cdot)}
\big|
\nabla u_0
- S_\varepsilon K_\varepsilon(\psi_{\varepsilon}\nabla u_0)\big|^2\Big)^{s'p}\Big\rangle^{\frac{1}{s'p}} \\
&\lesssim^{\eqref{pri:2.18}}_{\mu,\lambda_1,d,p}
\int_{\tilde{\Omega}_T}
\Big(\dashint_{U_\varepsilon(\cdot)}|h_j|^{2s}\Big)^{\frac{1}{s}}
\Big\langle\Big(\dashint_{U_\varepsilon(\cdot)}
\big|
\nabla u_0
- S_\varepsilon K_\varepsilon(\psi_{\varepsilon}\nabla u_0)\big|^2\Big)^{s'p}\Big\rangle^{\frac{1}{s'p}}\\
&\lesssim
\bigg(\int_{\tilde{\Omega}_T}
\big(\dashint_{U_\varepsilon(\cdot)}|
 h_j|^{2s}\big){\sigma}^{\frac{s}{s'}-2s}\bigg)^{\frac{1}{s}}
\bigg(\int_{\tilde{\Omega}_T}
 \Big\langle\Big(\dashint_{U_\varepsilon(\cdot)}|
 \nabla u_0
- S_\varepsilon K_\varepsilon(\psi_{\varepsilon}\nabla u_0)|^{2}\Big)^{s'p}\Big\rangle^{\frac{1}{p}}
 {\sigma}^{2s'-1}\bigg)^{\frac{1}{s'}}.
\end{aligned}
\end{equation}
A method similar to that used for $I_1$ in the estimate $\eqref{f:4.24}$ yields the following computation:
\begin{equation}\label{f:3-5-11}
\begin{aligned}
&\bigg(\int_{\tilde{\Omega}_T}
 \Big\langle\Big(\dashint_{U_\varepsilon(\cdot)}|
 \nabla u_0
- S_\varepsilon K_\varepsilon(\psi_{\varepsilon}\nabla u_0)|^{2}\Big)^{s'p}\Big\rangle^{\frac{1}{p}}
 {\sigma}^{2s'-1}\bigg)^{\frac{1}{s'}} \\
&\lesssim
\bigg(\int_{(O_{2\varepsilon})_T}|\nabla u_0|^{2s'}{\sigma}^{2s'-1}
\bigg)^{\frac{1}{s'}}
+ \varepsilon^2
\bigg(\int_{(\Omega\setminus O_{\varepsilon})_T}
\big(|\nabla^2 u_0|^{2s'} + |\partial_t u_0|^{2s'}
\big){\sigma}^{2s'-1}
\bigg)^{\frac{1}{s'}}\\
&\lesssim^{\eqref{pri:3.4-ap},\eqref{pri:3.5-ap}}R_0^{2-\frac{1}{s'}}
\varepsilon^{2-}\|\nabla f\|_{L^{2s'}(\Omega_T)}^2.
\end{aligned}
\end{equation}

By substituting the estimates $\eqref{f:3-5-7}$ and $\eqref{f:3-5-11}$
back into $\eqref{f:3-5-10}$, it is obtained that
\begin{equation*}
\bigg\langle
 \bigg(\int_{\Omega_T}
 \Big(\dashint_{U_\varepsilon(\cdot)}\big|
 (\nabla\phi^{*\varepsilon}_j h_j)\otimes \big(
 \nabla u_0
- S_\varepsilon K_\varepsilon(\psi_{\varepsilon}\nabla u_0)\big)\big|\Big)^2
 \bigg)^{p}\bigg\rangle^{\frac{1}{p}}
\lesssim
\varepsilon^{2-}\|R_0 \nabla f\|_{L^{2s'}(\Omega_T)}^2
\|\nabla h\|_{L^{2s}(\Omega_T)}^2,
\end{equation*}
and this together with $\eqref{f:3-5-12}$ implies the desired estimate $\eqref{pri:3-5-2b}$.

\textbf{Step 4.} Arguments for the estimates $\eqref{pri:3-5-3a}$ and $\eqref{pri:3-5-3b}$. The proof begins with the estimate $\eqref{pri:3-5-3a}$.
By a rescaling argument,
the auxiliary equations in $\eqref{pde:3-5-3}$ can be rewritten as
\begin{equation}\label{pde:3-5-3*}
\left\{\begin{aligned}
&\big(-\partial_t+\mathcal{L}_\varepsilon^*\big)(v^*_\varepsilon)
=\varepsilon^{d+2}\nabla\cdot\big[\varepsilon(a^{*\varepsilon}
\phi_{j}^{*\varepsilon}-\sigma_{j}^{*\varepsilon}
-a^{*\varepsilon}\partial\sigma_{(d+1)j}^{*\varepsilon})\nabla h_j \\
&\qquad\qquad
\qquad\qquad\qquad\quad-
\varepsilon^2 a^{*\varepsilon}\otimes\sigma_{(d+1)j}^{*\varepsilon}:\partial^2 h_j-\varepsilon^2\sigma_{(d+1)j}^{*\varepsilon}\partial_t h_j\big]
\quad \text{in}~~\tilde{\Omega}_T;\\
& v^*_\varepsilon =0 \quad \text{on}~~\tilde{S}_T;
\qquad
v^*_\varepsilon(\cdot,T) =0 \quad\text{on}~~\Omega,
\end{aligned}\right.
\end{equation}
where $\tilde{S}_T:=(\partial\Omega)\times
(-1,T]$ and
$v_\varepsilon^{*}(x,t):=\varepsilon v^{*}(x/\varepsilon,t/\varepsilon^2)=\varepsilon v^{*}(y,\tau)$ with $y\in\Omega/\varepsilon$ and $\tau\in(-1,T]/\varepsilon^2$.
It is also found that
\begin{equation*}
\begin{aligned}
&\bigg\langle
 \bigg(\int_{\tilde{\Omega}_T/\varepsilon}dz
 \Big(\dashint_{U_1(z)}
 \big|
 \big(\nabla v^* + \phi^{*}_j\nabla\check{h}_j
 +\nabla(\sigma_{(d+1)lj}^{*}\partial_{l}\check{h}_j)\big)
 \otimes\nabla\tilde{u}_\varepsilon\big|\Big)^2
 \bigg)^{p}\bigg\rangle^{\frac{1}{p}}\\
 & = \bigg\langle
 \bigg(\int_{\tilde{\Omega}_T}
 \Big(\dashint_{U_\varepsilon(\cdot)}\big|
 \big(\varepsilon^{-\frac{d+2}{2}}\nabla v_\varepsilon^{*}+\varepsilon^{1+\frac{d+2}{2}}\phi_j^{*\varepsilon}\nabla h_j+\varepsilon^{2+\frac{d+2}{2}}
 \nabla(\sigma_{(d+1)lj}^{*\varepsilon}\partial_{l}h_j)\big)\otimes \nabla u_\varepsilon\big|\Big)^2
 \bigg)^{p}\bigg\rangle^{\frac{1}{p}}.
\end{aligned}
\end{equation*}
Therefore, the right-hand side above can be bounded by
\begin{equation}\label{f:3-5-20}
\begin{aligned}
&\varepsilon^{-d-2}\underbrace{\bigg\langle
 \bigg(\int_{\tilde{\Omega}_T}
 \dashint_{U_\varepsilon(\cdot)}\big|
 \nabla v_\varepsilon^{*}|^2
 \dashint_{U_\varepsilon(\cdot)}
 |\nabla u_\varepsilon|^2
 \bigg)^{p}\bigg\rangle^{\frac{1}{p}}}_{I_1} \\
&\qquad\qquad +\varepsilon^{4+d}
 \underbrace{\bigg\langle
 \bigg(\int_{\tilde{\Omega}_T}\dashint_{U_\varepsilon(\cdot)}
 |\phi_j^{*\varepsilon}\nabla h_j|^2
 \dashint_{U_\varepsilon(\cdot)}|\nabla u_\varepsilon|^2\bigg)^p
 \bigg\rangle^{\frac{1}{p}}}_{I_2}\\
&\qquad\qquad \qquad\qquad
\varepsilon^{6+d}
 \underbrace{\bigg\langle
 \bigg(\int_{\tilde{\Omega}_T}\dashint_{U_\varepsilon(\cdot)}
 |\nabla(\sigma_{(d+1)lj}^{*\varepsilon}\partial_{l}h_j)|^2
 \dashint_{U_\varepsilon(\cdot)}|\nabla u_\varepsilon|^2\bigg)^p
 \bigg\rangle^{\frac{1}{p}}}_{I_3}.
\end{aligned}
\end{equation}

Now, the corresponding estimates are handled term by term.
$I_1$ is addressed first. It follows from Minkowski's inequality and H\"older's inequality that,
\begin{equation}\label{f:3-5-18}
\begin{aligned}
I_1
&\leq
 \int_{\tilde{\Omega}_T}
 \Big\langle\Big(
 \dashint_{U_\varepsilon(\cdot)}\big|
 \nabla v_\varepsilon^{*}|^2\Big)^{sp}
 \Big\rangle^{\frac{1}{sp}}
 \Big\langle
 \Big(\dashint_{U_\varepsilon(\cdot)}
 |\nabla u_\varepsilon|^2\Big)^{s'p}
 \Big\rangle^{\frac{1}{s'p}} \\
& \leq
\bigg(\int_{\tilde{\Omega}_T}
 \Big\langle\Big(
 \dashint_{U_\varepsilon(\cdot)}\big|
 \nabla v_\varepsilon^{*}|^2\Big)^{sp}
 \Big\rangle^{\frac{1}{p}}\bigg)^{\frac{1}{s}}
 \bigg(\int_{\tilde{\Omega}_T}
 \Big\langle
 \Big(\dashint_{U_\varepsilon(\cdot)}
 |\nabla u_\varepsilon|^2\Big)^{s'p}
 \Big\rangle^{\frac{1}{p}}\bigg)^{\frac{1}{s'}}.
\end{aligned}
\end{equation}
In view of the equations $\eqref{pde:3-5-3*}$, by using Theorem
$\ref{thm:C-Z}$ (with $\omega_\sigma =1$) and
Corollary $\ref{corollary:2.1}$, it is obtained that
\begin{equation}\label{f:3-5-19}
\begin{aligned}
&\bigg(\int_{\tilde{\Omega}_T}
 \Big\langle\Big(
 \dashint_{U_\varepsilon(x)}\big|
 \nabla v_\varepsilon^{*}|^2\Big)^{sp}
 \Big\rangle^{\frac{1}{p}}dx\bigg)^{\frac{1}{s}}\\
&\lesssim^{\eqref{pri:B}} \varepsilon^{2d+4}
\bigg(\int_{\tilde{\Omega}_T}
 \Big\langle\Big(
 \dashint_{U_\varepsilon(\cdot)}\big|
 \varepsilon\varpi_j^{*\varepsilon}\nabla h_j
 +\varepsilon^2\tilde{\varpi}_j^{*\varepsilon}(\partial_t,\partial^2)h_j|^2\Big)^{s\bar p}
 \Big\rangle^{\frac{1}{\bar p}}\bigg)^{\frac{1}{s}}\\
& \lesssim
\varepsilon^{2d+4}
\bigg(\int_{\tilde{\Omega}_T}
 \Big\langle
 \Big(\dashint_{U_\varepsilon(\cdot)}|
 \varepsilon\varpi_j^{*\varepsilon}|^{2s'}\Big)^{\frac{s\bar{p}}{s'}}
 \Big\rangle^{\frac{1}{\bar{p}}}
 \dashint_{U_\varepsilon(\cdot)}|\nabla h_j|^{2s}
 \bigg)^{\frac{1}{s}} \\
&\qquad\qquad\qquad\qquad +
\varepsilon^{2d+4}
\bigg(\int_{\tilde{\Omega}_T}
 \Big\langle
 \Big(\dashint_{U_\varepsilon(\cdot)}|
 \varepsilon^2\tilde{\varpi}_j^{*\varepsilon}|^{2s'}\Big)^{\frac{s\bar{p}}{s'}}
 \Big\rangle^{\frac{1}{\bar{p}}}
 \dashint_{U_\varepsilon(\cdot)}|(\partial_t,\partial^2) h_j|^{2s}
 \bigg)^{\frac{1}{s}}\\
& \lesssim^{\eqref{pri:2.18},\eqref{pri:2.19}}
\varepsilon^{2d+6}
\Big(\int_{\Omega_T}|\nabla h|^{2s}\Big)^{\frac{1}{s}}
+ \varepsilon^{2d+8}\mu_{d}^2(R_0/\varepsilon)
\Big(\int_{\Omega_T}|(\partial_t,\partial^2) h|^{2s}\Big)^{\frac{1}{s}}.
\end{aligned}
\end{equation}
On account of the first equation in $\eqref{pde:3-5-4}$, employing
Theorem $\ref{thm:C-Z}$ again, we have
\begin{equation}\label{f:3-5-13}
\begin{aligned}
\bigg(\int_{\Omega_T}
 \Big\langle
 \Big(\dashint_{U_\varepsilon(\cdot)}
 |\nabla u_\varepsilon|^2\Big)^{s'p}
 \Big\rangle^{\frac{1}{p}}\bigg)^{\frac{1}{s'}}
\lesssim^{\eqref{pri:B}}
\bigg(\int_{\Omega_T}
 \Big\langle
 \Big(\dashint_{U_\varepsilon(\cdot)}
 |f|^2\Big)^{s'\bar{p}}
 \Big\rangle^{\frac{1}{\bar{p}}}\bigg)^{\frac{1}{s'}}
\lesssim
\Big(\int_{\Omega_T}
 |f|^{2s'}\Big)^{\frac{1}{s'}}.
\end{aligned}
\end{equation}
Collecting the estimates $\eqref{f:3-5-18}$, $\eqref{f:3-5-19}$, and $\eqref{f:3-5-13}$, one arrives at
\begin{equation}\label{f:3-5-21}
\varepsilon^{-d-2} I_1 \lesssim \bigg[\varepsilon^{d+4}
\Big(\int_{\Omega_T}|\nabla h|^{2s}\Big)^{\frac{1}{s}}
+ \varepsilon^{d+6}\mu_{d}^2(R_0/\varepsilon)
\Big(\int_{\Omega_T}|(\partial_t,\partial^2) h|^{2s}\Big)^{\frac{1}{s}}
\bigg]
\Big(\int_{\Omega_T}
 |f|^{2s'}\Big)^{\frac{1}{s'}}.
\end{equation}

We continue to handle $I_2$ and $I_3$ in $\eqref{f:3-5-20}$, a computation similar to that for $I_1$ yields
\begin{equation*}
\begin{aligned}
I_2
&\leq
\int_{\tilde{\Omega}_T}
 \Big\langle\Big(
 \dashint_{U_\varepsilon(\cdot)}\big|
 \phi_j^{*\varepsilon}\nabla h_j|^2\Big)^{sp}
 \Big\rangle^{\frac{1}{sp}}
 \Big\langle
 \Big(\dashint_{U_\varepsilon(\cdot)}
 |\nabla u_\varepsilon|^2\Big)^{s'p}
 \Big\rangle^{\frac{1}{s'p}}\\
& \leq
\bigg(\int_{\tilde{\Omega}_T}
 \Big\langle\Big(
 \dashint_{U_\varepsilon(\cdot)}\big|
 \phi_j^{*\varepsilon}\nabla h_j|^2\Big)^{sp}
 \Big\rangle^{\frac{1}{p}}\bigg)^{\frac{1}{s}}
 \bigg(\int_{\Omega_T}
 \Big\langle
 \Big(\dashint_{U_\varepsilon(\cdot)}
 |\nabla u_\varepsilon|^2\Big)^{s'p}
 \Big\rangle^{\frac{1}{p}}dx\bigg)^{\frac{1}{s'}} \\
& \lesssim^{\eqref{pri:2.18},\eqref{f:3-5-13}}
\Big(\int_{\Omega_T}|\nabla h|^{2s}\Big)^{\frac{1}{s}}
\Big(\int_{\Omega_T}
 |f|^{2s'}\Big)^{\frac{1}{s'}},
\end{aligned}
\end{equation*}
and
\begin{equation*}
\begin{aligned}
I_3
&\leq
\int_{\tilde{\Omega}_T}
 \Big\langle\Big(
 \dashint_{U_\varepsilon(\cdot)}|\nabla(\sigma_{(d+1)lj}^{*\varepsilon}\partial_{l}h_j)
 |^2\Big)^{sp}
 \Big\rangle^{\frac{1}{sp}}
 \Big\langle
 \Big(\dashint_{U_\varepsilon(\cdot)}
 |\nabla u_\varepsilon|^2\Big)^{s'p}
 \Big\rangle^{\frac{1}{s'p}}\\
& \lesssim^{\eqref{pri:2.19},\eqref{f:3-5-13}} \bigg[\varepsilon^{-2}
\Big(\int_{\Omega_T}|\nabla h|^{2s}\Big)^{\frac{1}{s}}
+ \mu_{d}^2(R_0/\varepsilon)
\Big(\int_{\Omega_T}|\nabla^2h|^{2s}\Big)^{\frac{1}{s}}
\bigg]
\Big(\int_{\Omega_T}
 |f|^{2s'}\Big)^{\frac{1}{s'}}.
\end{aligned}
\end{equation*}
As a result, plugging the above two estimates and $\eqref{f:3-5-21}$ back into the estimate $\eqref{f:3-5-20}$,
we can derive the stated estimate $\eqref{pri:3-5-3a}$.

Finally, we turn to the estimate $\eqref{pri:3-5-3b}$. Before the proof is formally begun, as a preparation, we assert that there hold the following estimates:
\begin{subequations}
\begin{align}
&\bigg(\int_{\tilde{\Omega}_T}
\Big\langle
\Big(\dashint_{U_\varepsilon(\cdot)}\big|\varepsilon
 \varpi_k^{*\varepsilon}\varphi_j|^{2s'}\Big)^{\frac{s\bar{p}}{s'}}
 \Big\rangle^{\frac{s'}{s\bar{p}}}\bigg)^{\frac{1}{s'}}
\lesssim
\varepsilon^2
\Big(\int_{\Omega_T} |f|^{2s'}\Big)^{\frac{1}{s'}};
\label{f:3-5-14}\\
& \bigg(\int_{\tilde{\Omega}_T}
\Big\langle
\Big(\dashint_{U_\varepsilon(\cdot)}\big|\varepsilon^2
 \tilde{\varpi}_k^{*\varepsilon}\varphi_j|^{2s'}\Big)^{\frac{s\bar{p}}{s'}}
\Big\rangle^{\frac{s'}{s\bar{p}}}\bigg)^{\frac{1}{s'}}
\lesssim \varepsilon^4\mu_d^2(R_0/\varepsilon)
\Big(\int_{\Omega_T} |f|^{2s'}\Big)^{\frac{1}{s'}};
\label{f:3-5-17}\\
&\bigg(\int_{\tilde{\Omega}_T}
\Big\langle
\Big(\dashint_{U_\varepsilon(\cdot)}\big|\varepsilon
 \varpi_k^{*\varepsilon}\nabla\varphi_j|^{2s'}\Big)^{\frac{s\bar{p}}{s'}}
 \Big\rangle^{\frac{s'}{s\bar{p}}}\omega_{\sigma}\bigg)^{\frac{1}{s'}}
\lesssim R_0^{2-\frac{1}{s'}}
\varepsilon^{2-}\|\nabla f\|_{L^{2s'}(\Omega_T)}^2;
\label{f:3-5-15}\\
&\bigg(\int_{\tilde{\Omega}_T}
\Big\langle
\Big(\dashint_{U_\varepsilon(\cdot)}\big|\varepsilon^2
 \tilde{\varpi}_k^{*\varepsilon}(\partial_t,\partial^2)(\varphi_j)
 |^{2s'}\Big)^{\frac{s\bar{p}}{s'}}
 \Big\rangle^{\frac{s'}{s\bar{p}}}\omega_{\sigma}\bigg)^{\frac{1}{s'}}
\lesssim
R_0^{2-\frac{1}{s'}}
\varepsilon^{2-}\mu_d^2(R_0/\varepsilon)\|\nabla f\|_{L^{2s'}(\Omega_T)}^2.
\label{f:3-5-16}
\end{align}
\end{subequations}

We now show the proof of the estimate $\eqref{pri:3-5-3b}$.
By H\"older's inequality, Minkowski's inequality,
Corollary $\ref{corollary:2.1}$, and a triangle inequality, it is obtained that
\begin{equation}\label{f:3-5-22}
\begin{aligned}
&\bigg\langle
 \bigg(\int_{\mathbb{R}^{d+1}}
 \Big(\dashint_{Q_1(\cdot)}\big|
 \big(\nabla z_j^{*}+\phi_k^{*}\nabla(\check{h}_k\hat{\varphi}_j)
+\nabla(\sigma_{(d+1)lk}^{*}\partial_l(\check{h}_k\hat{\varphi}_j))\big)\otimes
 (\nabla\phi_j+e_j)\big|\Big)^2
 \bigg)^{p}\bigg\rangle^{\frac{1}{p}}\\
&\lesssim^{\eqref{pri:2.18}}
\int_{\mathbb{R}^{d+1}}
\Big\langle \big(\dashint_{Q_1(\cdot)}\big|
 \big(\nabla z^*+\phi_k^{*}\nabla(\check{h}_k\hat{\varphi})
 +\nabla(\sigma_{(d+1)lk}^{*}\partial_l(\check{h}_k\hat{\varphi}))
 \big)\big|^2\big)^{ps}
\Big\rangle^{\frac{1}{ps}}\\
&\lesssim \int_{\mathbb{R}^{d+1}}
\Big\langle \big(\dashint_{Q_1(\cdot)}\big|
 \nabla z^*\big|^2\big)^{ps}
\Big\rangle^{\frac{1}{ps}}
+
\int_{\mathbb{R}^{d+1}}
\Big\langle \big(\dashint_{Q_1(\cdot)}\big|\phi_k^{*}\nabla(\check{h}_k\hat{\varphi})
+\nabla(\sigma_{(d+1)lk}^{*}\partial_l(\check{h}_k\hat{\varphi}))\big|^2\big)^{ps}
\Big\rangle^{\frac{1}{ps}}.
\end{aligned}
\end{equation}

Using the annealed Calder\'on-Zygmund estimate again (see Theorem $\ref{thm:C-Z}$ for the case of $\mathbb{R}^{d+1}$), we can derive that
\begin{equation*}
\begin{aligned}
&\int_{\mathbb{R}^{d+1}}dz
\Big\langle \big(\dashint_{Q_1(z)}\big|
 \nabla z_j^*\big|^2\big)^{ps}
\Big\rangle^{\frac{1}{sp}}
\lesssim^{\eqref{pri:B}}
\int_{\mathbb{R}^{d+1}}dz
 \Big\langle\Big(
 \dashint_{Q_1(z)}\big|
 \varpi_k^{*}\nabla (\check{h}_k\hat{\varphi}_j)
 +\tilde{\varpi}_k^{*}(\partial_t,\partial^2)(\check{h}_k\hat{\varphi}_j)|^2
 \Big)^{s\bar p}
 \Big\rangle^{\frac{1}{s\bar p}}\\
&=\varepsilon^{d+2}\underbrace{\int_{\mathbb{R}^{d+1}}
 \Big\langle\Big(
 \dashint_{U_\varepsilon(\cdot)}\big|
 \varepsilon\varpi_k^{*\varepsilon}\nabla (h_k\varphi_j)|^2
 \Big)^{s\bar p}
 \Big\rangle^{\frac{1}{s\bar p}}}_{J_1}
 +
 \varepsilon^{d+2}\underbrace{\int_{\mathbb{R}^{d+1}}
 \Big\langle\Big(
 \dashint_{U_\varepsilon(\cdot)}\big|\varepsilon^2
 \tilde{\varpi}_k^{*\varepsilon}
 (\partial_t,\partial^2)(h_k\varphi_j)|^2
 \Big)^{s\bar p}
 \Big\rangle^{\frac{1}{s\bar p}}}_{J_2}.
\end{aligned}
\end{equation*}
It follows from H\"older's inequality and the estimates prepared above that
\begin{equation*}
\begin{aligned}
J_1
&\lesssim
\int_{\mathbb{R}^{d+1}}
\Big\langle
\Big(\dashint_{U_\varepsilon(\cdot)}\big|\varepsilon
 \varpi_k^{*\varepsilon}\varphi_j\nabla h_k|^2\Big)^{s\bar{p}}\Big\rangle^{\frac{1}{s\bar{p}}}
+
\int_{\mathbb{R}^{d+1}}
\Big\langle
\Big(\dashint_{U_\varepsilon(\cdot)}\big|\varepsilon
 \varpi_k^{*\varepsilon}\nabla \varphi_j h_k|^2\Big)^{s\bar{p}}\Big\rangle^{\frac{1}{s\bar{p}}} \\
& \lesssim
\int_{\tilde{\Omega}_T}
\Big\langle
\Big(\dashint_{U_\varepsilon(\cdot)}\big|\varepsilon
 \varpi_k^{*\varepsilon}\varphi_j|^{2s'}\Big)^{\frac{s\bar{p}}{s'}}
 \Big\rangle^{\frac{1}{s\bar{p}}}
 \Big(\dashint_{U_\varepsilon(\cdot)}|\nabla h_k|^{2s}\Big)^{\frac{1}{s}} \\
&\qquad\qquad +
\bigg(\int_{\tilde{\Omega}_T}
\Big\langle
\Big(\dashint_{U_\varepsilon(\cdot)}\big|\varepsilon
 \varpi_k^{*\varepsilon}\nabla\varphi_j|^{2s'}\Big)^{\frac{s\bar{p}}{s'}}
 \Big\rangle^{\frac{s'}{s\bar{p}}}{\sigma}^{2s'-1}\bigg)^{\frac{1}{s'}}
 \bigg(\int_{\tilde{\Omega}_T}\Big(\dashint_{U_\varepsilon(\cdot)}|h_k|^{2s}\Big)
 {\sigma}^{\frac{s}{s'}-2s}\bigg)^{\frac{1}{s}}\\
&\lesssim^{\eqref{f:3-5-14},\eqref{f:3-5-15}}_{\eqref{f:3-5-7}}
\varepsilon^2\bigg(\int_{\Omega_T}
|f|^{2s'}\bigg)^{\frac{1}{s'}}
\bigg(\int_{\Omega_T}
|\nabla h|^{2s}\bigg)^{\frac{1}{s}}
+ \varepsilon^{2-}
\bigg(\int_{\Omega_T}
|R_0\nabla f|^{2s'}\bigg)^{\frac{1}{s'}}
\bigg(\int_{\Omega_T}
|\nabla h|^{2s}\bigg)^{\frac{1}{s}}.
\end{aligned}
\end{equation*}
A computation similar to that given for $J_1$ leads to
\begin{equation*}
\begin{aligned}
J_2
&\lesssim
\int_{\mathbb{R}^{d+1}}
\Big\langle
\Big(\dashint_{U_\varepsilon(\cdot)}\big|\varepsilon^2
 \tilde{\varpi}_k^{*\varepsilon}(\partial_t,\partial^2)(\varphi_j)h_k|^2
 \Big)^{s\bar{p}}\Big\rangle^{\frac{1}{s\bar{p}}}
 +
\int_{\mathbb{R}^{d+1}}
\Big\langle
\Big(\dashint_{U_\varepsilon(\cdot)}\big|\varepsilon^2
 \tilde{\varpi}_k^{*\varepsilon}\varphi_j(\partial_t,\partial^2)(h_k)|^2
 \Big)^{s\bar{p}}\Big\rangle^{\frac{1}{s\bar{p}}}\\
&\lesssim^{\eqref{f:3-5-16},\eqref{f:3-5-7}}_{\eqref{f:3-5-17}} \varepsilon^{2-}\mu_d^2(R_0/\varepsilon)
\bigg(\int_{\Omega_T}
|R_0\nabla f|^{2s'}\bigg)^{\frac{1}{s'}}
\bigg(\int_{\Omega_T}
|\nabla h|^{2s}\bigg)^{\frac{1}{s}}\\
&\qquad\qquad\qquad\qquad + \varepsilon^{4}\mu_d^2(R_0/\varepsilon)
\bigg(\int_{\Omega_T}
|f|^{2s'}\bigg)^{\frac{1}{s'}}
\bigg(\int_{\Omega_T}
|(\partial_t,\partial^2)(h)|^{2s}\bigg)^{\frac{1}{s}}.
\end{aligned}
\end{equation*}

Then, by combining the above two estimates, it is found that
\begin{equation*}
\begin{aligned}
\int_{\mathbb{R}^{d+1}}dz
\Big\langle \big(\dashint_{Q_1(z)}\big|
 \nabla z_j^*\big|^2\big)^{ps}
\Big\rangle^{\frac{1}{sp}}
&\lesssim
\varepsilon^{d+4-}\mu_d^2(R_0/\varepsilon)
\bigg(\int_{\Omega_T}
|R_0\nabla f|^{2s'}\bigg)^{\frac{1}{s'}}
\bigg(\int_{\Omega_T}
|\nabla h|^{2s}\bigg)^{\frac{1}{s}}\\
&+ \varepsilon^{d+6}\mu_d^2(R_0/\varepsilon)
\bigg(\int_{\Omega_T}
|f|^{2s'}\bigg)^{\frac{1}{s'}}
\bigg(\int_{\Omega_T}
|(\partial_t,\partial^2)(h)|^{2s}\bigg)^{\frac{1}{s}}.
\end{aligned}
\end{equation*}
It is noted that the estimate of the last term in $\eqref{f:3-5-22}$ can be absorbed into the estimates for $J_1$ and $J_2$. Therefore, the above estimate has already implied the desired estimate $\eqref{pri:3-5-3b}$.

Finally, it is pointed out that the proofs of the above assertions from $\eqref{f:3-5-14}$ to $\eqref{f:3-5-16}$ merely rely on the conclusion of Corollary $\ref{corollary:2.1}$ and Lemmas $\ref{lemma:5}$, $\ref{lemma:2.4}$, and \ref{lemma:3.2-ap}. To be short, here we briefly indicate the proof of the estimate $\eqref{f:3-5-15}$, i.e.,
\begin{equation*}
\begin{aligned}
&\bigg(\int_{\tilde{\Omega}_T}
\Big\langle
\Big(\dashint_{U_\varepsilon(\cdot)}\big|\varepsilon
 \varpi_k^{*\varepsilon}\nabla\varphi_j|^{2s'}\Big)^{\frac{s\bar{p}}{s'}}
 \Big\rangle^{\frac{s'}{s\bar{p}}}\omega_{\sigma}\bigg)^{\frac{1}{s'}}
\lesssim^{\eqref{pri:2.18},\eqref{pri:3.42a},\eqref{pri:2.17}}
\Big(\int_{\Omega_T} |\nabla(\eta\partial_ju_0)|^{2s'}
{\sigma}^{2\underline{s}'-1}
\Big)^{\frac{1}{s'}} \\
&\lesssim
\bigg(\int_{(O_{2\varepsilon})_T}|\nabla u_0|^{2s'}{\sigma}^{2\underline{s}'-1}
\bigg)^{\frac{1}{s'}}
+\varepsilon^2
\bigg(\int_{(\Omega\setminus O_{\varepsilon})_T}|\nabla^2 u_0|^{2s'}{\sigma}^{2\underline{s}'-1}
\bigg)^{\frac{1}{s'}}
\lesssim^{\eqref{pri:3.4-ap},\eqref{pri:3.5-ap}} R_0^{2-\frac{1}{s'}}
\varepsilon^{2-}\|\nabla f\|_{L^{2s'}(\Omega_T)}^2,
\end{aligned}
\end{equation*}
and the remaining estimates can be obtained similarly. This concludes the entire proof.
\end{proof}

\subsection{Proof of Theorems $\ref{thm:1.1}$ and  $\ref{thm:1.2}$}\label{subsection:5.4}

\noindent
Given that the properties and lemmas required to prove Theorems $\ref{thm:1.1}$ and $\ref{thm:1.2}$ are all provided with detailed proofs
in this section, we mainly present the flow of the proofs in Fig. $\ref{pic:5.1}$ and Fig. $\ref{pic:5.2}$, respectively.

\begin{figure}[htbp]
\centering 
\resizebox{0.85\textwidth}{!}{
\begin{tikzpicture}[
    node distance=0.7cm and 0.8cm,
    >=Stealth, 
    box/.style={
        rectangle, draw,
        minimum width=3.2cm, minimum height=0.9cm,
        align=center, text width=3cm
    },
    smallbox/.style={
        rectangle, draw,
        minimum width=3.5cm, minimum height=0.7cm,
        align=center, text width=3cm
    },
    bigbox/.style={
    rectangle, draw,
        minimum width=6cm, minimum height=1.5cm,
        align=center, text width=5.5cm
    },
    bbigbox/.style={
    rectangle, draw,
        minimum width=8cm, minimum height=1.8cm,
        align=center, text width=7.5cm
    },
    cbigbox/.style={
    rectangle, draw, dashed,
    inner sep=10pt, 
    rounded corners=3pt 
    }
]

\node[smallbox] (n11) {Lemma $\ref{lemma:3.1}$};
\node[smallbox, below=of n11] (n21) {Lemma $\ref{lemma:4.3}$};
\node[smallbox, below=of n21] (n31) {Lemma $\ref{lemma:4.4}$};
\node[smallbox, below=of n31] (n41) {Lemma $\ref{lemma:4.2}$};

\node[smallbox, left=1.5cm of n11] (n10) {Theorem $\ref{thm:1.0}$};
\node[smallbox, left=1.5cm of n21] (n20) {Lemma $\ref{lemma:2.4}$};
\node[smallbox, left=1.5cm of n31] (n30) {Lemma $\ref{lemma:2.5}$};
\node[smallbox, left=1.5cm of n41] (n40) {Lemma $\ref{lemma:5}$};

\node[smallbox, left=of n20, draw=white, text opacity=0] (n2r) {};
\node[box, below=0.15cm of n2r] (n3r) {Lemma $\ref{lemma:3.4}$};

\node[box, right=3cm of n21] (n22) {Proposition $\ref{P:4.2}$};
\node[box, right=3cm of n31] (n23) {Proposition $\ref{P:4.1}$};


\node[cbigbox,fit=(n10) (n20) (n30) (n40)] (f1*){};

\node[dashed, draw, fit = (n11) (n21), inner sep = 0.2cm](ff1){};

\node[dashed, draw, fit= (n31) (n41), inner sep=0.2cm] (ff3) {};

\node[box] (n31) at([xshift=10cm,yshift=1.5cm]ff3.east){\textbf{Theorem \ref{thm:1.1}}};



\draw[->] (n3r) to[bend left=15] node[pos=0.5, left, sloped] {} (ff1);
\draw[->] (n3r) to[bend right=15] node[pos=0.8, above, sloped] {
$+$Theorem  $\ref{thm:C-Z}$} (ff3);

\draw[->] (ff3)
-- node[pos=0.5, above, sloped] {$+$Lemma \ref{lemma:3.2-ap}} (n23);
\draw[->] (ff1)
-- node[pos=0.5, below, sloped] {$+$Lemma \ref{lemma:3.2-ap}} (n22);

\draw[->] (n23)
-- (n31);

\draw[->] (n22)
-- (n31);




%


%

\end{tikzpicture} %
}
\caption{On the proof structure of Theorem $\ref{thm:1.1}$} 
\label{pic:5.1}
\end{figure}

\begin{figure}[htbp]
\centering 
\resizebox{0.9\textwidth}{!}{
\begin{tikzpicture}[
    node distance=0.7cm and 0.8cm,
    >=Stealth, 
    box/.style={
        rectangle, draw,
        minimum width=3.2cm, minimum height=0.9cm,
        align=center, text width=3cm
    },
    smallbox/.style={
        rectangle, draw,
        minimum width=3.5cm, minimum height=0.7cm,
        align=center, text width=3cm
    },
    bigbox/.style={
    rectangle, draw,
        minimum width=6cm, minimum height=1.5cm,
        align=center, text width=5.5cm
    },
    bbigbox/.style={
    rectangle, draw,
        minimum width=8cm, minimum height=1.7cm,
        align=center, text width=7.5cm
    },
    cbigbox/.style={
    rectangle, draw, dashed,
    inner sep=10pt, 
    rounded corners=3pt 
    }
]

\node[bigbox, draw=white, text opacity=0] (k11) {}; 
\node[smallbox, right=of k11, draw=white, text opacity=0] (k12) {}; 
\node[bigbox, right=of k12] (K13) {Lemma $\ref{lemma:4.4}$};
\node[smallbox, right=of K13, draw=white, text opacity=0] (k14) {}; 
\node[bigbox, right=2cm of k14, draw=white, text opacity=0] (k15) {}; 
\node[bigbox,below=of k11] (K21) {Theorem $\ref{thm:C-Z}$};
\node[smallbox, right=of K21,draw=white, text opacity=0] (k22) {Lemma}; 
\node[bbigbox, below=of K13] (K23) {Theorem $\ref{thm:1.0}$;\qquad Lemma $\ref{lemma:2.4}$};
\node[smallbox, right=of K23, draw=white, text opacity=0] (k24) {Lemma}; 
\node[bigbox, below=3.8cm of k15] (k25) {\textbf{Theorem \ref{thm:1.2}}};

\node[smallbox, below=1.5cm of K21] (K31) {Corollary $\ref{corollary:2.1}$};
\node[smallbox, right=of K31,draw=white, text opacity=0] (k32) {Lemma}; 
\node[bbigbox, below=of K23] (K33) {Lemma $\ref{lemma:2.5}$;\qquad Lemma $\ref{lemma:5}$};
\node[smallbox, right=of K33, draw=white, text opacity=0] (k34) {Lemma}; 
\node[smallbox, right=of k34,draw=white, text opacity=0] (k25) {}; 

\node[smallbox,below=of K31] (K41) {Lemma \ref{lemma:3.2-ap}};
\node[bigbox, below=of K33] (K43) {Proposition $\ref{P:3.1}$};


\node[cbigbox,fit=(K31) (K41)] (f1){};

\node[cbigbox, fit=(K23) (K33)] (f2) {~\\~\textbf{Weights} $\&$ \textbf{Correctors}~};

\draw[->,thick] (K21) to[bend right=30]
coordinate[pos=0.5] (mid1)
node[pos=0.5, below, sloped]
 {} (K43);
\draw[->] (K21) to[bend right=18] node[pos=0.5, left, sloped] {} (K13);


\draw[thick,
decoration={markings, mark=at position 0.7 with {\arrow{>}}},
postaction={decorate}] (f2) to[bend right=25]
    (mid1) node[midway, right] {};

\draw[->] (K13) to[bend left=18] node[pos=0.5, above, sloped] {$+$ Lemma \ref{lemma:3.2-ap}} (k25);

\draw[->,thick] (K43) -- (k25);

\draw[->,dashed]  (K13) -- (K43);

%
%
%

\node at ($(K21.south)!0.5!(f1.north)$) {$+$};

\node at ($(K21.east)!0.7!(f2.west)$) {$+$};

\end{tikzpicture} %
}
\caption{On the proof structure of Theorem $\ref{thm:1.2}$} 
\label{pic:5.2}
\end{figure}

\medskip
\noindent
\textbf{Proof of Theorem $\ref{thm:1.2}$.}
The desired estimate $\eqref{pri:1.4}$ can be easily derived from
the corresponding result $\eqref{pri:3-5-1}$ in Proposition $\ref{P:3.1}$, while the estimate $\eqref{pri:1.3}$ follows from
Lemma $\ref{lemma:4.4}$ coupled with Lemma $\ref{lemma:3.2-ap}$
as shown in Fig. $\ref{pic:5.2}$, and the details indeed have been presented in this section.
\qed

\bigskip
\noindent
\textbf{Acknowledgements}~
The second author would like to express the gratitude to
Prof. Felix Otto and Prof. Yu Gu for
the insightful discussions on the estimates of correctors during the summer of 2025 and 2026.
The first author was supported by the National Natural Science Foundation of China (Grant NO. 12371096).
The second author was supported by the Natural Science Foundation of Gansu Province, China (Grant No. 26JRRA174).
This project partially supported by the Tianyuan Fund for Mathematics of the National Natural Science Foundation of China (Grant No. 12326102).

\bigskip
\noindent
\textbf{Data Availability Statement}  Data sharing is not applicable to this article as no datasets were generated or analysed during the current study.

\bigskip
\noindent
\textbf{Declarations}

\noindent
\medskip
\textbf{Conflicts of Interest}
Both authors confirm that we do not have any conflict of interest.

\section*{Appendix}  
\addcontentsline{toc}{section}{Appendix}  

\renewcommand{\thesubsection}{\Alph{subsection}}  
\setcounter{subsection}{0}  

\numberwithin{equation}{subsection}  
\renewcommand{\theequation}{\thesubsection.\arabic{equation}}  

\makeatletter
\newtheoremstyle{boldstyle}  
  {\topsep}   
  {\topsep}   
  {\normalfont}  
  {0pt}       
  {\bfseries} 
  {.}         
  {5pt plus 1pt minus 1pt}  
  {}          
\makeatother

\theoremstyle{boldstyle}
\newtheorem{appendixlemma}{Lemma}[subsection]
\renewcommand{\theappendixlemma}{\thesubsection.\arabic{appendixlemma}}

\subsection{Properties of weights and interpolations}
\begin{appendixlemma}[\cite{Duoandikoetxea01}]\label{lemma:weight}
Let $\omega\in A_p$, $1\leq p<\infty$. Then, we have the following properties:
\begin{subequations}
\begin{align}
& [\tilde{\omega}]_{A_p} = [\omega]_{A_p},
~\text{where}~\tilde{\omega} = \omega(\varepsilon\cdot);  \label{f:apw-1}\\
& \omega(Q)\big(\frac{|S|}{|Q|}\big)^{p}\leq [\omega]_{A_p}\omega(S)
\quad \forall S\subset Q\subset\mathbb{R}^{d+1}. \label{f:apw-2}
\end{align}
\end{subequations}
Also, there exists $\epsilon>0$, depending only on $p$ and $[\omega]_{A_p}$, such that
for any $Q$, it holds that
\begin{subequations}
\begin{align}
&\Big(\dashint_{Q}\omega^{1+\epsilon}\Big)^{\frac{1}{1+\epsilon}}
\lesssim_{d,p,[\omega]_{A_p}}\dashint_{Q}\omega;\label{f:apw-3} \\
& \omega(S)\lesssim_{d,p,[\omega]_{A_p}}  \big(\frac{|S|}{|Q|}\big)^{\delta_0}\omega(Q),
\qquad \forall S\subset Q,
\label{f:apw-4}
\end{align}
\end{subequations}
where $\delta_0 = \epsilon/(1+\epsilon)$. Moreover, given any $1<p<\infty$ and $c_0>1$
there exist positive constants $C=C(d,p,c_0)$ and $\gamma=\gamma(d,p,c_0)$ such that
for all $\omega\in A_p$ we have
\begin{equation}\label{f:apw-5}
  [\omega]_{A_p}\leq c_0
\quad \Rightarrow \quad
  [\omega]_{A_{p-\gamma}}\leq C.
\end{equation}
\end{appendixlemma}

\begin{appendixlemma}\label{lemma:4.1-ap}
For any $1\leq q<\infty$, let
\begin{equation*}
  \|u\|_{V_{q,2}(\mathbb{R}^{d+1})}
 :=\sup_{t\in\mathbb{R}}\|u(\cdot,t)\|_{L^2(\mathbb{R}^d)}
  + \|\nabla u\|_{L^q(\mathbb{R}^{d+1})}.
\end{equation*}
Then, there exists $C_d>0$, depending only on $d$ and $q$, such that
\begin{subequations}
\begin{align}
&\|u\|_{L^{\frac{q(d+2)}{d}}(\mathbb{R}^{d+1})}
 \leq C_d\|u\|_{V_{q,2}(\mathbb{R}^{d+1})};
 \label{pri:6.8-ap}  \\
&\|\mu_{d} u\|_{L^{\frac{q(d+2)}{d}}(\mathbb{R}^{d+1})}
\lesssim_d \sup_{t\in\mathbb{R}}\|u\|_{L^2(\mathbb{R}^d)}
+\|\mu_{d}^{\alpha}\nabla u\|_{L^q(\mathbb{R}^{d+1})};
\label{pri:6.9-ap}
\end{align}
\end{subequations}
where $\alpha := 1+2/d$, 
and $\mu_d$ is given as
in Theorem $\ref{thm:1.1}$.
\end{appendixlemma}

\begin{proof}
Arguments for $\eqref{pri:6.8-ap}$.
Let $q^*:=\frac{q(d+2)}{d}=q/\theta$ and $\theta:=\frac{d}{d+2}$.
Starting from Galiardo-Nirenberg inequality (see e.g., \cite[Theorem 1.1]{Duarte-DrumondSilva-23}), it follows that
\begin{equation}\label{f:6.29-ap}
 \|u(\cdot,t)\|_{L^{q^*}(\mathbb{R}^{d})}
 \lesssim_{d,q}\|\nabla u(\cdot,t)\|_{L^q(\mathbb{R}^d)}^\theta
 \|u(\cdot,t)\|_{L^2(\mathbb{R}^d)}^{1-\theta} \qquad\forall t\in\mathbb{R},
\end{equation}
where $q^*, q, \theta$ satisfy the relationship
$1/q^{*}=\theta\big(1/q-1/d\big)+(1-\theta)/2$. Then, one can derive that
\begin{equation*}
\begin{aligned}
 \int_{\mathbb{R}}dt\|u(\cdot,t)\|_{L^{q^*}(\mathbb{R}^{d})}^{q^*}
& \lesssim^{\eqref{f:6.29-ap}}
 \int_{\mathbb{R}}dt
 \|\nabla u(\cdot,t)\|_{L^q(\mathbb{R}^d)}^{\theta q^{*}}
 \|u(\cdot,t)\|_{L^2(\mathbb{R}^d)}^{(1-\theta)q^{*}}\\
& \lesssim \sup_{t\in\mathbb{R}}\|u(\cdot,t)\|_{L^2(\mathbb{R}^d)}^{(1-\theta)q^{*}}
\int_{\mathbb{R}}dt
 \|\nabla u(\cdot,t)\|_{L^q(\mathbb{R}^d)}^{q}
\lesssim \|u\|_{V_{q,2}(\mathbb{R}^{d+1})}^{q^{*}},
\end{aligned}
\end{equation*}
where the second inequality is due to $q^*\theta = q$, and
the last one follows from Young's inequality.
Taking the $q^*$-th root on the both sides of the above
inequality leads to the stated result $\eqref{pri:6.8-ap}$.

Arguments for $\eqref{pri:6.9-ap}$. To this end, it suffices to modify $\eqref{f:6.29-ap}$ as follows:
\begin{equation*}
 \|\mu_*u(\cdot,t)\|_{L^{q^*}(\mathbb{R}^{d})}
 \lesssim_{d,q}\|\mu_*^{\frac{1}{\theta}}\nabla u(\cdot,t)\|_{L^q(\mathbb{R}^d)}^\theta
 \|u(\cdot,t)\|_{L^2(\mathbb{R}^d)}^{1-\theta} \qquad\forall t\in\mathbb{R},
\end{equation*}
which is known as the weighted Galiardo-Nirenberg inequality (see e.g., \cite[Theorem 1.2]{Duarte-DrumondSilva-23}\footnote{The weighted
Galiardo-Nirenberg inequality is originally developed
by C.S. Lin for power law weights. Concerned with
the non-homogeneous
weights like $\mu_*$, it was recently shown by
Duarte and Drumond Silva \cite[Theorem 1.7]{Duarte-DrumondSilva-23}.}). The remainder of the proof is exactly the same as that given for $\eqref{pri:6.8-ap}$, and the details are omitted here.
\end{proof}

\subsection{Basic a priori estimates}

\begin{appendixlemma}[Caccioppoli's inequality]\label{lemma:5.1}
Let $u$ and $g$ together with $f$ be associated with
$\mathcal{A}_{\beta}^* (u_\beta) =
\nabla\cdot f + \beta g$
or $\mathcal{A}_{\beta} (u_\beta) = \nabla\cdot f +\beta g$ in
$Q_{4R}$ with $R>0$ and $\beta\geq 0$,
where $\mathcal{A}_{\beta}$ and $\mathcal{A}_{\beta}^*$ are given in
$\eqref{massive-corrector}$, $\eqref{pde:5.3-ap}$, respectively. Then, it holds that
\begin{subequations}
\begin{align}
\dashint_{Q_R}|(\sqrt{\beta} u_\beta,\nabla u_{\beta})|^2
&\lesssim_{\mu,d}
 R^{-2}\dashint_{Q_{2R}}|u_\beta|^2 + \dashint_{Q_{2R}}|(\sqrt{\beta} g,f)|^2;\label{pri:6.1-ap}  \\
\sup_{s\in I_{R}}\dashint_{B_{R}}|u_{\beta}(\cdot,s)|^2
&\lesssim_{\mu,d}
\dashint_{Q_{2R}}|u_\beta|^2 + R^{2}\dashint_{Q_{2R}}|(\sqrt{\beta} g,f)|^2,
\label{pri:6.2-ap}
\end{align}
\end{subequations}
where we recall the notation $I_r:=(-r^2,r^2)$. Moreover, if $R>r>0$,
and let $(u_{\beta})_{r}:=K_r(u_\beta)$, there exists $\theta\in (0,1)$, depending on $\mu,d$, such that
\begin{equation}\label{pri:6.3-ap}
\begin{aligned}
\dashint_{Q_R}|\nabla u_\beta|^2
&\leq
\theta \dashint_{I_{2R}}dt\dashint_{\partial B_{2R}}dx\dashint_{B_r(x)}\big(|\nabla u_\beta|^2+|f|^2\big)(\cdot,t) \\
&+ C_{\theta}\Big(\frac{r}{R}\Big)^2\dashint_{Q_{3R}}|\nabla u_{\beta}|^2 +
 C_{\theta}\dashint_{Q_{2R}}|\nabla (u_{\beta})_r|^2
 + C_\theta\dashint_{Q_{2R}}|(\sqrt{\beta} g,f)|^2,
\end{aligned}
\end{equation}
where $C_\theta$ depends on $\theta,d$, and $\mu$.
\end{appendixlemma}

It is recalled that $C_0^1(Q_{R})$ is the Banach space of functions with continuously one-order derivative with respect to spatial or time variables, requiring its element to vanish near $\partial Q_{R}$.

\begin{proof}
The idea of the proof is standard, and we provide a proof for the sake of the reader's convenience.
Let $\varphi\in C_0^1(Q_{2R})$ be a cut-off function satisfying
$\varphi = 1$ in $Q_{R}$ and $|\nabla\varphi|^2+|\partial_t\varphi|\lesssim_d R^{-2}$.
We start from the weak formulation
\begin{equation}\label{f:6.0-ap}
 \int_{Q_{2R}}\beta u_\beta v-\partial_t u_\beta v
 +a^*\nabla u_\beta\cdot \nabla v =
 \int_{Q_{2R}}\beta g v - f\cdot\nabla v
 \qquad \forall v\in C^{1}_{0}(Q_{2R}).
\end{equation}
Then, by choosing different test functions, the proof is divided into three steps.

\textbf{Step 1.} Arguments for $\eqref{pri:6.1-ap}$.
Let $v = \varphi^2 u_\beta$. On the one hand, it is found that
\begin{equation}\label{f:6.1-ap}
 \int_{Q_{2R}} \partial_t u_\beta\varphi^2 u_\beta
 = \frac{1}{2}\int_{Q_{2R}}\partial_t\big(u^2_\beta\varphi^2\big) - \int_{Q_{2R}}\varphi\partial_t \varphi u^2_\beta
 = - \int_{Q_{2R}}\varphi\partial_t\varphi u^2_\beta.
\end{equation}
On the other hand, a routine computation leads to
\begin{equation}\label{f:6.2-ap}
\begin{aligned}
&\int_{Q_{2R}}\big(a^*\nabla u_\beta+ f\big)\cdot\nabla v
=\int_{Q_{2R}}\big(a^*\nabla u_\beta+f\big)\cdot\big(\nabla u_\beta\varphi^2 + 2\varphi\nabla\varphi u_\beta\big)\\
&=\int_{Q_{2R}}\varphi^2 a^*\nabla u_\beta\cdot\nabla u_\beta
-\int_{Q_{2R}}f\cdot\nabla u_\beta\varphi^2
+\int_{Q_{2R}}2\varphi a^{*}\nabla u_\beta\cdot\nabla\varphi u_\beta
-\int_{Q_{2R}}2\varphi f\cdot\nabla\varphi u_\beta\\
&\geq \frac{\mu}{4}\int_{Q_{2R}}\varphi^2 |\nabla u_\beta|^2
-C_\mu \int_{Q_{2R}}\varphi^2|f|^2
-C_\mu \int_{Q_{2R}}|\nabla\varphi|^2|u_\beta|^2,
\end{aligned}
\end{equation}
and
\begin{equation}\label{f:6.2-2-ap}
\begin{aligned}
\beta\int_{Q_{2R}} g v
= \beta\int_{Q_{2R}} \varphi^2 g u_\beta
\leq \frac{\beta}{2}\int_{Q_{2R}}\varphi^2u_\beta^2
+  \beta \int_{Q_{2R}}\varphi^2 g^2.
\end{aligned}
\end{equation}

Combining the estimates $\eqref{f:6.0-ap}$, $\eqref{f:6.1-ap}$, $\eqref{f:6.2-ap}$, and $\eqref{f:6.2-2-ap}$, we obtain that
\begin{equation*}
 \int_{Q_{2R}}\varphi^2
 \big(\beta u_\beta^2 + |\nabla u_\beta|^2\big)
 \lesssim_{\mu,d} \int_{Q_{2R}}(|\partial_t\varphi| + |\nabla \varphi|^2)|u_\beta|^2
 +  \int_{Q_{2R}}\varphi^2\big(|f|^2+\beta g^2\big),
\end{equation*}
which implies the stated estimate $\eqref{pri:6.1-ap}$ upon dividing the both sides by $R^{d+2}$.

\textbf{Step 2.} Arguments for $\eqref{pri:6.2-ap}$. Take $v=\textbf{1}_{s<t}\varphi^2u_\beta$ in $\eqref{f:6.0-ap}$
with any fixed $s\in I_{2R}$. Then, the left-hand side of $\eqref{f:6.1-ap}$ turns into
\begin{equation}\label{f:6.3-ap}
\begin{aligned}
\int_{Q_{2R}}\partial_t u_\beta v
&= \int_{I_{2R}}\textbf{1}_{s<t}\int_{B_{2R}}\partial_t
u_\beta\varphi^2 u_\beta \\
&= \frac{1}{2}\Big(\int_{B_{2R}}\varphi^2u_\beta^2(\cdot,t)-
\int_{B_{2R}}\varphi^2u_\beta^2(\cdot,s)\Big)
-\int_{I_{2R}}\textbf{1}_{s<t}\int_{B_{2R}}\varphi\partial_t\varphi u_\beta^2.
\end{aligned}
\end{equation}
One can choose $t\in I_{2R}$ such that $\int_{B_{2R}}\varphi^2u_\beta^2(\cdot,t)=0$, and note that
\begin{equation*}
  \bigg|\int_{I_{2R}}\textbf{1}_{s<t}\int_{B_{2R}}\varphi\partial_t\varphi u_\beta^2\bigg|
  \leq \frac{C}{R^2}\int_{Q_{2R}}|u_\beta|^2.
\end{equation*}

In view of the estimates of $\eqref{f:6.0-ap}$, $\eqref{f:6.2-ap}$, $\eqref{f:6.2-2-ap}$, and $\eqref{f:6.3-ap}$, for any $s\in I_{2R}$, one can derive that
\begin{equation*}
 \int_{B_{R}}u_\beta^2(\cdot,s) + \int_{Q_{R}}
 |(\sqrt{\beta} u_\beta,\nabla u_\beta)|^2
 \lesssim \frac{1}{R^2}\int_{Q_{2R}}|u_\beta|^2 + \int_{Q_{2R}}|(f,\sqrt{\beta} g)|^2.
\end{equation*}
This gives the desired estimate $\eqref{pri:6.2-ap}$ by multiplying $R^{-d-2}$ on the both sides above.

\textbf{Step 3.} Arguments for $\eqref{pri:6.3-ap}$. For a new parameter $r\in(0,R)$, we consider
\begin{equation*}
   \tilde{u}_\beta(x,t) = u_\beta(x,t) - c(t) \qquad \text{and}\qquad
   c(t):=\dashint_{B_{2R}} K_r(u_\beta)(\cdot,t),
\end{equation*}
which satisfies the following equations:
\begin{equation*}
\beta\tilde{u}_\beta  -\partial_t \tilde{u}_\beta - \nabla\cdot a^*\nabla \tilde{u}_\beta
  = \nabla\cdot f + \beta \big(g-c(t)\big) + \partial_t c
  \qquad \text{in}\quad Q_{2R}.
\end{equation*}

It is noted that $\nabla\tilde{u}_\beta = \nabla u_\beta$, and the demonstration is similar to that given in \textbf{Step 1}.
The weak formulation $\eqref{f:6.0-ap}$ can be rewritten as
\begin{equation}\label{f:6.0*-ap}
 \int_{Q_{2R}}
 \big(\beta\tilde{u}_\beta-\partial_t \tilde{u}_\beta
  \big) v = \int_{Q_{2R}}\big(a^*\nabla \tilde{u}_\beta+f\big)\cdot\nabla v + \int_{Q_{2R}}
  \big(\beta(g-c(t))+\partial_t c \big)v
 \qquad \forall v\in C^{1}_{0}(Q_{2R}).
\end{equation}
Now, we choose $v=\varphi^2\tilde{u}_\beta$. Similar to the computation given for $\eqref{pri:6.1-ap}$, there holds
\begin{equation}\label{f:6.1*-ap}
\begin{aligned}
\dashint_{Q_R}|(\sqrt{\beta} \tilde{u}_\beta,\nabla \tilde{u}_{\beta})|^2
\lesssim_{\mu,d}
 R^{-2}\dashint_{Q_{2R}}|\tilde{u}_\beta|^2 + \dashint_{Q_{2R}}|(\sqrt{\beta} g,f)|^2
 + \Big|\dashint_{Q_{2R}}
 \big(\partial_t c -\beta c(t)\big)v\Big|.
\end{aligned}
\end{equation}
We proceed to estimate the last term in the right-hand side of $\eqref{f:6.1*-ap}$. It is noted that
$\partial_t c -\beta c(t) = \dashint_{B_{2R}} K_r(\partial_t u_\beta -\beta u_\beta)(\cdot,t)$, and by recalling the equation that $u_\beta$ satisfies, it follows that
\begin{equation*}
\begin{aligned}
\partial_t c -\beta c(t) &= -\dashint_{B_{2R}}
 K_r(\nabla\cdot(f+a^*\nabla u_\beta)+\beta g)(\cdot,t) \\
 &=-\dashint_{B_{2R}}\nabla\cdot K_r(a^*\nabla u_\beta+f)(\cdot,t)
 -\beta \dashint_{B_{2R}}
 K_r(g)(\cdot,t) \\
 &\lesssim_{d,\mu} R^{-1}\dashint_{\partial B_{2R}}dx\dashint_{B_r(x)} \big(|\nabla u_\beta|+|f|\big)(\cdot,t)
 + \beta\dashint_{B_{2R}}dx\dashint_{B_r(x)}|g|(\cdot,t).
\end{aligned}
\end{equation*}
This implies that
\begin{equation}\label{f:6.3*-ap}
\begin{aligned}
&\bigg|\dashint_{Q_{2R}}\big(\partial_tc-\beta c(t)\big) v\bigg|
\lesssim_{\mu,d} \frac{1}{R}\dashint_{I_{2R}}dt
\bigg(\dashint_{\partial B_{2R}}dx\dashint_{B_r(x)} \big(|\nabla u_\beta|+|f|\big)(\cdot,t)\bigg)
\dashint_{B_{2R}}|\tilde{u}_\beta(\cdot,t)|\\
&+ \beta\dashint_{I_{2R}}dt\dashint_{B_{2R}}dx\dashint_{B_r(x)} |g(\cdot,t)|\dashint_{B_{2R}}|\tilde{u}_\beta(\cdot,t)|\\
&\leq \theta \dashint_{I_{2R}}dt\dashint_{\partial B_{2R}}dx\dashint_{B_r(x)} \big(|\nabla u_\beta|^2+|f|^2\big)(\cdot,t) + \frac{C_\theta}{R^2}\dashint_{Q_{2R}}|\tilde{u}_\beta|^2
+\frac{\beta}{4}\dashint_{Q_{2R}}|\tilde{u}_\beta|^2
+ \beta\dashint_{Q_{3R}}|g|^2,
\end{aligned}
\end{equation}
where Young's inequality with $\theta\in(0,1)$ is employed for the second inequality.

Combining the estimates $\eqref{f:6.1*-ap}$ and $\eqref{f:6.3*-ap}$,
we arrive at
\begin{equation}\label{f:6.4-ap}
\begin{aligned}
\dashint_{\mathcal{C}_{2R}^*}|(\sqrt{\beta}\tilde{u}_\beta,\nabla u_\beta)|^2
&\lesssim_{\mu,d} \theta \dashint_{I_{2R}}dt\dashint_{\partial B_{2R}}dx\dashint_{B_r(x)} \big(|\nabla u_\beta|^2+|f|^2\big)(\cdot,t)\\
&+ \frac{C_\theta}{R^2}\dashint_{Q_{2R}}|\tilde{u}_\beta|^2 +\dashint_{Q_{3R}}|(\sqrt{\beta} g,f)|^2.
\end{aligned}
\end{equation}
To obtain the desired estimate $\eqref{pri:6.3-ap}$, it follows from Lemma $\ref{lemma:2.5}$ and Poincar\'e's inequality that
\begin{equation}\label{f:6.5-ap}
\begin{aligned}
\dashint_{Q_{2R}}|\tilde{u}_\beta|^2
&\lesssim  \dashint_{Q_{2R}}|u_\beta-(u_\beta)_r|^2 + \dashint_{Q_{2R}}|(u_\beta)_r - c|^2 \\
&\lesssim^{\eqref{pri:2.10}} r^2 \dashint_{Q_{3R}}|\nabla u_\beta|^2 + R^2 \dashint_{Q_{2R}}|\nabla (u_\beta)_r|^2.
\end{aligned}
\end{equation}
Plugging the estimate $\eqref{f:6.5-ap}$ back into $\eqref{f:6.4-ap}$, we finally derive the desired estimate
$\eqref{pri:6.3-ap}$.
\end{proof}

\begin{appendixlemma}[Meyer-type estimates]\label{lemma:5.2}
Let $\beta\geq 0$, $p_0>0$, $|2-p|\ll 1$.
Suppose that $u_\beta$, $g$, and $f$ are associated with
$\mathcal{A}_{\beta}^* (u_\beta) =
\nabla\cdot f + \beta g$
(or $\mathcal{A}_{\beta} (u_\beta) = \nabla\cdot f +\beta g$) in
$Q_{4R}$ with $R>0$,
where $\mathcal{A}_{\beta}$ and $\mathcal{A}_{\beta}^*$ are given in
$\eqref{massive-corrector}$ and $\eqref{pde:5.3-ap}$, respectively.
Then, for any $\alpha>1$, the interior estimate holds:
\begin{equation}\label{pri:5.9-ap}
\Big(\dashint_{Q_R}|(\sqrt{\beta} u_\beta,\nabla u_{\beta})|^p
\Big)^{\frac{1}{p}}
\lesssim_{\mu,d,p,p_0,\alpha} \Big(\dashint_{\alpha Q_R}|(\sqrt{\beta} u_\beta,\nabla u_{\beta})|^{p_0}
\Big)^{\frac{1}{p_0}}
+ \Big(\dashint_{\alpha Q_R}|(\sqrt{\beta} g,f)|^p
\Big)^{\frac{1}{p}}.
\end{equation}
Similarly, if $u_\beta$ satisfies
the equations $\mathcal{A}_{\beta}^* (u_\beta) =
\nabla\cdot f + \beta g$
(or $\mathcal{A}_{\beta} (u_\beta) = \nabla\cdot f +\beta g$) in
$D_{4R}$ with $u_\beta = 0$ on $\Delta_{4R}$, then the corresponding
boundary estimate holds:
\begin{equation}\label{pri:5.12-ap}
\Big(\dashint_{D_R}|(\sqrt{\beta} u_\beta,\nabla u_{\beta})|^p
\Big)^{\frac{1}{p}}
\lesssim_{\mu,d,p,p_0,\alpha} \Big(\dashint_{\alpha D_R}|(\sqrt{\beta} u_\beta,\nabla u_{\beta})|^{p_0}
\Big)^{\frac{1}{p_0}}
+ \Big(\dashint_{\alpha D_R}|(\sqrt{\beta} g,f)|^p
\Big)^{\frac{1}{p}}.
\end{equation}
Moreover, let $u_\beta$, $g$, and $f$ be associated with
$\mathcal{A}_{\beta}^* (u_\beta) =
\nabla\cdot f + \beta g$
(or $\mathcal{A}_{\beta} (u_\beta) = \nabla\cdot f +\beta g$) in
$\mathbb{R}^{d+1}$. Then, there holds
\begin{equation}\label{pri:5.10-ap}
\Big(\int_{\mathbb{R}^{d+1}}|(\sqrt{\beta} u_\beta,\nabla u_{\beta})|^p
\Big)^{\frac{1}{p}}
\lesssim_{\mu,d,p}
\Big(\int_{\mathbb{R}^{d+1}}|(\sqrt{\beta} g,f)|^p
\Big)^{\frac{1}{p}}.
\end{equation}
\end{appendixlemma}

\begin{proof}
The proofs are standard (since included in \cite{Armstrong-Bordas-Mourrat-18,Auscher-Bortz-Egert-Saari-19}) and detailed proofs are not provided here. Roughly speaking, based on $\eqref{pri:6.1-ap}$, one can first establish $\eqref{pri:5.9-ap}$ for $p_0=2$ by using Gehring-type lemma (see e.g., \cite[Lemma B.5]{Armstrong-Bordas-Mourrat-18}).
Then, a convexity argument (see \cite[pp.173]{Fefferman-Stein72}) leads to the cases $p_0\in(0,2)$, and this method is independent of PDEs. By the same token, one can have the boundary estimate $\eqref{pri:5.12-ap}$.
The estimate $\eqref{pri:5.10-ap}$ consequently follows from $\eqref{pri:5.9-ap}$.
\end{proof}

\begin{appendixlemma}[Reverse H\"older inequality]\label{lemma:6.1}
Let $\Omega\subset\mathbb{R}^d$ be a bounded $C^1$ domain with $d\geq 2$.
Assume that $v$ is a solution of $\big(\partial_t-\nabla\cdot\bar{a}\nabla\big) v = 0$ in $D_{2r}$ with $v=0$ on $\Delta_{2r}$ with $r>0$. Then, for any
$2<p<\infty$, $p_0>0$, and $\alpha\in(1,2]$, it holds that
\begin{equation}\label{pri:5.11-ap}
\Big(\dashint_{D_r}|\nabla v|^p\Big)^{\frac{1}{p}}
\lesssim_{\mu,d,p,p_0,\alpha} \Big(\dashint_{D_{\alpha r}}|\nabla v|^{2}\Big)^{\frac{1}{2}}.
\end{equation}
\end{appendixlemma}

\begin{proof}
See e.g., \cite[Theorem 9.1]{Geng-Shen-15}, or \cite[Lemma 5.4]{Wang-Xu-25}
for a different proof strategy.
\end{proof}

\begin{appendixlemma}[Fractional integral estimate]\label{lemma:6.1-n}
Let $1<p<q<\infty$ and $1/p = 1/q+1/(d+2)$. 
Let $v$ be a solution to $\big(\partial_t-\Delta\big)v = \nabla\cdot g$ in $\mathbb{R}^{d+1}$. Then, there holds the 
anisotropic fractional integration estimate: 
\begin{equation}\label{pri:6.1-n}
 \|v\|_{L^q(\mathbb{R}^{d+1})} \lesssim \|g\|_{L^p(\mathbb{R}^{d+1})}. 
\end{equation}
\end{appendixlemma}
\begin{proof}
By the notion of parabolic Riesz potential \cite[Theorem 3.2]{Aliev-Sekin-20} (see also \cite[Theorem 3.1]{Gopala Rao-77}), denoted by $A_2^{\alpha}$, it is not hard to verify the stated estimate $\eqref{pri:6.1-n}$. Let $\Gamma$ be the heat kernel, and the solution $u$ can be controlled by
\begin{equation}
|u(x,t)|\lesssim \int_{\mathbb{R}^{d+1}}s_{+}^{-\frac{1}{2}}
\Gamma(y,s)f(x-y,t-s)dyds=:A_2^{\frac{1}{2}}(f)(x,t), 
\qquad \text{where}~s_{+}:=0\vee s 
\end{equation}
(see e.g. \cite[Definition 2.3]{Aliev-Sekin-20}), 
and we mention that $|\nabla_y\Gamma(y,s)|\lesssim s^{-1/2}\Gamma(y,s)$ for any $s>0$, and we therefore take $\alpha = 1/2$ correspondingly. We are done. 
\end{proof}

It is recalled that ${_0H^{1,1/2}_q(\partial_{\shortparallel}\Omega_T)}$ denotes
the Sobolev space of functions possessing one spatial derivative and half
of a time derivative in $L^q(\partial_{\shortparallel}\Omega_T)$, with  elements required to vanish on $\partial\Omega\times\{t=0\}$.

\begin{appendixlemma}[Layer and co-layer type estimates]\label{lemma:3.2-ap}
Let $\Omega\subset\mathbb{R}^d$ be a bounded $C^1$ domain with $d\geq2$,
$1<q<\infty$, and $T>0$. Given
$F\in L^q(\Omega_T)$ and
$g\in {_{0}H^{1,1/2}_{q}}(\partial_{\shortparallel}\Omega_T)$
satisfying $\|g\|_{H^{1,1/2}_{q}(\partial_{\shortparallel}\Omega_T)} = 1$,
let $u_0$ be the weak solution of $(\emph{DP}_0)$ in $\eqref{pde:3.1*}$. Then, the following layer-type estimates hold:
\begin{subequations}
\begin{align}
&\|\nabla u_0\|_{L^q(\boxbox_{2c\varepsilon})}
\lesssim\varepsilon^{\frac{1}{q}}\Big\{\|F\|_{L^q(\Omega_T)}
+ 1\Big\};
\label{pri:3.1-ap}\\
&\|\delta^{1/2}\nabla u_0\|_{L^2(\boxbox_{2c\varepsilon})}
\lesssim\varepsilon\Big\{\|F\|_{L^2(\Omega_T)}
+ 1\Big\};
\label{pri:3.2-ap}\\
&\|\omega_\sigma^{1/q}\nabla u_0\|_{L^q((O_{2c\varepsilon})_T)}
\lesssim\varepsilon^{\underline{q}/q}\Big\{\|F\|_{L^q(\Omega_T)}
+ 1\Big\},
\label{pri:3.4-ap}
\end{align}
\end{subequations}
in which $\omega_\sigma$ is given as in Theorem $\ref{thm:C-Z}$
with $0<\underline{q}<q$ and $\delta(z)
:=\text{dist}(z,\partial\Omega_T)$. Also, the following
co-layer type estimates hold:
\begin{subequations}
\begin{align}
\max\Big\{\|\nabla^2 u_0\|_{L^q((\Omega\setminus O_{c\varepsilon})_T)},
~\|\partial_t u_0\|_{L^q((\Omega\setminus O_{c\varepsilon})_T)}\Big\}
&\lesssim \varepsilon^{\frac{1}{q}-1}\Big\{\|F\|_{L^q(\Omega_T)}
+ 1\Big\};
\label{pri:3.3-ap}\\
\max\Big\{\|\omega_{\sigma}^{1/q}\nabla^2 u_0\|_{L^q((\Omega\setminus O_{c\varepsilon})_T)},
~\|\omega_{\sigma}^{1/q}\partial_t u_0\|_{L^q((\Omega\setminus O_{c\varepsilon})_T)}\Big\}
&\lesssim \varepsilon^{\underline{q}/q-1} \Big\{\|F\|_{L^q(\Omega_T)}
+ 1\Big\};
\label{pri:3.5-ap}\\
\|\delta^{\frac{1}{2}}(\nabla^2u_0+\partial_t u_0)\|_{L^2((\Omega\setminus O_{c\varepsilon})_T)}
+ \|\delta^{-\frac{1}{2}}\nabla u_0\|_{L^2(\Sigma_{c\varepsilon}^T)}
&\lesssim \ln^{\frac{1}{2}}(1/\varepsilon)\Big\{\|F\|_{L^2(\Omega_T)}
+ 1\Big\},
\label{pri:3.6-ap}
\end{align}
\end{subequations}
where $c\in[1,8]$, and the multiplicative constant depends at most on $\mu,d,q,T$, and $\Omega$. In particular, the estimates
$\eqref{pri:3.2-ap}$ and $\eqref{pri:3.6-ap}$ still hold
for a Lipschitz cylinder, and the symmetry condition
$a=a^*$ is required in such a case.
\end{appendixlemma}

\begin{proof}
The main idea in the proofs is analogous to that applied to elliptic operators but more involved, and we refer the
reader to \cite{Shen18,Xu-16} for the original idea.
We mainly prove $\eqref{pri:3.1-ap}$ and $\eqref{pri:3.6-ap}$. The corresponding weighted estimates (except of the second term in
the left-hand side of $\eqref{pri:3.6-ap}$) only require minor modifications to the original proofs, and therefore we leave them to the reader.
The proof is divided into six steps.

\textbf{Step 1.} Outline of the proof of $\eqref{pri:3.1-ap}$ and reduction.
Let $t_0:=(2c\varepsilon)^2$ and $t_*:=T$. By recalling the definition of
$\boxbox_{2c\varepsilon}$, it suffices to divide it to the lateral layer, denoted by $(O_{2c\varepsilon})_T$, and the time layer (see Fig.\ref{pic:1.2}). Correspondingly,
the triangle inequality leads to
\begin{equation}\label{f:3.8-ap}
\begin{aligned}
\|\nabla u_0\|_{L^q(\boxbox_{2c\varepsilon})}
\leq \|\nabla u_0\|_{L^q((O_{2c\varepsilon})_T)}
+ \sum_{t\in\{t_0,t_*\}}\Big(\int_{t-{(2c\varepsilon)}^2}^t
\int_{\Sigma_{2c\varepsilon}}|\nabla u_0|^q\Big)^{\frac{1}{q}}.
\end{aligned}
\end{equation}
Once we can establish the following estimates:
\begin{subequations}
\begin{align}
& \|\nabla u_0\|_{L^q((O_{2c\varepsilon})_T)}\lesssim
\varepsilon^{\frac{1}{q}} \Big\{\|F\|_{L^q(\Omega_T)}+1\Big\};
\label{f:3.6-a-ap} \\
& \sum_{t\in\{t_0,t_*\}}\Big(\int_{t-{(2c\varepsilon)}^2}^t
\int_{\Sigma_{2c\varepsilon}}|\nabla u_0|^q\Big)^{\frac{1}{q}}
\lesssim \varepsilon^{\frac{1}{q}} \Big\{\|F\|_{L^q(\Omega_T)}+1\Big\},
\label{f:3.6-b-ap}
\end{align}
\end{subequations}
by substituting $\eqref{f:3.6-a-ap}$ and $\eqref{f:3.6-b-ap}$ into the estimate $\eqref{f:3.8-ap}$, we will obtain the stated estimate $\eqref{pri:3.1-ap}$.

To show the estimates $\eqref{f:3.6-a-ap}$ and $\eqref{f:3.6-b-ap}$,
the solution $u_0$ is divided into two parts, i.e., $u_0 = v+w$,
where $v$ and $w$ satisfy
the following equations (i) and (ii), respectively.
\begin{equation}\label{pde:5.5}
(\text{i})\left\{\begin{aligned}
\partial_t v + \mathcal{L}_0(v) &= \tilde{F} &~&\text{in}~ \mathbb{R}^{d}\times\mathbb{R}_{+};\\
v(\cdot,0) & = 0 &~& \text{on}~ \mathbb{R}^d,
\end{aligned}\right.
\qquad
\quad
(\text{ii}) \left\{\begin{aligned}
\partial_t w + \mathcal{L}_0(w) &= 0 &\quad&\text{in}\quad \Omega_T;\\
w & = g-v &~& \text{on}\quad \partial_{\shortparallel}\Omega_T;\\
w(\cdot,0) & = 0 &~& \text{on}\quad \Omega,
\end{aligned}\right.
\end{equation}
where $\tilde{F}$ is the zero extension of $F$ from $\Omega_T$ to
$\mathbb{R}_{+}\times\mathbb{R}^d$, and $\mathbb{R}_{+}:=(0,\infty)$.
By the well-posedness of (i), it can be observed that
$g-v\in {_{0}H_p^{1,1/2}}(\partial_{\shortparallel}\Omega_T)$. It is
noted that the existence of $w$ has been shown in \cite[Theorem 3.4]{Fabes-Riviere-79}, and the
second equality in (ii) is to be understood in the sense of
the nontangential convergence (see e.g., \cite{Brown-89,Shen18}).

For $\eqref{f:3.6-a-ap}$ and $\eqref{f:3.6-b-ap}$,
the key ingredients are as follows:
\begin{subequations}\label{}
\begin{align}
&\|v\|_{W^{2,1}_{q}(\Omega_{T})}
\lesssim_{\mu,d,q,R_0,T} \|F\|_{L^q(\Omega_T)};
\label{f:3.7-ap}\\
&\|(\nabla w)^*\|_{L^q(\partial_{\shortparallel}\Omega_T)}
\lesssim \Big\{\|F\|_{L^q(\Omega_T)}+1\Big\},
\label{f:3.13-ap}
\end{align}
\end{subequations}
where the definition of the nontangential maximal function is given
as in $\eqref{def:nontangential}$ (see also Fig.\ref{pic:1.3}).
The stated estimate $\eqref{f:3.6-a-ap}$
follows from the above two estimates (see Step 3), while
one can also handle the time layer type estimate as follows:
\begin{equation}\label{f:3.30-ap}
\begin{aligned}
&\sum_{t\in\{t_0,t_*\}}\int_{t-(2c\varepsilon)^2}^t
\int_{\Sigma_{2c\varepsilon}}|\nabla u_0|^q
\lesssim \varepsilon^2
\sum_{t\in\{t_0,t_*\}}
\int_{\Sigma_{2c\varepsilon}}|\nabla u_0(\cdot,t')|^q
\lesssim  \varepsilon
\sum_{t\in\{t_0,\cdots,t_N,t_*\}}
\varepsilon\int_{\Sigma_{2c\varepsilon}}|\nabla u_0(\cdot,t')|^q\\
&\lesssim \varepsilon \int_{0}^{T}dt
\int_{\Sigma_{2c\varepsilon}}|\nabla u_0(\cdot,t)|^q
\lesssim \varepsilon \bigg\{\int_{0}^{T}dt
\int_{\Omega}(|\nabla^2 v|+|v|)^q(\cdot,t)
+\|(\nabla w)^*\|_{L^q(\partial_{\shortparallel}\Omega_T)}^q\bigg\}\\
&\lesssim^{\eqref{f:3.7-ap},\eqref{f:3.13-ap}} \varepsilon
\Big\{\|F\|_{L^q(\Omega_T)}^q
+1 \Big\},
\end{aligned}
\end{equation}
where we trivially insert $N$ points in the time direction
in the second inequality above to use the definition of Riemann integral\footnote{Since $\nabla u_0|_{x\in\Sigma_{2c\varepsilon}}(x,t)$ is smooth with respect to time variable $t$ due to the interior regularity of $u_0$, it is fine to understand it
in the sense of Riemann integral.} for the third inequality. Then,
taking the $q$-th root on the both sides of $\eqref{f:3.30-ap}$
consequently yields the stated estimate $\eqref{f:3.6-b-ap}$.
This ends the arguments for $\eqref{pri:3.1-ap}$.

\textbf{Step 2.} Arguments for $\eqref{f:3.7-ap}$ and $\eqref{f:3.13-ap}$.
Concerning the equations (i) of $\eqref{pde:5.5}$,
It can be observed that the equation $(\text{i})$ of $\eqref{pde:5.5}$, as well as $v$, possesses a zero
extension for $t<0$. Therefore, one can understand $D_t^{1/2}\nabla v$
in the Fourier sense, i.e.,
\begin{equation*}
D_t^{1/2}\nabla v = \mathcal{F}^{-1}\big(\kappa(\cdot,\cdot)\mathcal{F}(\tilde{F})\big)
\quad \text{with~a~Mihlin~multiplier}\quad \kappa(\tau,\xi):=\frac{\sqrt{|\tau|}\xi}{i\tau+\xi\cdot\bar{a}\xi}.
\end{equation*}
Then, using Mihlin's theorem, it follows that
\begin{equation*}
 \|D_t^{1/2}\nabla v\|_{L^q(\mathbb{R}_{+}\times\mathbb{R}^d)}
 \leq \|D_t^{1/2}\nabla v\|_{L^q(\mathbb{R}^{d+1})}
 \lesssim_{\mu,d,q}\|\tilde{F}\|_{L^q(\mathbb{R}^{d+1})}
 \lesssim\|F\|_{L^q(\Omega_T)}.
\end{equation*}
By the same token, it further holds that
\begin{equation}\label{f:5.27-ap}
 \|D_t^{1/2}\nabla v\|_{L^q(\mathbb{R}_{+}\times\mathbb{R}^d)}
+\|\nabla^2 v\|_{L^q(\mathbb{R}_{+}\times\mathbb{R}^d)}
+\|\partial_t v\|_{L^q(\mathbb{R}_{+}\times\mathbb{R}^d)}
 \lesssim_{\mu,d,q}\|F\|_{L^q(\Omega_T)}.
\end{equation}
This, together with $\|v\|_{L^q(\Omega_T)}\lesssim_T\|\partial_tv\|_{L^q(\Omega_T)}$,
consequently leads to the stated estimate $\eqref{f:3.7-ap}$.
%
%

We now proceed to study the equations (ii) of $\eqref{pde:5.5}$.
On account of \cite[Theorem 3.4]{Fabes-Riviere-79} (or
\cite[Theorem 6.1]{Brown-89} in the case of Lipschitz domains with $q=2$,
where the symmetry condition $a=a^*$ is additionally required), one
can first derive that
\begin{equation}\label{f:5.25-ap}
\begin{aligned}
\|(\nabla w)^*\|_{L^q(\partial_{\shortparallel}\Omega_T)}
&\lesssim \|g-v\|_{H^{1,1/2}_{q}(\partial_{\shortparallel}\Omega_T)}\\
&\lesssim
\|g\|_{H^{1,1/2}_{q}(\partial_{\shortparallel}\Omega_T)}
+\|\nabla v\|_{L^{q}(\partial_{\shortparallel}\Omega_T)}
+\|D_t^{1/2} v\|_{L^{q}(\partial_{\shortparallel}\Omega_T)}
+ \|v\|_{L^{q}(\partial_{\shortparallel}\Omega_T)}.
\end{aligned}
\end{equation}

By the trace theorem with respect to the spatial variable, it is obtained that
\begin{equation}\label{f:5.26-ap}
\begin{aligned}
&\|\nabla v\|_{L^{q}(\partial_{\shortparallel}\Omega_T)}
+\|D_t^{1/2} v\|_{L^{q}(\partial_{\shortparallel}\Omega_T)}
+ \|v\|_{L^{q}(\partial_{\shortparallel}\Omega_T)}\\
&\lesssim  \Big\{\|v\|_{W^{2,1}_q(\Omega_T)}
+ \|\nabla D_t^{1/2} v\|_{L^q(\Omega_T)} + \|D_t^{1/2} v\|_{L^q(\Omega_T)}\Big\}\\
&\lesssim  \Big\{\|v\|_{W^{2,1}_q(\Omega_T)}
+ \|D_t^{1/2}\nabla v\|_{L^q(\mathbb{R}_{+}\times\mathbb{R}^{d})}
+ \|\partial_t v\|_{L^q(\mathbb{R}_{+}\times\mathbb{R}^{d})}^{\frac{1}{2}}
\|v\|_{L^q(\mathbb{R}_{+}\times\mathbb{R}^{d})}^{\frac{1}{2}}\Big\},
\end{aligned}
\end{equation}
where the interpolation inequality on time\footnote{
That is $\|D_t^{1/2} u\|_{L^q(I)}\lesssim \|\partial_t u\|_{L^q(I)}^{\frac{1}{2}}\| u\|_{L^q(I)}^{\frac{1}{2}}$
for the case $|I|=\infty$, and we refer the reader to
\cite[Corollary 1.27]{Leoni-23}.} is also employed for the second inequality.
Then, the following is claimed:
\begin{equation}\label{f:3.12-ap}
 \|v\|_{L^q(\mathbb{R}_{+}\times\mathbb{R}^{d})}
 \lesssim \|\tilde{F}\|_{L^{s}(\mathbb{R}_{+};L^r(\mathbb{R}^d))} \lesssim_{R_0,T}\|F\|_{L^q(\Omega_T)},
\end{equation}
where $r,s\in[1,q)$ satisfy $\frac{d}{2}\big(\frac{1}{r}-\frac{1}{q}\big)
+\big(\frac{1}{s}-\frac{1}{q}\big)=1$.
Thus, plugging $\eqref{f:3.12-ap}$, $\eqref{f:5.26-ap}$, $\eqref{f:5.27-ap}$, and $\eqref{f:3.7-ap}$
back into \eqref{f:5.25-ap} yields the stated estimate
$\eqref{f:3.13-ap}$.

\textbf{Step 3.} Arguments for $\eqref{f:3.6-a-ap}$.
Let $S_r:=\{x\in\Omega:\text{dist}(x,\partial\Omega)=r\}$
denote the level set of $\Omega$. For any $r\in[0,2c\varepsilon]$ and $t>0$, it follows from the trace theorem that
\begin{equation*}
\begin{aligned}
\int_{S_r}dS_r|\nabla v(\cdot,t)|^q
\lesssim\bigg\{\int_{\Sigma_r}|\nabla^2 v(\cdot,t)|^q
+ \int_{\Sigma_r}|\nabla v(\cdot,t)|^q \bigg\}
\lesssim \bigg\{\int_{\Omega}|\nabla^2 v(\cdot,t)|^q
+ \int_{\Omega}|\nabla v(\cdot,t)|^q \bigg\},
\end{aligned}
\end{equation*}
where the multiplicative constant is independent of $r$ and $t$. By the co-area formula, we further have
\begin{equation}\label{f:3.10-ap}
\begin{aligned}
\Big(\int_{(O_{2c\varepsilon})_{T}} |\nabla v|^q\Big)^{\frac{1}{q}}
&=
\bigg(\int_{0}^Tdt\int_0^{2c\varepsilon}dr\int_{S_r}dS_r|\nabla v(\cdot,t)|^q\bigg)^{\frac{1}{q}} \\
&\lesssim\varepsilon^{\frac{1}{q}}
\Big(\int_{\Omega_T}|\nabla^2 v|^q+|\nabla v|^q \Big)^{\frac{1}{q}}
\lesssim^{\eqref{f:3.7-ap}}
\varepsilon^{\frac{1}{q}}
\Big\{\|F\|_{L^q(\Omega_T)}+1\Big\}.
\end{aligned}
\end{equation}


By using the co-area formula again, it is also obtained that
\begin{equation}\label{f:3.11-ap}
\begin{aligned}
\Big(\int_{(O_{2c\varepsilon})_T}
|\nabla w|^q \Big)^{\frac{1}{q}}
\lesssim \varepsilon^{\frac{1}{q}}\big\|(\nabla w)^*\big\|_{L^q(\partial_{
\shortparallel}\Omega_T)}
\lesssim^{\eqref{f:3.13-ap}}\varepsilon^{1/q}
\Big\{\|F\|_{L^q(\Omega_T)}+1\Big\}.
\end{aligned}
\end{equation}
Combining the estimates $\eqref{f:3.10-ap}$ and $\eqref{f:3.11-ap}$ implies
the stated estimate $\eqref{f:3.6-a-ap}$.

\textbf{Step 4.} Show the estimate $\eqref{pri:3.3-ap}$.
We now turn to the so-called co-layer type estimate.
Since $\partial_t u_0 + \mathcal{L}_0(u_0) = F$ in $\Omega_T$, it is sufficient to prove the
estimate $\eqref{pri:3.3-ap}$ for the quantity
\begin{equation*}
\int_{(\Omega\setminus O_{c\varepsilon})_T}
|\nabla^2 u_0|^q.
\end{equation*}
In view of $u_0 = v+w$, the above integral could be bounded  by
\begin{equation}\label{f:3.19-ap}
\begin{aligned}
\int_{0}^{T^*}\int_{\tilde{\Omega}}
|\nabla^2 v|^q
&+ \int_{0}^{T}\int_{\Omega\setminus O_{c\varepsilon}}
|\nabla^2 w|^q \\
&\lesssim^{\eqref{f:3.7-ap}} \|F\|_{L^q(\Omega_T)}^q
+\underbrace{
\int_{0}^{T}\int_{O_{c_0}\setminus O_{c\varepsilon}}
|\nabla^2 w|^q}_{I_1}
+\underbrace{\int_{0}^{T}\int_{\Omega\setminus O_{c_0}}
|\nabla^2 w|^q}_{I_2},
\end{aligned}
\end{equation}
where $c_0=R_0/100$. From the interior estimate\footnote{The estimate
$\eqref{f:3.22-ap}$ follows from Caccioppoli's inequality and
Sobolev embedding theorem, where it is noted that $w$ possesses the zero extension for $t<0$. The technique for averaging from $L^2$ to $L^1$ is also standard, and we refer the reader to \cite[pp.173]{Fefferman-Stein72}
for a convexity argument.}:
\begin{equation}\label{f:3.22-ap}
  |\nabla^2 w(z)|
  \lesssim [\sigma(z)]^{-1}\dashint_{Q_{\sigma(z)}(z)}|\nabla w|,
  \qquad \sigma(z):=\text{dist}(x,\partial\Omega)
  ~\text{and}~z=(x,t)\in\Omega_{T},
\end{equation}
it follows that
\begin{equation}\label{f:3.20-ap}
I_2 \lesssim_{c_0}^{\eqref{f:3.22-ap}} \int_{0}^{T}\int_{\Omega\setminus O_{2c\varepsilon}}
|\nabla w|^q
\lesssim \int_{0}^{T}\int_{\Omega}
|\nabla (u_0-v)|^q \lesssim^{\eqref{f:3.7-ap}}  \Big\{\|F\|_{L^q(\Omega_T)}^q+1\Big\},
\end{equation}
where the facts that $u_0 = v+w$ and
$\|\nabla u_0\|_{L^q(\Omega_T)}\lesssim \|F\|_{L^q(\Omega_T)}+1$ are  employed in the second and third inequalities, respectively.
We proceed to estimate $I_1$,
\begin{equation}\label{f:3.21-ap}
\begin{aligned}
I_1
&\lesssim^{\eqref{f:3.22-ap}}
\int_0^Tdt\int_{\partial\Omega}dS|(\nabla w)^*(\cdot,t)|^q
\int_{c\varepsilon}^{c_0}\frac{dr}{r^q}
\lesssim^{\eqref{f:3.13-ap}} (c\varepsilon)^{1-q}\Big\{
\|F\|_{L^q(\Omega_T)}^q
+ 1\Big\}.
\end{aligned}
\end{equation}

Combining the estimates
$\eqref{f:3.19-ap}$, $\eqref{f:3.20-ap}$, and $\eqref{f:3.21-ap}$ yields
\begin{equation}\label{f:3.23-ap}
 \bigg(\int_{(\Omega\setminus O_{c\varepsilon})_T}|\nabla^2 u_0|^q\bigg)^{\frac{1}{q}}
 \lesssim \varepsilon^{\frac{1}{q}-1}\Big\{\|F\|_{L^q(\Omega_T)}
+ 1 \Big\},
\end{equation}
which implies the stated estimate $\eqref{pri:3.3-ap}$ by noting that
$c\in[1,8]$ assumed in the theorem.

\textbf{Step 5.} Arguments for $\eqref{f:3.12-ap}$. By the semigroup representation $v(\cdot,t) = \int_{0}^{t}ds e^{(s-t)\mathcal{L}_0}\tilde{F}(\cdot,s)$,
it follows from the semigroup estimates that
\begin{equation}\label{f:8.4}
 \|v(\cdot,t)\|_{L^q(\mathbb{R}^d)}
 \lesssim \int_{0}^{t}ds\|e^{(s-t)\mathcal{L}_0}\tilde{F}(\cdot,s)\|_{L^q(\mathbb{R}^d)}
 \lesssim \int_{0}^{t}ds(t-s)^{-\sigma}\|\tilde{F}(\cdot,s)\|_{L^r(\mathbb{R}^d)},
\end{equation}
where $\sigma:=\frac{d}{2}\big(\frac{1}{r}-\frac{1}{q}\big)>0$.
Then, by applying the Hardy-Littlewood-Sobolev inequality to the right-hand side of $\eqref{f:8.4}$, it holds that
\begin{equation}\label{f:8.5}
 \|v\|_{L^q(\mathbb{R}_{+}\times\mathbb{R}^d)}^q
 \lesssim \int_{\mathbb{R}_{+}}dt\Big(\int_{0}^{t}ds
 (t-s)^{-\sigma}\|\tilde{F}(\cdot,s)\|_{L^r(\mathbb{R}^d)}\Big)^q
 \lesssim \|\tilde{F}\|_{L^s(\mathbb{R}_+;L^r(\mathbb{R}^d))}^q,
\end{equation}
where $1 = \frac{1}{s}-\frac{1}{q}+\sigma$, and the desired estimate
$\eqref{f:3.12-ap}$ is obtained. It is noted that the validity
of the estimate $\eqref{f:8.5}$ requires that $\sigma < 1$, and it leads to  requiring $s\in[1,q)$.

\textbf{Step 6.}
Arguments for the second term in the left-hand side of $\eqref{pri:3.6-ap}$. We start from the decomposition of the integral region, i.e,
\begin{equation*}
\begin{aligned}
\int_{\Sigma_{c\varepsilon}^T}
|\nabla u_0|^2 \delta^{-1}
&=\bigg\{\int_{(c\varepsilon)^2}^{c_0^2}\int_{\Sigma_{c\varepsilon}}
+\int_{T-c_0^2}^{T-(c\varepsilon)^2}\int_{\Sigma_{c\varepsilon}}
+\int_{c_0^2}^{T-c_0^2}\int_{\Sigma_{c\varepsilon}\setminus\Sigma_{c_0}}
+\int_{c_0^2}^{T-c_0^2}\int_{\Sigma_{c_0}}\bigg\}|\nabla u_0|^2 \delta^{-1}\\
&=:J_1 + J_2 + J_3 +J_4.
\end{aligned}
\end{equation*}

Obviously, the easiest term is $J_4$ since $\delta\geq c_0$ in $\Sigma_{c_0}\times (c_0^2, T-c_0^2)$, and the relatively easy one is
$J_3$ since we can take $\delta(z)=\sigma(x)=\text{dist}(x,\partial\Omega)$
for $z\in (\Sigma_{c\varepsilon}\setminus\Sigma_{c_0})\times
(c_0^2, T-c_0^2)$. Therefore, we have
\begin{equation}\label{f:3.51-ap}
J_4 \leq c_0^{-1}\int_{0}^{T}\int_{\Omega}
|\nabla u_0|^2
\lesssim \|F\|_{L^2(\Omega_T)} + 1,
\end{equation}
and an analogous argument as given for $\eqref{f:3.6-a-ap}$ in Step 3 leads to
\begin{equation}\label{f:3.52-ap}
\begin{aligned}
J_3
&= \int_{c^2}^{T-c^2}dt\int_{c\varepsilon}^{c_0}\frac{dr}{r}\int_{S_r}
dS_r|\nabla u_0(\cdot,t)|^2\\
&\lesssim \ln(c_0/\varepsilon)\Bigg\{
\int_{0}^T\int_{O_{c_0}}\big(|\nabla^2 v|^2
+|\nabla v|^2\big)dxdt
+ \int_{0}^T\int_{\partial\Omega}|(\nabla \bar{w})^*|^2dSdt\Bigg\}\\
&\lesssim^{\eqref{f:3.7-ap},\eqref{f:3.13-ap}} \ln(R_0/\varepsilon)
\Big\{\|F\|_{L^2(\Omega_T)} + 1\Big\}.
\end{aligned}
\end{equation}

We now turn to $J_1$ and $J_2$. In fact,
by a change of the time variable, the study on $J_2$ can be reduced to investigating $J_1$. To carry out the analysis, $J_1$ is split into two parts:
\begin{equation*}
\underbrace{\int_{4\varepsilon^2}^{c_0^2}dt
\int_{\Sigma_{2\varepsilon}\cap\{\delta = t^{\frac{1}{2}}\}}|\nabla u_0(\cdot,t)|^2
t^{-\frac{1}{2}}}_{J_{11}}
\quad\text{and}\quad
\underbrace{\int_{4\varepsilon^2}^{c_0^2}dt
\int_{\Sigma_{2\varepsilon}\cap\{\delta = \sigma\}}
dx|\nabla u_0(x,t)|^2 [\text{dist}(x,\partial\Omega)]^{-1}}_{J_{12}}.
\end{equation*}
It is clear that the same arguments as those used for $J_3$ can be adopted to derive
\begin{equation}\label{f:3.53-ap}
 J_{12} \lesssim \ln(R_0/\varepsilon)\Big\{\|F\|_{L^2(\Omega_T)} + 1\Big\}.
\end{equation}
For $J_{11}$, set $t_k = 2^k\varepsilon$ and
$N_0 = \log_2(c_0/\varepsilon)$, noting that $2^{N_0}\varepsilon = c_0$. Hence, it is obtained that
\begin{equation}\label{f:3.29-ap}
 J_{11} \leq \sum_{k=1}^{N_0}\int_{t_k^2}^{t_{k+1}^2}
\int_{\Sigma_{t_k}\cap\{\delta=s^{\frac{1}{2}}\}}
|\nabla u_0|^2 s^{-\frac{1}{2}}dxds :=\sum_{k=1}^{N_0} \mathcal{K}_k
\lesssim \log_2(c_0/\varepsilon)\Big\{\|F\|_{L^2(\Omega_T)}+1\Big\},
\end{equation}
provided that there uniformly holds $
\mathcal{K}_k\lesssim \big\{\|F\|_{L^2(\Omega_T)}+1\big\}$
with respect to $k$. To see this, we start from a basic inequality that
\begin{equation}\label{f:3.24-ap}
\begin{aligned}
& \max_{0\leq t\leq T}\|\nabla u_0\|_{L^2(\Sigma_{t_k})}
 \lesssim \Big\{
 \int_{(\Omega\setminus O_{t_k})_T}\big(|\nabla^2 u_0|^2
 +|\nabla u_0|^2 + |u_0|^2 +|\partial_t u_0|^2\big)\Big\}^{\frac{1}{2}}\\
&\lesssim \Big(\int_{(\Omega\setminus O_{t_k})_T}|\nabla^2 u_0|^2\Big)^{\frac{1}{2}} +
\Big\{\|F\|_{L^2(\Omega_T)} + 1\Big\}
\lesssim^{\eqref{f:3.23-ap}}
\frac{1}{\sqrt{t_k}}\Big\{\|F\|_{L^2(\Omega_T)} + 1\Big\},
\end{aligned}
\end{equation}
where the first inequality follows from \cite[Theorem 4, pp.306]{LCE}\footnote{One can replace the extension theorem in the proof with Stein's extension theorem \cite[Theorem 5, pp.181]{Stein-70}
to reduce the requirements for boundary regularity.}, and
the energy estimate is employed in the second.
Hence, it follows that
\begin{equation*}
\begin{aligned}
 \mathcal{K}_{k}
& = \int_{t_k^2}^{t_{k+1}^2}
 \int_{\Sigma_{t_k}\cap\{\delta=s^{\frac{1}{2}}\}}|\nabla u_0|^2 s^{-\frac{1}{2}} dx ds \\
&\leq \max_{0\leq t\leq T}\|\nabla u_0\|_{L^2(\Sigma_{t_k})}^2
\int_{t_k^2}^{t_{k+1}^2}\frac{ds}{\sqrt{s}}
\lesssim^{\eqref{f:3.24-ap}}
\Big\{\|F\|_{L^2(\Omega_T)}^2 + 1\Big\},
\end{aligned}
\end{equation*}
which completes the argument for $\eqref{f:3.29-ap}$. Now,
collecting the estimates $\eqref{f:3.51-ap}$, $\eqref{f:3.52-ap}$,
$\eqref{f:3.53-ap}$, and $\eqref{f:3.29-ap}$ yields the
stated estimate on the second term in the left-hand side of $\eqref{pri:3.6-ap}$. This completes the whole proof.
\end{proof}

\end{document}